\def\Dbar{\overline{\partial}}
\def\cross{{\rm{cr}}}

\def\dist{\rm{dist}}

\def\dist{{\rm{dist}}}
\def\diam{{\rm{diam}}}
\def\supp{{\rm{supp}}}
\def\Mod{{\rm{Mod}}}

\def\lc2{{\rm{cap}}}

\def\z{{\bold z}}
\def\w{{\bold w}}

\def\v{{\bold v}}

\def\boldalpha{{\bold \alpha}}

\def\Bbb{\mathbb}
\def\cal{\mathcal}

\def\disk{{\Bbb D}}
\def\ball{{\Bbb B}}
\def\circle{{\Bbb T}}

\def\reals{{\Bbb R}}

\def\real{{\Bbb R}}
\def\plane{{\Bbb R^2}}
\def\complexes{{\Bbb C}}
\def\sphere {{\Bbb S}}

\def\uhp{{\Bbb H}}
\def\uhs{{\Bbb H}^3}

\def\diag{\rm{diag}}
\def\FFT{{\rm{FFT}}}
\def\IFFT{{\rm{IFFT}}}

\def\OmegaFB{\Omega^{\rm fb}}

\def\OmegaRC{\Omega^{\rm rc}}

\documentclass[12pt]{amsart}  

\usepackage{graphicx,amssymb,amsthm,verbatim,amsfonts,amscd}

\newtheorem{thm}{Theorem}
\newtheorem{cor}[thm]{Corollary}
\newtheorem{lemma}[thm]{Lemma}

\newcounter{conj}
  
\theoremstyle{definition}
   
\newcounter{ques}

\begin{document}

\baselineskip=18pt

\title{ Conformal mapping in  linear time}

\author {Christopher J. Bishop}
\address{C.J. Bishop\\
         Mathematics Department\\
         SUNY at Stony Brook \\
         Stony Brook, NY 11794-3651}
\email {bishop@math.sunysb.edu}
\thanks{The  author is partially supported by NSF Grant DMS 04-05578.}

\date{August 13, 2009}

\maketitle

\begin{abstract}
Given any $\epsilon >0$ and any  planar region $\Omega$
 bounded by a simple $n$-gon $P$ we construct  a $(1+\epsilon)$-quasiconformal 
map between $\Omega$ and the unit disk in time  $C(\epsilon) n$.
 One can take $C(\epsilon) = C + C \log \frac 1 \epsilon \log \log
\frac 1 \epsilon$.


 \end{abstract}

\clearpage
\setcounter{page}{1}
\renewcommand{\thepage}{\arabic{page}}

\section{Introduction} \label{new-intro}

If $\Omega$ is a proper, simply connected plane domain, then by the 
Riemann mapping theorem  there is  a conformal map
$f: \disk \to \Omega$, but for most domains there is no
simple, explicit formula. In this paper we will 
show that there is ``almost'' such a formula in the sense
that there is a linear time algorithm for computing the 
conformal map with estimates on time and accuracy that are 
independent of the geometry of the particular domain.
Thus the computational complexity of conformal mapping is
linear in the following sense.

\begin{thm} \label{thmQC}
Given a simply connected domain $\Omega$ bounded by 
an $n$-gon we can compute  the conformal 
map $f: \disk \to \Omega$ to within quasiconformal 
error $\epsilon$ in time $O(n \cdot
 p \log p)$ where $p =O(\log \frac 1 \epsilon)$.
\end{thm}

The phrases ``can compute'' and ``quasiconformal error'' 
require some explanation in order
to make this a precise mathematical statement.
A unit of work consists of  an infinite precision arithmetic operation or 
an evaluation of $\exp$ or $\log$.
We will 
cover the unit disk by   $O(n)$ regions (disks and annuli) and in each
region approximate the conformal map using a $p$-term power or 
Laurent series and some elementary functions.
Combining these using a partition of unity will give  a
 $(1+\epsilon)$-quasiconformal map from $\disk$ to $\Omega$.
Our series  converge  geometrically fast on 
the associated regions,  and so each series has 
$p \sim \log \frac 1 \epsilon$ terms in general.
 The fastest  known methods for 
multiplication, division, composition or inversion of 
power series use the fast Fourier transform, and the time to 
perform an FFT on a $p$-term power series is 
${\rm{FFT}}(p) = O(p \log p)$,  so Theorem 
\ref{thmQC} says we only need  $O(1)$ such operations per
vertex.

The Schwarz-Christoffel formula (see Appendix  \ref{background}) provides 
a formula for the conformal map onto a polygon, but involves unknown 
parameters (the conformal preimages of the vertices).  Thus, it is 
not really a solution of the mapping problem, but simply reduces it 
to finding the $n$ conformal prevertices.
Suppose $\Omega$ is bounded by a simple $n$-gon
with vertices $\v=\{v_1,\dots,v_n\}$, let $f: \disk \to \Omega$ 
be  conformal and 
let $\z = f^{-1}(\v)$ be the conformal prevertices. 
A more concrete version of Theorem \ref{thmQC} is:

\begin{thm}\label{main}
Given any $\epsilon>0$ there is a $C=C(\epsilon) < \infty$ so that 
if $\Omega$ is bounded by a
simply polygon $P$ with $n$ vertices we can find points 
$ \w = \{w_1, \dots, w_n\}  \subset \circle$ so that 
\begin{enumerate}
\item All $n$ points in $\w$ can be computed in at most $ C n$ 
     steps.  
\item $d_{QC}(\w,\z) < \epsilon $ where $\z$ are the true conformal 
   prevertices.
\end{enumerate}
Here 
$d_{QC}(\w,\z) = \inf\{ \log K:  \exists \text{  
$K$-\rm{quasiconformal} } h: \disk \to \disk 
   \text{ such that }  h(\z) = \w. \}$ The constant $C(\epsilon)$ 
may be taken to be 
$C + C\log  \frac 1\epsilon  \cdot \log \log \frac 1 \epsilon$ 
 where $C$ is independent 
of $\epsilon$ or $n$.
\end{thm}

Note that $d_{QC}(\w,\z)=0$ iff the $n$-tuples are M{\"o}bius 
images of each other.  It is not hard to see that this 
 happens iff the corresponding 
polygons are linear images of each other, and so this is 
a natural metric for the problem. 
Quasiconformal approximation 
implies uniform approximation but is stronger; not only are
the points of $\w$ within $O(\epsilon)$ of the corresponding 
points of $\z$, but  the relative arrangement 
of $\w$ approximates the corresponding arrangement of $\z$ equally well
at every scale (see Lemma \ref{QC-near-Id}).



We will  define a quadratically convergent iteration on 
$n$-tuples in $\circle$  and 
provide a starting point from which it is guaranteed to converge
with an estimate independent of the domain.
Although there are various details to check, each of  basic ideas
involved 
is fairly easy to explain and involves a  geometric 
construction.  We will discuss  these briefly here, 
leaving the details and difficult cases for the rest of the paper.

The  first 
idea is to  consider the so called ``iota-map'', $\iota: P \to \circle$ 
 to obtain an $n$-tuple $\w = \iota(\v)  \subset \circle $ 
that is only a bounded $d_{QC}$-distance $K$ from the true prevertices
(it is known from \cite{Bishop-ExpSullivan} that we can take 
$K \leq 7.82$). The definition of this map and the proof
that it has the desired approximation properties are  motivated by results 
 from  hyperbolic 3-dimensional geometry, but we can  
give a simple, geometric description  in the plane.
We approximate our polygon by a finite collection 
of medial axis disks (these are subdisks of the domain whose
boundary hits the boundary of the domain in at least two points).
 The union of these disks, $\Omega$,  
 can be written as a union 
of a single $D$  disk and a collection of disjoint crescents.
See Figure \ref{cd2lines-intro}.
Each crescent is foliated by circular arcs orthogonal to 
its two boundary arcs.  Following leaves of this foliation 
gives the desired  map $\iota: \partial \Omega \to \partial D$.
  The initial approximation 
by a union of disks is unnecessary,
but convenient for various reasons (the $\iota$ map for a 
polygon can be computed directly, using the medial axis 
of the polygon e.g., \cite{Bishop-fast}).
The construction of $\iota$ in linear time
depends on the fact that 
 the medial axis of a $n$-gon can be computed in linear time, a 
result of Chin, Snoeyink and Wang \cite{CSW-99}.

\begin{figure}[htbp]
\centerline{
 \includegraphics[height=2in]{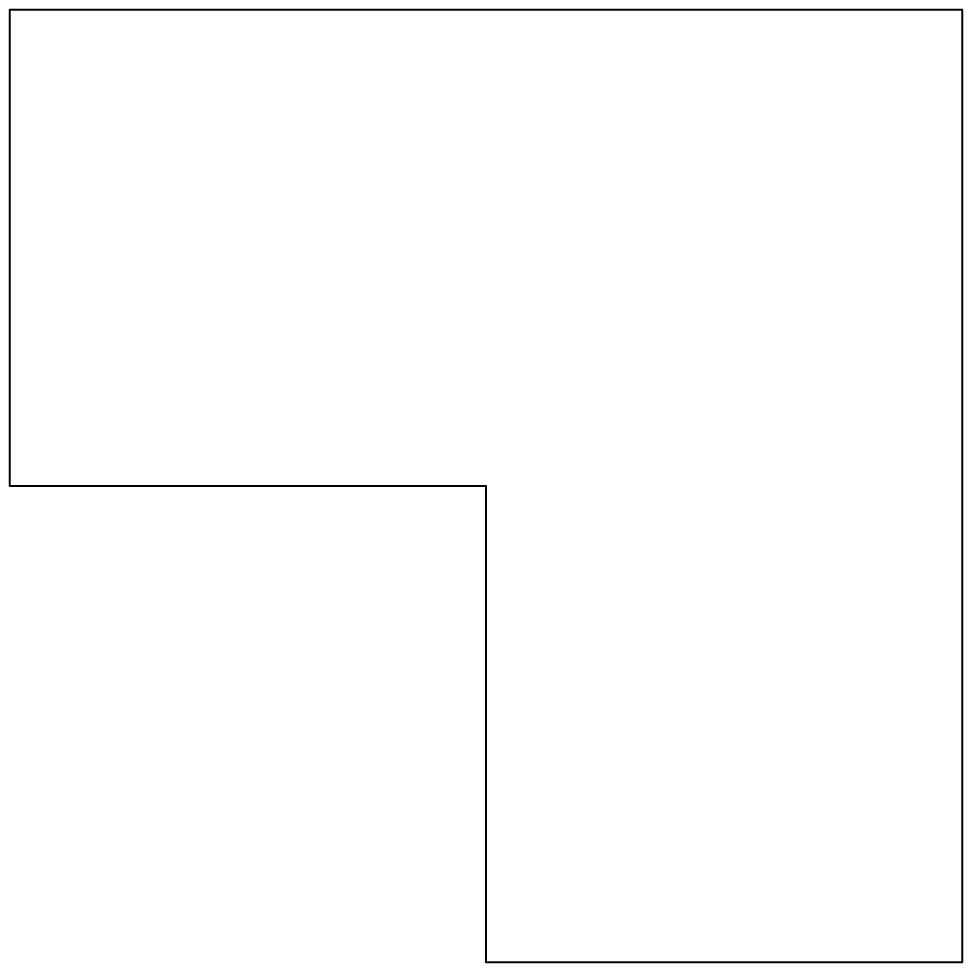}
$\hphantom{xxx}$
 \includegraphics[height=2in]{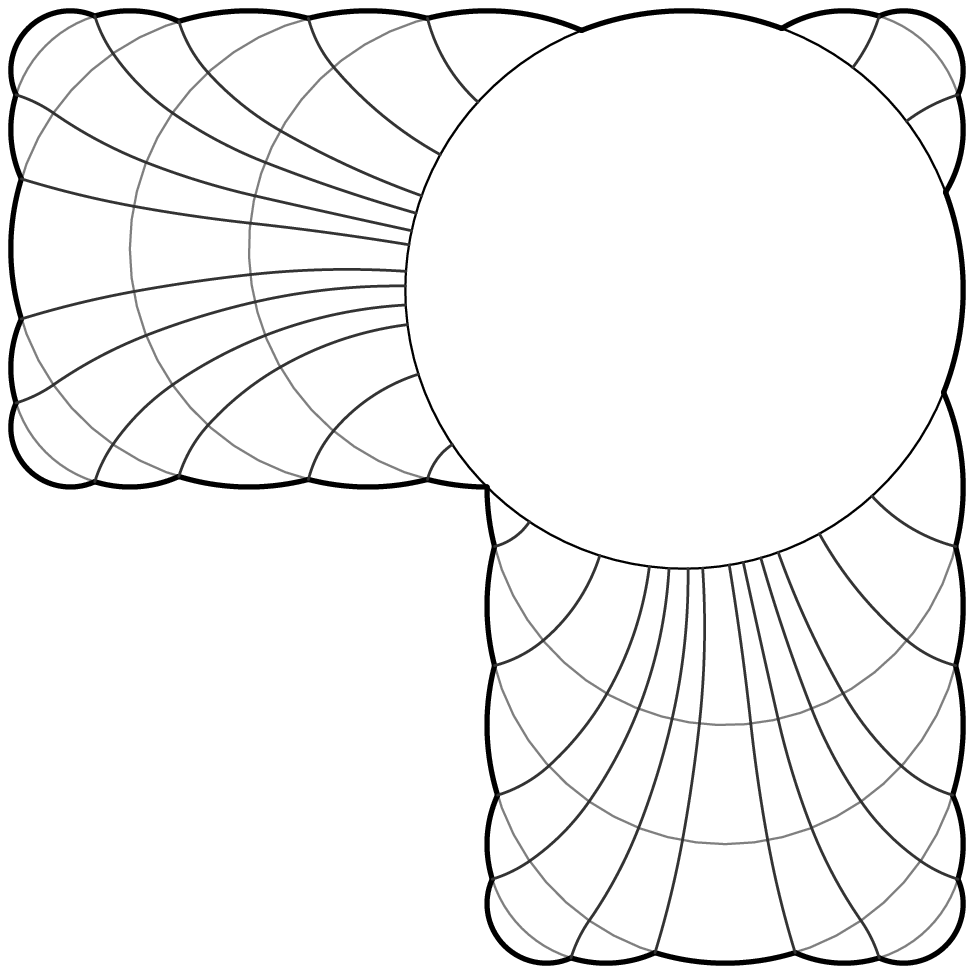}
}
\caption { \label{cd2lines-intro}
          An example  where we have approximated a domain  by a union 
         of disks; written the new domain $\Omega$ as a disjoint  union of 
          one disk  $D$ and several crescents; and used circular arcs 
         orthogonal to the crescents to define a flow from 
         $\partial \Omega$ to $\partial D$. The resulting map 
         is close to the Riemann map with estimates independent 
         of the domain.}
\end{figure}


The next idea is to decompose polygons into pieces, again 
following a motivation from hyperbolic geometry.
A standard technique in the theory of hyperbolic manifolds 
is to partition the manifold into its thick and thin parts
(based on the length of the shortest non-trivial loop through each  point). 
See Figure \ref{surfaces}.
 Thin parts
often cause technical difficulties, but this is  partially 
compensated for by the fact that  there are only a few possible 
types of thin parts and each has a well understood shape.
  Thus  we can 
think of the manifold as consisting of some 
``interesting'' thick parts   attached to some
annoying, but explicitly described, thin parts.
The manifold is considered especially nice if it is 
thick, i.e.,  no thin parts occur.

\begin{figure}[htbp]
\centerline{
 \includegraphics[height=1.5in]{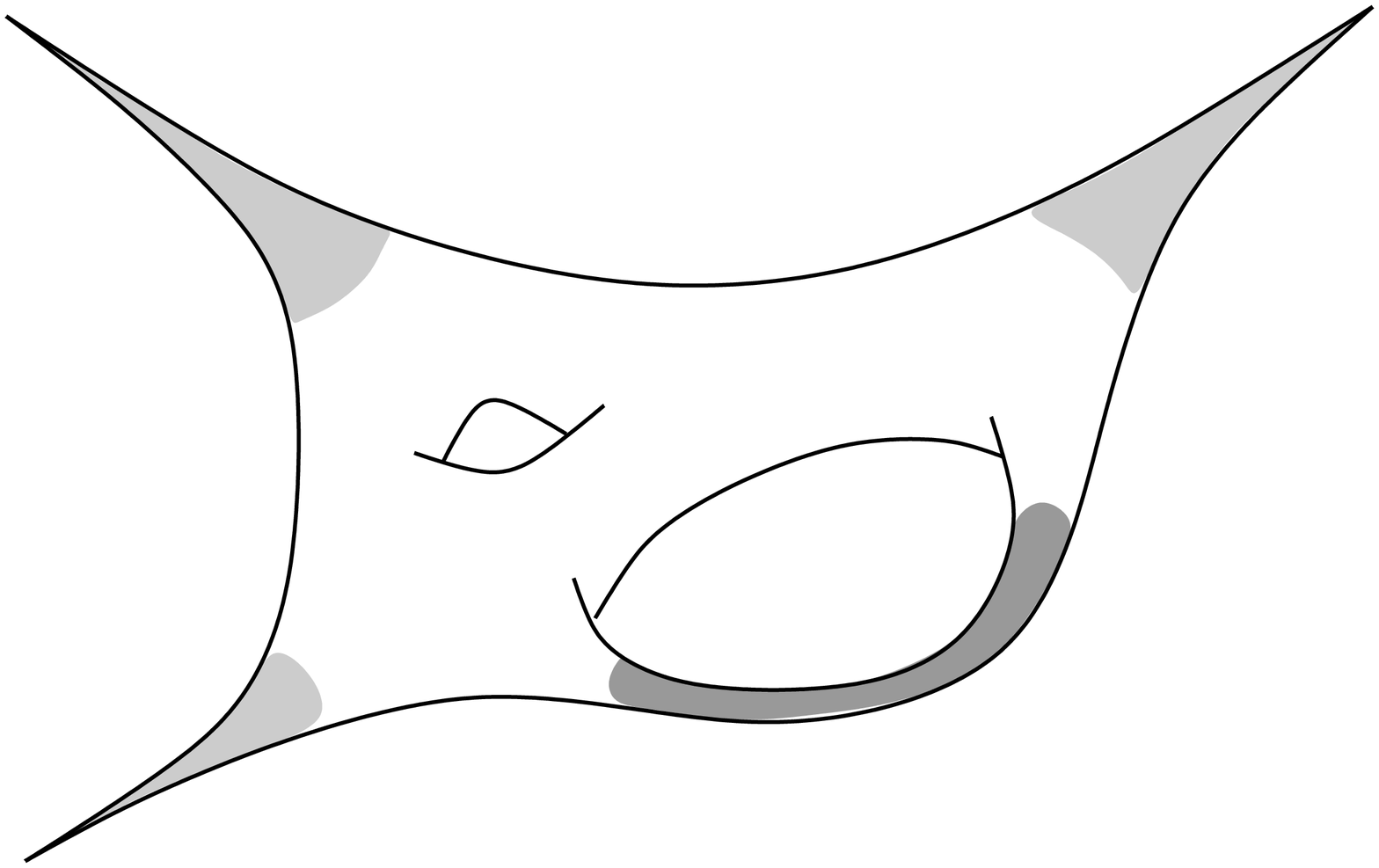}
$\hphantom{xxxx}$
 \includegraphics[height=1.5in]{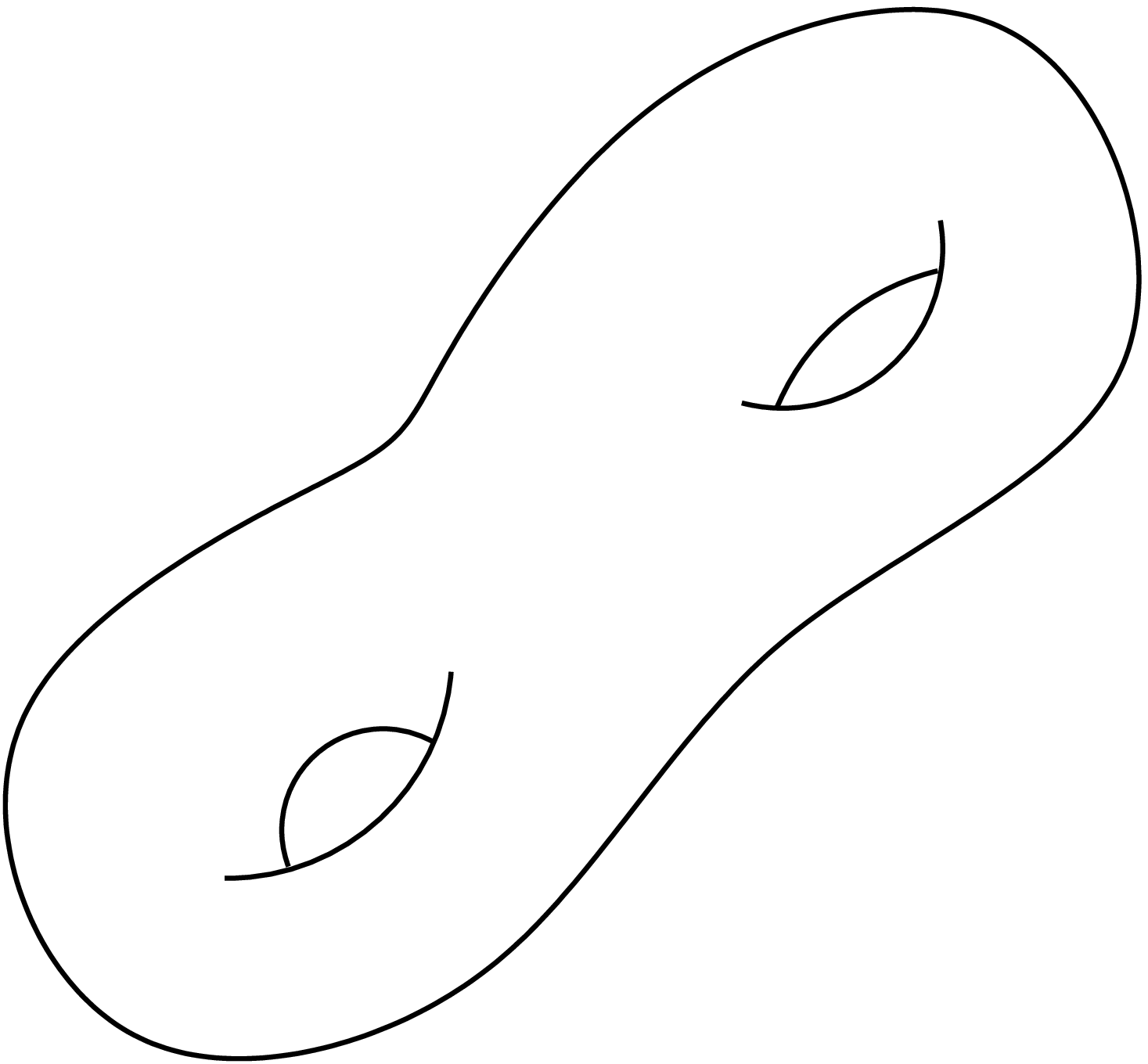}
}
\caption { \label{surfaces}
             On the left is a surface with one hyperbolic thin part 
  (darker) and three parabolic thin parts (lighter). On the right
  is a ``thick'' surface with no thin parts.
   }
\end{figure}

\begin{figure}[htbp]
\centerline{
 \includegraphics[height=1.75in]{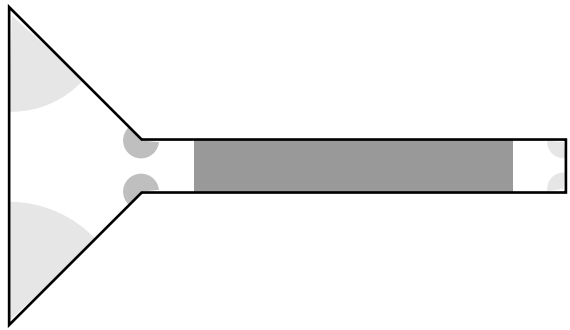}
}
\caption { \label{thin-parts}
            A  polygon with one hyperbolic thin part (darker) and six 
  parabolic thin parts, which we further divide into two groups 
  corresponding to interior vertex angles $< \pi$ and $> \pi$. 
}
\end{figure}

We will describe an analogous decomposition 
of a polygon into thick and thin parts. The thin 
parts occur when the extremal length between two edges 
is very small (roughly this means the Euclidean distance 
inside the domain
between the edges is small compared to their Euclidean 
diameters). This occurs whenever the edges are adjacent, but we 
shall be mostly interested in thin parts corresponding to 
non-adjacent edges and we denote the two cases as parabolic 
and hyperbolic respectively, in analogy to the thin parts 
of a Riemann surface (in that case, parabolic thin parts are
non-compact and have
one boundary component attached to the thick part of the surface;
hyperbolic thin parts are compact and have two boundary 
components, both attaching to the thick part of the surface).
See Figure \ref{thin-parts}.
The parabolic thin parts look like sectors,  and 
the hyperbolic thin parts look like generalized 
quadrilaterals (with two sides 
on the boundary of the given polygon).
We say the polygon is thick if no hyperbolic 
thin parts occur.
See Figure \ref{thin-intro} for 
various ways hyperbolic thin parts can arise.

\begin{figure}[htbp]
\centerline{
 \includegraphics[height=1.25in]{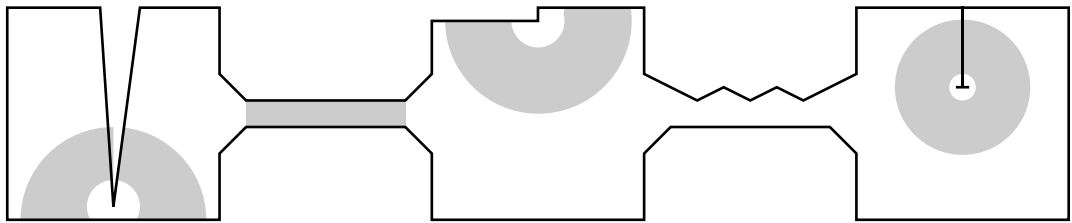}
}
\caption { \label{thin-intro}
            The five hyperbolic thin parts of this polygon are shaded gray.
            The channel on the right is not thin because there
             are many vertices lining one side of it.
            The complementary 
             white regions are ``thick''; one of our  strategies 
             is to compute mappings 
             onto thick domains and ``glue'' them together across the thin 
             connecting regions.  }
\end{figure}

  As with manifolds,  the thin 
parts of polygons  cause technical difficulties.
 However, our thin parts 
can only have a small number of simple shapes
and the  conformal 
maps from the disk into a thin part can be well approximated 
by explicit formulas. Thus they are
``well understood''.  Indeed, much of the algorithm 
described in this paper will only be applied to the remaining 
thick parts, making them the ``interesting'' part of the 
polygon. Thus, as with hyperbolic manifolds, polygons 
will be divided into interesting thick parts, attached 
to annoying, but well understood, thin parts.


The next idea concerns how to represent a  
map onto $\Omega$.  A conformal
map onto  a polygon has a convergent power series on $\disk$,
but since $f'$ is discontinuous at the prevertices,  
it converges slowly and the number of terms needed
for a given accuracy  depends on the geometry of 
the image domain.  For convenience we will replace the 
disk by the upper half-plane $\uhp$, and we will  represent a map
$f : \uhp \to \Omega$ by breaking $\uhp$ into $O(n)$
simple pieces and using a $p$-term power series or 
Laurent expansion on each 
piece that represents $f$ with  error $\leq 2^{-p}$,
 independent of the geometry of $\Omega$.
The series are combined using 
a partition of unity to give a single quasiconformal map whose dilatation can 
be computed and corrected for to give an improved approximation.
 The decomposition ${\cal W}$  of $\uhp$ is accomplished by 
taking the hyperbolic convex hull of  the point set $S$ (our 
current prevertex approximation) and covering it by $O(n)$ Whitney boxes, 
Carleson squares  and 
regions we call arches  
 and then dividing the remaining regions, which all lie outside 
the convex hull, into  $O(n)$ Carleson 
boxes.  See Figure \ref{decom-intro}.

\begin{figure}[htbp]
\centerline{
 \includegraphics[height=2in]{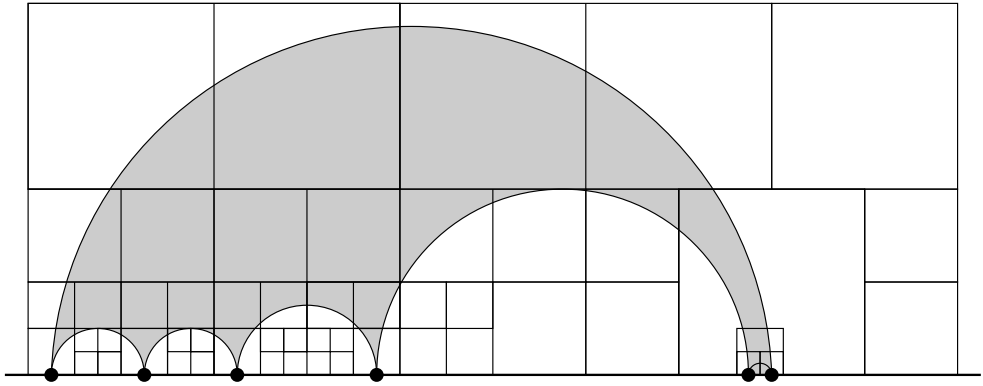}
}
\caption { \label{decom-intro}
Decompose the plane by first covering the convex hull of the 
prevertices.  This picture 
has one arch on the right hand side. Arches correspond to clusters 
of two or more prevertices which are isolated in a precise way.
}
\end{figure}

More precisely, 
an  $\epsilon$-representation of a polygon is a 
triple $(S,{\cal W}, {\cal F})$, where $S$ is a $n$-tuple 
in $\reals$ (the prevertices), ${\cal W}$ is  a 
decomposition of $\uhp$ into $O(n)$ simple pieces and ${\cal F}$ is a
collection of functions, one for each piece of our decomposition.
These functions consist of
 $p = O(|\log \epsilon|)$ terms of a  series expansion 
on each piece, and a choice of a certain elementary function for each 
piece, which is the identity or a power function in most cases.
Moreover, we require that functions for adjacent pieces agree 
to within $\epsilon$ (in a certain metric) along the common 
boundary. 

We will  prove a Newton type iteration
for improving  $\epsilon$-representations.
We will show that there is an absolute constant
 $\epsilon_0 >0 $ (independent of $n$ and $\Omega$) so that 
if $ \epsilon < \epsilon_0$, then  we can quickly
 improve a $\epsilon$-representation 
to a $\epsilon^2$-representation. 
  Thus starting with a $n$-tuple at distance
$\epsilon_0$ from the true answer, it only takes $O(\log \log  \frac 1\epsilon )$ 
iterations to reach 
accuracy $\epsilon$. 
The main problem is to estimate the time needed to perform each 
iteration.

Combining the functions  in ${\cal F}$ with a piecewise
polynomial partition of unity gives 
a $(1+\epsilon)$-quasiconformal map from $F:\uhp \to \Omega$. 
Let  $\mu = \overline {\partial }F/\partial F$ be the Beltrami 
dilatation of $F$.  Then $\| \mu\|_\infty = O(\epsilon)$
and   $\mu$  can be explicitly 
computed from the series expansions in ${\cal F}$ and the partition of 
unity. If we could solve a Beltrami equation to find a 
mapping $H$ of the upper half-plane to itself so
that $\mu_H = \mu_F$,  then $F \circ H^{-1}$ would be the 
desired  conformal map.
We can't solve this equation exactly in finite time, but we can solve
$\mu_H = \mu_F + O(\|\mu_F\|_\infty^2)$ in linear time using 
the fast multipole method of Greengard and Rokhlin.
Thus $F \circ H^{-1}$ will be $(1+O(\epsilon^2))$-quasiconformal, 
and this  is the improved representation.
Each iteration consists of approximately solving an
equation  $\Dbar H = \mu$ by evaluating
  $ p =O(\log \frac 1 \epsilon)$ terms of 
the power series of a  Beurling transform of $\mu$  on  $n$ disks.
 Using a fast multipole method and fast 
manipulation of power series, we can do each 
iteration in time $O(np \log p) = 
O(n\log  \frac 1\epsilon  \cdot \log \log \frac 1 \epsilon)$.
Moreover, since   $p = |\log \epsilon|$ 
increases geometrically with each iteration, the total work is 
dominated by the final iteration, which gives the desired estimate.

Another basic idea of the paper 
  deals with how to improve  our initial $n$-tuple (provided by 
the $\iota$ map) that is at 
most distance $K$ 
from the correct answer to an $n$-tuple that is within
the distance $\epsilon_0$ required by the  Newton type iteration.
  This 
is accomplished by connecting  our domain to the unit disk by a chain 
of $N+1 = O(1/\epsilon_0) $ 
regions $\disk = \Omega_0, \dots, \Omega_N = \Omega$.
  As before, it is 
convenient to work with a domain that is a finite 
union of disks (such domains are also called ``finitely 
bent'' for reasons that will be clear when we discuss the 
dome of a domain later).

In this case there is a ``normal crescent'' decomposition of $\Omega$.
If $\Omega = \cup D_k$ is a finite union of disks and 
$\partial D_1 \cap \partial D_2 \cap  \partial \Omega \ne \emptyset$, 
then the corresponding crescent is the subregion of $\Omega$ 
bounded by circular arcs perpendicular to $\partial D_1$ and 
$\partial D_2$ at their  two intersection points. 
Removing  every such 
crescent  from a finitely bent domain 
leaves a collection of ``gaps''. See 
 Figure \ref{gap-cres-dome}. This decomposition 
has a natural interpretation in terms of 
3-dimensional hyperbolic geometry.
Each planar domain $\Omega$ is associated to a surface $S$
(called the dome of $\Omega$) in 
the upper have space  that is the upper envelope of all hemispheres
with base disk in $\Omega$. If $\Omega$ is a finite union of disks then 
the dome is a finite union of geodesic faces that meet along hyperbolic 
geodesics called the bending lines. There is a map $R: \Omega \to S$ 
(the hyperbolic  nearest point projection onto $S$) and the gaps in $\Omega$ 
are simply the points that map to faces of $S$ and the crescents are 
points that map to bending lines. 

\begin{figure}[htbp]
\centerline{
\includegraphics[height=1.5in]{poly8.ps}
$\hphantom{xxxxxxxxx}$
 \includegraphics[height=1.5in]{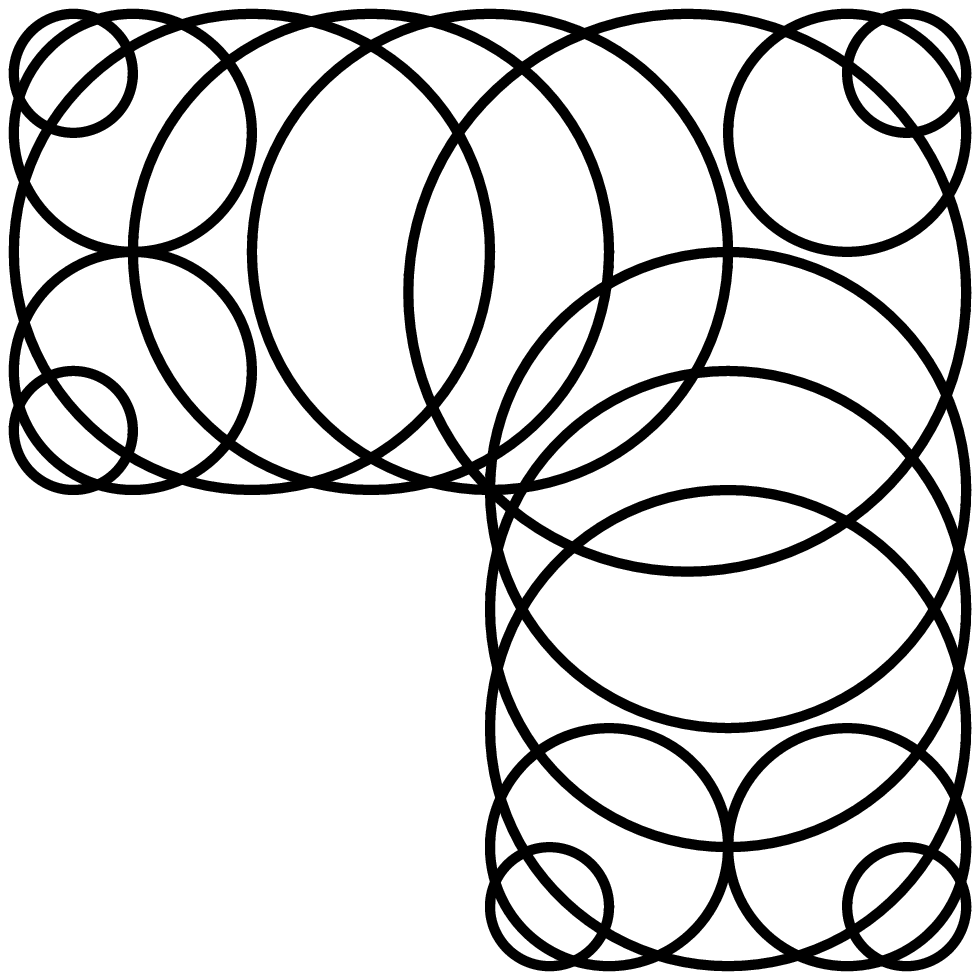}
 }
\vskip.2in
 \centerline{
 \includegraphics[height=1.5in]{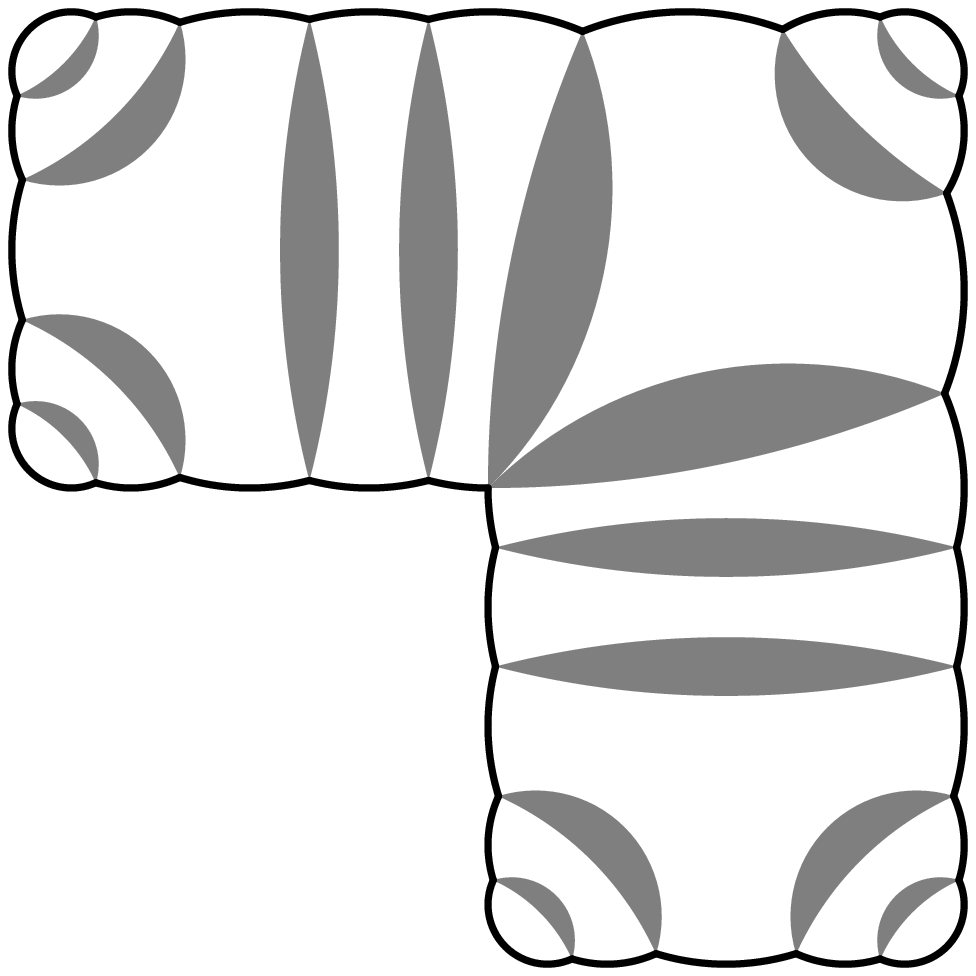}
$\hphantom{xxx}$
 \includegraphics[height=1.5in]{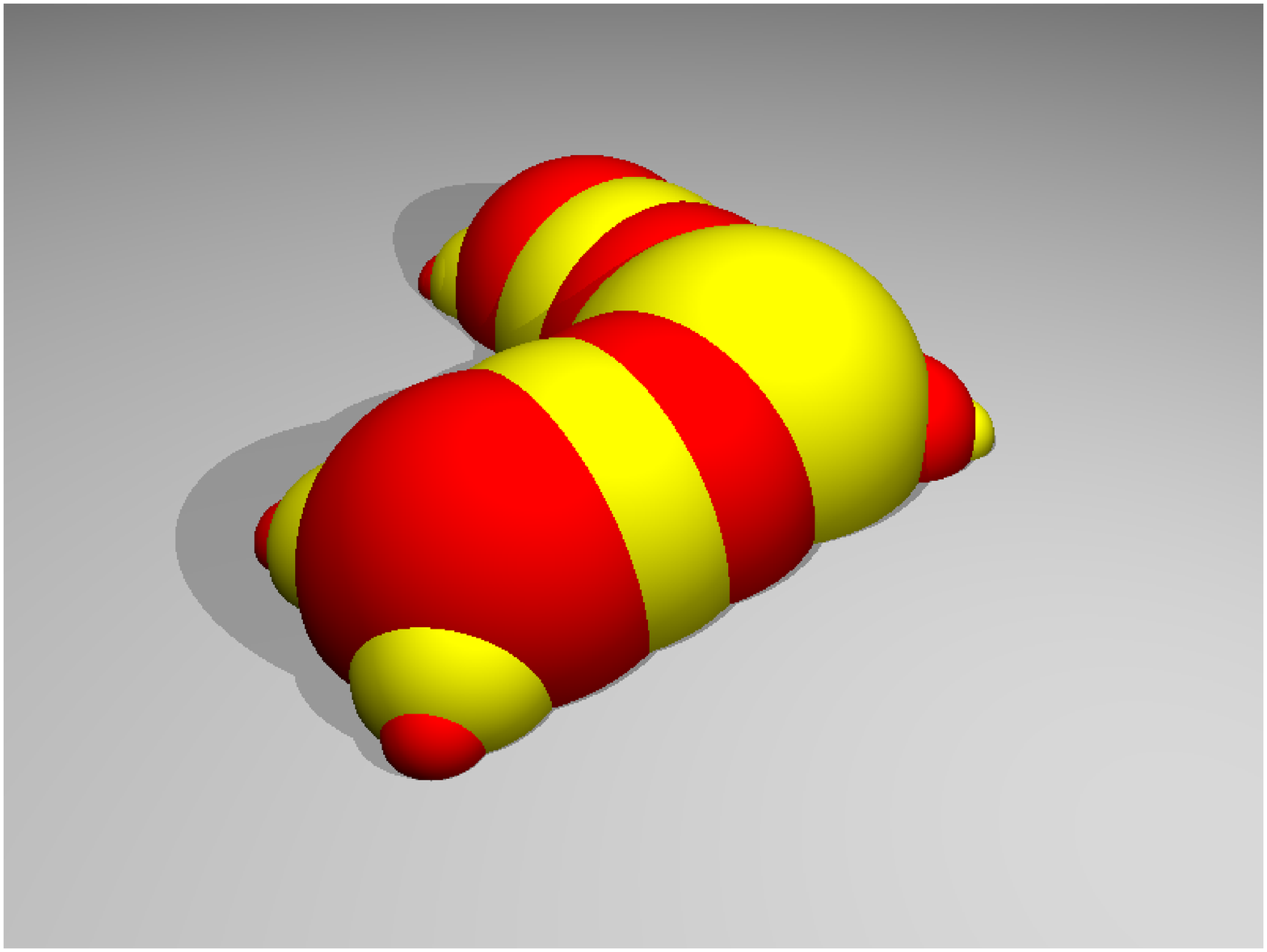}
}
\caption { \label{gap-cres-dome}
      A polygon, an approximation by a
       finite union of disks, its 
         normal crescent decomposition and 
	  the corresponding dome.}
\end{figure}

Given the  normal crescent decomposition of the 
domain,  we can build a one parameter family of regions by 
varying the angles of the crescents (this procedure is called ``angle 
scaling'').
When the angles have all been 
collapsed to zero, the resulting domain is the disk.
See Figure \ref{poly8-intro} for an  example of such a chain.
More examples are illustrated in Figures \ref{scale2} to \ref{disk-poly3}.

\begin{figure}[htbp] 
\centerline{ 
\includegraphics[height=1.5in]{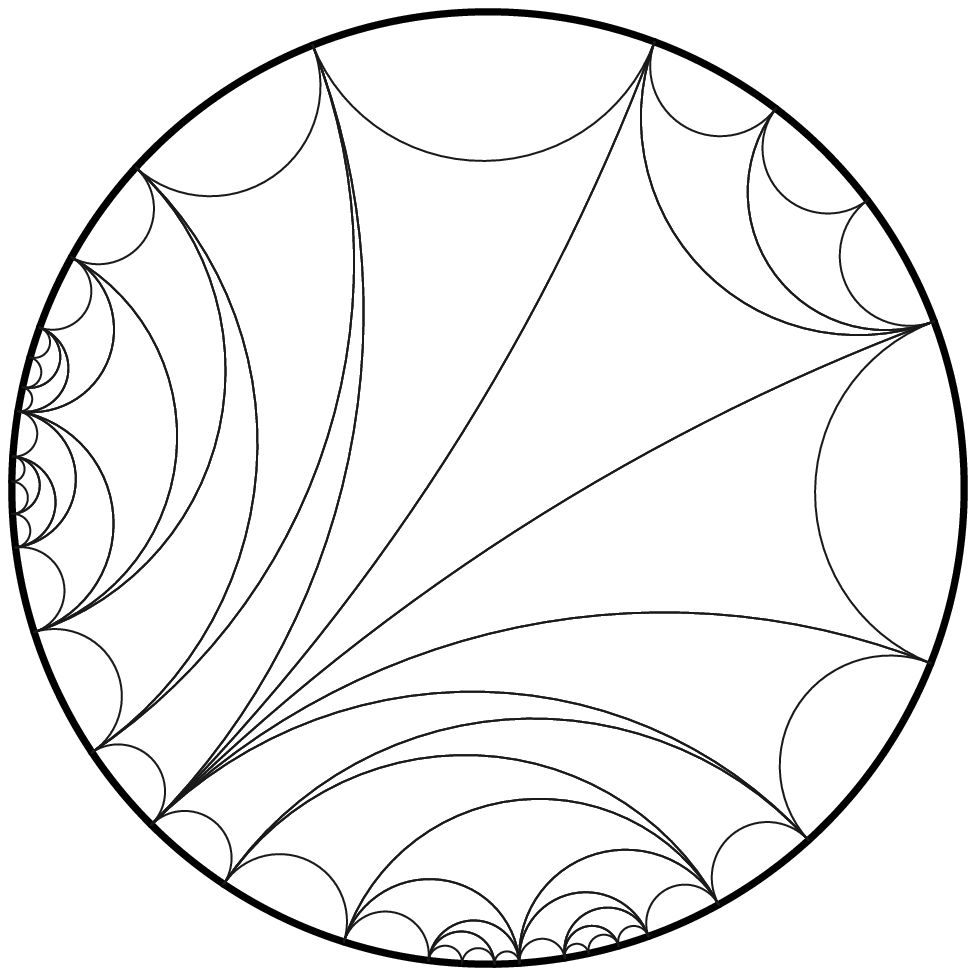}
$\hphantom{xxx}$
\includegraphics[height=1.5in]{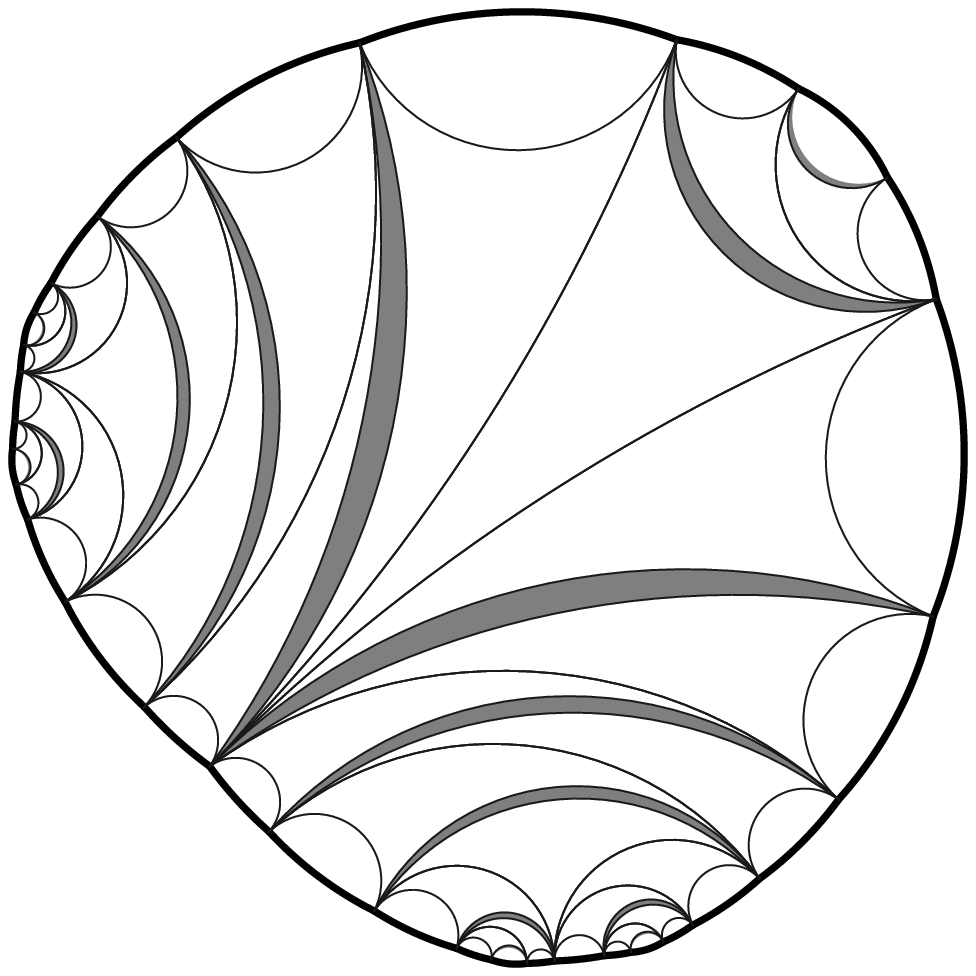}
$\hphantom{xxx}$
\includegraphics[height=1.5in]{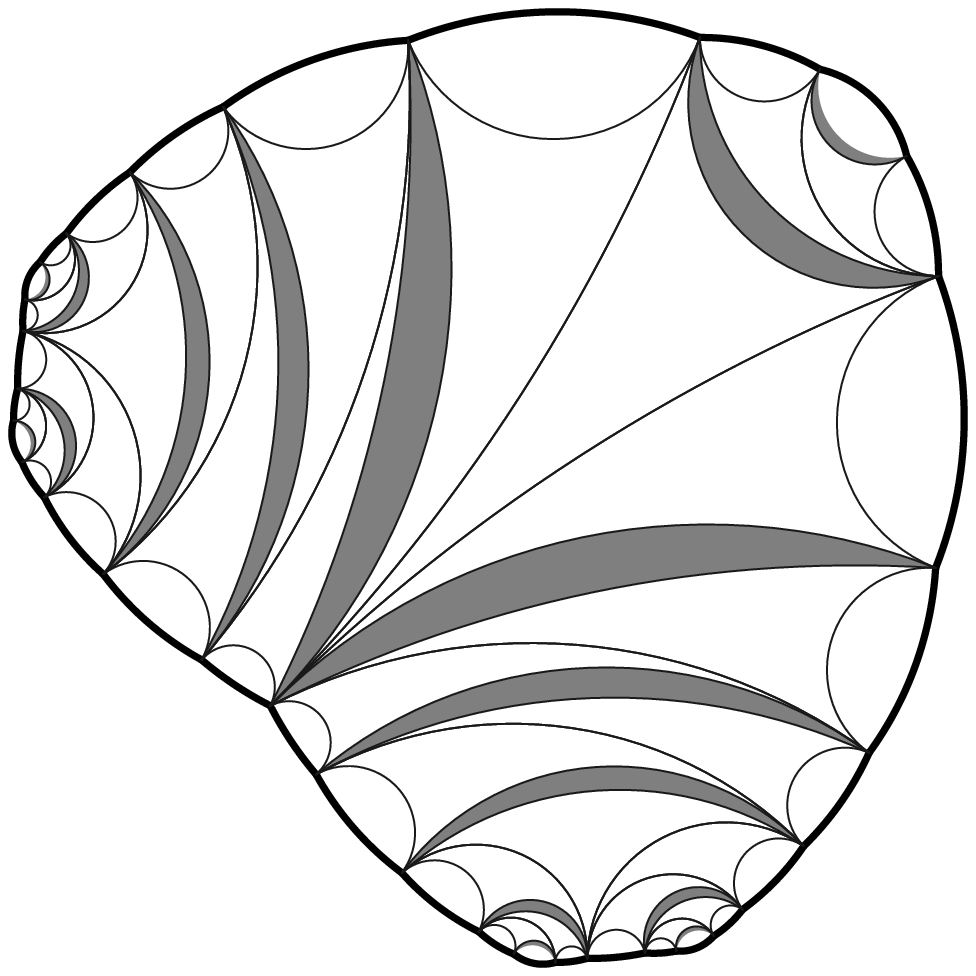}
 }
\centerline{ 
\includegraphics[height=1.5in]{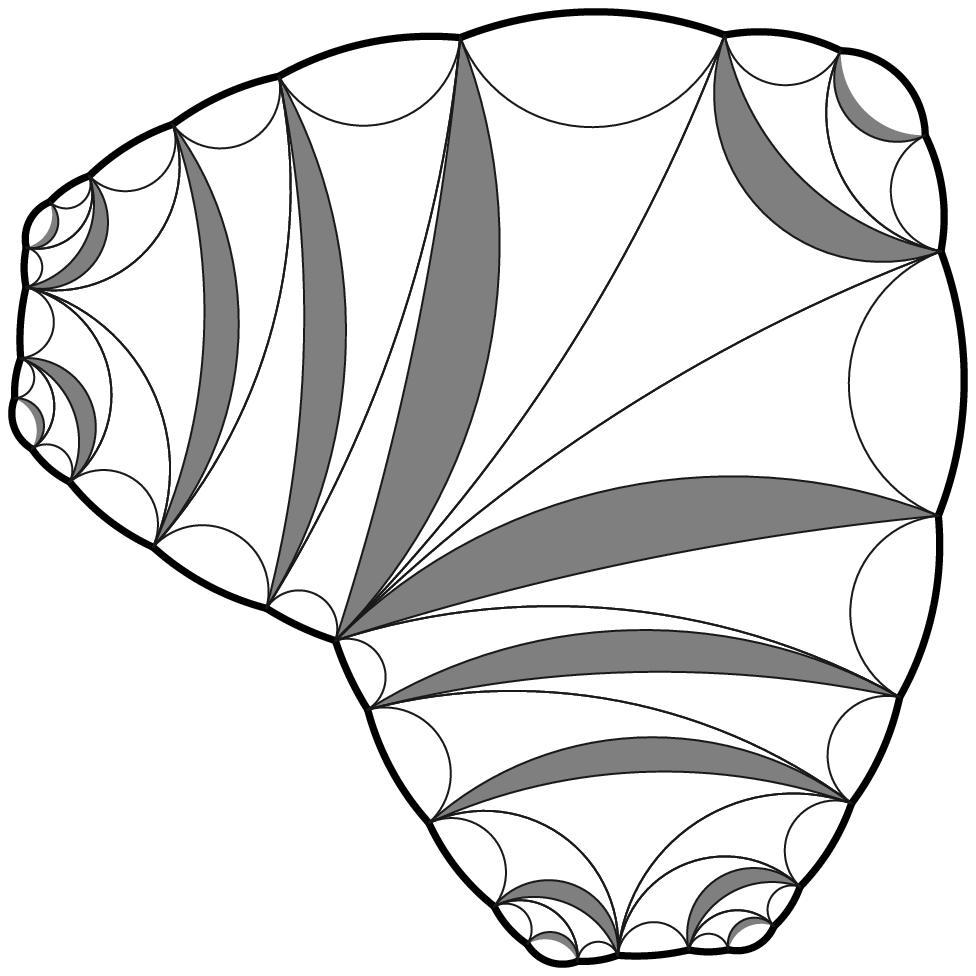}
$\hphantom{xxx}$
\includegraphics[height=1.5in]{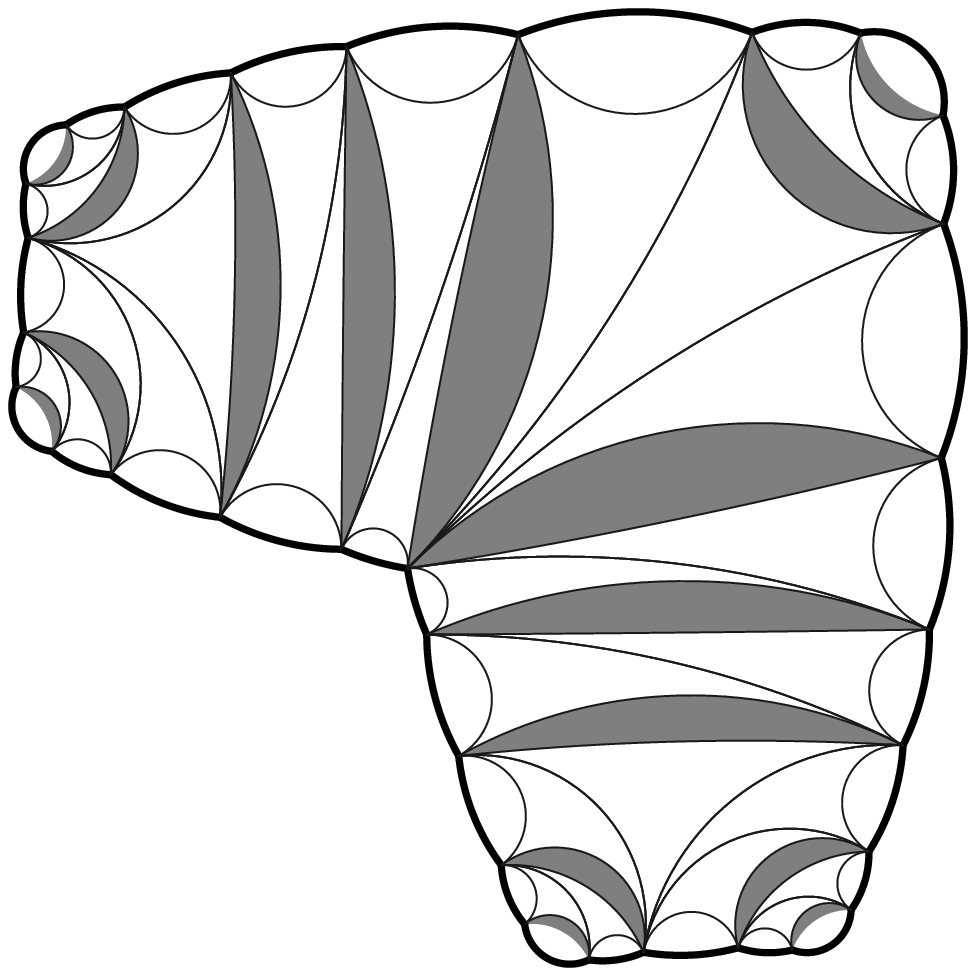}
$\hphantom{xxx}$
\includegraphics[height=1.5in]{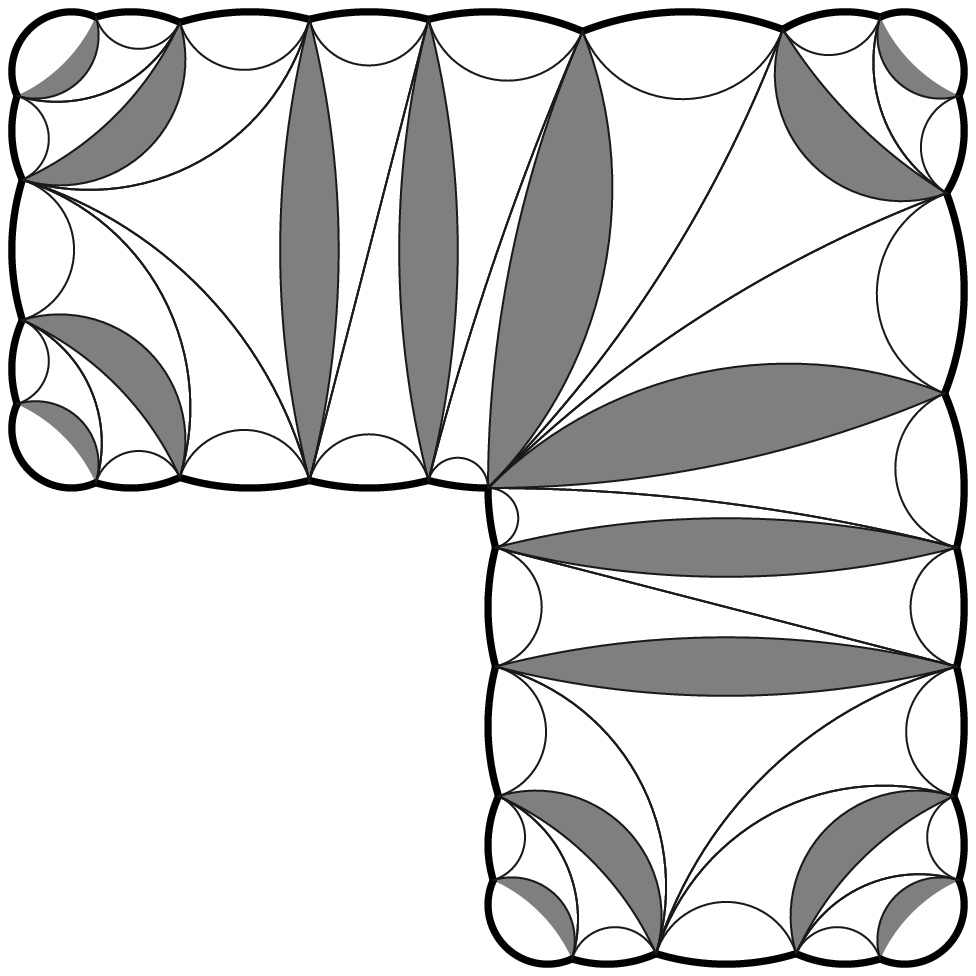}
 }
\caption{\label{poly8-intro}
Deforming the disk to an approximate polygon. The gaps have 
been subdivided to make our decomposition  into a  hyperbolic triangulation 
of the disk.
}
\end{figure}

We shall see that each domain in the chain  is mapped to the next by an 
explicit map $g_k: \Omega_k \to \Omega_{k+1}$ 
with small quasiconformal constant.
This will allow us to convert
an $\epsilon$-representation for one domain in our chain into 
a $2 \epsilon$-representation for the next one.
 We can then use our improvement iteration to improve $2 \epsilon$
to $\epsilon$ and repeat the process. In this way we can start with 
a representation of the disk (which is easy to find) and finish with 
one for $\Omega$ (which is what we want), in a uniformly bounded
number of steps.

There are (at least) two 
alternatives for approximating  a conformal map $f : \disk 
\to \Omega$:  approximate it by conformal maps 
$f_n : \disk \to \Omega_n$ where $\Omega_n$ converges to 
$\Omega$, or approximate by maps $f_n: \disk \to \Omega$ which 
are not conformal, but converge to the conformal map.
The first approach is natural when 
dealing with Schwarz-Christoffel maps since a choice of 
parameters defines a conformal map onto  a region with 
the right angles, but perhaps the wrong side lengths. 
We then adjust the parameters to get a better approximation 
to the target domain. There are various heuristics for 
doing this that work in practice, but the relation between 
the parameters and the geometry of the image can be subtle and 
I have not seen how to prove convergence for any such method.
In this paper, I take the second choice above.  Information 
about the geometry of $\Omega$ is built directly into our 
approximating functions, and our iteration merely has to 
force the approximation to be ``more conformal''; this 
can be done  without reference to $\Omega$,
and hence with estimates independent of $\Omega$. This choice 
also leads naturally to the representation of these maps
using power series on Whitney-Carleson decompositions to enforce 
the desired  boundary conditions. 

The $\epsilon$-representations used to approximate conformal maps 
onto polygons can also approximate maps onto other domains,  as
long as each boundary point has a neighborhood which is an 
image of a half-disk by an explicit conformal map. The 
algorithm is just  a way of computing a global conformal map from knowing 
the local maps around each boundary point; we deal with polygons since 
the local maps are trivial. The work needed in general is $O(N)$, 
where $N$ is the   number of simple 
disks needed to cover the boundary (a simple disk is one so that 
$2D \cap \Omega$ can be explicitly mapped to a half-disk, with the 
boundary going to the line segment). In the case of polygons, 
we can reduce $O(N)$ to $O(n)$ by using arches,
 but this requires ``conformally straightening'' two boundary arcs 
simultaneously. For polygons, we do this with $3$-parameter Schwarz-Christoffel 
maps, but  it  may not be easy to do for curved boundary segments.
For example, local boundary maps are also easy to find 
for circular arc polygons,  but I don't know how to 
 ``straighten''  pairs of circular arcs (unless they
 happen to lie on intersecting circles).
 Thus the method of this paper will 
compute an $(1+\epsilon)$-quasiconformal map onto a circular arc
$n$-gon in time $O(N |\log \epsilon \log \log \epsilon|)$, but I 
don't yet see how to reduce $N$ to $n$.
The special case of finitely bent domains  (unions of disks)
will be discussed later in detail, and conformal maps onto such 
domains will be computed as part of the proof of Theorems 
\ref{thmQC} and \ref{main}.

This paper is part of a series of papers
 that have studied hyperbolic geometry 
and its relation to conformal mappings 
\cite{Bishop-Bowen},
\cite{Bishop-BrenConj}, \cite{Bishop-ExpSullivan}, 
\cite{Bishop-fast}, \cite{Bishop-crdt}.  Along the way, many people 
have  contributed helpful comments, advice and encouragement 
including Raphy Coifman, Tobin Driscoll, David Epstein, John Garnett,  
 Peter Jones, Al Marden, Vlad Markovic, Joe Mitchell,  Nick Trefethen,
Jack  Snoeyink and  Steve Vavasis. Many thanks to them and 
the others who helped me reach the results described here.

Also special thanks to the referees who made a tremendous 
effort reading and evaluating the manuscript. Their thoughtful and
extensive remarks touched on everything from typos to
the  overall strategy of the proof, and prompted  a 
rewriting which simplified parts of the proof  and improved 
the exposition.
The longer the paper, the 
more important (and more difficult) good writing becomes, and I 
very much appreciate their help in making this a better paper.



\section{Summary of the proof  } \label{the-proof}

Now we will summarize our method for computing conformal maps. 
I hope that even without the precise definitions,
 this sketch  will help motivate what follows and 
give a ``map'' for reading the rest of the paper.

Suppose $\Omega$ is a simply connected domain with 
a polygonal boundary with $n$ sides.
Let $\epsilon_0$ be the radius of convergence of 
our Newton-type iteration for representations
 (see Lemma \ref{Newton-radius}).
Compute the medial axis of $\Omega$ and use it to  break $\Omega$
into  $O(n)$ thick and thin pieces 
(see Section \ref{thick-thin}).
Fix a thick piece  $\Omega^{\rm{thick}}$  
and approximate it by a finitely bent 
region $\OmegaFB$ using Lemma \ref{FB-app} with a 
``flattening map'' that is $(1+\delta)$-quasiconformal.
 Compute the
corresponding bending lamination (Section \ref{lamin=linear}),  
normal crescent decomposition (Section \ref{gaps-crescents}) and 
the chain of finitely bent angle scaling  domains $\Omega_0= \disk, 
\dots, \Omega_N = \OmegaFB$.
We will prove  that
if $\delta$ is small enough and $N$ is large enough (depending
only on $\epsilon_0$), then: 
\begin{enumerate}
\item (Starting point) 
     We can construct an $\epsilon_0/2$ representation of $\Omega_0 = \disk$
         (trivial).
\item (Composition step) Given an $\epsilon_0/2$ representation of 
     $\Omega_k$ we can construct an $ \epsilon_0$
     representation of $\Omega_{k+1}$ (Lemma \ref{compose rep FB}).
\item  (Improvement step) Given an $\epsilon_0 $ representation of 
      $\Omega_k$ we can compute an $\epsilon_0/2$ representation 
      of $\Omega_k$ (Lemma \ref{Newton-radius}).
\item (Final conversion) Given an $\epsilon_0/2$-representation of 
      $\Omega_N = \OmegaFB$ we 
      can construct a $\epsilon_0$ representation of $\Omega$
      (Lemma \ref{final rep}).
\item (Iterate to desired accuracy) Given any $\epsilon < \epsilon_0$
      and a $\epsilon_0$-representation of $\Omega$, we can compute an 
      $\epsilon$-representation of $\Omega$ (Lemma \ref{Newton-radius}).
\end{enumerate}

It is important to note that in steps (1)-(4) we only compute maps 
to a fixed accuracy; just enough to 
use it as a good starting point for the map onto the 
next element.   Thus the precise timing of these steps is 
unimportant, as long as it is linear in $n$ with a constant 
depending only  on $\epsilon_0, \delta, N$.
All  these constants will be  be chosen 
independent of $n$ and the geometry of $\Omega$, so the total 
work to get an $\epsilon_0$-representation of $\Omega$ is 
$O(n)$ with an constant independent of $n$ and $\Omega$.

At the final step we use Lemma \ref{Newton-radius} to  
iterate until we reach the desired $\epsilon$. 
By Lemma \ref{Newton-radius}, the $k$th iteration
 gives accuracy $ \epsilon_0^{2^k}$ 
and takes time $O( n  2^{k}  k )$ to perform (with constant
depending only on the fixed number $\epsilon_0$).
Thus  
$O(\log \frac 1 \epsilon)$ iterations 
are needed to reach accuracy $\epsilon$. Since the 
time per iteration grows exponentially at each step, the total time 
is dominated by the final step, which is 
$O(n \log \frac 1 \epsilon \log \log \frac 1 \epsilon)$. 

In an earlier version of this paper, the chain of domains 
consisted of polygons inscribed in the angle scaling family 
of finitely bent domains. This was awkward, but avoided 
some complications of extending the idea of $\epsilon$-representations
from polygons to finitely bent domains. This version
deals with these complications, in return for a cleaner 
presentation of the angle scaling chain and inductive steps.

The paper divides roughly into five parts: (1) an expository 
introduction to the  medial axis, $\iota$-map and angle 
scaling, (2) the construction of the bending lamination,
the associated decomposition of $\uhp$ and our 
representation of conformal maps, (3) the thick/thin 
decomposition of polygons and the special properties of 
thin polygons, (4) constructing the chain of domains 
connecting $\disk$ to $\Omega$ and implementing the composition 
step on representations,
 and  (5) our iteration for 
improving representations based on finding approximate
solutions of the Beltrami 
equation  by the multipole method.
More precisely, the remaining sections are:

\begin{description}
\item [Section \ref{Domes+MA1}] We    introduce the medial axis 
                   and the hyperbolic dome.
\item [Section \ref{Domes+MA2}] We  discuss Thurston's observation 
                    that the dome of a simply connected domain  $\Omega$ is 
                    isometric to the hyperbolic disk. We show  how this gives 
                    a mapping $\iota$ from $\partial \Omega$ to $\circle$.

\item [Section \ref{gaps-crescents}]
          We introduce the gap/crescent decomposition of a 
       finitely bent domain, the corresponding bending lamination on the
       disk and construct the angle scaling  chain of domains that connects the 
      disk to $\Omega$.

\item [Section \ref{proof-of-SEM}]
      We show that elements of the angle scaling chain are close 
      in a uniform quasiconformal sense. This is one of  the key ideas that makes 
      the whole method work with uniform estimates.

\item [Section \ref{PM-maps}]
      We prove a technical result used in the Section \ref{proof-of-SEM}. 
     We introduce the idea of piecewise M{\"o}bius maps and $\epsilon$-Delaunay 
      triangulations to prove that a map which is close to M{\"o}bius
      transformations locally has a global approximation by a 
      hyperbolic bi-Lipschitz function.

\item [Section \ref{lamin=linear}]
 We   show the bending lamination of a finitely bent domain can be computed 
        in linear time.
\item [Section \ref{cover}] We cover the bending
 lamination  by $O(n)$ ``simple''
                  regions.
\item [Section \ref{extend-decom}] 
    We refine this covering  and extend it 
    to a decomposition   of $\uhp$.

\item [Section \ref{epsilon-reps}]  We  define an $\epsilon$-representation 
          of a polygonal domain and show such a representation
         corresponds to a  $1+O(\epsilon)$-quasiconformal map onto the domain.

\item [Section \ref{thick-thin}] We define thick and thin polygons 
     and show that any polygon with $n$ sides can be decomposed 
     into thick and thin pieces with  a total of  $O(n)$ sides, 
     and with certain estimates on the overlaps of the pieces.
     We also record some approximation results for conformal maps 
     onto thin polygons.

\item [Section \ref{reps of FB}] We show how to approximate 
  thick polygons by finitely bent domains and define 
   $\epsilon$-representations of conformal maps onto such 
   domains. We use the approximation to define an angle 
   scaling family.

\item[Section \ref{really thick}] We show that if a polygon satisfies 
   a strong form of thickness, its finitely bent approximation satisfies 
   a weak form.  We use this to show how a representation of one 
   element of the angle scaling  family can be used to construct  
   a representation of the next element.

\item [Section \ref{iterate}] Assuming we can approximately solve a certain 
Beltrami equation we show how to update a $\epsilon$-representation 
to a $\epsilon^2$-representation.

\item[Section \ref{reduce-to-Dbar}] We reduce solving the  Beltrami 
   problem  to solving a $\overline{\partial} $ problem.

\item [Section \ref{multipole}] We show how to quickly solve 
the $\Dbar$-problem by computing the Beurling transform
of a function  using the fast multipole method.


\item [Appendix  \ref{background}] Background on conformal maps, hyperbolic 
                        geometry and quasiconformal mappings. 
			Non-analysts may wish to review some of this 
			material before reading the rest of the paper.
\item [Appendix  \ref{fast-power-series}]  We  review  known results 
                about power series and  show that $O(p \log p)$ 
		suffices for all the manipulations needed by the 
		algorithm.


\end{description}


\section{Domes and the medial axis: an introduction}\label{Domes+MA1}

Here we  introduce two closely related geometric 
objects associated to any planar domain $\Omega$: the dome
of $\Omega$ and the medial axis of $\Omega$.  We start with the dome.

Given a closed set $E$ in the plane, we let $C(E)$ denote the 
convex hull of $E$ in the hyperbolic upper half-space, $\uhs=\reals^3_+$.
This is the convex hull in $\uhs$ of all the infinite hyperbolic
geodesics that have both endpoints in $E$ (recall these are 
exactly the circular arcs in $\uhs$ that are orthogonal to 
$\reals^2= \partial \uhs$). One really needs to take the convex hull
of the geodesics ending in $E$ and not just the union of these 
geodesics; for example, if $E$ consists of three 
points, then there are three such geodesics and these form the 
``boundary'' of an ideal triangle whose interior is also 
in the convex hull of $E$.

 The complement of $C(E)$ is a 
union of hyperbolic half-spaces.
   There is one component of 
$\uhs \setminus C(E)$  for each complementary component $\Omega$ of 
$E$ and this component is the union of hemispheres whose bases are disks 
in $\Omega$ 
(also include  half-planes and disk 
complements if $\Omega$ is unbounded).
For example, when $E$ is the boundary of a square, 
the lower and upper boundaries of $C(E)$ are 
illustrated in Figure \ref{square-CH}.

\begin{figure}[htbp]
\centerline{ 
\includegraphics[height=2in]{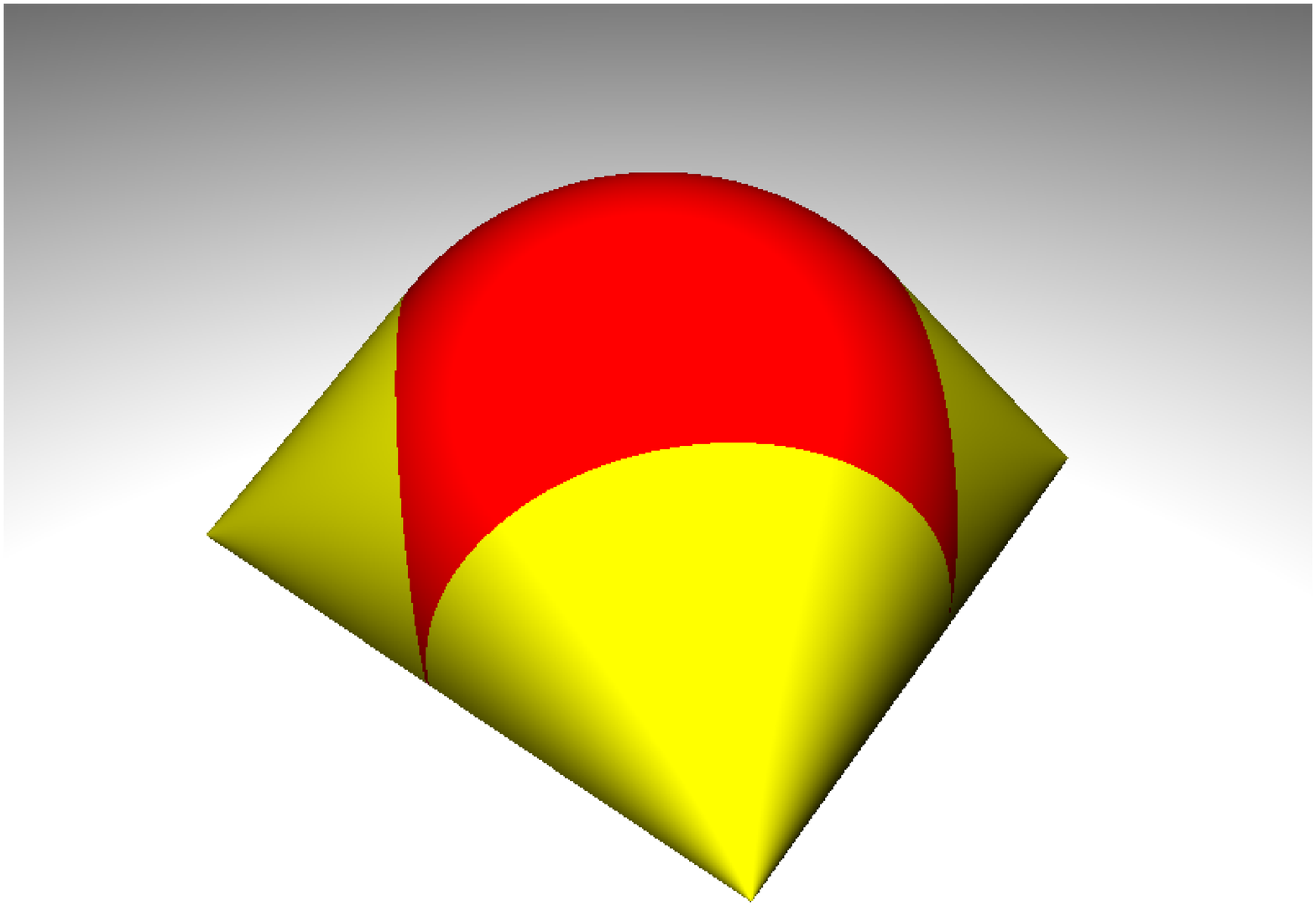}
$\hphantom{xxxxx}$
\includegraphics[height=2in]{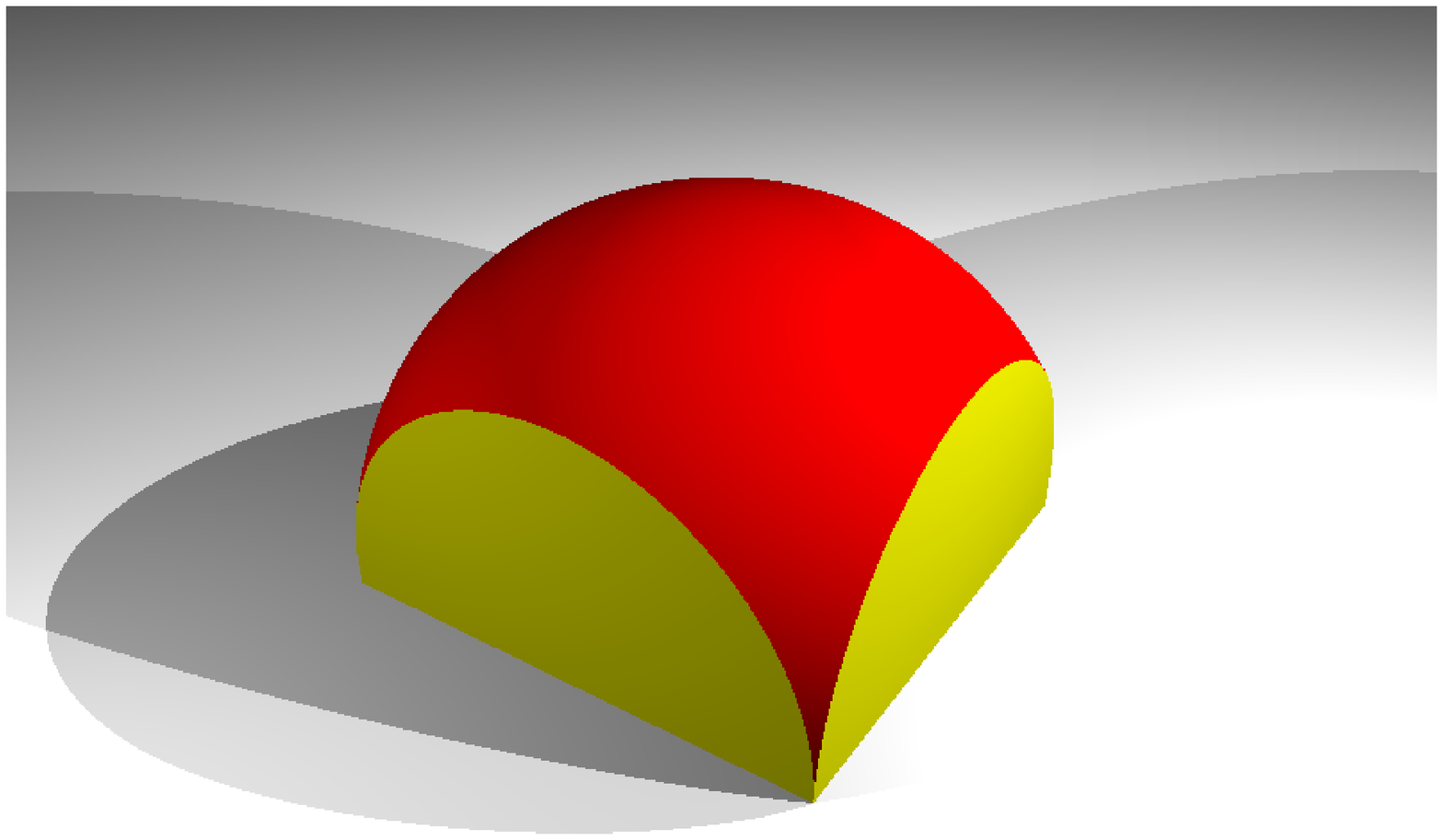}
 }
\caption{ \label{square-CH} The lower and upper boundaries of the 
    hyperbolic convex hull of  the boundary of  
    a square (left and right figures respectively).
 The lower boundary consists 
of one geodesic face (dark)  and four Euclidean cones (lighter). The upper 
boundary has five geodesic faces (one hemisphere and four 
vertical). The outside of the square is a finitely bent domain, 
but the inside is not.}
\end{figure}

In this paper we will focus exclusively on the case of
a single, bounded, simply connected domain 
$\Omega$. In this case, the dome of 
$\Omega$ is the unique boundary component of the convex set $C(\Omega^c)$.
The dome is fairly easy to draw because of the description of
$W$ as a union of Euclidean hemispheres  with bases in $\Omega$.
Moreover, 

\begin{lemma} \label{support-planes}
Suppose $S_\Omega$ is the dome of a simply connected, proper
plane domain $\Omega$. Then for every $x \in S_\Omega$ there is an
open hyperbolic  half-space $H$ disjoint from $S_\Omega$ so that $x \in
\partial H \cap S_\Omega$. For any such half-space,
 $\partial H \cap S_\Omega$ contains
an infinite geodesic, and its base disk (or half-plane)
has boundary that hits 
$\partial \Omega$ in at least two points.
\end{lemma}

\begin{proof}
Let $W = C(\Omega^c)$  be the hyperbolic convex hull of $\Omega^c$,
so $S_\Omega = \partial W$.
By definition, $W$ is the intersection of all closed
half-spaces that contain it, and from this it is easy to see that
any boundary point on $W$ is on the boundary of some closed
half-space that contains $W$. Thus $x$ is also on the boundary
of the complementary open half-space $H$ (which must be disjoint
from $W$). The base of $H$ on $\reals^2$ is a half-plane or
a disk and by conjugating by a M{\"o}bius transformation, if
necessary, we assume it is the  unit disk $D=\disk$ and that
$H$ contains the point $z=(0,0,1) \in  S_D$.
  Clearly $\partial D$
hits $\partial \Omega$ in at least one point, for otherwise
its closure would be contained in another open disk in $\Omega$, 
whose dome would be strictly higher than $S_D$, contradicting 
that $z \in S_D \cap S_\Omega $.
 In fact, $\partial D$ must
hit  $\partial \Omega$ in at least two points. For suppose
it only hit at one point, say $(1,0) \in \reals^2$. Then
for $\epsilon>0$ small enough the disk $D(-2 \epsilon, 1+\epsilon)$
would also be in $\Omega$ and its dome would strictly separate
$z$ from $S_\Omega$.
Thus $\partial D$ hits $\partial \Omega$ in at least two
points and the geodesic in $\reals^3_+$ between these points lies
on the $\partial H \cap  S_D$, as desired.

\end{proof}

Thus each point on the dome is also on the dome of a
disk in $\Omega$ whose boundary hits $\partial \Omega$
in at least two points.
  Such a 
disk is called a ``medial axis disk'' for $\Omega$ and the set of 
centers of such disks is called the medial axis of $\Omega$, 
denoted $\rm{MA}(\Omega)$. 
(The centers, together with the radii, is usually called the medial 
axis transform of $\Omega$, $\rm{MAT}(\Omega)$). 
It is easy to see that the $\Omega$ is the union of its
medial axis disks, and so it determined by $\rm{MAT}(\Omega)$.

The dome is easiest to visualize when 
$\Omega$ is a finite union of disks, e.g., see Figure  \ref{FB1}.
Such a domain will be called ``finitely bent'' because the dome 
consists of a finite union of geodesic faces (each contained on 
a geodesic plane in $\uhs$, i.e., a Euclidean hemisphere or vertical 
plane) which are joined along infinite geodesics called the bending 
geodesics.  
\begin{figure} [htbp]
\centerline{
 \includegraphics[height=1.75in]{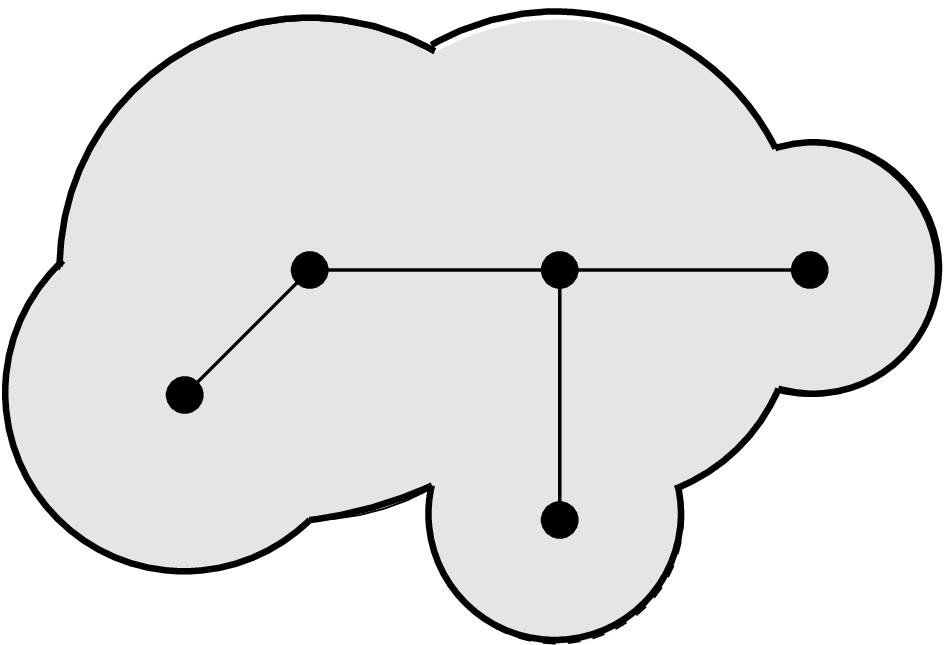}
$\hphantom{xxxxx}$
 \includegraphics[height=1.75in]{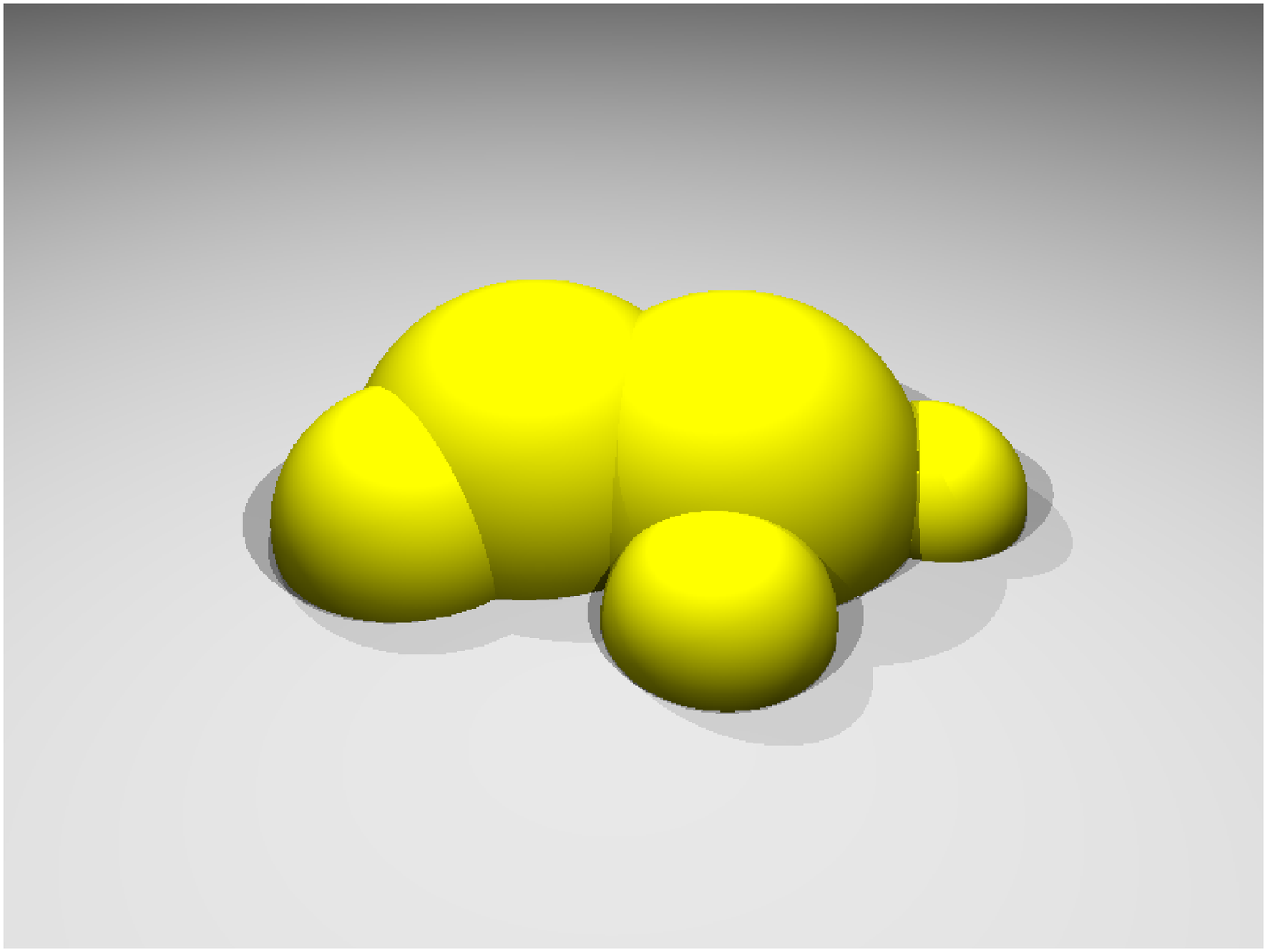}
}
\caption{\label{FB1} A finitely bent domain, its medial axis  and its dome}
\end{figure}

When we are given a finitely bent domain $ \Omega$ we shall always 
assume we are given a complete  list of disks in 
$\Omega$ whose boundaries hit $\partial \Omega$ in at least 
three points.  Then every face of the dome corresponds to a 
hemisphere that has one of these disks as its base.  This is 
slightly different than just giving a list of disks whose union 
is $\Omega$; in Figure \ref{FB2} we show a domain that is a union 
of four disks $\Omega= D(1,1) \cup D(i,1) \cup  D(-1,1) \cup
D(-i,1)$ but that contains a fifth disk, $D(0, \sqrt{2})$, which 
also corresponds to a face on the dome of $\Omega$.

\begin{figure}
\centerline{
 \includegraphics[height=1.75in]{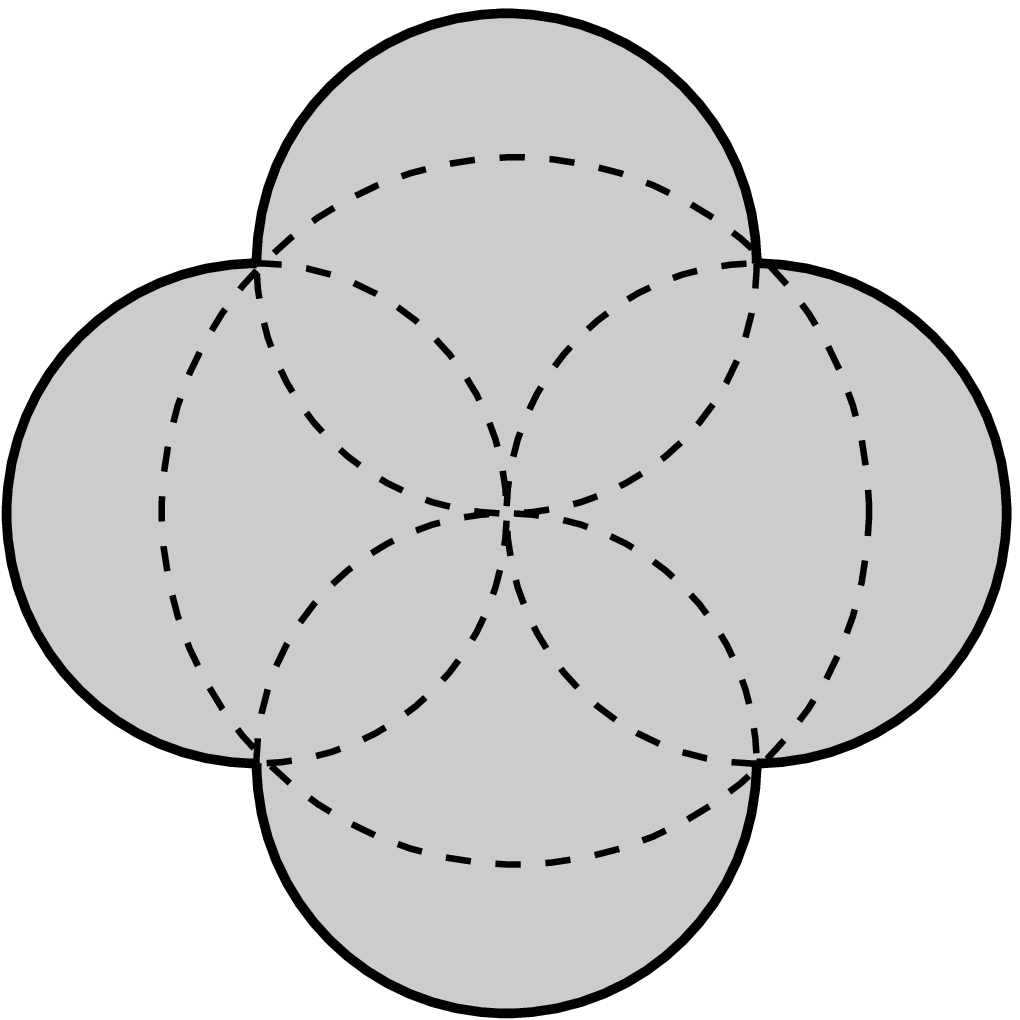}
$\hphantom{xxxxx}$
 \includegraphics[height=1.75in]{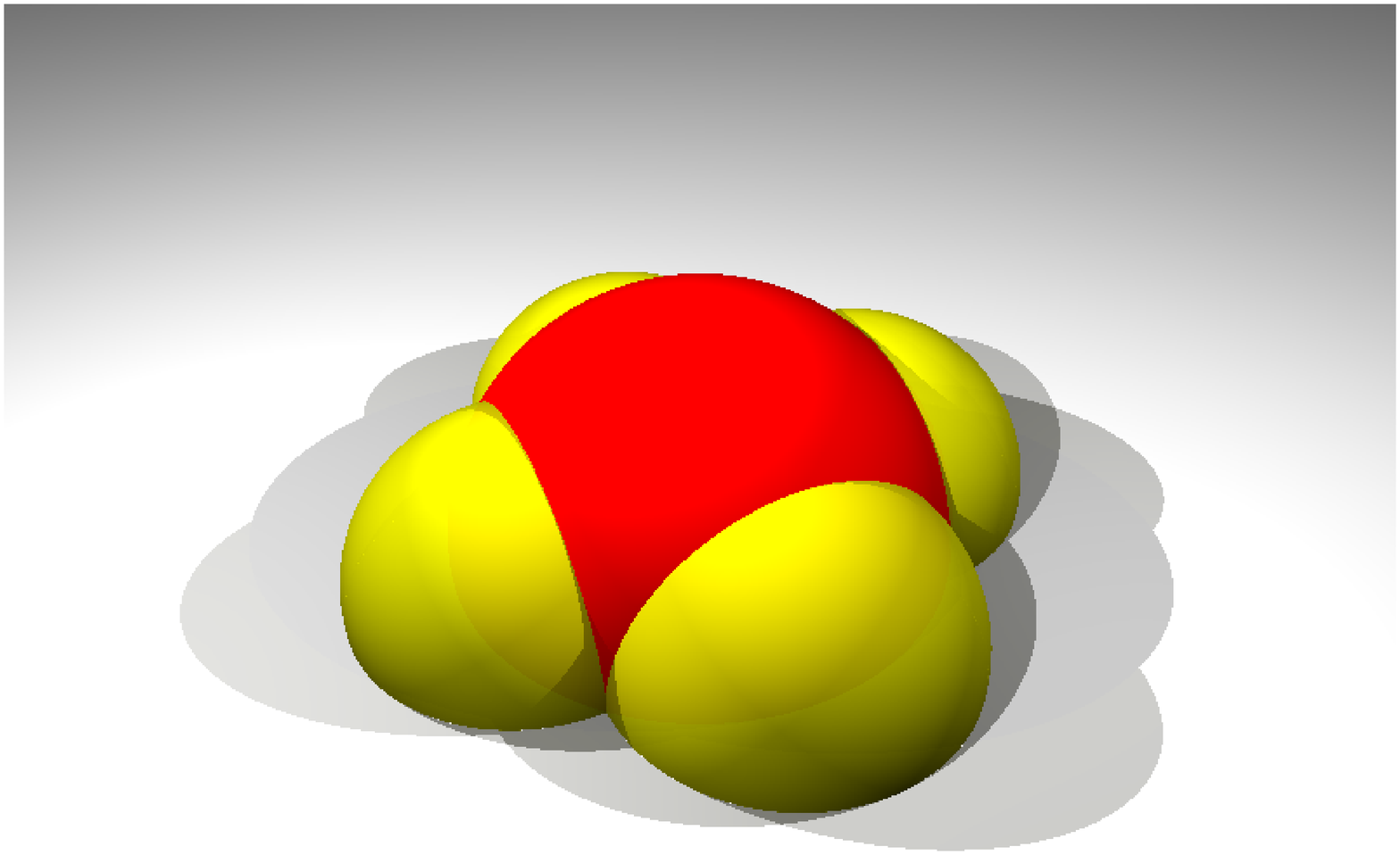}
}
\caption{\label{FB2} A domain 
        that is a union of four disks, but 
    which has five faces on the dome because of a ``hidden'' maximal disk.}
\end{figure}

The faces of the dome of a finitely bent domain 
form the vertices of a finite tree, with 
adjacency defined by having an infinite geodesic edge in common.
This induces a tree structure on the maximal disks in the 
base domain: disks that hit exactly two boundary points are 
interior points of edges of the tree and disks that hit three or 
more points are the vertices. 

\begin{lemma} \label{tree-lemma}
For any tree the number of vertices of degree three or greater 
is less than the number of degree  one vertices.
\end{lemma}
 The proof is easy and left to the reader (remove a degree one 
vertex and use induction). So if $\Omega$ can be written as a 
union of $n$ disks in any way,
there are at most $2n$ vertices of the medial axis.

 For polygons the medial axis 
is also a finite tree, but now there are three 
types of edges:
 (1) edge-edge 
bisectors that are straight 
line segments equidistant from two edges, (2) point-point
bisectors, which are  straight line segments equidistant 
from two vertices,  or (3) point-edge bisectors, which are
 parabolic arcs equidistant from a vertex and 
and an edge. For an $n$-gon the medial axis has at most $O(n)$
vertices (it is not hard to show $2n+3$ works).

To illustrate these ideas we show a few polygons, along with 
their medial axes and their domes. 
The dome of a polygon is naturally divided into kinds of  pieces: (1)  a
hyperbolic geodesic face corresponding to a vertex of the medial axis 
of degree three or more (2) a cylinder or cone   corresponding 
to sweeping a hemisphere along  a bisector of two edges or 
(3) sweeping a hemisphere along the parabolic arc of a point-edge
bisector.
Disks corresponding to the interiors of point-point bisector edges 
do not contribute to the dome since
the union of the two disks at the endpoints  of this edge
contain all the disks corresponding to the interior points.

  In the dome 
of a convex polygon, only the first two types of  pieces
can occur. These are illustrated in  Figure \ref{convex}. 
 The third type of medial axis arc can occur 
in non-convex domains, as illustrated in the polygonal ``corner''
in Figure \ref{corner}.

\begin{figure}[htbp]
\centerline{ 
\includegraphics[height=1.75in]{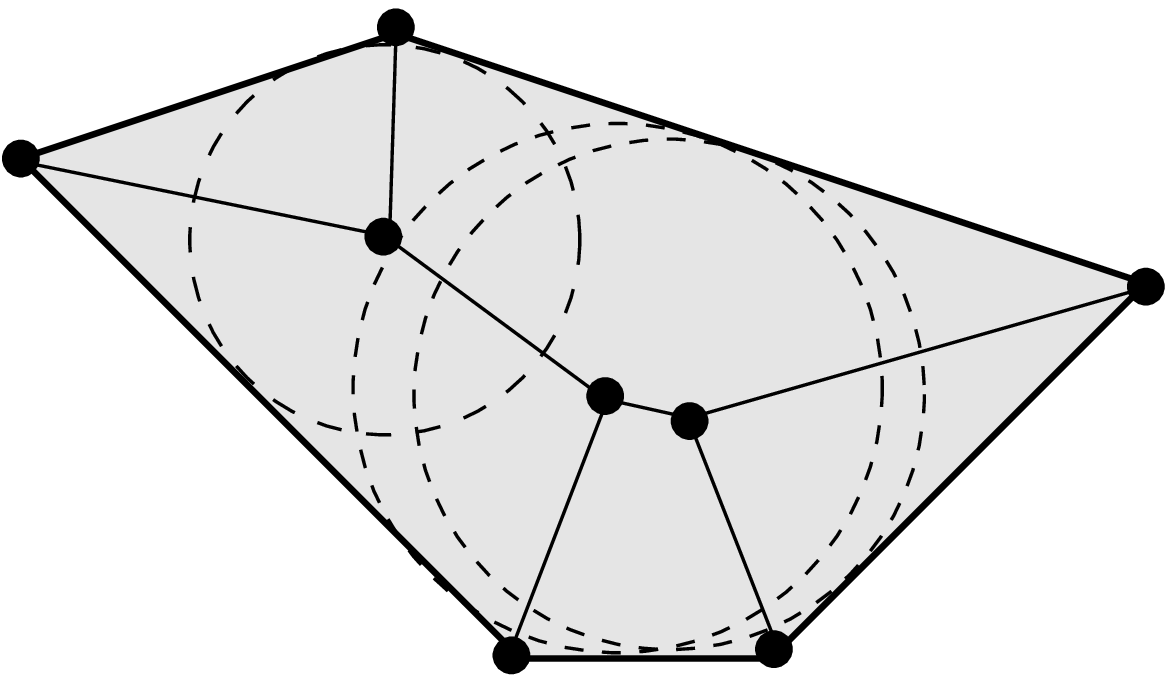}
}
\centerline{
\includegraphics[height=2in]{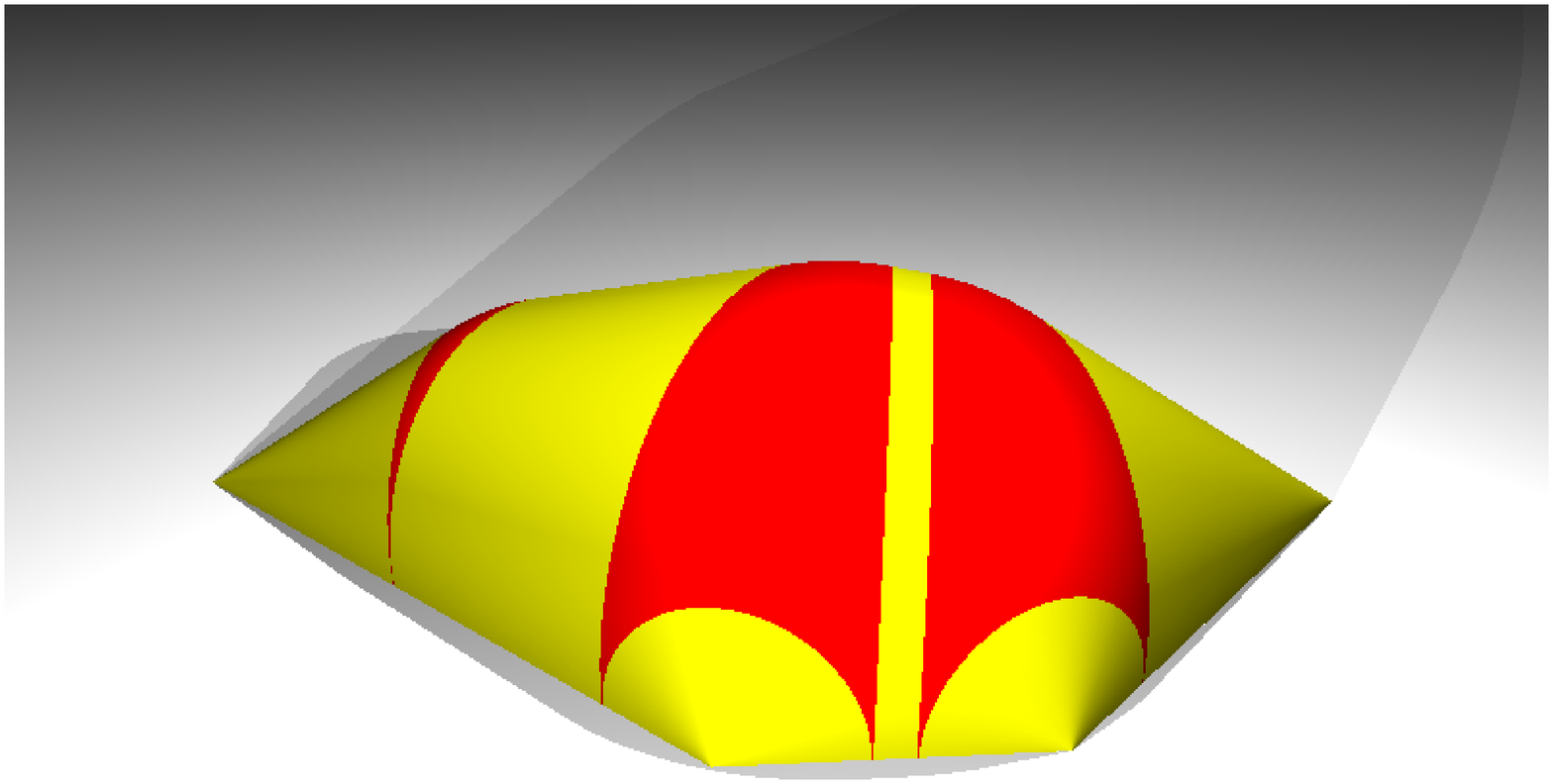}
$\hphantom{xxx}$
\includegraphics[height=2in]{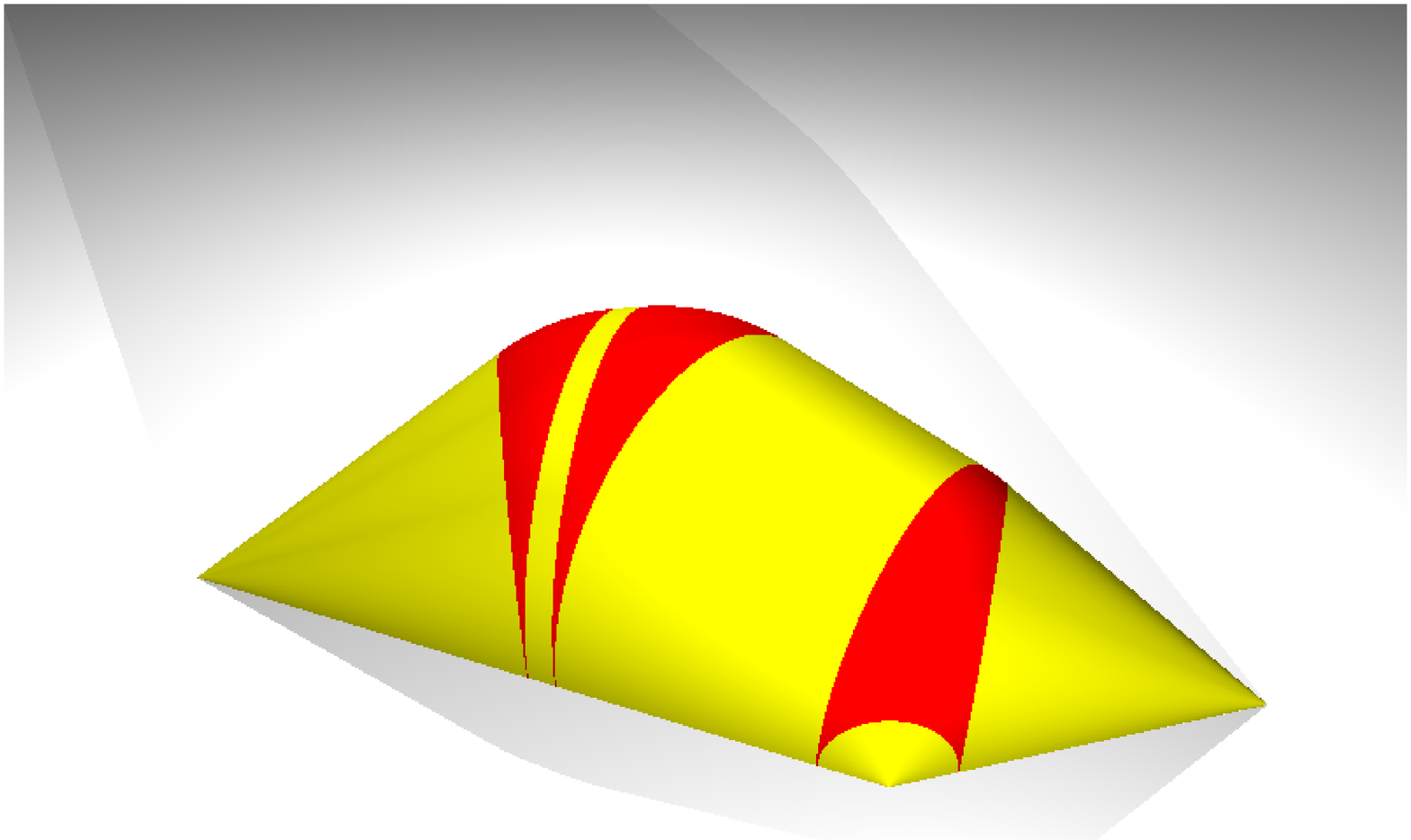}
 }
\caption{ \label{convex}
The medial axis and dome of a convex region. 
This dome has three
geodesic faces that are shaded darker (these correspond to 
vertices of the  medial axis);
 the lighter parts of the 
dome are Euclidean cones that correspond to edges of the medial axis.
The dome is shown from two different directions}
\end{figure}

\begin{figure}
\centerline{ 
	\includegraphics[height=2in]{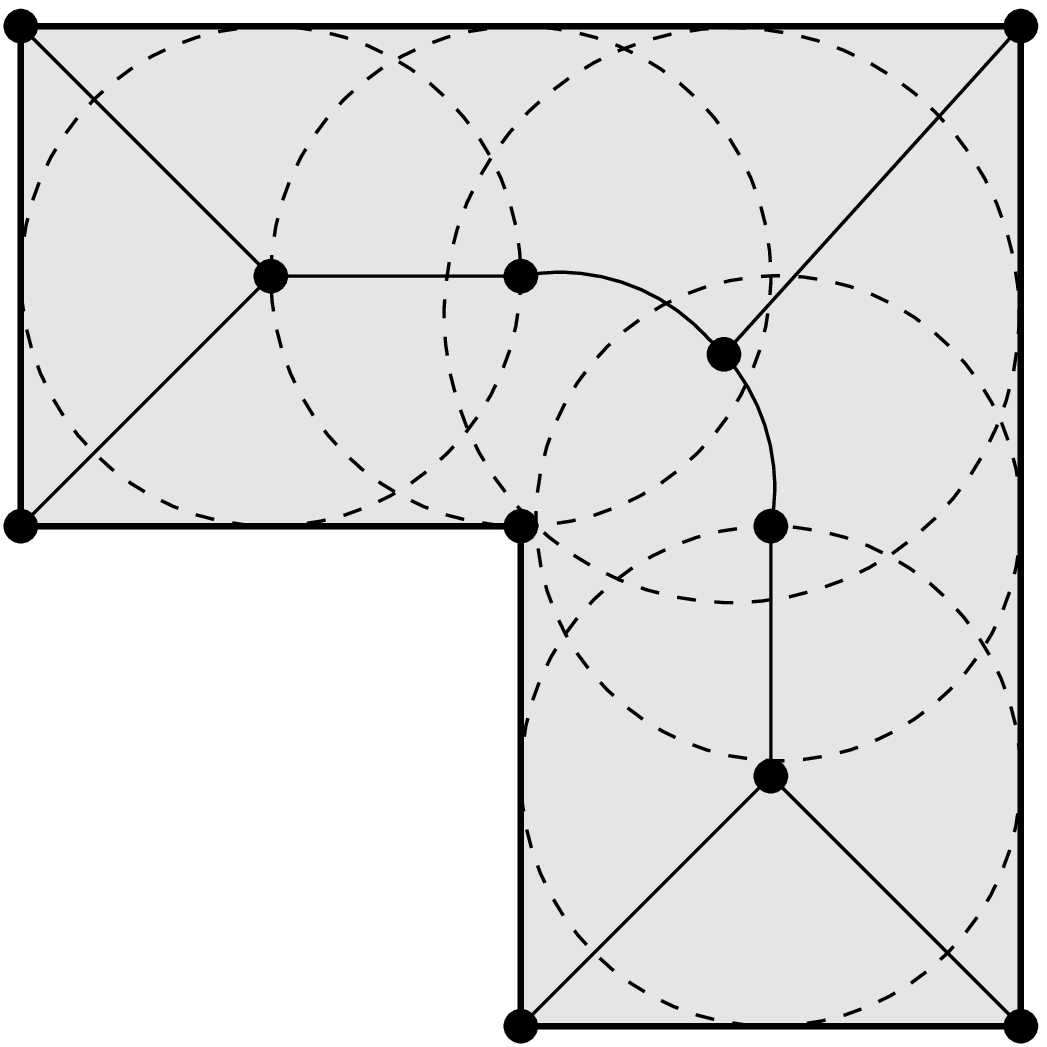}
$\hphantom{xx}$
\includegraphics[height=2in]{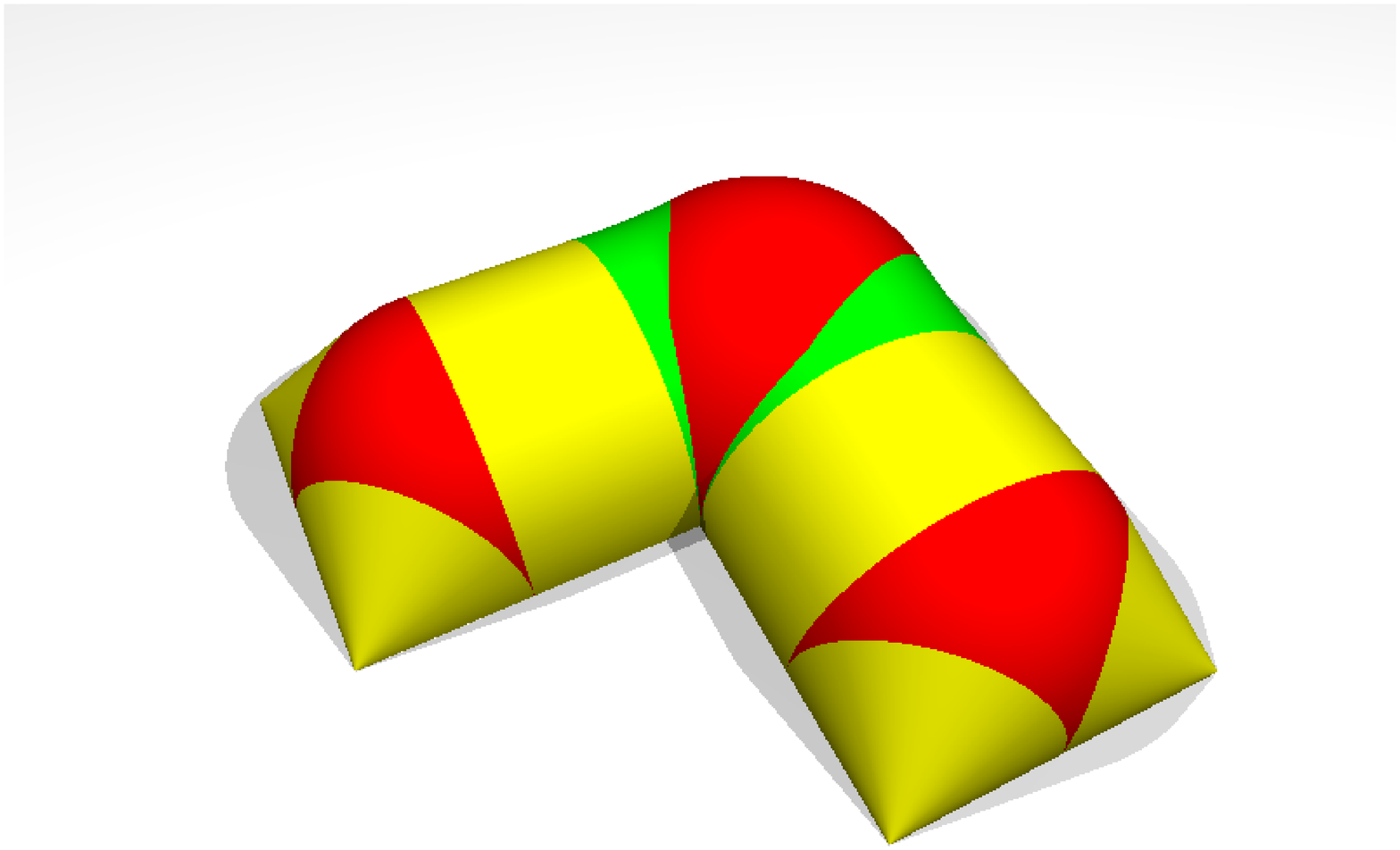}
 }
\caption{\label{corner}
The dome of a ``corner''. The darkest shading are 
geodesic faces (vertices of the medial axis);
 the lightest are Euclidean cones or cylinders (edge-edge bisectors
in the medial axis).
The medium shading illustrates 
the third type of medial axis edge that can occur: the 
parabolic bisector of a point and a line.}
\end{figure}

The medial axis also suggests a way of approximating any domain 
by a finite union of disks; simply take a finite subset of the 
medial axis so that the corresponding union of medial axis disks
is connected.
The medial axis of such a union consists of one vertex for each 
geodesic face in the dome and straight lines connecting the vertices 
corresponding to adjacent faces. 
 A polygon, its medial axis and a  finitely bent approximation 
are shown in Figure \ref{P_2}.  In Figure \ref{P_2dome} we show 
the domes of the polygon and its approximation.
The process of approximating a polygon by a finitely bent 
region will be discussed in greater detail in 
Section \ref{reps of FB}. Alternate approximations of 
polygons by disks coming from circumcircles of a triangulation
of the polygon are used in \cite{Bishop-crdt}, \cite{DV98}.

\begin{figure}[htbp]
\centerline{ 
\includegraphics[height=1.6in]{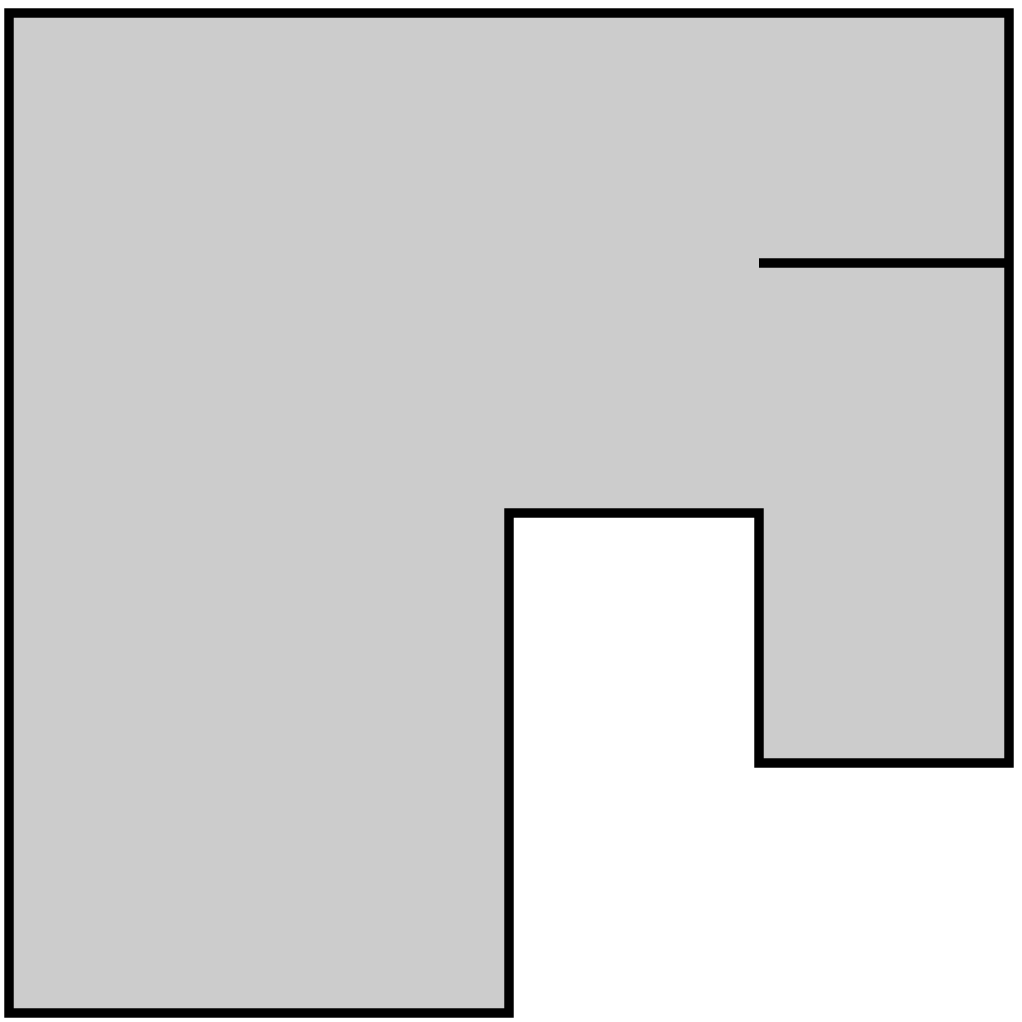}
$\hphantom{xxx}$
\includegraphics[height=1.6in]{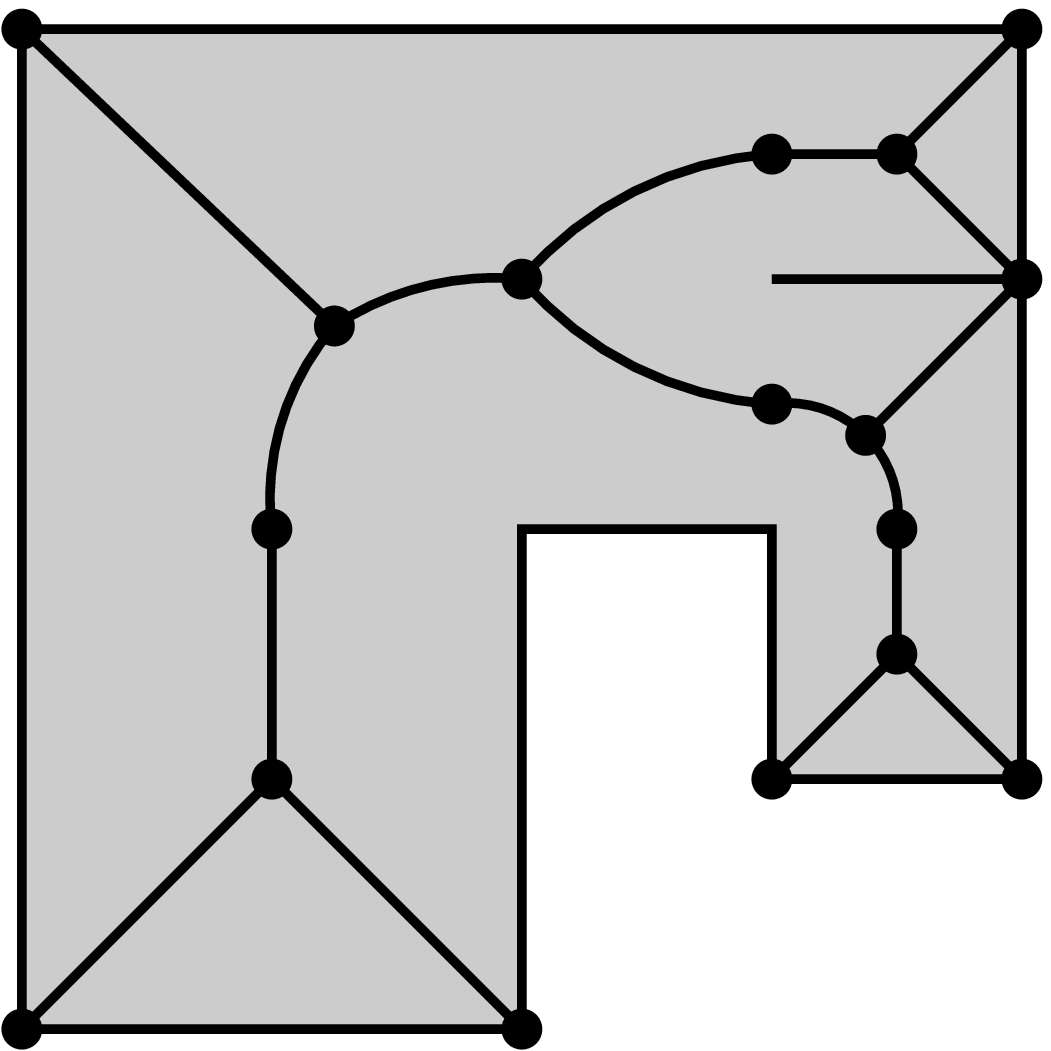}
$\hphantom{xxx}$
\includegraphics[height=1.6in]{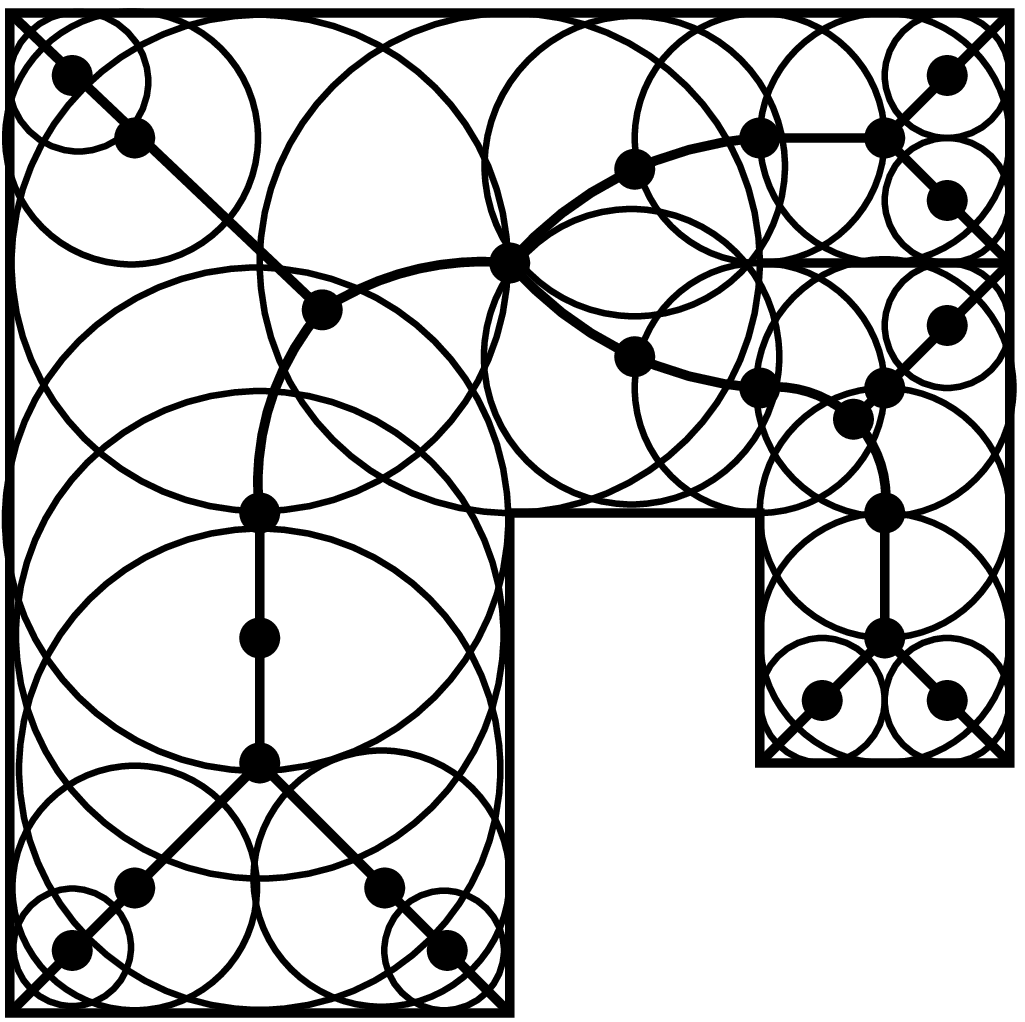}
}
\caption{ \label{P_2}
A (non-simple) polygon, its medial axis and a 
finitely bent approximation.}
\end{figure}

\begin{figure}[htbp]
\centerline{ 
\includegraphics[height=2in]{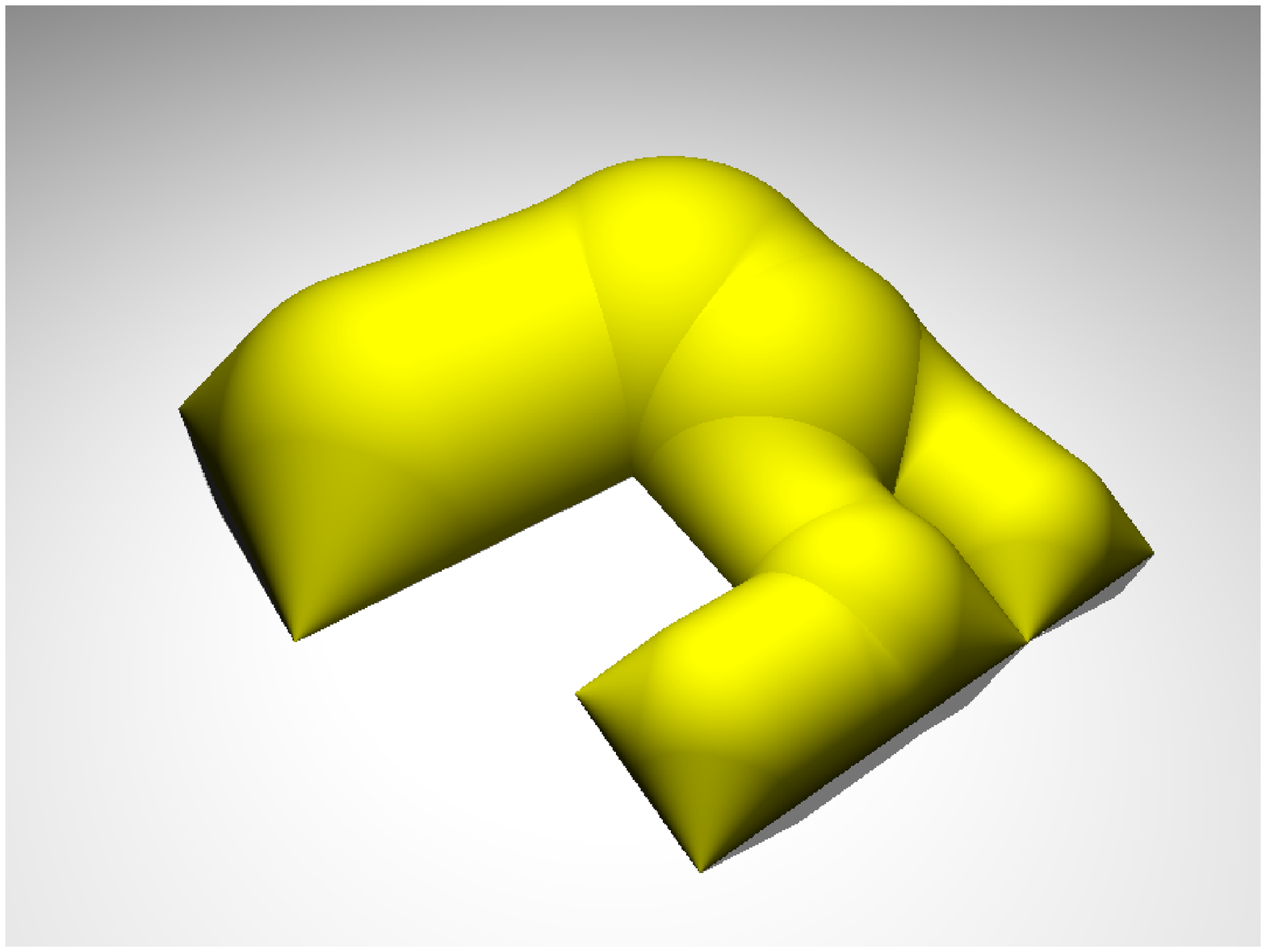}
$\hphantom{xx}$
\includegraphics[height=2in]{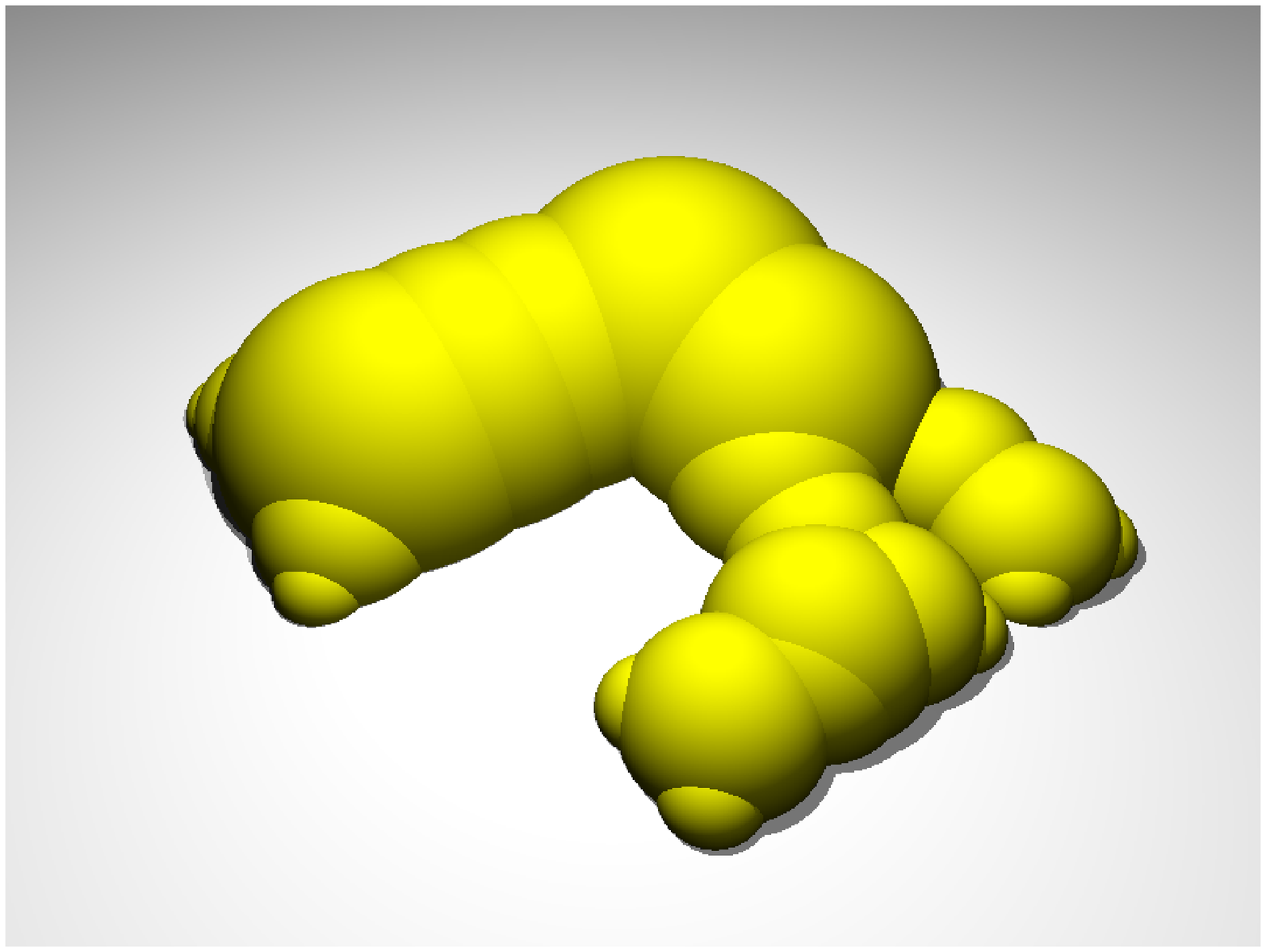}
 }
\vskip .25in
\centerline{ 
\includegraphics[height=2in]{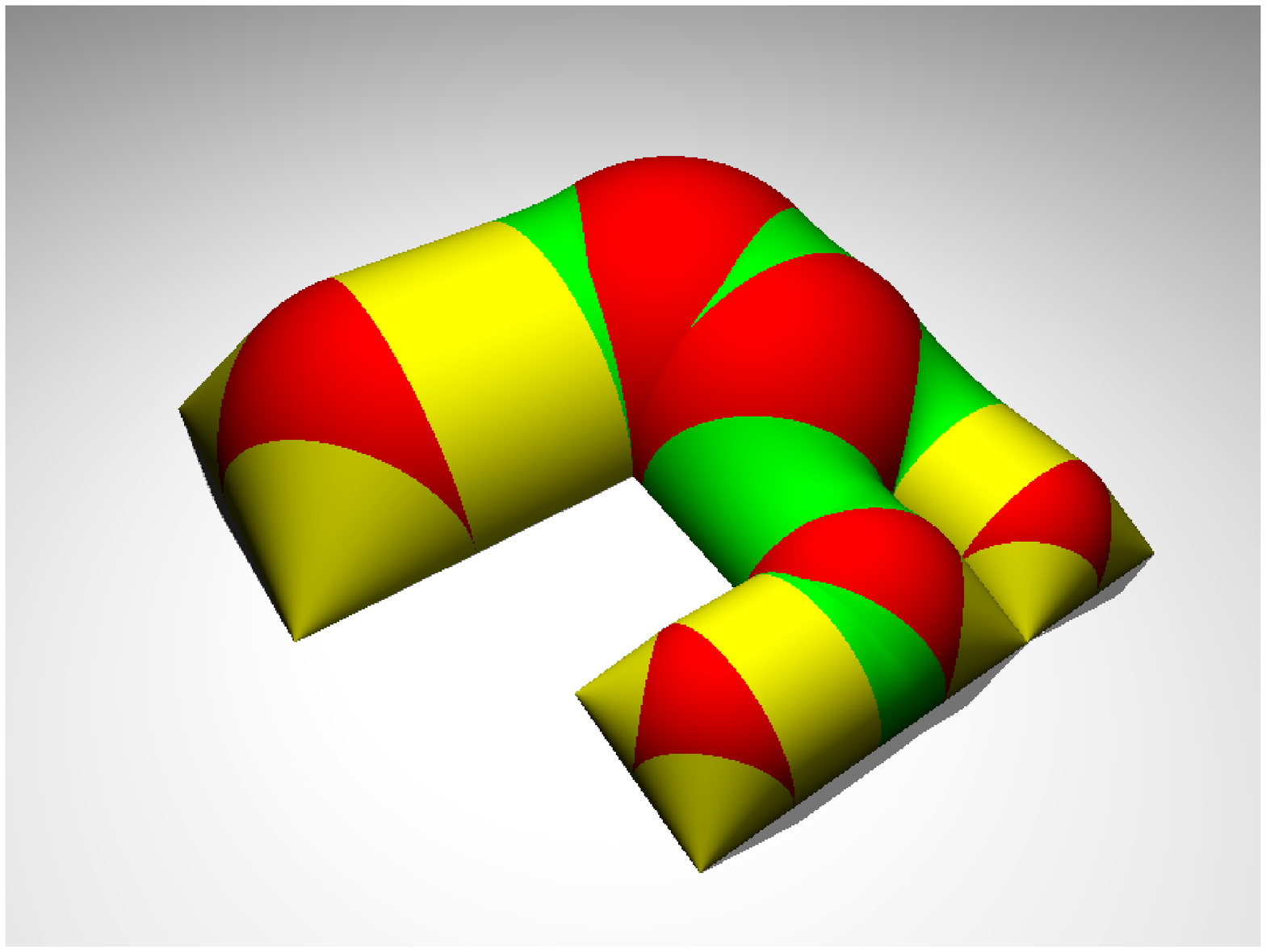}
$\hphantom{xx}$
\includegraphics[height=2in]{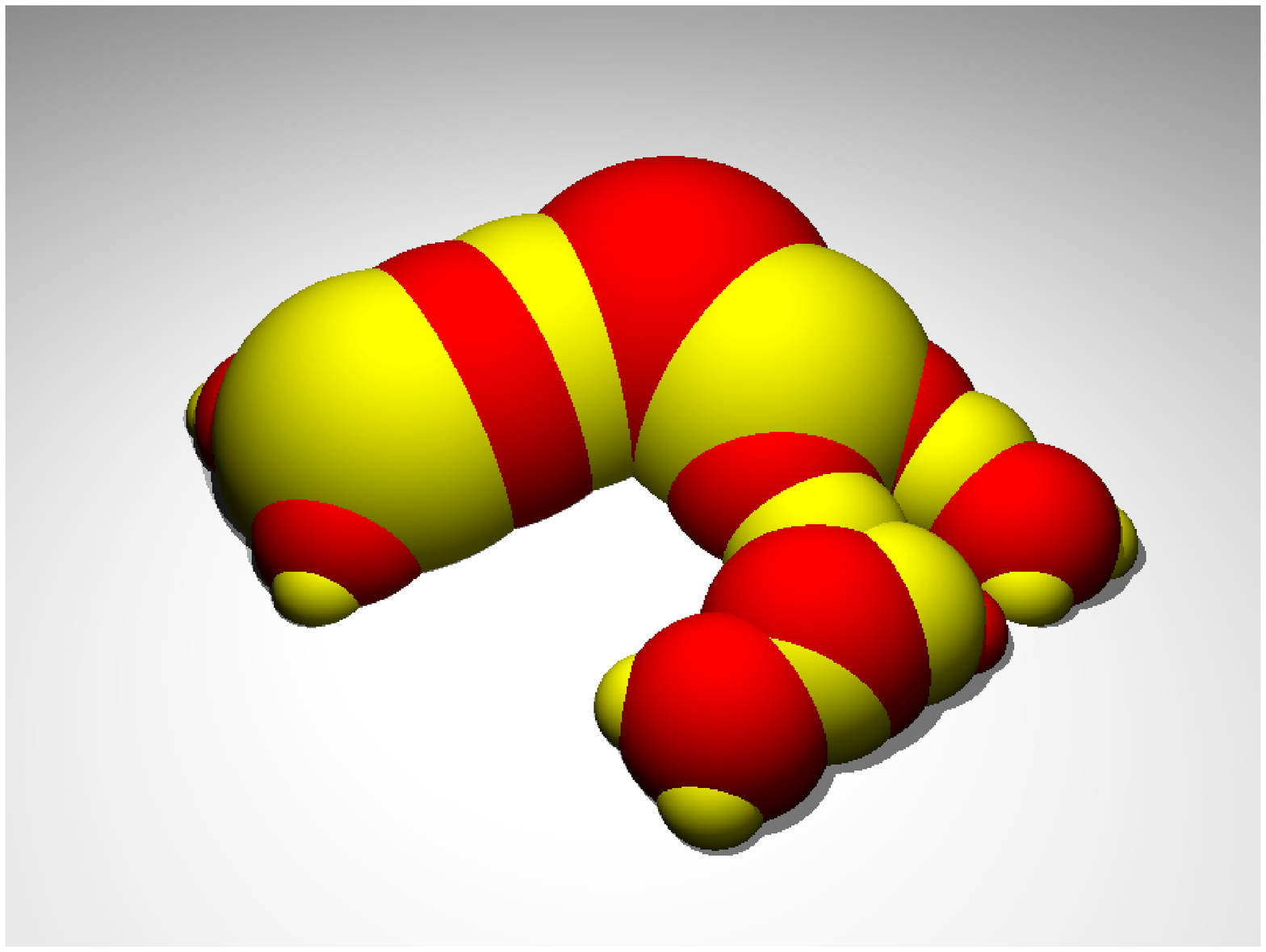}
 }
\caption{ \label{P_2dome}
On the top  left is the dome of the polygon 
$P_2 $ and on the top right is the dome of the finitely 
bent approximation $\Omega_2$.
Below each, we have redrawn the domes, but with different 
sections shaded differently. For $P_2$, 
 regions corresponding to different edges 
of the medial axis colored differently. 
For $\Omega_2$ the dome is a union of geodesic faces (which 
form the vertices of a tree) and
adjacent faces are shaded in alternating  colors.}
\end{figure}

The medial axis is a fundamental concept of geometry that 
seems to have been rediscovered many times and 
goes by several names: medial axis, skeleton, symmetry set, 
 cut locus (defined as the closure of the 
medial axis in \cite{Wolter}), equidistant set, 
 ridge set (think of an island where the elevation 
is proportional to the distance to the sea), wildfire set
(think of a fire started simultaneously along the boundary 
that burns inward at a constant rate).
The earliest reference I am aware of is a 1945 
paper of Erd{\"o}s \cite{Erdos}, where he proves the medial axis
(he calls it ``$M_2$'') of a planar domain has 
Hausdorff dimension $1$.

 In some parts of the 
literature the medial axis is  confused with the set of centers of maximal 
disks in $\Omega$, that,  following \cite{Fremlin},  we will call  the 
central set of $\Omega$. For polygons the two sets are the same, 
but in general they are not (e.g., the parabolic 
region $\Omega=\{ (x,y): y > x^2\}$ contains a maximal disk  
 that is only tangent at the origin).  
More dramatically, the medial axis of a planar domain always has $\sigma$-finite 
$1$-dimensional measure \cite{Fremlin}, but the central set can have 
Hausdorff dimension 2, \cite{Bishop-Hakobyan}.
Some papers in the mathematical literature that deal with 
the medial axis include 
\cite{Banchoff-Giblin87},
\cite{Bruce-Giblin-Gibson},
\cite{Duan-Rees},
  \cite{Evans-Harris},
 \cite{Giblin-00},
 \cite{Giblin-OShea},
  \cite{Hormander}, 
  \cite{Milman}, 
\cite{Milman-Waksman},
 \cite{Thom}.

In the computer science literature the medial axis is 
credited to  Blum who introduced it to describe
biological shapes \cite{Blum67}, \cite{Blum73}, \cite{Blum-Nagel}.
A few papers consider the theory of the medial axis
(e.g., \cite{Choi-Choi-Moon97},
 \cite{Choi-Seidel01}, \cite{Choi-Seidel02},
 \cite{Choi-Lee},
 \cite{SPW96},
 \cite{Wolter}), but most deal with algorithms for
computing it and with applications to areas like
pattern recognition, robotic motion, control of cutting tools, sphere
packing and  mesh generation.
A sample 
 of such papers includes:
\cite{Chazal-Lieutier04A},
\cite{Chazal-Soefflet},
\cite{Chiang-Hoffmann},
\cite{CKM99},
\cite{Evans-M-M98},
\cite{Gaudeau-Boiron79},
\cite{Patrikalkis},
\cite{Hoffmann92},
\cite{Hoffmann-Dutta},
\cite{propellers},
\cite{Lee82},
\cite{Lee-Drysdale81},
\cite{Lee-Horng},
\cite{MP93},
\cite{MP94},
\cite{PM-book},
\cite{Pottman-Wallner},
\cite{Preparata77},
\cite{SPB95},
\cite{SPB},
\cite{Wang00},
\cite{Wu99},
\cite{Yap87}.

Given a finite collection of disjoint  sets (called
sites), the corresponding Voronoi diagram  divides the plane
according to which site a point is closest to.  The medial
axis of a polygon $P$ is a Voronoi diagram for the interior of $P$
where the sites are
the complementary arcs in $P$ of the convex
vertices (i.e., interior angle $< \pi$) and distance is measured within $P$.
Equivalently, one can compute  the medial axis by 
taking the Voronoi diagram for the polygon with 
all edges and vertices as sites and then removing 
the cell boundaries that terminate at a  concave
vertex (one with angle $\leq \pi$). See Figure 
\ref{MAvsVD}.
 Thus the medial axis can be computed by using
algorithms for computing generalized Voronoi diagrams.
Voronoi diagrams  were defined by Voronoi in
\cite{Voronoi}, but go back at least to Dirichlet
\cite{Dirichlet} (indeed,
in the theory of Kleinian groups the
Voronoi cells of an orbit are called Dirichlet fundamental
domains).
 For more about Voronoi diagrams  see  e.g.,
\cite{Aurenhammer91},
\cite{Aurenhammer-Klein},
 \cite{BE92},
\cite{Fortune92}, \cite{Fortune97},
\cite{O'Rourke98},
\cite{Preparata-Shamos85}.
\begin{figure} [htbp]
\centerline{
\includegraphics[height=1.5in]{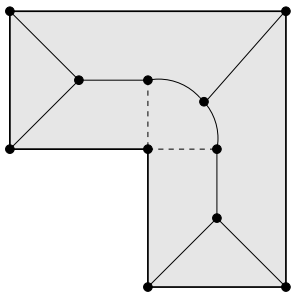}
}
\caption{\label{MAvsVD} This shows the Voronoi cells when all the 
edges and vertices are sites.  However, the dashed edges must 
 be removed to give the medial axis.  }
\end{figure}

It is a theorem of Chin, Snoeyink and Wang that the 
medial axis of a simple $n$-gon can be computed in 
$O(n)$ time.  I am not aware that their 
$O(n)$ algorithm has been implemented, since it depends
on the intricate algorithm of Chazelle that
triangulates polygons in linear time. However, other 
asymptotically slower methods (e.g., $O(n \log n)$) have 
been implemented, and in practice the computation of the 
medial axis in $\reals^2$ is not considered a ``bottleneck''.
See   \cite{Yao-Rokne91}, \cite{Yap87}.

The basic strategy of the linear time algorithm of
 Chin, Snoeyink and Wang is fairly
 simple (although the details are not): (1) decompose the 
polygon into simpler polygonal pieces called monotone
histograms using at most 
$O(n)$ new edges, (2) compute the Voronoi diagram for each 
piece with work $O(k)$ for a  piece with $k$ sides
 and finally (3)  merge the Voronoi diagrams of the pieces
using at most $O(n)$ work.

The first step is accomplished using a celebrated result of
 Chazelle \cite{Chazelle1991} that one can  
cut interior of $P$ into trapezoids with vertical sides in linear 
time (this is equivalent to triangulating the polygon in linear time).
 Klein and Lingas \cite{Klein-Lingas96} showed how to use 
Chazelle's result to cut a polygon into ``pseudo-normal 
histograms''; Chin, Snoeyink and Wang then show how to 
cut these into monotone histograms.

The next step is to show that the Voronoi diagram of a monotone
histogram can be computed in linear time.  The argument given 
in \cite{CSW-99} follows the elegant argument of Aggarwal, 
Guibas, Saxe and Shor for the case of convex domains.  In 
\cite{AGSS} the four authors  use duality to reduce the problem 
to finding the three dimensional convex hull of $n$ points 
whose vertical projections onto the plane are the vertices 
of a convex polygon.  

The final step is to merge the Voronoi diagrams of all 
the pieces. The merge lemma used in \cite{CSW-99} states:

\begin{lemma}
Let $Q$ be a polygon that is divided into two subpolygons
$Q_1$ and $Q_2$ by a diagonal $e$ (i.e., an line segment in 
$P$ whose endpoints are vertices of $P$). Let $S_1$ be a 
subset sites (vertices and edges) in $Q_1$ and $S_2$ a 
subset of sites in $Q_2$. Given the Voronoi diagrams for 
$S_1$ in $Q_1$ and for $S_2$ in $Q_2$, one can obtain the 
Voronoi diagram for $S = S_1 \cup S_2$ in $Q$ in time 
proportional to number of Voronoi edges for $S_1$ and $S_2$ 
that intersect $e$ and the number of new edges that are 
added.
\end{lemma}

This type of result was first 
 used by Shamos and Hoey \cite{Shamos-Hoey75} and has been
adapted by many authors since.


\section{The dome is the disk } \label{Domes+MA2}

 The two main results about the 
dome of $\Omega$ say that (1) it  is isometric to the hyperbolic disk
and (2) it is ``almost isometric'' to the base domain $\Omega$.
More precisely, equip the dome with the hyperbolic path metric
$\rho_S$
(shortest hyperbolic length of a path connecting two points and 
staying on the surface).

\begin{thm}[Thurston, \cite{Thurston}] \label{Dome=disk}
Suppose $\Omega$ is a simply connected plane domain (other 
than the whole plane or  the complement of a circular arc) and let $S$ be 
its dome.
Then 
$(S,\rho_S)$ is isometric to the hyperbolic unit disk. We will 
denote the isometry by $\iota: S \to \disk$.
\end{thm}

\begin{thm}[Sullivan \cite{Sullivan81}, Epstein-Marden \cite{EM87}]
\label{SEM}
Suppose $\Omega$ is a simply connected plane domain (other than than 
the whole plane or the complement of a circular arc).
There is a $K$-quasiconformal map $\sigma: \Omega \to S$ that 
extends continuously to the identity on the boundary  ($K$ is independent 
of $\Omega$).
\end{thm}

In fact, there is a biLipschitz map between 
$\Omega$ and its 
dome (each with their hyperbolic metric; see Theorem \ref{QC-bdy}),
but we will only use the 
quasiconformal version of the result.
We place the additional restriction that $\Omega$ is not the
complement of a circular arc because in that case the convex 
hull of $\partial \Omega$ is a hyperbolic half-plane and the dome 
should be interpreted as two copies of this  half-plane joined along its edge 
with  bending angle  $\pi$. In order to simplify the discussion here,
we simply omit this case (with the correct interpretations the results
above still  hold in this case; this is discussed in complete detail
in Section 5 of \cite{EMM-earthquakes}). 

Explicit estimates of the constant in the
Sullivan-Epstein-Marden 
theorem are given elsewhere in the literature.
For example, it is proven in \cite{Bishop-ExpSullivan} that one 
can take $K=7.82$.  The estimates $K \approx 80$ and $ K 
\leq  13.88$ are given in \cite{EM87} and \cite{EMM-earthquakes} respectively.

Although we will not use it here, it is worth noting that both 
these theorems have their origin in the theory hyperbolic
of $3$-manifolds. Such a manifold $M$ is a quotient of the hyperbolic 
half-space, $\uhs$, by a discrete group $G$ of isometries. The orbit 
of any point under this group accumulates only on the boundary 
of the half-space and the accumulation set (which is independent 
of the orbit except in trivial cases) is called the limit set $\Lambda$. The complement 
$\Omega$ of $\Lambda$ in the boundary of hyperbolic space
 is called the ordinary set. The group 
$G$ acts discontinuously on $\Omega$ and $\partial_\infty M = \Omega/G$ 
is called the ``boundary at infinity'' of $M$. This is a
Riemann surface (possibly with branch points).  The manifold
 $M$ contains 
 closed geodesics and the closed convex hull of these is called
the convex core of $M$ and denoted $C(M)$. 
The lift of the convex core to $\uhs$ 
is the hyperbolic convex hull of the limit set and 
its boundary is the dome of the ordinary set. Thus $\partial C(M)$ 
is just the quotient of this dome by the group $G$.
Theorem \ref{Dome=disk}
implies that the boundary of $C(M)$ is a surface of 
constant negative curvature, i.e., is
 isomorphic to the hyperbolic disk modulo a 
group of isometries.  Theorem \ref{SEM} says that $\partial_\infty M$
and $\partial C(M)$ are homeomorphic, indeed, are quasiconformal
images of each other with respect to their hyperbolic metrics. 
This fact was needed in the proof  of Thurston's hyperbolization
theorem for 3-manifolds that fiber over the circle.


The proof of Theorem \ref{Dome=disk} for finitely bent domains simply 
consists of observing that if we deform the dome by bending it 
along a bending geodesic, we don't change the path  metric at all.
Moreover, a finite number of such deformations converts a finitely 
bent  
dome into a hemisphere, and this is obviously isomorphic to the 
hyperbolic disk.  More precisely, we are using the following 
simple lemma.

\begin{lemma}
Suppose two surfaces $S_1, S_2$ in $\uhs$
 are joined along a infinite hyperbolic 
geodesic and suppose $\sigma $ is an elliptic M{\"o}bius transformation 
of $\uhs$ that fixes this geodesic. Then  a map to another 
surface that equals
 the identity  on $S_1$ and equals  $\sigma$ on $S_2$ is an isometry between 
the path metric on $S_1 \cup S_2$ and the path metric on the image. 
\end{lemma}

\begin{proof}
This becomes  obvious is one normalizes so that the 
geodesic in question becomes a vertical line and $\sigma$ 
becomes a (Euclidean) rotation around it, since it is then
clear that the length of any path is left unchanged.
\end{proof}

\begin{figure}[htbp] \label{rotate}
\centerline{ 
\includegraphics[height=1.5in]{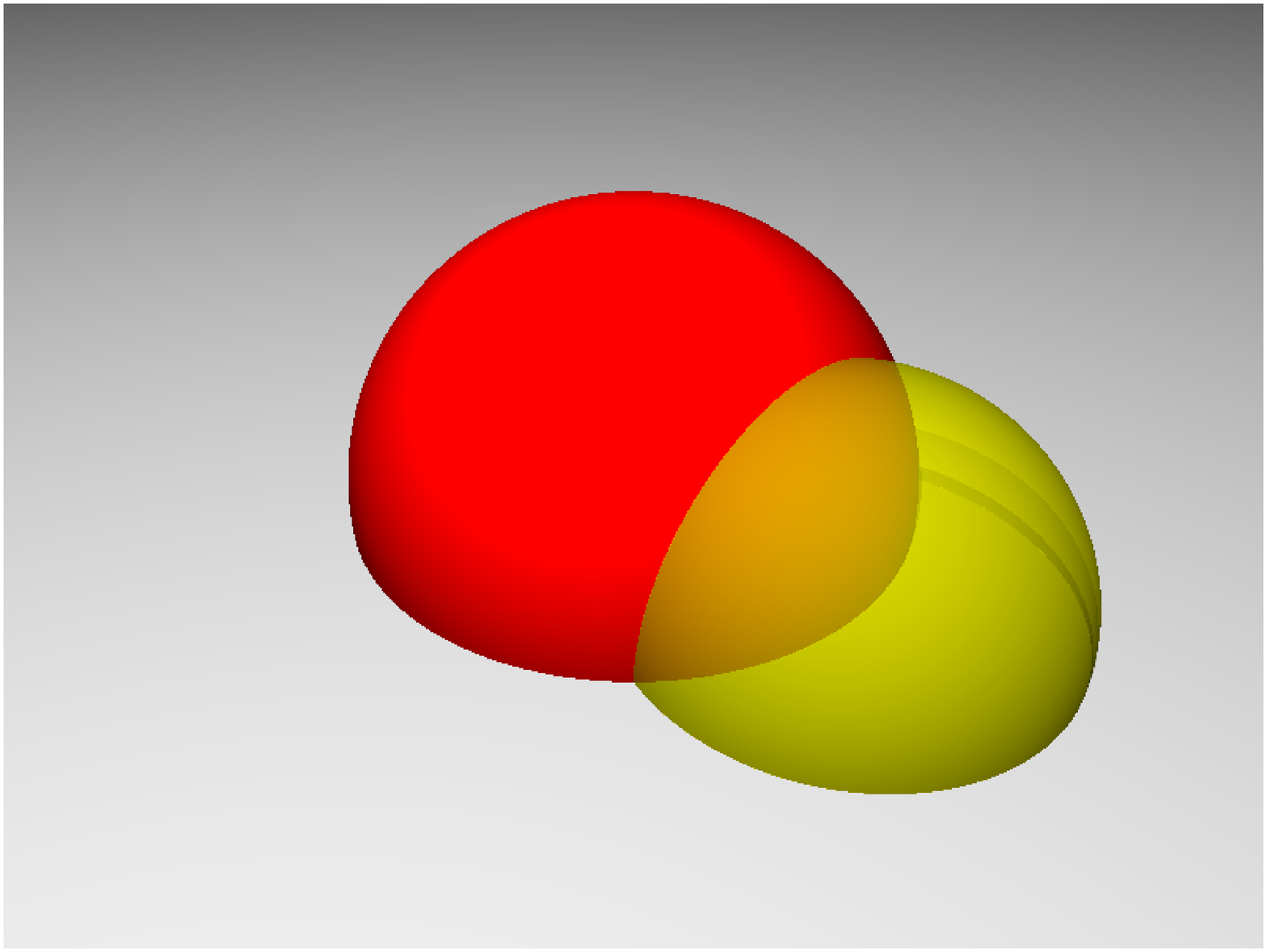}
$\hphantom{xxx}$
\includegraphics[height=1.5in]{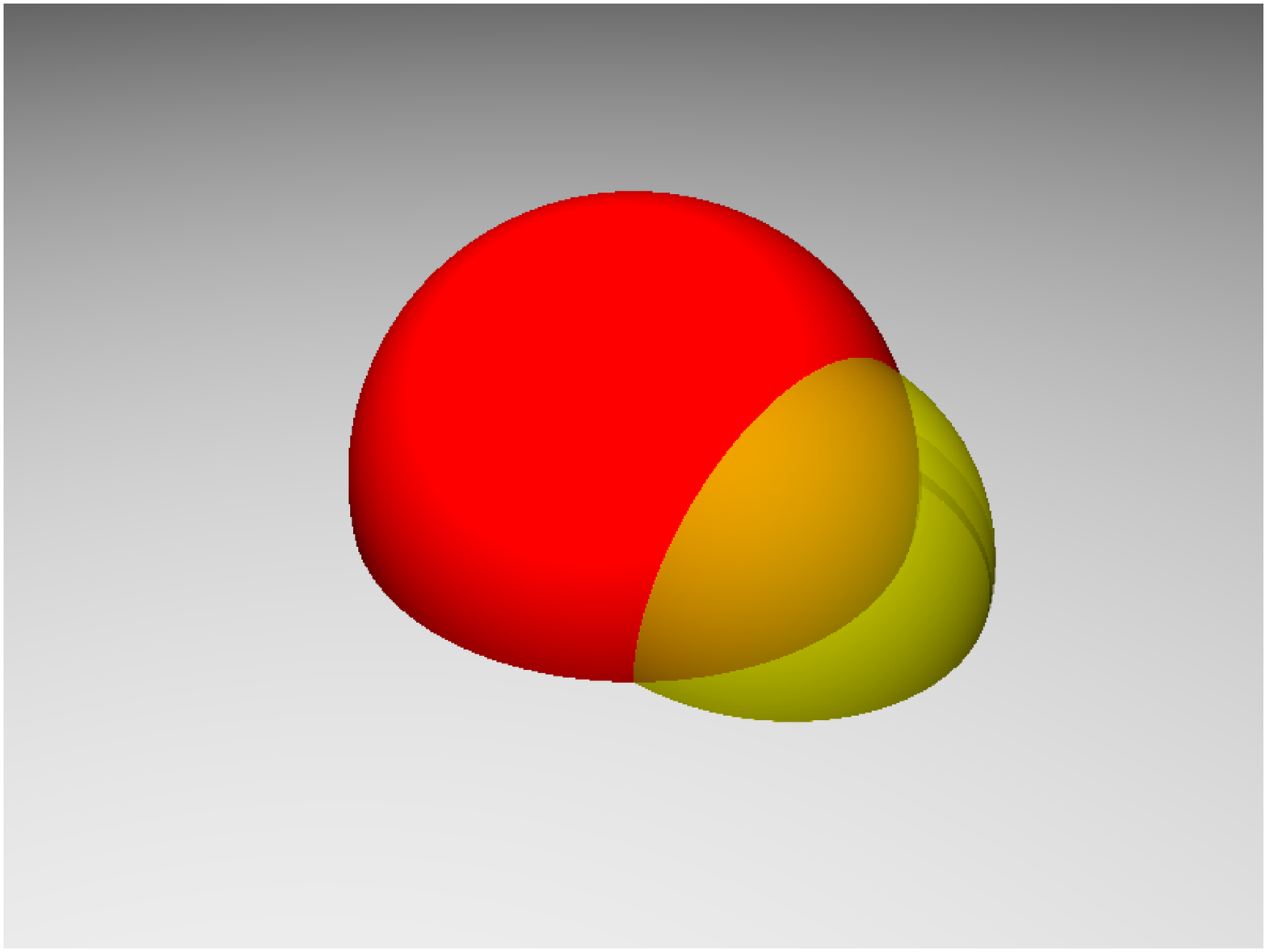}
}
$\hphantom{xxx}$
\centerline{
\includegraphics[height=1.5in]{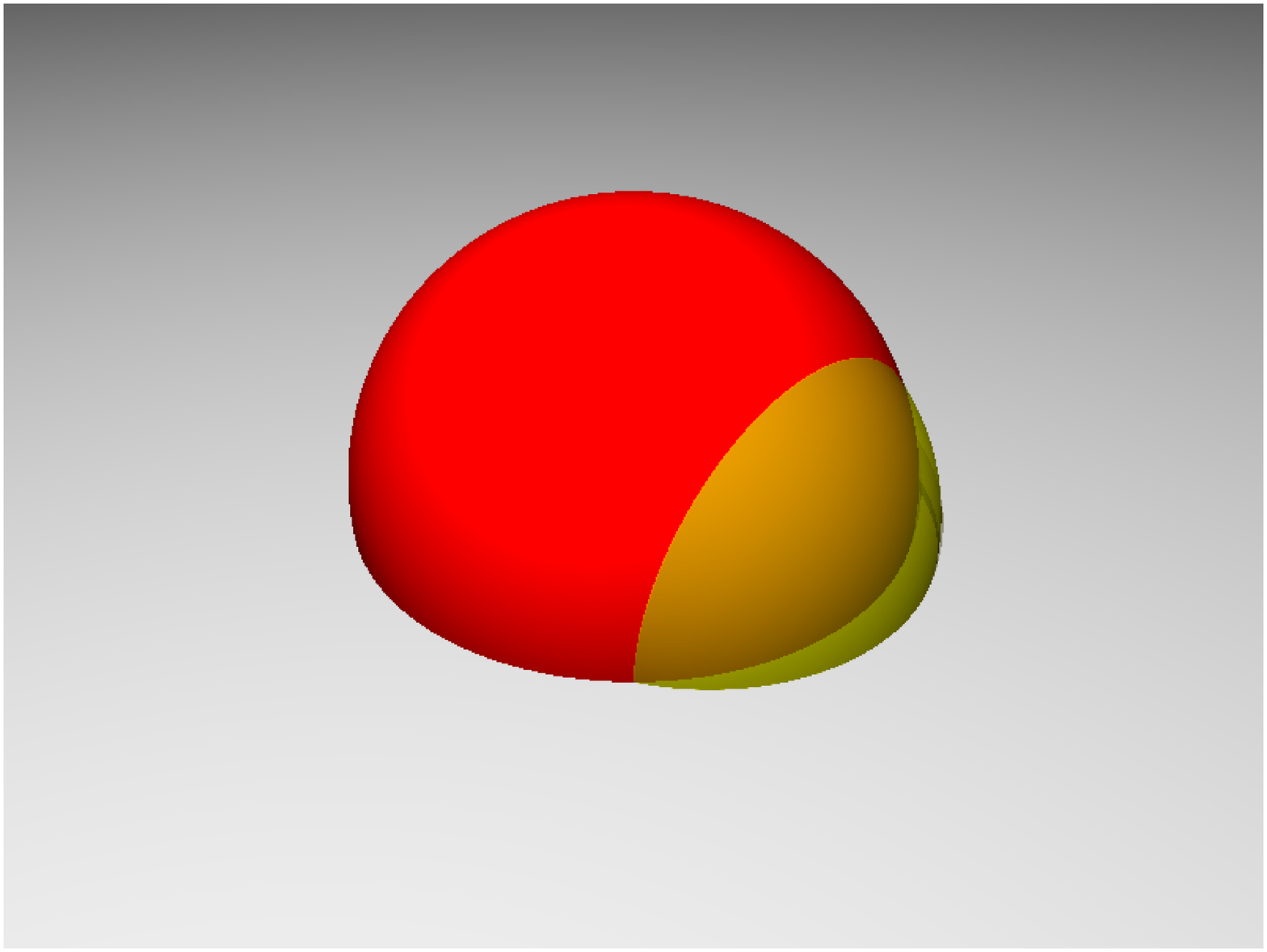}
$\hphantom{xxx}$
\includegraphics[height=1.5in]{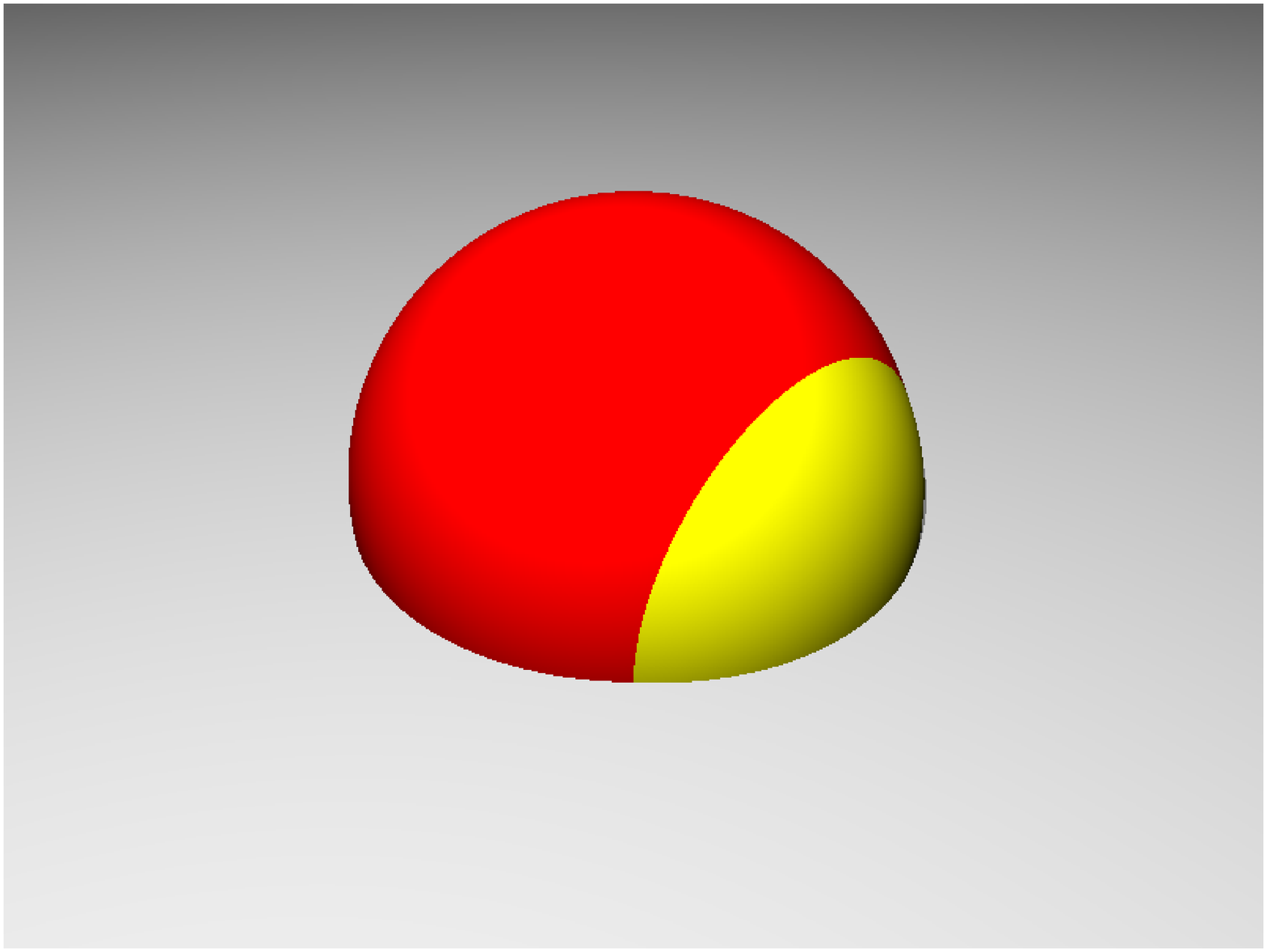}
}
\caption{  A dome consisting of two geodesic faces joined along 
an infinite geodesic.  By bending the dome along the geodesic we 
get a one-parameter,  isometric family of surfaces ending with 
 a hemisphere, which is obviously isometric to 
the hyperbolic disk.}
\end{figure}

Theorem \ref{Dome=disk} then follows by taking a finitely bent surface and 
``unbending'' it one geodesic at a time, i.e., we can map it to 
a hemisphere by a series of maps, each of which is an isometry 
by the lemma. Since a hemisphere is isometric to the disk, we are done.
In Figure \ref{rotate} we illustrate the bending along a geodesic 
for a dome with two faces.

This proof gives us a  geometric interpretation of the map 
$\iota : \partial \Omega \to \partial \disk$. The 
disks making up a finitely bent domain have a tree structure and
if $\Omega$ is finitely bent then we fix a root disk  $D_0$ and 
write $\Omega = D_0 \cup_j D_j \setminus D_j^*$, where $D_j^*$ 
denotes the parent disk of $D_j$.  This gives $\Omega \setminus D_0$ 
as a union of crescents.
See Figure \ref{cd2lines}.
We call these ``tangential'' crescents since  one edge of the 
crescent follows $\partial \Omega$ near each vertex
 (and to differentiate them 
from the ``normal'' crescents we will introduce later).

Each crescent in the tangential crescent decomposition
 has an ``inner edge''
(the one in the boundary of $D_j^*$) and an ``outer edge''
(the other one) and there is a unique elliptic M{\"o}bius
transformation that maps the outer edge to the inner one, 
fixing the two vertices of the crescent (this is just the 
restriction to the plane of the M{\"o}bius transformation of 
$\uhs$ that removes the bending along the corresponding 
bending geodesic). The map $\iota: 
\partial \Omega \to \partial \disk$ is the composition of these
maps along a path of crescents that connects an arc on 
$\partial \Omega$ to an arc on $\partial \disk$. 

An alternate way to think of this is to foliate each crescent 
$D_j \setminus D_j^*$ by circular arcs that are orthogonal to 
both boundary arcs. This gives a foliation of $\Omega \setminus 
D_0$ by piecewise circular curves that connect $x \in \partial 
\Omega$ to $\iota(x) \in \partial \disk$.
On the left of Figure \ref{cd2lines} 
we have sketched the foliation  in each of the 
crescents for a particular finitely bent domain  (but without 
attempting to line up the leaves in different crescents) and on the 
right we have plotted the trajectories of a couple of boundary points
that correspond to the vertices of the polygon we have 
approximated. This is the description given in the introduction.
Some further examples are illustrated in Figure \ref{more flows}.

\begin{figure}[htbp]
\centerline{
 \includegraphics[height=2.5in]{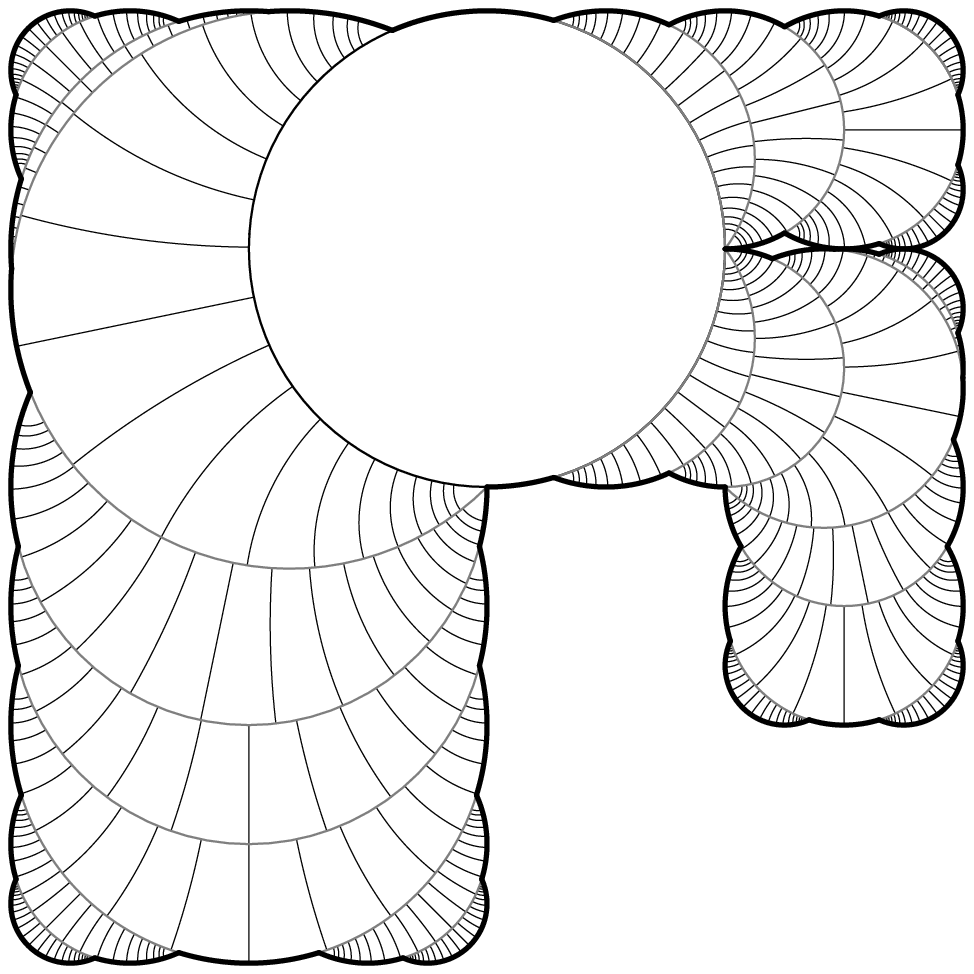}
$\hphantom{xxxxx}$
 \includegraphics[height=2.5in]{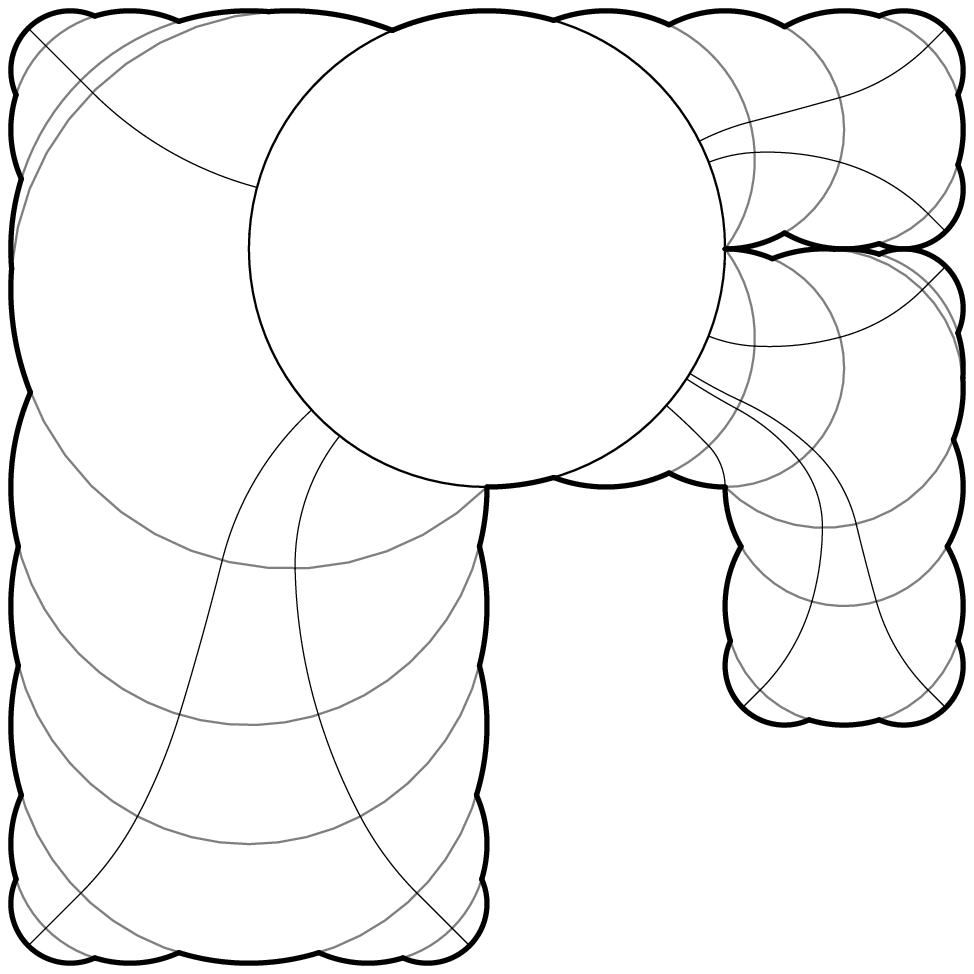}
}
\caption { \label{cd2lines}
             On the left is the foliation by orthogonal arcs in the
            tangential crescents.
	     On the right we start at the 
           vertices on the boundary of $\Omega_2$ follow 
          the corresponding trajectories of the vertices. Where
           these trajectories land on the circle are the $\iota$
           images  of the vertices.
	   }
\end{figure}

\begin{figure}[htbp]
\centerline{
 \includegraphics[height=2.5in]{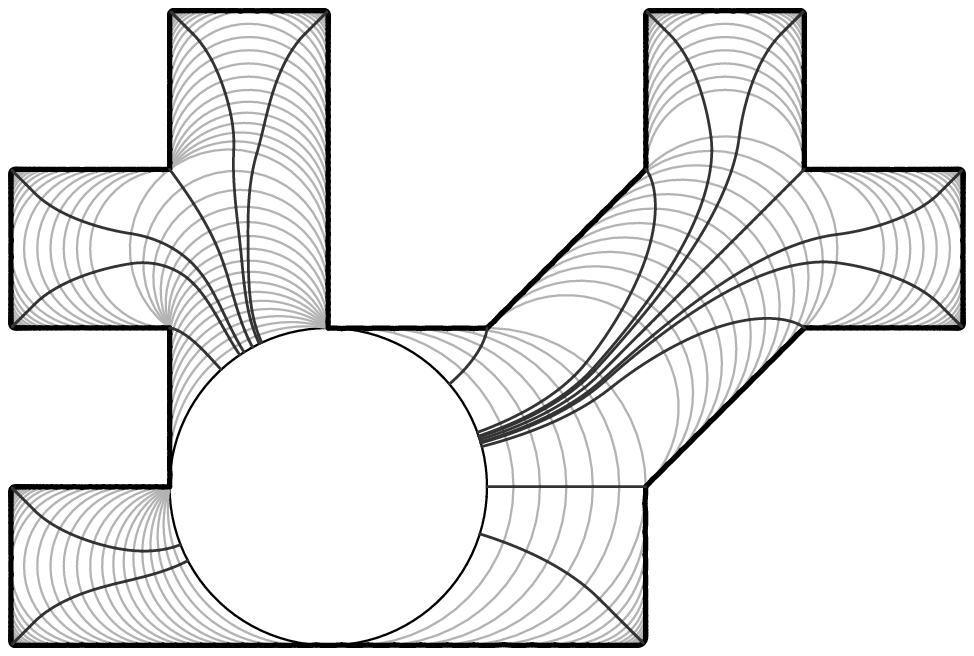}
  $\hphantom{xxxx}$ 
 \includegraphics[height=2.5in]{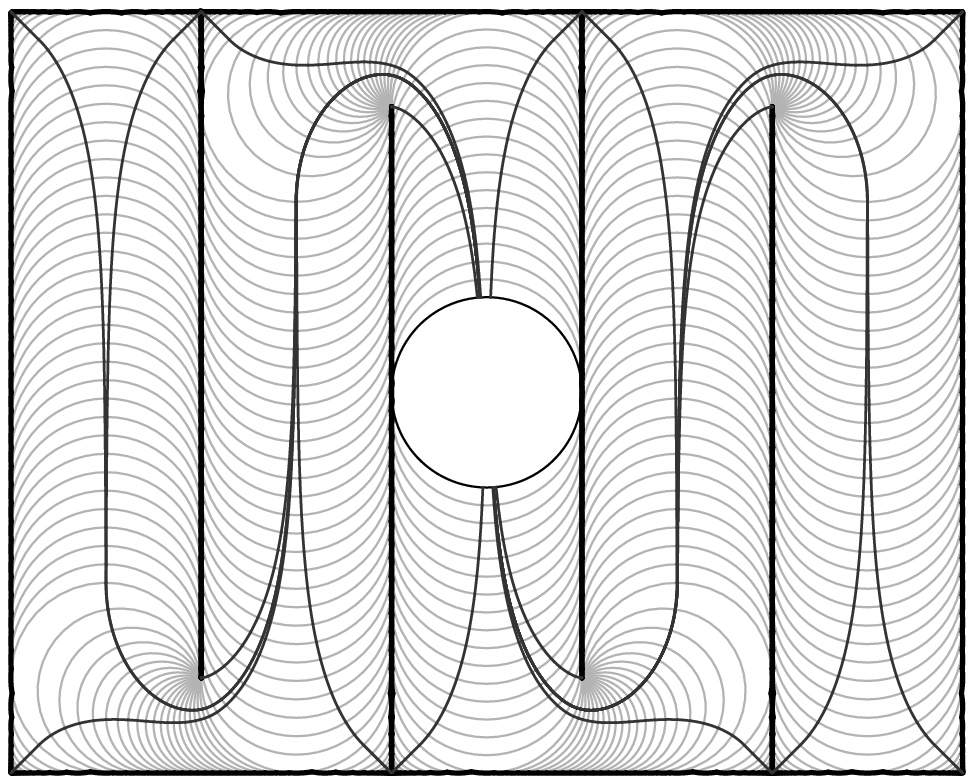}
}
  \caption{ \label{more flows} 
   The medial axis flow for two more polygons which have 
    been approximated by unions of medial axis disks. This 
     flow defines the iota map from $\partial \Omega$ to 
     the chosen root disk of the medial axis.
   } 
\end{figure}

Theorem \ref{SEM}  implies that 
the mapping $\iota: \partial \Omega \to \partial \disk$ has 
a quasiconformal extension to a map $\Omega \to \disk$
 that is $K$-quasiconformal 
with a bound $K$, independent of $\Omega$. Thus the 
geometric map we have described above is a rough approximation 
to the boundary values of the Riemann map.
  It is surprising (at least to the 
author) that there is such a simple, geometrically defined map 
that is  close to the Riemann map with estimates independent 
of the domain.


\section{The nearest point retraction and normal crescents}
 \label{gaps-crescents}

In the previous section we defined the map $\iota$ and interpreted 
it geometrically by  collapsing tangential crescents. In this section 
we will interpret $\iota$ as collapsing crescents from 
 a different decomposition of $\Omega$ that more closely 
approximates the geometry of the dome.
 

Recall that $S$ is the boundary 
of a convex set in $\uhs$,  so that the nearest point retraction 
defines a Lipschitz map of the complement of this set onto 
its boundary. This map can 
be extended to $\Omega \subset \partial \uhs = \reals^2$  as 
follows:  given a point $z \in \Omega$, 
define nearest point retraction $R : \Omega \to S$ by
expanding horoball tangent at $z \in \Omega$ until
it first hits $S$ at $R(z)$ (a horoball in $\uhs$ is a
Euclidean  ball tangent to the boundary). See Figure \ref{retract}.

\begin{figure}[htbp]
\centerline{
	\includegraphics[height=1.5in]{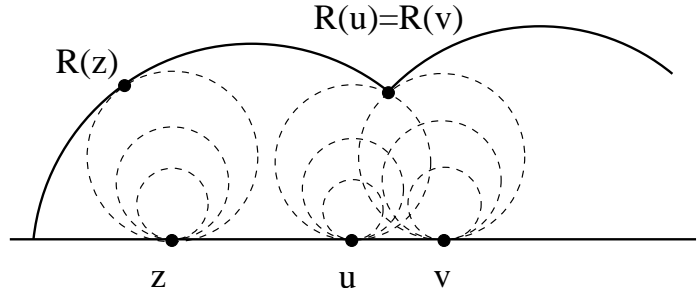}
	}
\caption{ \label{retract} Defining the retraction map $R : \Omega \to S$:
expand a sphere tangent at $z$ until it touches $S$ at $R(z)$.
This map need not be 1-1.}
\end{figure}

Note that the map need not be 1-to-1, i.e., two points in $\Omega$
 can map to the same point on the dome.
Thus it can't always be quasiconformal or even
be a homeomorphism. However, it is always a quasi-isometry
with bounds independent of $\Omega$ and this implies that 
there is a quasiconformal map from $\Omega$ to its
dome with the same boundary values by Theorem \ref{QC-bdy}.
This implies Theorem \ref{SEM}, 
e.g., see \cite{Bishop-Bowen}. 
Moreover, $R$ is quasiconformal in 
some special cases; e.g., Epstein, Marden and Markovic prove in 
\cite{EMM-convexregions} that for Euclidean convex domains the retraction
map is $2$-quasiconformal.

This map is called the nearest point retraction because it is the 
continuous extension to the boundary of the map in $\uhs$ that sends
a point to the nearest  point of $S$ in the hyperbolic metric, 
$\rho_{\uhs}$. 
See Appendix \ref{background} for the definition of the hyperbolic metric 
on $\disk$ and $\uhs$.   The surface 
$S$ has an important related metric, $\rho_S$. This is the hyperbolic 
path metric on $S$ defined by taking the shortest hyperbolic 
length of all paths that connect two points and stay on $S$.
Clearly $\rho_{\uhs} |_S \leq \rho_S$. The base domain of $\Omega$ 
has its own hyperbolic metric, $\rho_\Omega$, obtained by transporting the 
hyperbolic metric on $\disk$ by any conformal map.

The nearest retraction map $R $ is $C$-Lipschitz from 
$\rho_\Omega$ to $\rho_S$ for 
some $C < \infty$ (e.g., see \cite{Bishop-ExpSullivan}).
  It is easy to prove this for some 
$C$; the  sharp estimate of $C=2$
 is given in \cite{EMM-QH+CHB}
and earlier results are given in
 \cite{Bridgeman-Canary}, \cite{Canary2001}, \cite{EM87}.

Now suppose $\Omega$ is a finitely bent domain.
Then the dome $S$ of $\Omega$ is a finite union of 
geodesic faces. On the interior of each face 
the retraction map has a well defined inverse and
the images of the faces under $R^{-1}$ are called the 
``gaps''. 
The inverse images of the bending geodesics 
are crescents that separate the gaps.
 These are called ``normal crescents'' since their two
boundary arcs are perpendicular to the two arcs of 
$\partial \Omega$ that meet at the common vertex.
Therefore, we will call this decomposition of $\Omega$ 
the ``normal crescent decomposition''.
Refer back to Figure \ref{gap-cres-dome}; that picture shows a
polygon, a finitely bent approximation, the normal  crescent 
decomposition and the dome.
See Figure \ref{GC} for  more examples of gap/crescent  decompositions.

\begin{figure}[htbp]
\centerline{
   \includegraphics[height=2.0in]{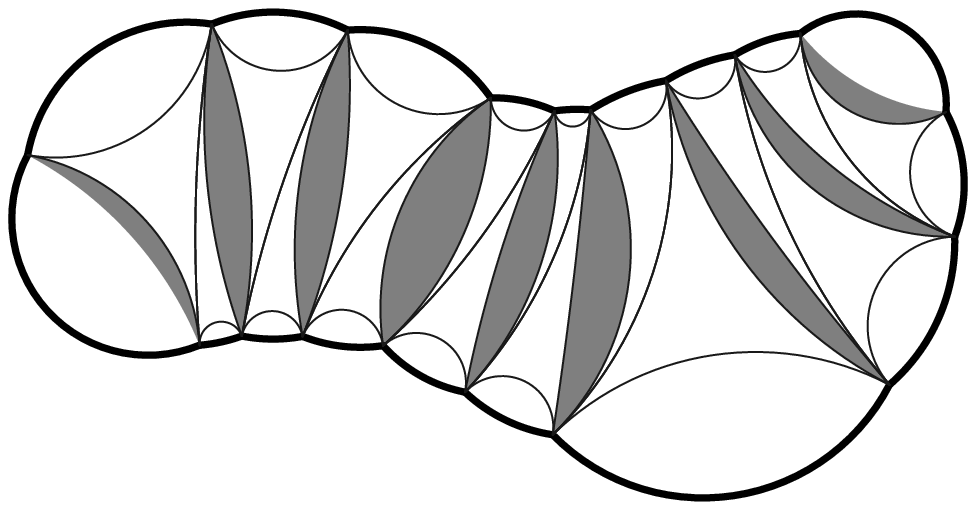}
      $\hphantom{xxx}$
   \includegraphics[height=2.0in]{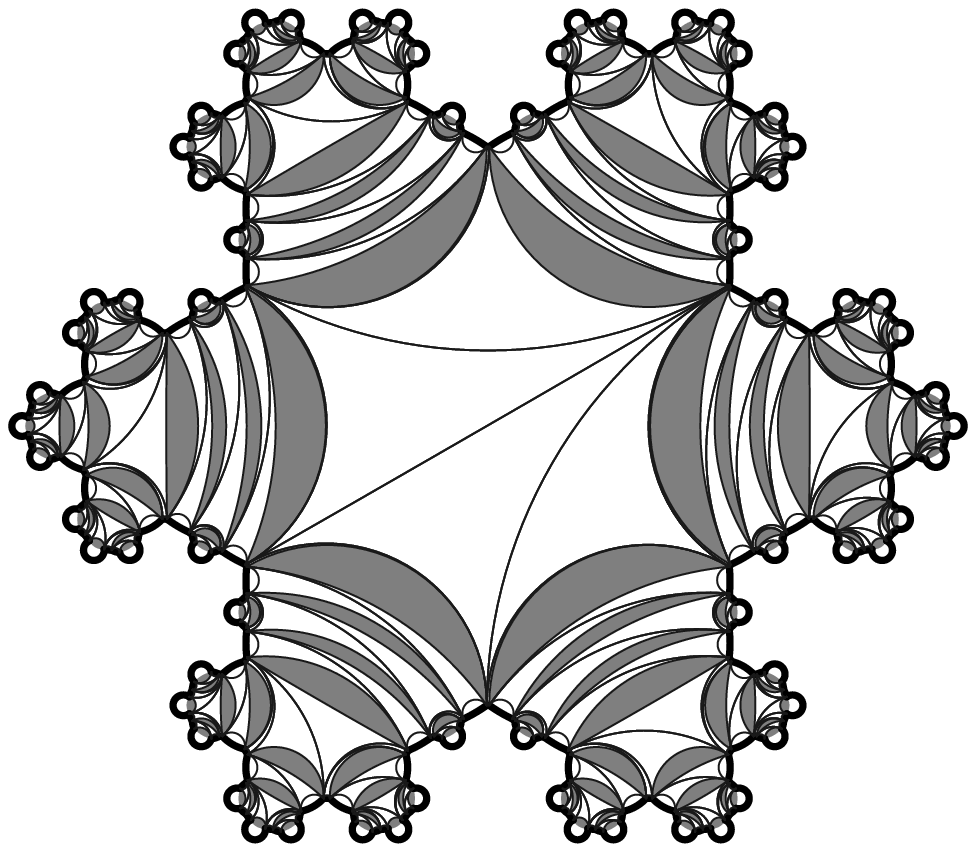}
}
\centerline{
 \includegraphics[height=2.0in]{cd8t=10.ps}
      $\hphantom{xxx}$
   \includegraphics[height=2.0in]{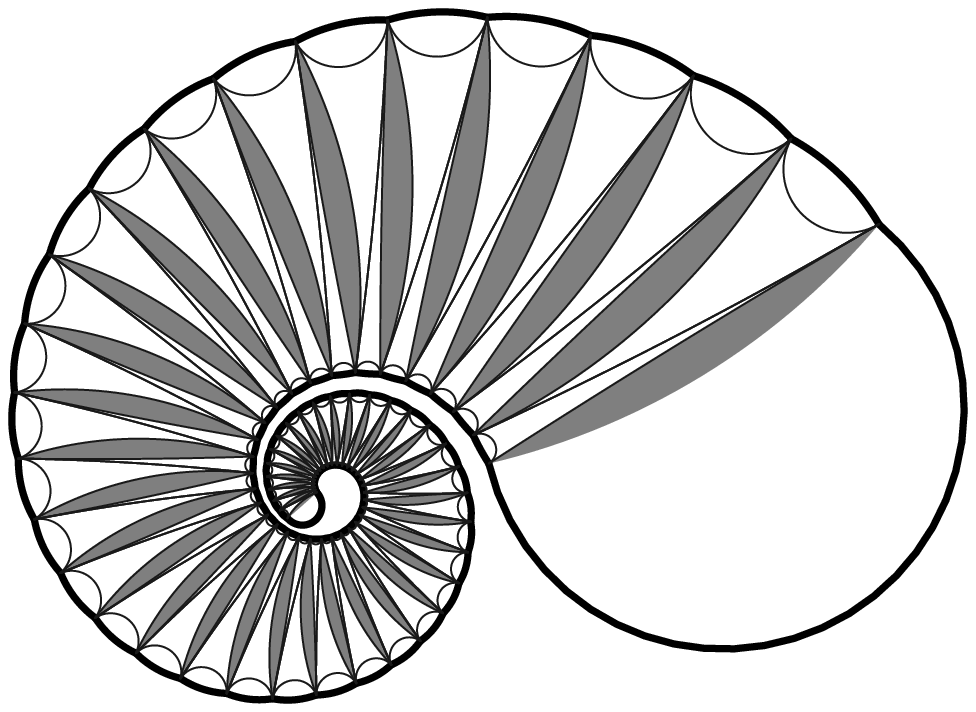}
}
\caption{ \label{GC} 
Normal crescent decompositions for some finitely bent domains.
Also drawn are  arcs triangulating the gaps. These are added 
to make the bending lamination complete (see Section \ref{lamin=linear}).
}
\end{figure}

 If a gap $G$ corresponds to a face $F \subset S$
then $G \subset D$, the disk in $\Omega$ that 
is the base of the hyperplane containing the face $F$.
We will call $D$ the ``base disk'' of $G$
Moreover, $G$ is the hyperbolic convex hull in $D$ of the 
set where $F$ meets $\partial \Omega$.
The angle of a normal crescent $C$ is the same 
as the angle made by the faces of the dome that meet at the 
corresponding bending geodesic.  
 $C$ is foliated by circular arcs that are orthogonal 
to both boundary arcs and each of these arcs is collapsed
to single point by $R$. Thus for a finitely bent domain $\Omega$, 
$R$ will never be a homeomorphism (unless $\Omega$ is a 
disk).

The two vertices of each normal crescent are also 
the vertices of a  crescent in the tangential 
crescent decomposition of $\Omega$. Moreover, corresponding 
crescents from the two decompositions have the same angle, 
and hence are simply images of each other by a $\pi/2$ 
elliptic rotation around the two common vertices. 
See Figure \ref{2-decoms}.
Collapsing the two types of crescents simply gives the two different 
continuous extensions to the interior  of the same map on the 
boundary (namely $\iota$).

\begin{figure}[htbp] 
\centerline{
 \includegraphics[height=2in]{cd2lines.ps}
$\hphantom{xx}$
 \includegraphics[height=2in]{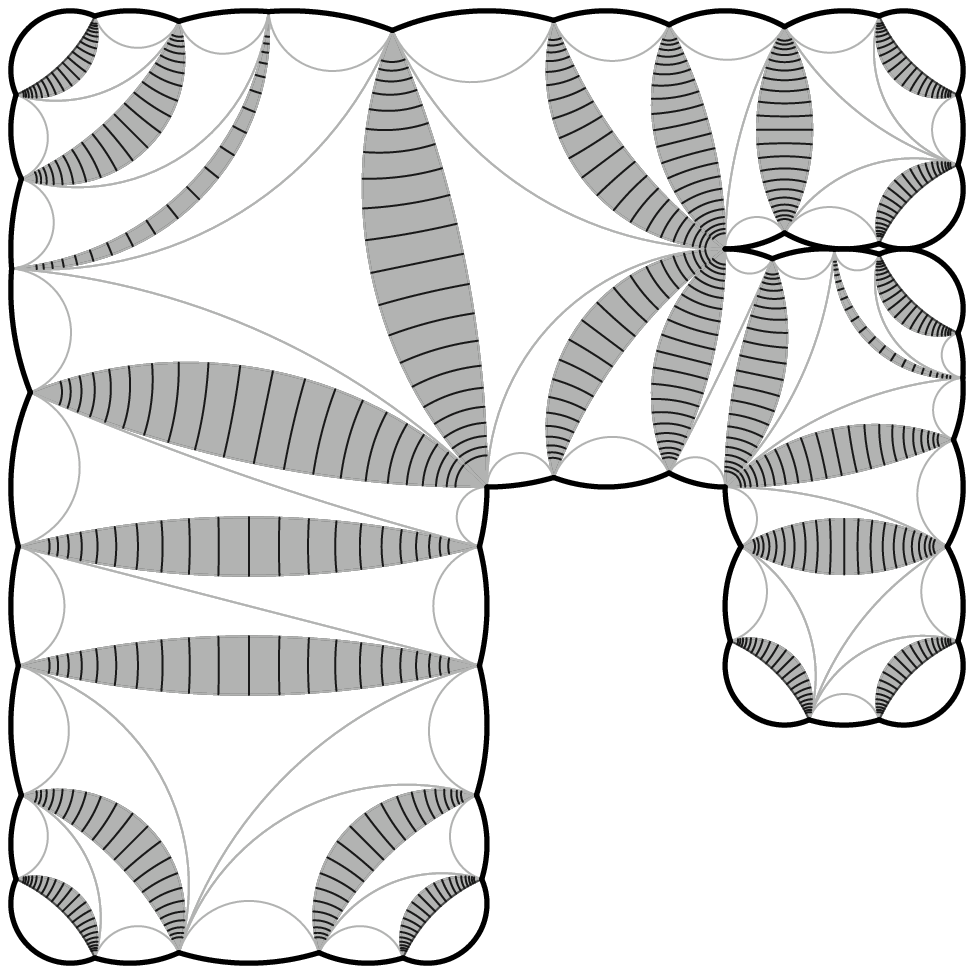}
 }
\caption{\label{2-decoms} 
The tangential and normal crescent decomposition for a domain.
There is a 1-to-1 correspondence between crescents in the two 
pictures; corresponding crescents have the same vertices and same
angle, but are rotated by $\pi/2$.}
\end{figure}

Both decompositions cut $\Omega$ into a ``disk'' and a union 
of crescents. In the tangential decomposition, it is 
a single connected disk, but in the normal decomposition
the disk itself is broken into pieces called the gaps.
The map $\varphi= \iota \circ R: \Omega \to S \to \disk$
 is M{\"o}bius on each gap and collapses every crescent to 
a hyperbolic geodesic in $\disk$, thus the disk is written 
as a union of M{\"o}bius  images of gaps. For example, see
Figures \ref{fol-square}.  
The picture on the left shows a normal crescent decomposition
of a square and on the right are the $\varphi$ images of the 
gaps in the disk.  The images of the crescents is a finite 
union of geodesics that is called the ``bending lamination''
of $\Omega$.  If we record the angle of each crescent and 
assign it to the corresponding geodesic in the bending 
lamination, then we get a ``measured lamination'', and this 
data is enough to recover $\Omega$, up to a M{\"o}bius image.
We will discuss laminations further in Section \ref{lamin=linear}.

\begin{figure}[htbp]
\centerline{ 
             \includegraphics[height=2in]{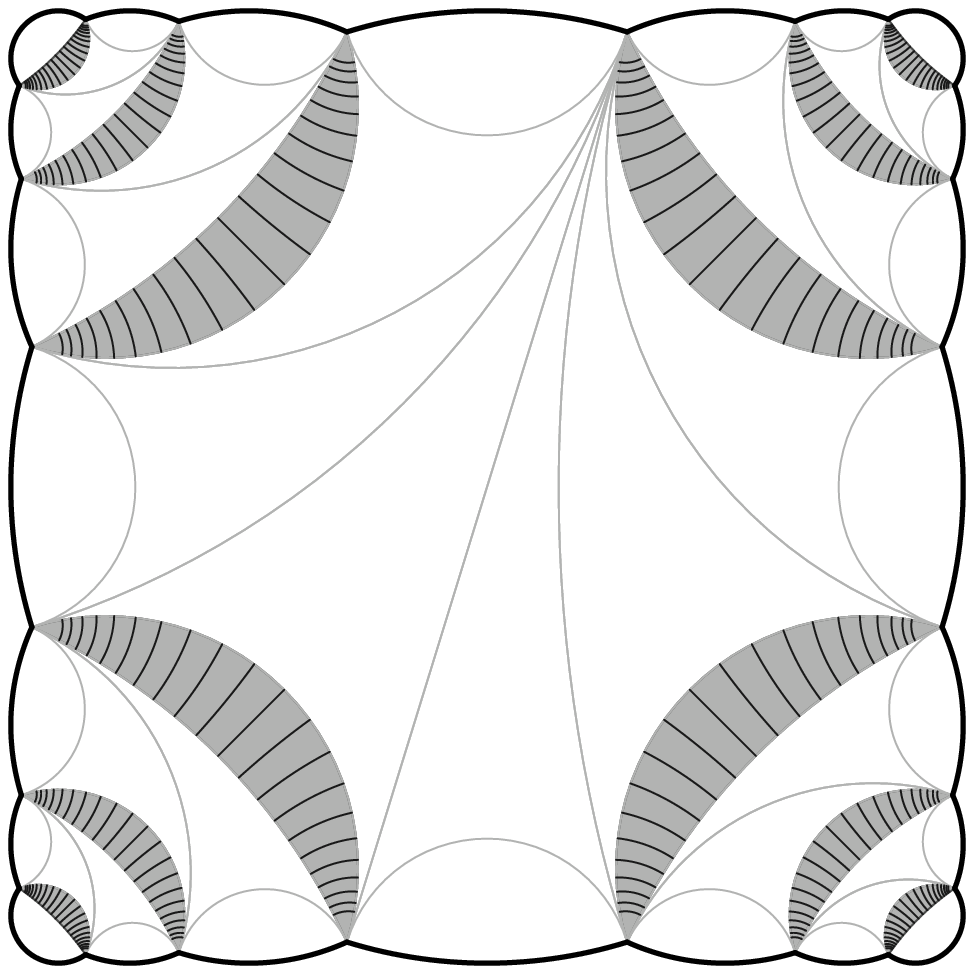}
             $\hphantom{xxx}$
             \includegraphics[height=2in]{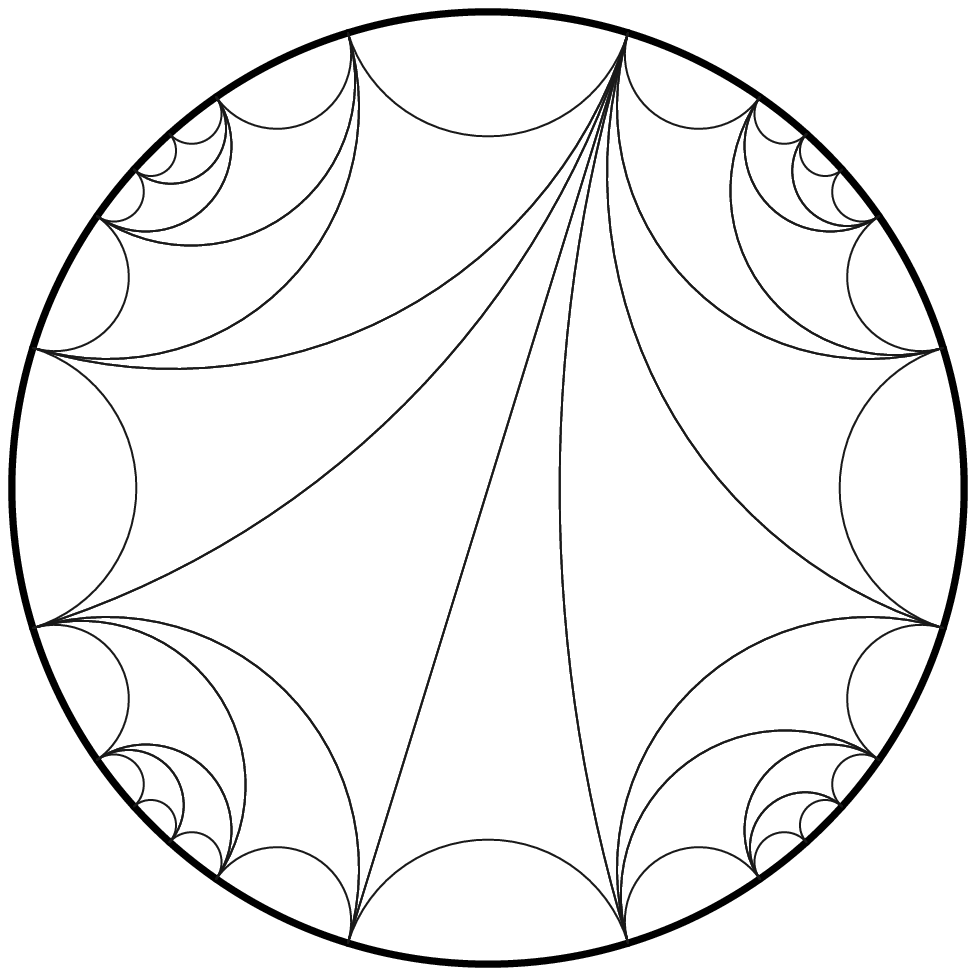}
 }
\caption{ \label{fol-square}
A normal crescent decomposition of a square and the corresponding 
bending lamination in the disk. We can recover the decomposition 
from the lamination by ``thickening'' each geodesics to a crescent 
of the correct angle.}
\end{figure}

We can recover the normal crescent decomposition from the bending 
lamination by ``thickening'' each bending geodesic to a crescent 
of the correct angle, and moving the gaps by the corresponding 
elliptic transformations.   If we do this continuously, we obtain 
a family of domains connecting the disk to $\Omega$.
For $0\leq t \leq 1$, 
let $\Omega_t$ be the domain obtained by replacing a crescent of  
angle $\alpha$ in the normal decomposition  by a crescent
or angle $t \alpha$. 
See Figures \ref{scale2} to \ref{disk-poly3} for some examples 
of these 1-parameter families.
In general, the intermediate domains need not be planar, but 
we can think of them as Riemann surfaces that are constructed 
by gluing together crescents and gaps of given sizes along 
their edges. 
Figure \ref{disk-poly3} shows an example where the intermediate domains are 
not planar (one sees some small overlap for parameter value 
$t=.99$; bigger overlaps could be produced by other
examples).

\begin{figure}[htbp] \label{scale1}
\centerline{ 
\includegraphics[height=1.5in]{cd1t=0.ps}
$\hphantom{xxx}$
\includegraphics[height=1.5in]{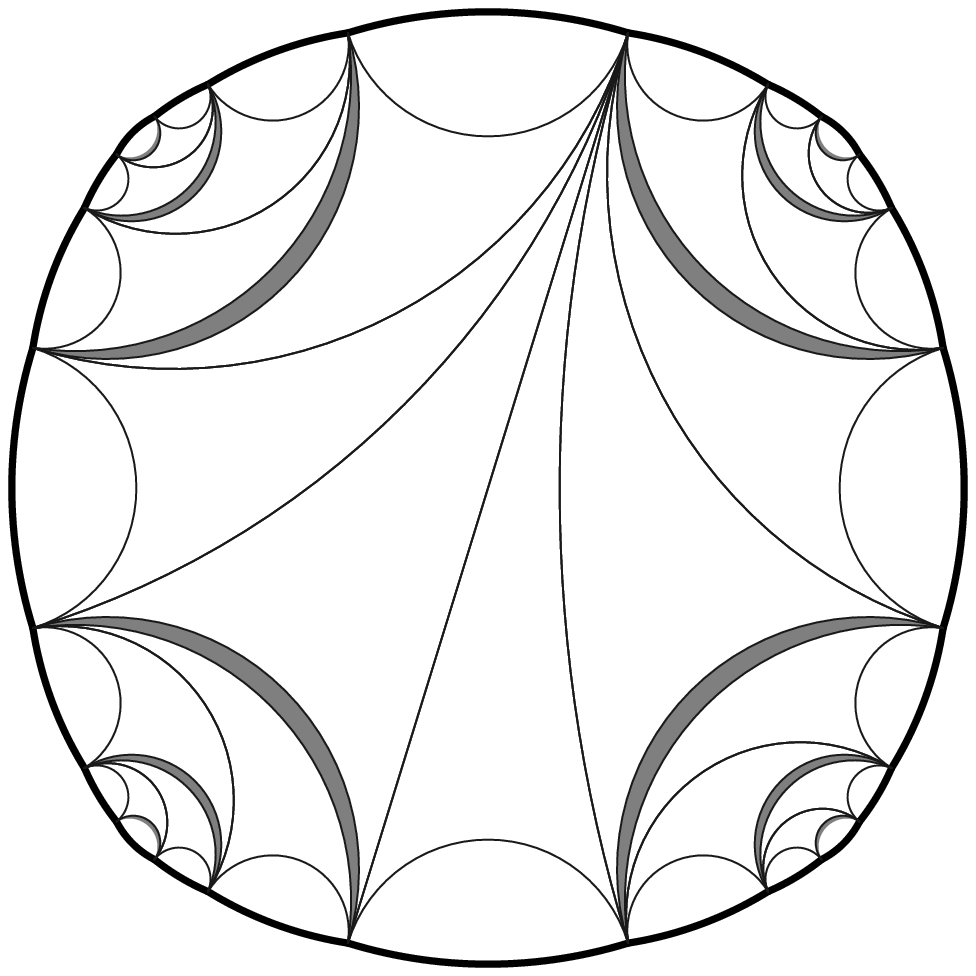}
$\hphantom{xxx}$
\includegraphics[height=1.5in]{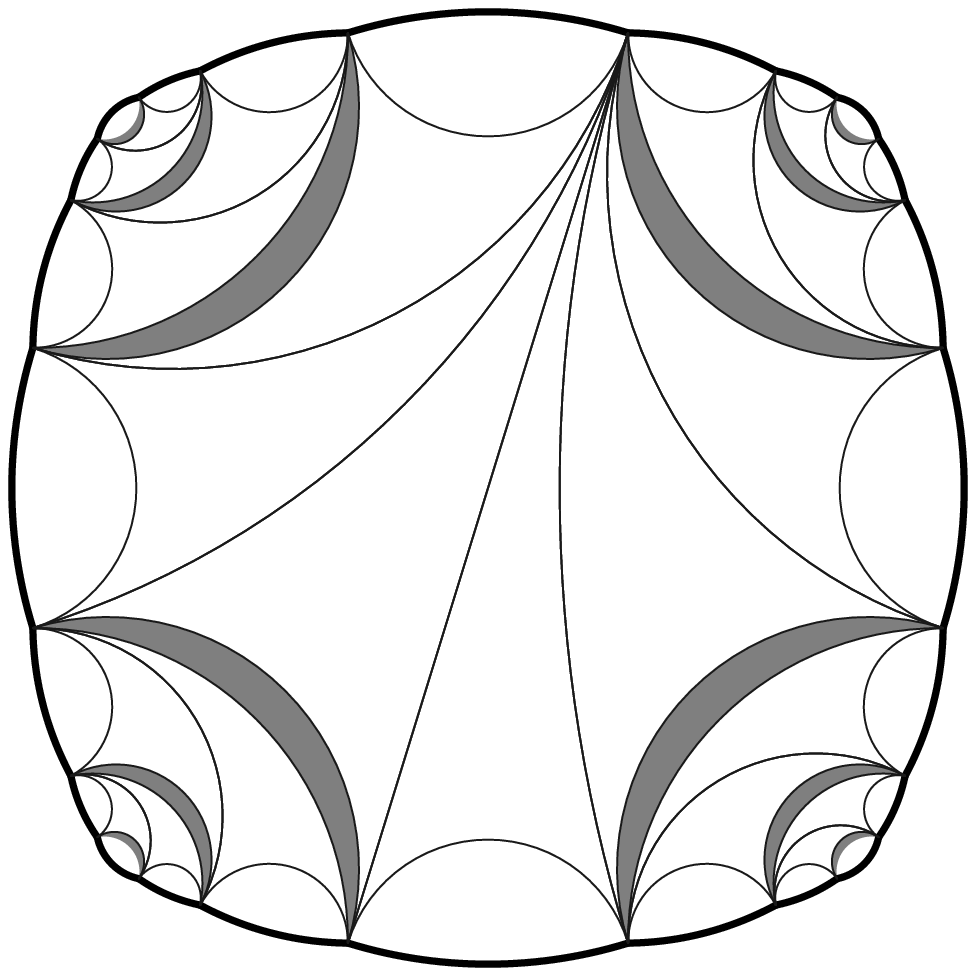}
 }
\centerline{ 
\includegraphics[height=1.5in]{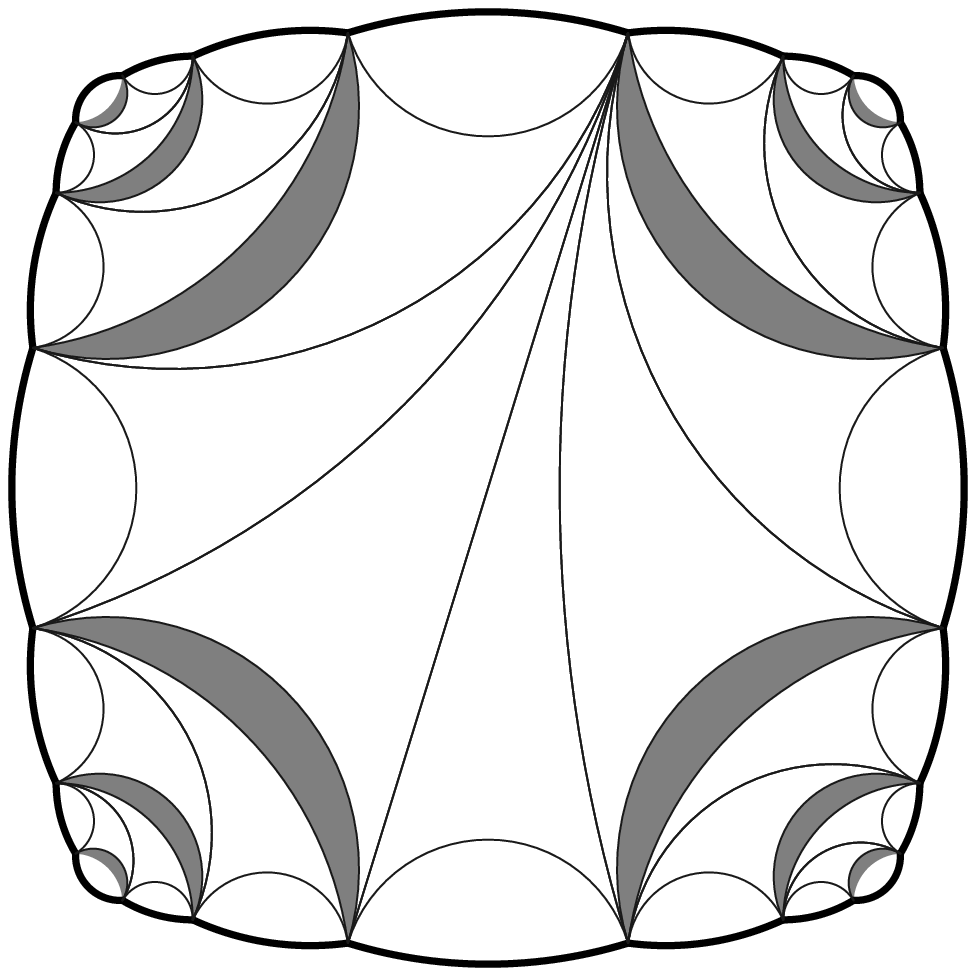}
$\hphantom{xxx}$
\includegraphics[height=1.5in]{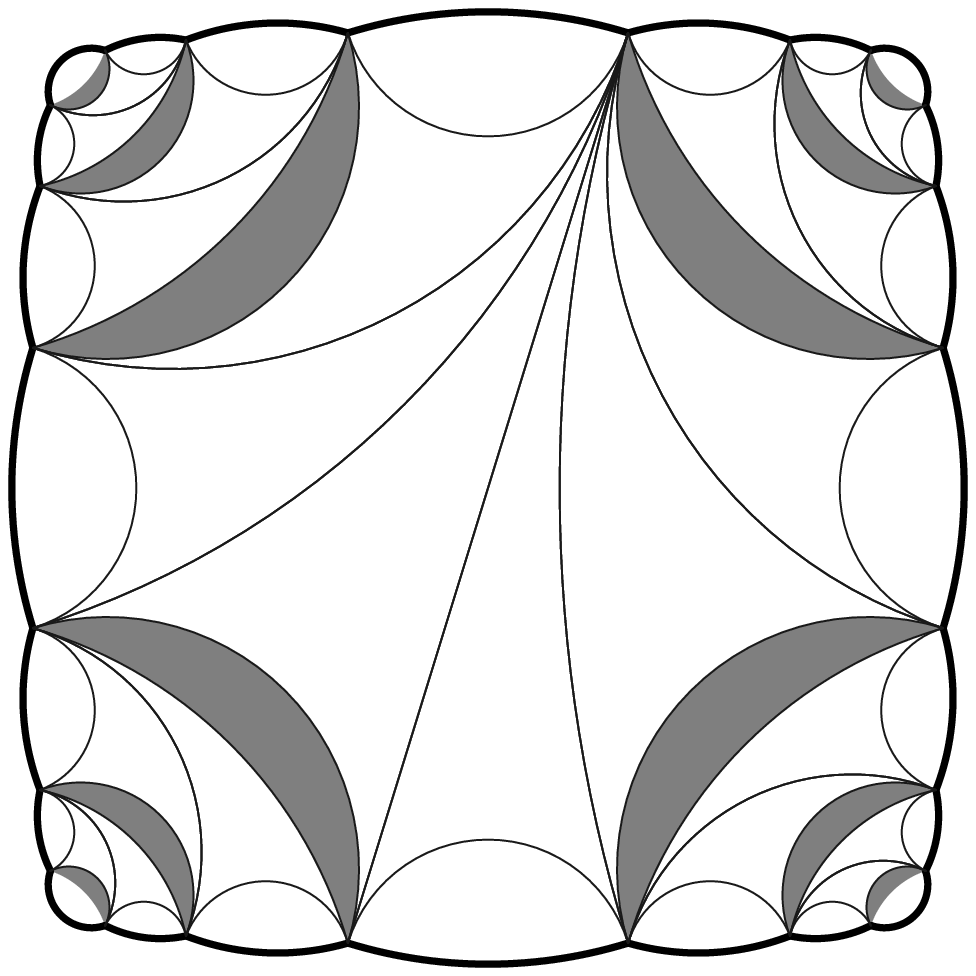}
$\hphantom{xxx}$
\includegraphics[height=1.5in]{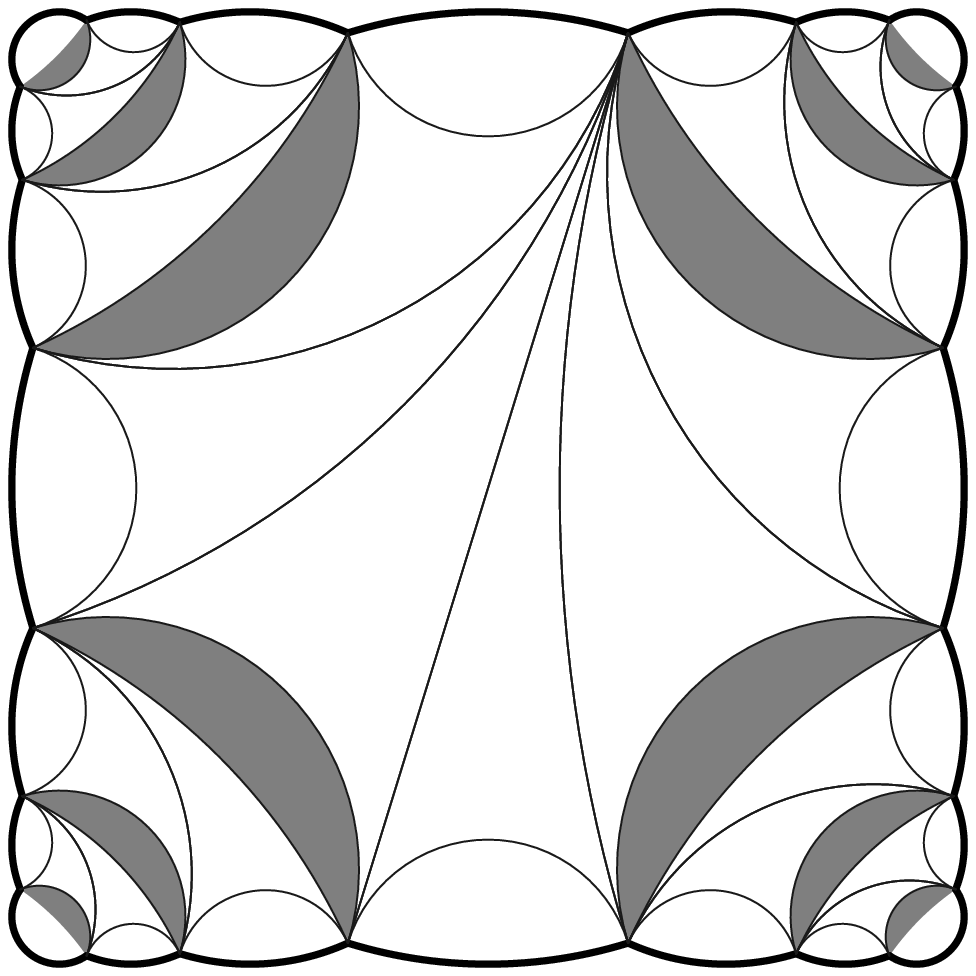}
 }
\caption{ \label{scale2}
The one parameter family connecting the disk 
to a finitely bent approximation of the square.
In each picture the angles have been 
multiplied by $t = 0,.2,.4,.6,.8,1$}
\end{figure}
\begin{figure}[htbp] 
\centerline{ 
\includegraphics[height=1.5in]{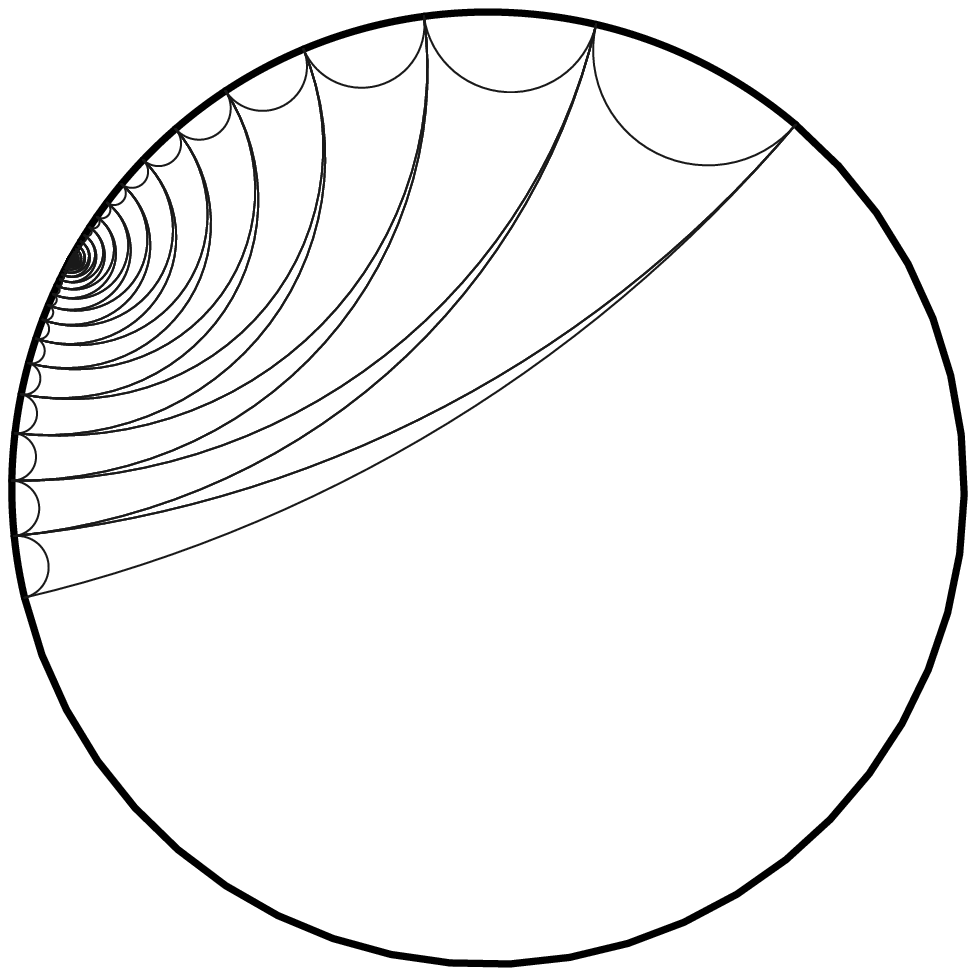}
$\hphantom{xxx}$
\includegraphics[height=1.5in]{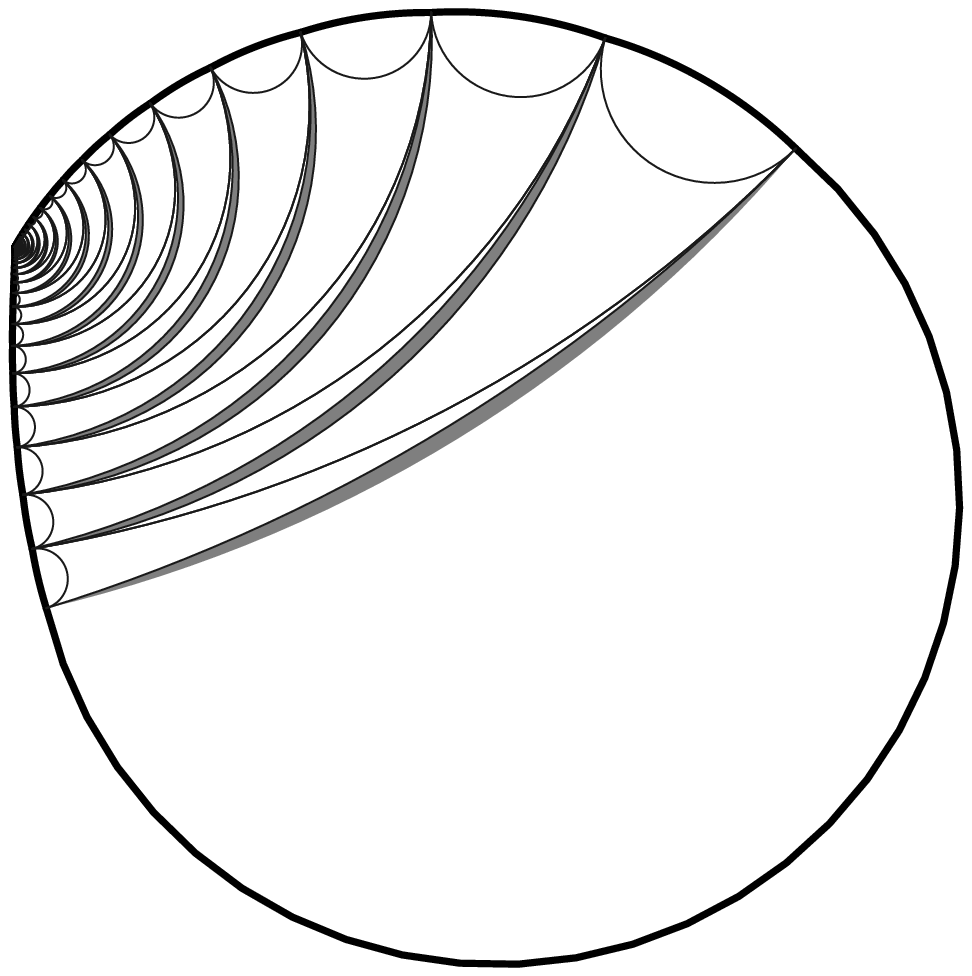}
$\hphantom{xxx}$
\includegraphics[height=1.5in]{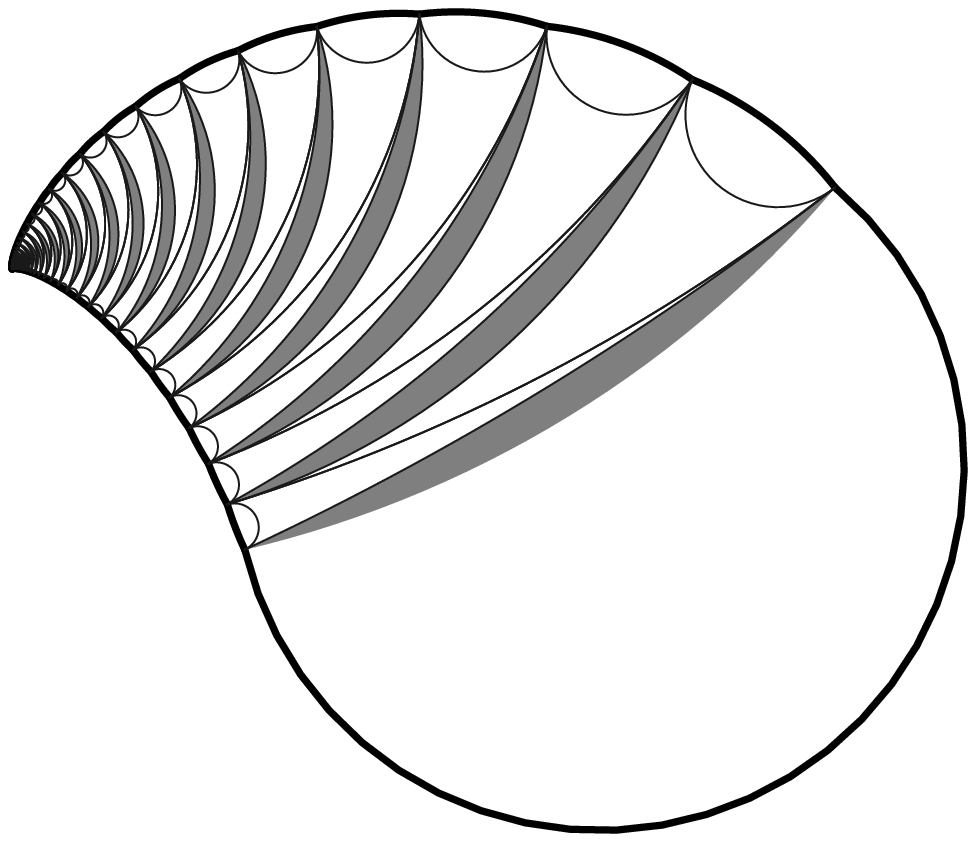}
 }
\centerline{ 
\includegraphics[height=1.5in]{cd7t=6.ps}
$\hphantom{xxx}$
\includegraphics[height=1.5in]{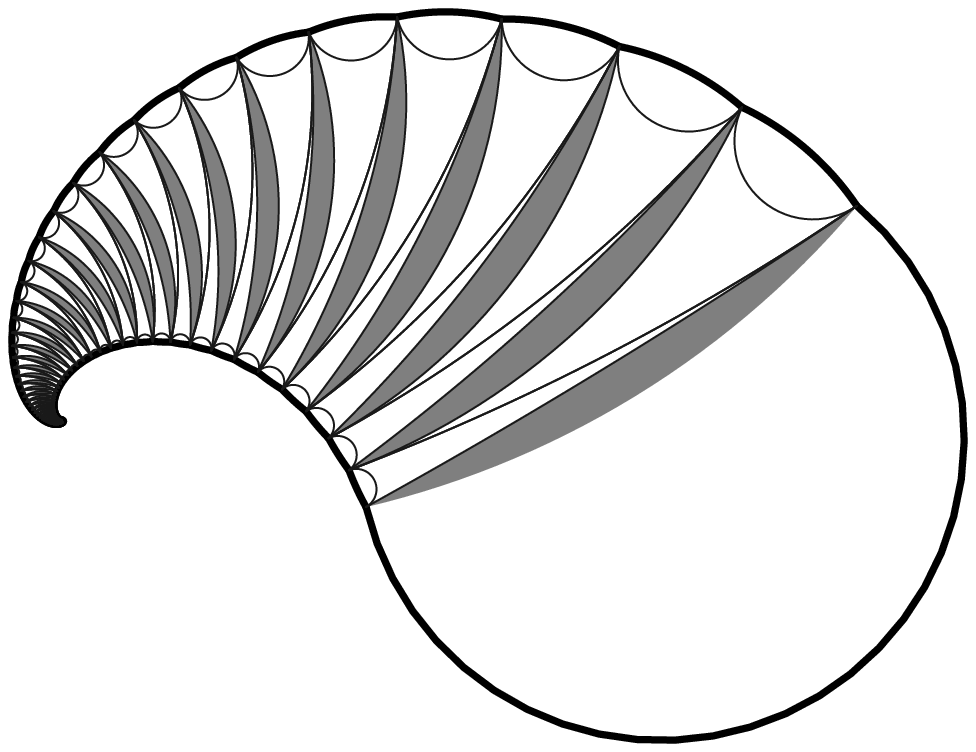}
$\hphantom{xxx}$
\includegraphics[height=1.5in]{cd7t=10.ps}
 }
\caption{\label{poly7}
An approximate logarithmic spiral with  $t = 0,.2,  .4, .6,  .8, 1$.
Logarithmic spirals were used by Epstein and Markovic in 
\cite{EM-log-spiral} to disprove Thurston's $K=2$ conjecture.
They showed that (in a precise sense) certain spirals have too much gray.
}
\end{figure}
\begin{figure}[htbp] \label{scale4}
\centerline{ 
\includegraphics[height=1.5in]{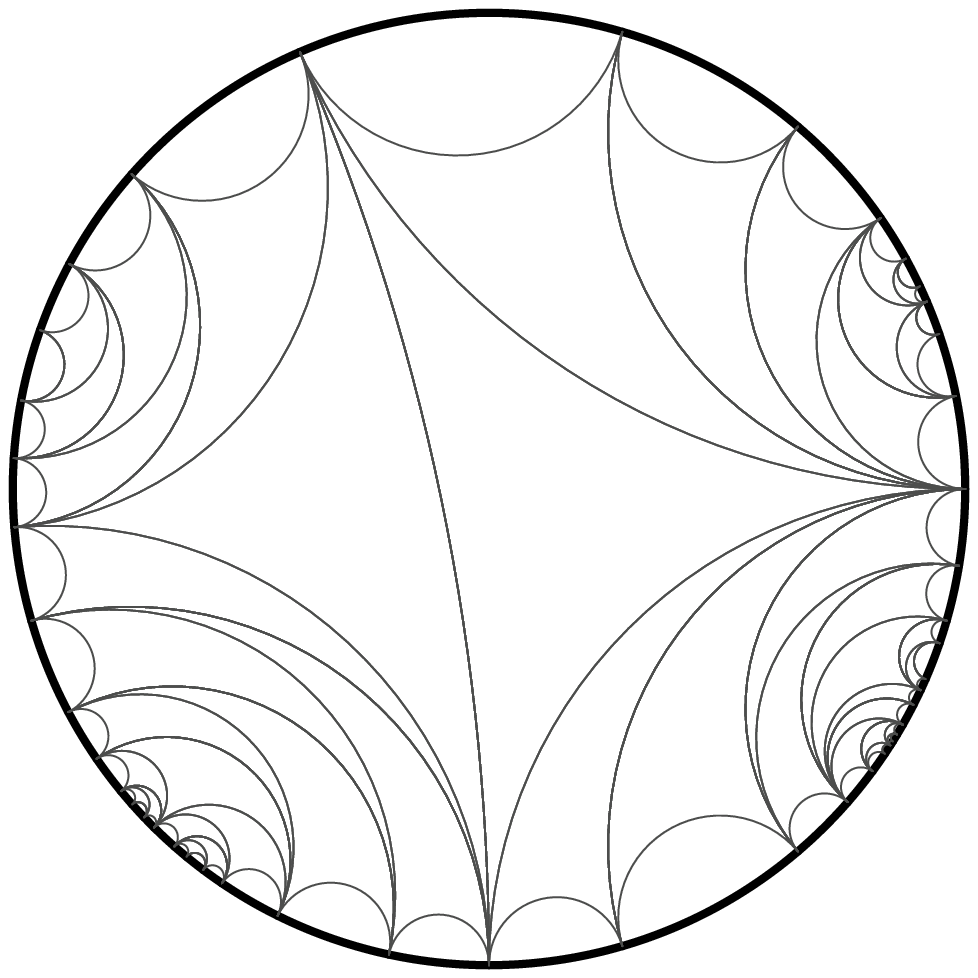}
$\hphantom{xxx}$
\includegraphics[height=1.5in]{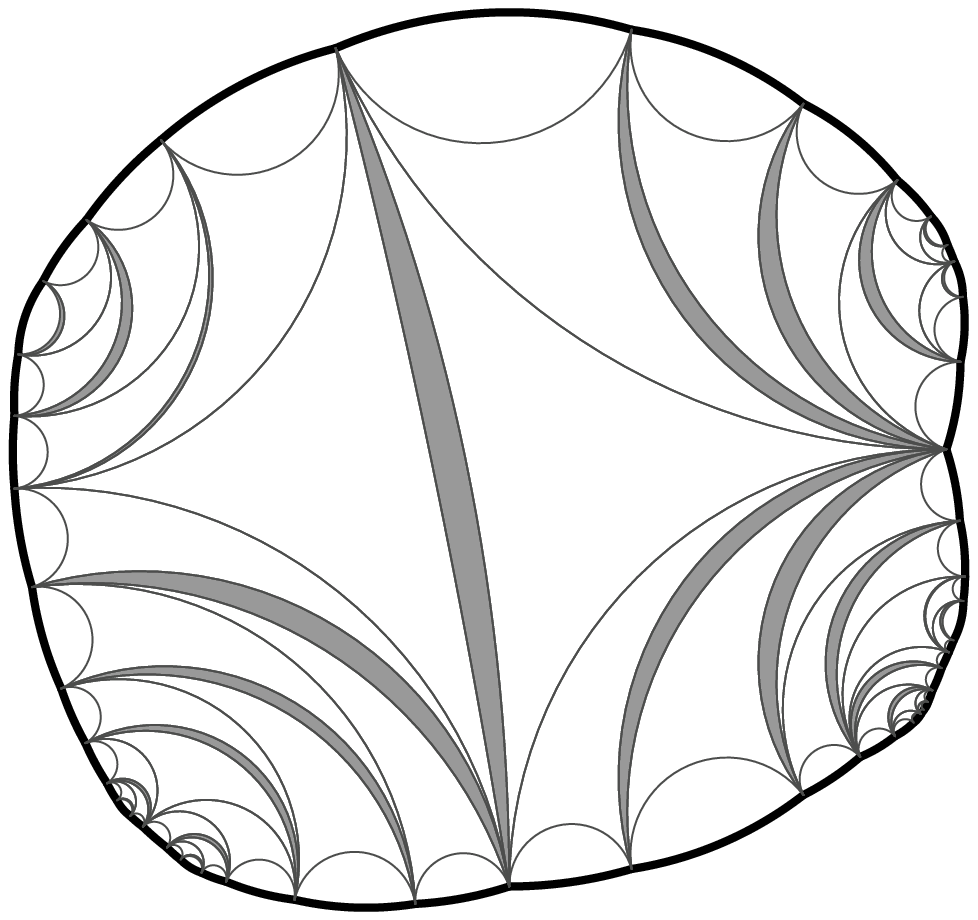}
$\hphantom{xxx}$
\includegraphics[height=1.5in]{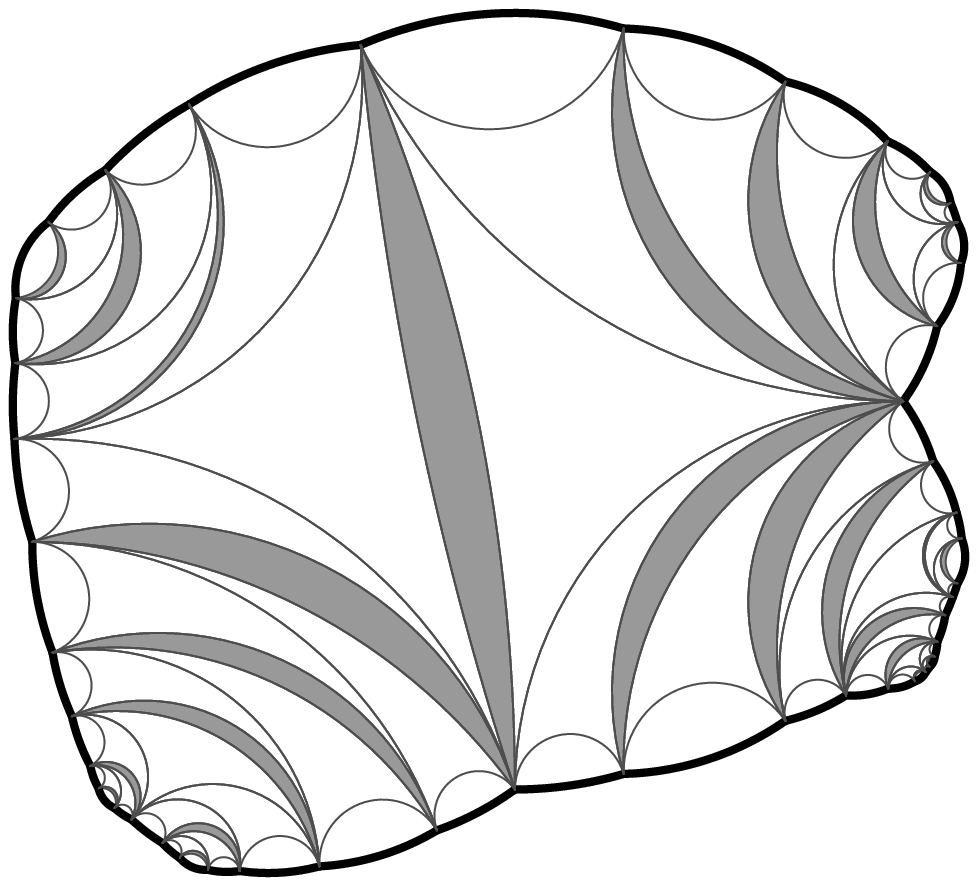}
 }
\centerline{ 
\includegraphics[height=1.5in]{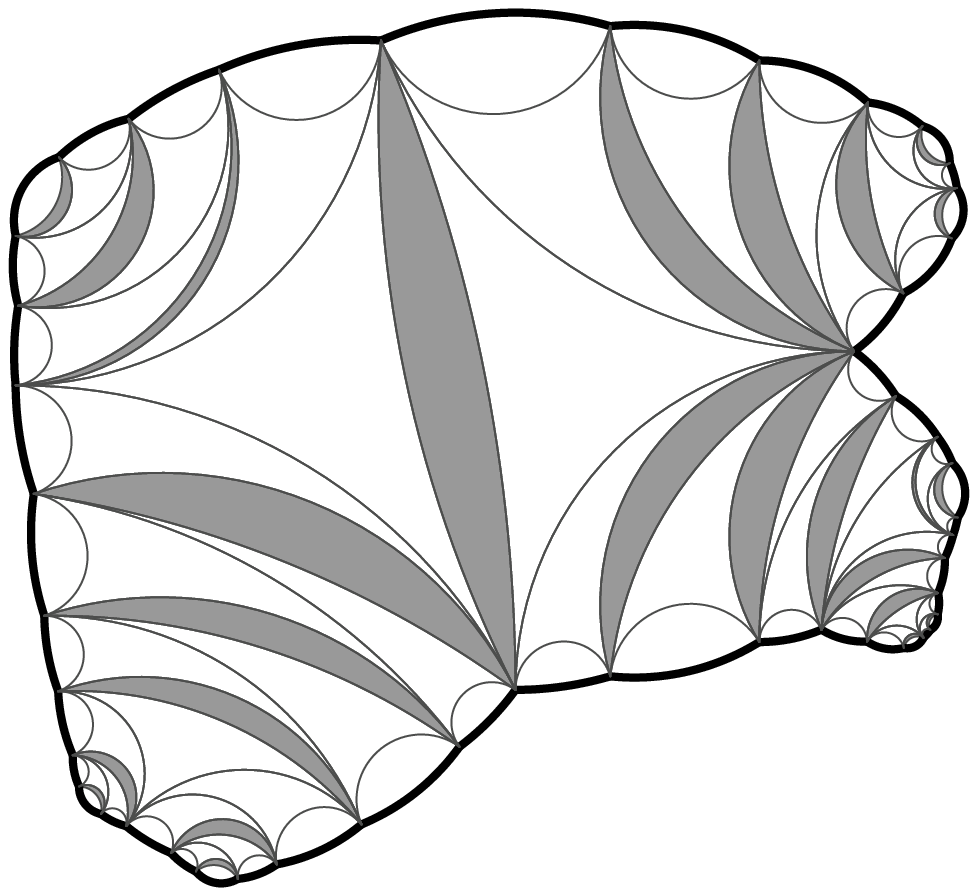}
$\hphantom{xxx}$
\includegraphics[height=1.5in]{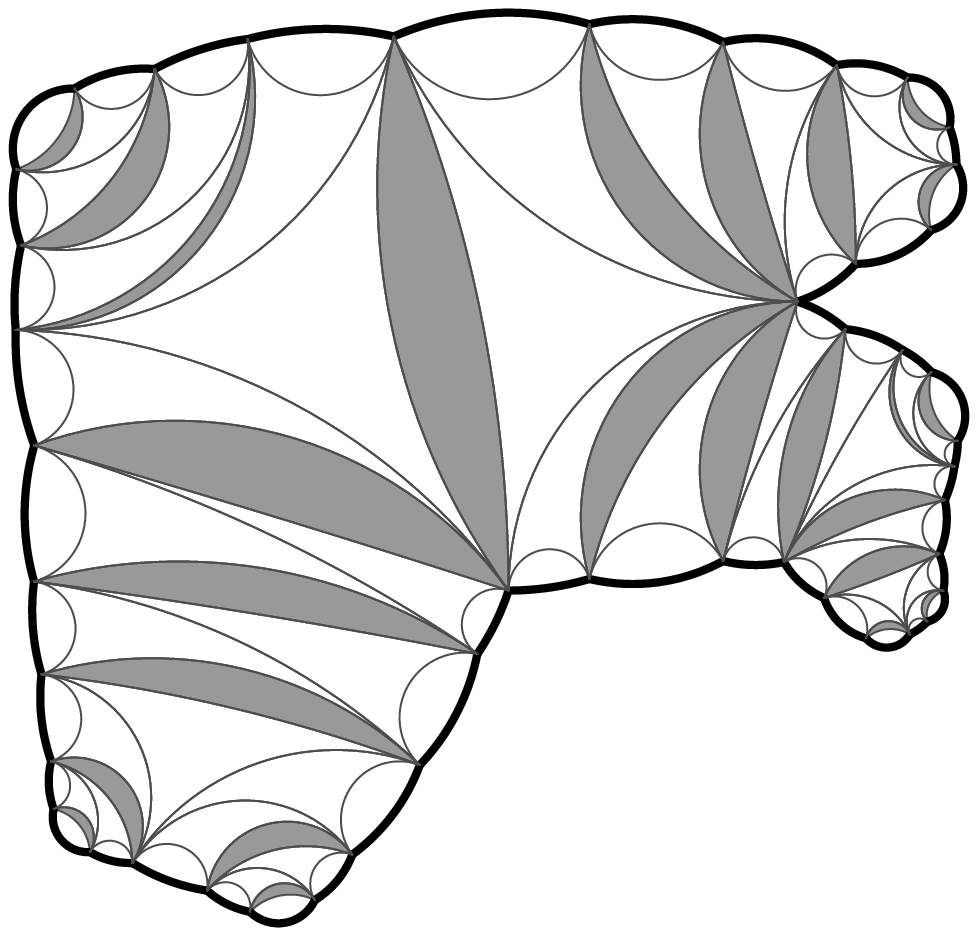}
$\hphantom{xxx}$
\includegraphics[height=1.5in]{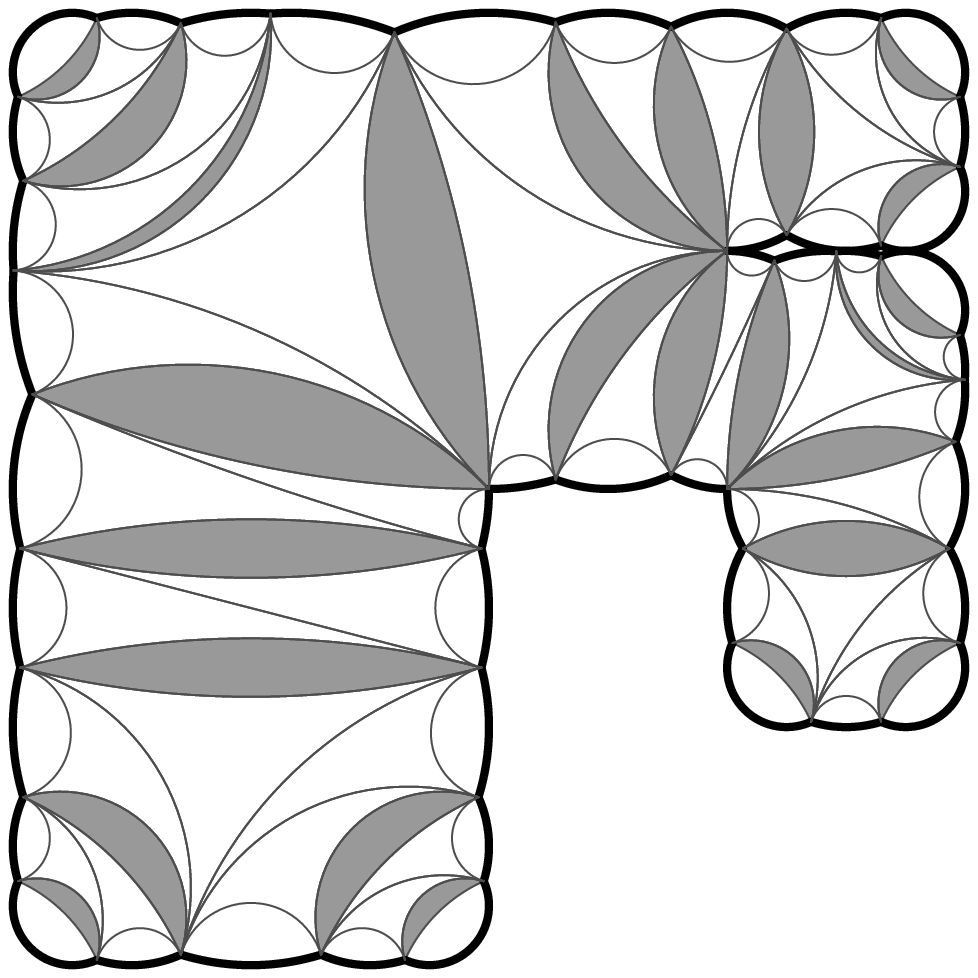}
 }
\caption{\label{poly2}
 The domain from Figure \ref{2-decoms}  with  $t = 0,.2,  .4, .6,  .8, 1$.
}
\end{figure}
\begin{figure}[htbp] 
\centerline{ 
\includegraphics[height=1.5in]{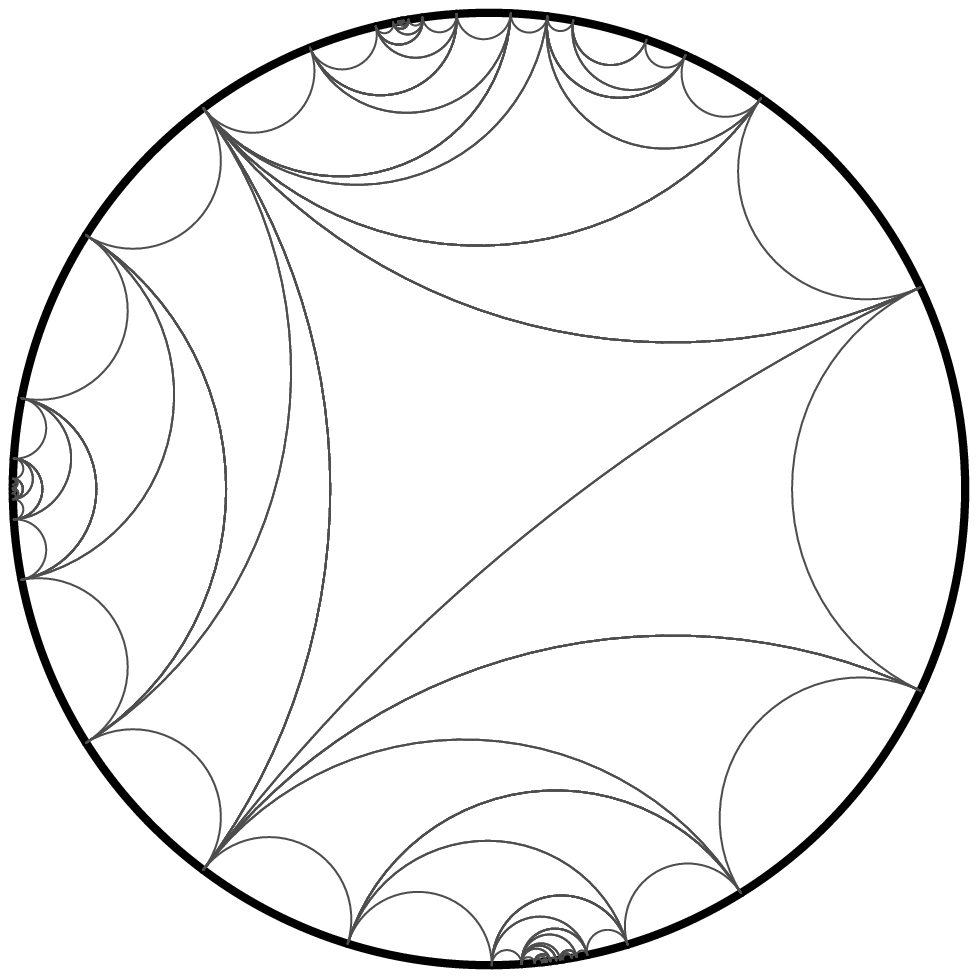}
$\hphantom{xxx}$
\includegraphics[height=1.5in]{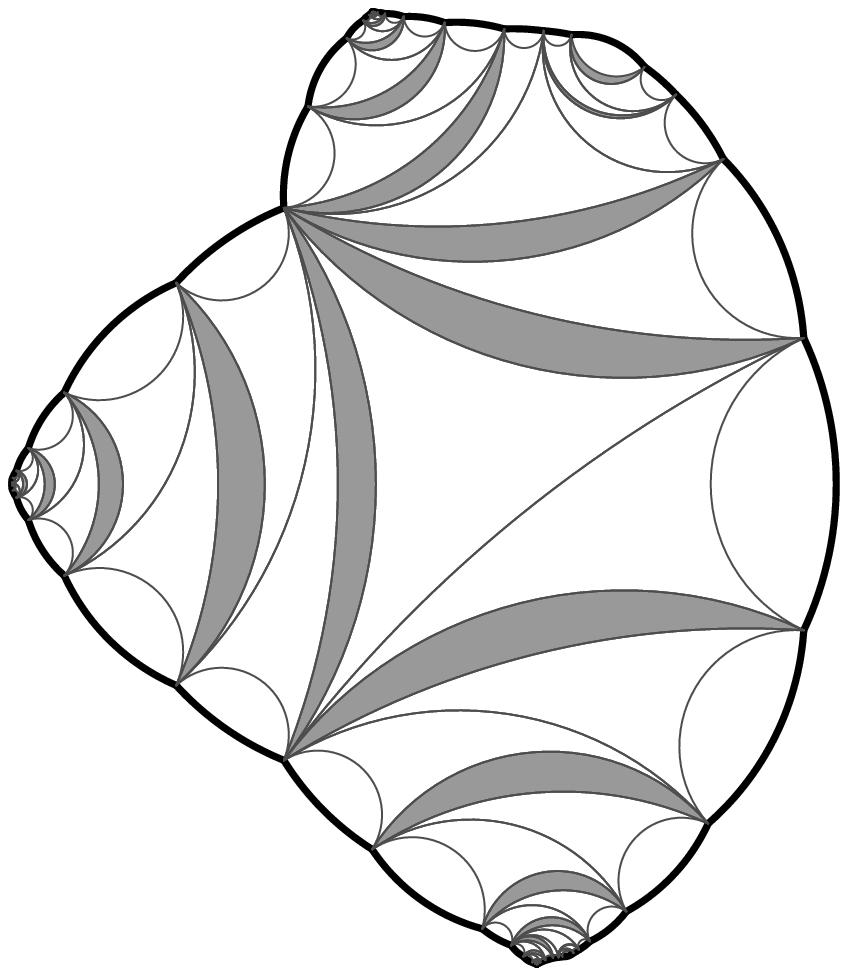}
$\hphantom{xxx}$
\includegraphics[height=1.5in]{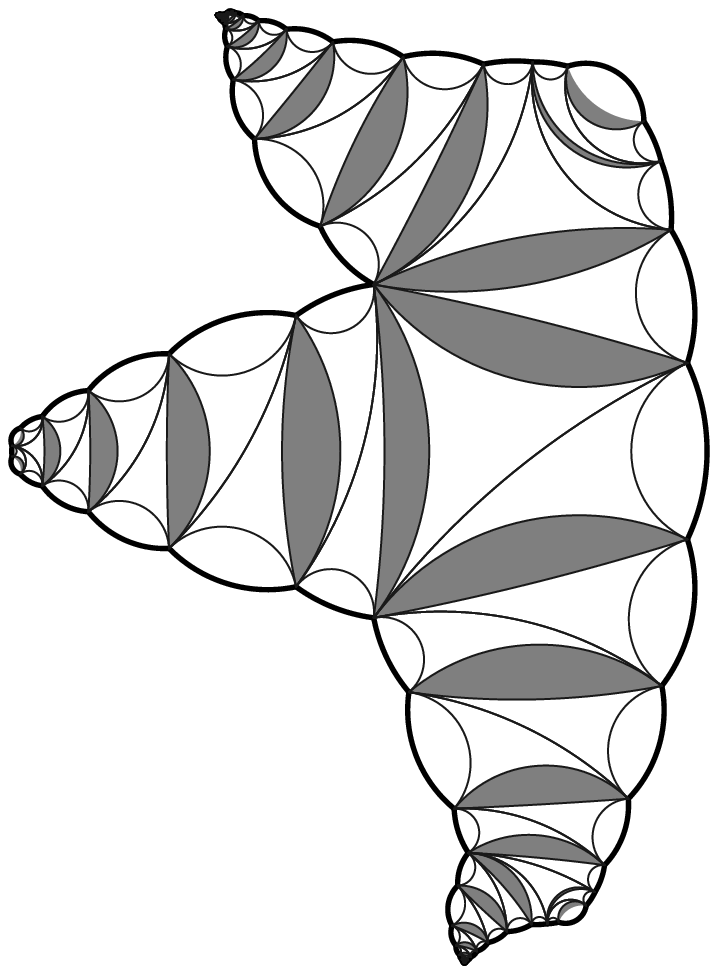}
 }
\centerline{ 
\includegraphics[height=1.5in]{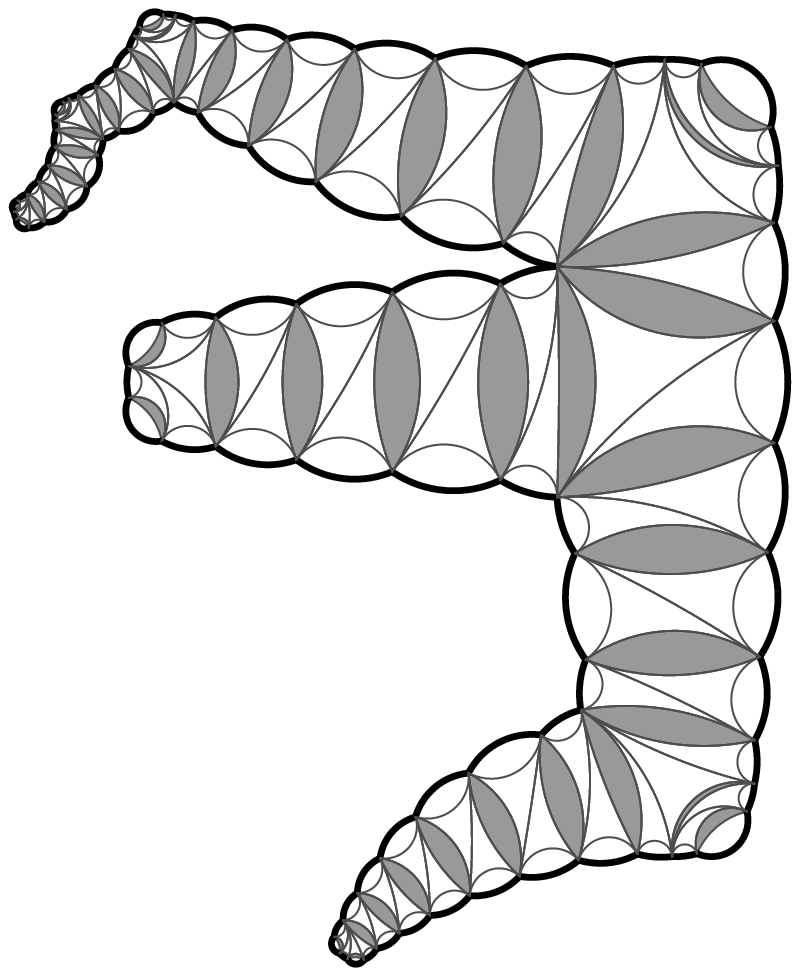}
$\hphantom{xxx}$
\includegraphics[height=1.5in]{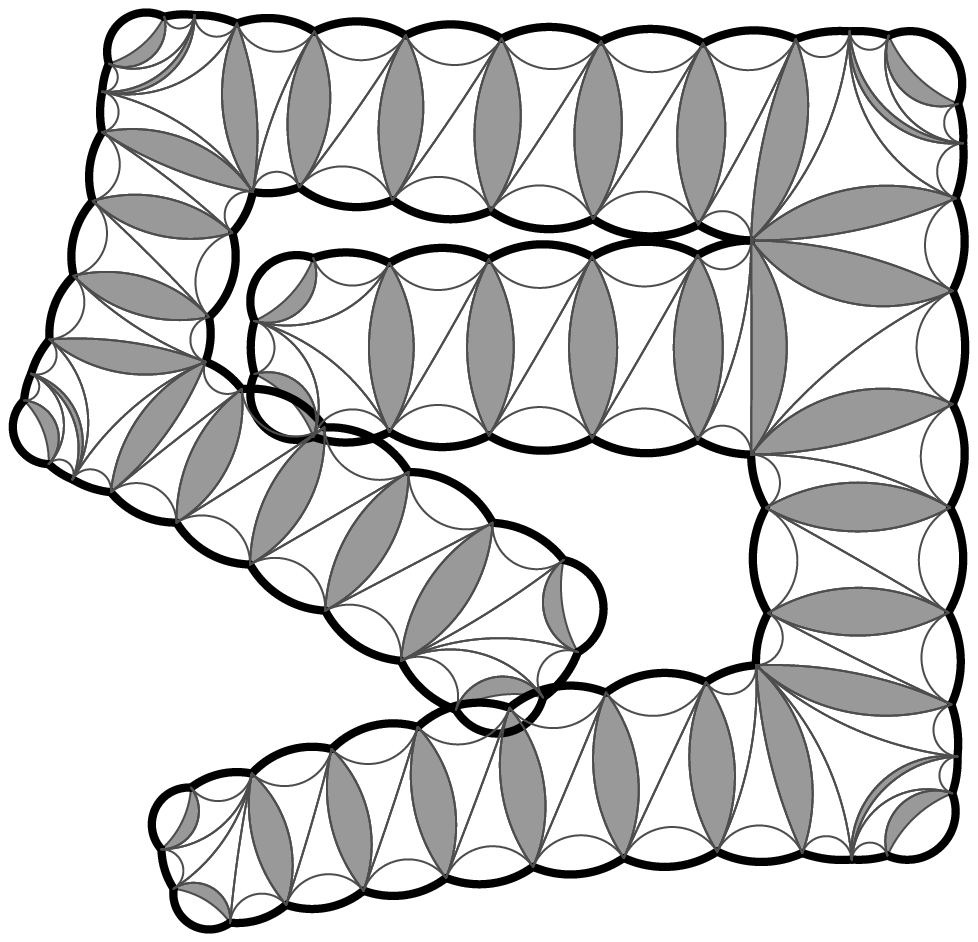}
$\hphantom{xxx}$
\includegraphics[height=1.5in]{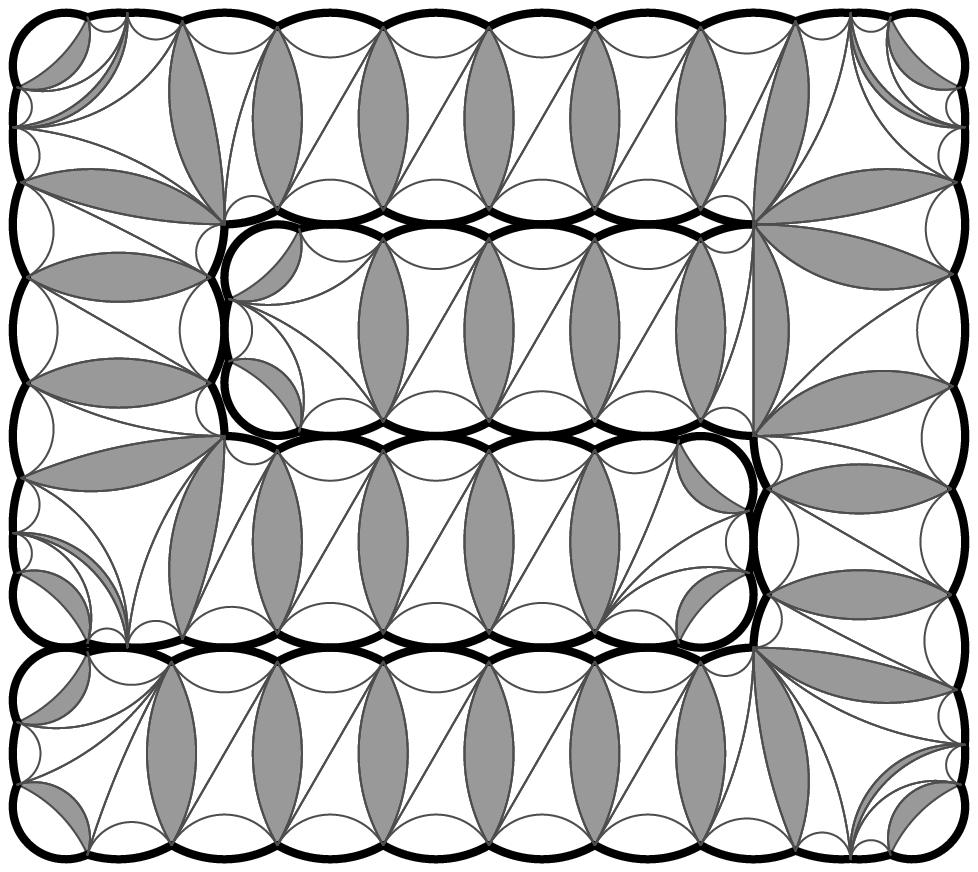}
 }
\caption{\label{disk-poly3}
An example where intermediate domains need not
be planar.  The pictures correspond to multiplying the angles 
by $t = 0, .4, .8, .95, .99, 1$.
Note that the parameter must be very close to 
$1$ before we see the longer corridors clearly.}
\end{figure}

Given a pair of domains $\Omega_s, \Omega_t$ with $0\leq 
s < t \leq 1$, let $\iota_{s,t}: \partial \Omega_t
\to \partial \Omega_s$ be the obvious boundary map obtained
multiplying  the angle of each crescent by $s/t$.
We will extend this boundary map to the interiors
by writing each crescent $C$ in $\Omega_t$ of angle 
 $\alpha$ as a union of crescents
$C_1$, of angle $\alpha s/t$ and $C_2$, of angle 
$\alpha(1-s/t)$. On $C_1$ we collapse each leaf of the 
$E$-foliation to a point (hence $C_1$ is maps to a circular 
arc) and we let our map be M{\"o}bius on $C_2$.
By continuity, 
this M{\"o}bius transformation would have to agree with the 
map on the gap that is adjacent to $C_2$. We will let 
$\varphi_{s,t}: \Omega_t \to \Omega_s$ denote this map.
Let $\rho_s= \rho_{\Omega_s}$ denote the hyperbolic
metric on $\Omega_s$.
%
Suppose $N$ is a large integer 
and choose points 
$t_0=0, t_1 = \frac 1N, \dots , t_N =1 $. Let 
$\Omega_k = \Omega_{t_k}$ for $k=0,\dots,N$.
Let $\varphi_k:\Omega_{k+1} \to \Omega_k$ be defined 
by $\varphi_k = \varphi_{\frac kn, \frac{k+1}n}$.


\section{$\varphi_{s,t} $ is a quasi-isometry} \label{proof-of-SEM}

As noted before, the retraction map $R: \Omega \to S$ is a quasi-isometry. 
 Thus $\varphi = \iota \circ R
: \Omega \to \disk$  is also a quasi-isometry between 
the hyperbolic metrics.
The same is true for the maps
$\varphi_{s,t}$  for any $0 \leq s < t \leq 1$,
with constant bounded by $O(|s-t|)$. 
This result is the goal of this section and the next.

Consider the bending lamination $\Gamma$ associated to a 
finitely bent domain $\Omega$.
Suppose a hyperbolic $r$-ball hits 
 geodesics  in $\Gamma$ with angles $\alpha_1,\dots, 
\alpha_m$.  We want to show that there is an upper
bound $\sum_j \alpha_j \leq B(r)$ that only depends on $r$.
See
 \cite{Bishop-ExpSullivan}, \cite{EM87} for some variations of this 
 idea.
Estimates of $B$ are also closely tied to results of 
Bridgeman \cite{Bridgeman-H3}, \cite{Bridgeman-H2} on 
bending of surfaces in hyperbolic spaces.
 Here we shall give a simple conceptual proof 
without an explicit estimate. The number of 
bending geodesics that hit the $r$-ball  has no uniform bound
(if it 
did the lemma would be trivial since every crescent 
has angle $\leq 2 \pi$). However, the total bending 
of these geodesics is bounded in terms of $r$.
This result (together with Lemma \ref{Newton-radius}) is 
one of the main pillars on which the whole paper rests;
the uniform estimate of bending  eventually becomes 
the uniform estimates of time and accuracy given in 
Theorems \ref{thmQC} and \ref{main}.

\begin{lemma} \label{B-est}
There is a $C < \infty$ so that $ B(r) \leq C e^{3 r}.$
\end{lemma}

\begin{proof}
Suppose  $\Omega$ is normalized so $\infty \not \in \Omega$.
The normalization implies that if $\gamma$ is a 
 bending geodesic  in $\uhs$ that  hits the plane at $1$ and $-1$,
 then  the corresponding crescent is
in the unit disk. Moreover, an easy estimate shows that
a  crescent with vertices $\pm 1$ and  angle $\alpha$
has area  $  \geq c \alpha$ for some fixed $c >0$. 

If $\tilde \gamma$ is a bending geodesic 
with angle $\beta$
that passes within hyperbolic distance $r$ of $ (0,0,1)$ then 
the ``highest'' point of $\tilde \gamma$ has Euclidean height 
at least $e^{-r}$ above the plane $\reals^2$. Thus
its two endpoints on the plane are at least $2 \cdot e^{-r}$ apart. 
Moreover at least one endpoint  
must be  contained in the disk of diameter $e^r$ around the origin
(if not, then $\tilde \gamma$ lies outside the hemisphere with this 
disk as its base, which means the hyperbolic distance to $(1,0,0)$ 
is $\geq r$). 

Thus the part of the crescent corresponding to $\tilde \gamma$
inside the ball $B(0,e^r +1)$ 
has area at least $c e^{-r} \beta$. 
 Consider the set of all bending 
geodesics that come within hyperbolic distance $r$ of the point 
$(0,0,1) \in \uhs$ and let $\{\alpha_n\}$ be an enumeration of the 
bending angles. Since the crescents are 
disjoint we deduce 
$ \sum_n \alpha_n \leq \frac 1c \pi e^r (e^{r}+1)^2 \leq Ce^{3r}$, as desired.
(Note that this argument is not sharp since the crescents can 
have small area only when then are  close to the origin.)
\end{proof}

The following simple lemma quantifies the fact that an elliptic
M{\"o}bius transformation with small rotation angle is close to the 
identity.

\begin{lemma} \label{elliptic-est}
Suppose $\sigma$ is an elliptic M{\"o}bius transformation 
with fixed points $a,b$ and rotation angle $\theta$.
If $r= \max(|z-a|,|z-b|)\leq A |b-a|\leq |b-a|/(4 \theta)$
and $|\theta| \leq \frac 14$,  then we have 
$$   |z - \sigma(z)| \leq  2 A^2 |\theta| |z-a|,$$
where $C$ depends only on $A$.
\end{lemma}

\begin{proof}
This is an explicit computation. The conclusion is invariant 
under scaling, so we may assume $a=1$, $b=-1$,
in which case $\sigma$ has the form 
$ \sigma(z) = \tau^{-1}(\lambda \tau(z))$ where 
$\lambda = e^{i \theta}$ and $\tau(z) = (z-1)/(z+1)$. Doing 
some arithmetic, and using $|1-\lambda| \leq |\theta|$, we get 
$$ |\sigma(z) - z| = |\frac{ (1-\lambda)-(1-\lambda)z^2}
                           { (1+\lambda)+ (1-\lambda)z}|
                   \leq |\theta|\frac{|1-z^2|}{1-|\theta|-|\theta||z|}
                    \leq 2A^2 |\theta| |z-1| ,$$
if $|\theta| \leq \frac 14$ and $ |\theta z|\leq \frac 1{4 }$.
\end{proof}

The following is the main result of this section. Recall that $R: \Omega
\to S$ denotes the nearest point retraction discussed in the previous 
section.

\begin{lemma} \label{varphi-est}
Suppose $r>0$ is given. There is an $\epsilon>0$, 
depending only on $r$, so that if $0\leq s < t \leq 1$ 
and $|s-t| \leq \epsilon$ then the following holds.
Suppose $G_1$ and $G_2$ are gaps in the normal crescent 
decomposition  of the finitely bent domain $\Omega_s$
such that  $\rho_S(R(G_1), R(G_2)) \leq r$.
 Suppose 
$\tau_j$ are M{\"o}bius transformations so that 
$\varphi_{s,t}^{-1}|_{G_j} = \tau_j$ for $j =1,2$. Then 
$$     \rho_{t}( \tau_1(z) , \tau_2(z) ) \leq 
                   C_r |t-s|,$$
for every $z  \in \Omega_s$ with $\rho_{s}(z, G_1) 
\leq r$.
\end{lemma}

\begin{proof}
The statement is invariant under renormalizing by
 M{\"o}bius transformations
so we may assume that $G_1$ has base disk $\disk$, that 
$z_1= 0 \in G_1$ is within $2r$ of  $G_2$, 
and that $\tau_1$ is the identity.

 Then 
$\tau_2$ is a composition of the elliptic transformations 
$\{ \sigma_j\}$
that correspond to the normal crescents $\{C_j\}$  that 
separate $G_1$ and $G_2$. By Lemma \ref{B-est}, the
measure of the bending geodesics separating $G_1$ and 
$G_2$ is at most $B(r)$.

Since $\rho_S(C_j , 0) \leq r$ for all $j$,   
$C_j$ has diameter $\geq e^{-r}$ and one vertex
is contained 
within $D(0, e^r)$ by  the proof of 
Lemma \ref{B-est}.  By Lemma \ref{elliptic-est}
this means that $\sigma_j$ moves points in $D(0,C)$ at most 
$C|\theta_j| $ with  $C$ depending only on $r$, assuming 
$\theta_j$ is small enough (depending only on $r$).
Thus 
\begin{eqnarray} \label{tau-1}
|\tau_2(z)- z| \leq  {C_r}|s-t|  \sum_j |\theta_j| 
  = O(|s-t|),
\end{eqnarray}
for $|z| \leq C$, assuming $|s-t|$ is small enough, depending only 
on $r$. 

If $\rho_t(0,z) \leq r$, then $|z|\leq A_r$ and $\dist(z, \partial 
\Omega_s) \geq B_r > 0$ with estimates that 
only depend on $r$ (see Lemma \ref{dist-to-bdy}, Appendix 
\ref{background}). Thus for 
$|s-t|$ small enough,
$\rho_t(z,0) \leq r$ and  $|z-w| \leq \epsilon$ 
imply  $\rho_t(z,w) \leq C_r \epsilon$. Hence for a given $r$ 
we can choose  $|s-t|$ so small that (\ref{tau-1})  implies 
$\rho_t(\tau_2(z),z) \leq 
O(|s-t|)$, (with constant depending on $r$).
\end{proof}

\begin{lemma} \label{varphi=QI}
$\varphi_{s,t}$ is a quasi-isometry with constant $O(|s-t|)$.
\end{lemma}

This follows immediately from the following technical result 
that will be proven in Section \ref{PM-maps}. It also follows 
from a careful reading of 
 \cite{Bishop-ExpSullivan}, which gives an explicit construction 
of a quasiconformal map from $\disk$ to a finitely bent domain 
$\Omega$ with boundary values $\varphi^{-1}$.  The method can 
be adapted to give an explicit  map $\Omega_s \to \Omega_t$  
that is quasiconformal with constant $O(|s-t|)$. 


\begin{thm} \label{near-Mobius}
Suppose $\Omega_0, \Omega_1$ are simply connected and
 $\varphi: \Omega_0 \to \Omega_1$ has the following property:
there is a $0< C < \infty$ so that given any hyperbolic
$C$-ball $B$ in $\Omega_0$, there is a M{\"o}bius transformation 
$\sigma$ so that $\rho_{\Omega_0}(z,\sigma(\varphi(z))) \leq \epsilon$
for every $z \in B$. Then there is a 
hyperbolic $(1+O(\epsilon))$-biLipschitz map $\psi: \Omega_0 \to \Omega_1$ 
so that $\sup_{z \in \Omega_0}
\rho_{\Omega_1}(\varphi (z), \psi(z)) \leq O(\epsilon)$.
In particular,   $\varphi$ is a quasi-isometry between the 
hyperbolic metrics with constant $O(\epsilon)$.
\end{thm}


\begin{cor} \label{psi=sigma}
There is a (hyperbolically) $(1+O(|s-t|)$-biLipschitz
 map $\psi_{s,t} : \Omega_s \to \Omega_t$ so that
$\psi_{s,t} = \varphi_{s,t}^{-1}$ on the boundary.
If $G$ is a gap or crescent and $\varphi_{s,t}^{-1}$ 
is the M{\"o}bius transform $\sigma$
on $G$, then $\rho_t(\psi_{s,t}(z) ,\sigma(z)) \leq O(|s-t|)$ for $z \in G$
\end{cor}


\section{Piecewise M{\"o}bius maps and $\epsilon$-Delaunay triangulations}
 \label{PM-maps}

Here we prove Theorem \ref{near-Mobius} from Section \ref{proof-of-SEM}.

If we want to approximate a map $f$ between polygons, 
a convenient thing to do is to decompose the interior 
into triangles, and approximate by a map that is linear
on each triangle. If $f$ is already linear in 
some subregion, we can arrange for the approximation to 
agree with it on the triangles that lie inside this subregion.

   We would like to do the same thing for finitely bent domains.
One problem is that the maps we wish to approximate 
are M{\"o}bius in some regions rather than linear, and a piecewise 
linear approximation will not preserve this.   We could  try to 
approximate circular arcs by line segments and M{\"o}bius 
transformations by linear maps,  but instead  we will slightly alter  the
idea of  piecewise linear approximation.

 Given 
a triangle $T$, let $D$ be the disk containing the three 
vertices on its boundary and let $\tilde T$ be the ideal hyperbolic 
triangle in $D$ with these three vertices.
We will say that a triangulation is $\epsilon$-Delaunay if whenever
two Euclidean triangles  $T_1, T_2$ meet along an edge $e$, the sum of 
two  the angles not incident on $e$ is at most $\pi-\epsilon$.
This means that between $\tilde T_1$ and $\tilde T_2$ there
is a crescent of angle at least $\epsilon$.
See Figure \ref{Del-quad}.
 A $0$-Delaunay triangulation is the same as 
the usual notion of  a Delaunay triangulation. Delaunay triangulations play 
an important role in computational geometry (see e.g., 
\cite{BE92}, \cite{Fortune92}, \cite{Fortune97}, \cite{Rajan94}).

\begin{figure}[htbp] 
\centerline{ 
	\includegraphics[height=2.0in]{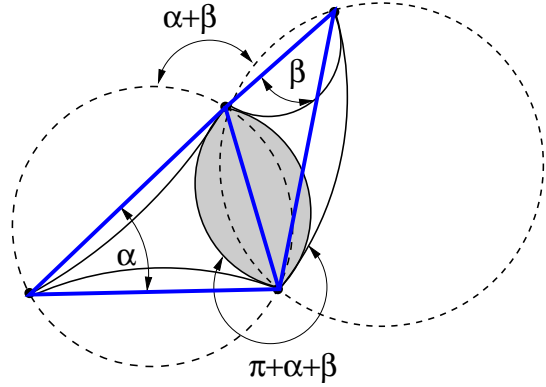}
	}
\caption{\label{Del-quad} Two triangles share an edge and 
the angles opposite the edge sum to less than $\pi$. There is 
then a crescent of angle $\pi-(\alpha+\beta)$ that separates
the ideal hyperbolic triangles associated the two Euclidean 
triangles.  }
\end{figure}

\begin{lemma}
$\epsilon$-Delaunay triangulations are invariant under 
M{\"o}bius transformations.
\end{lemma}
\begin{proof}
The $\epsilon$-Delaunay condition is equivalent to saying 
that if $T_1$ and $T_2$ are adjacent triangles then the 
boundaries of the corresponding disks $D_1$ and $D_2$ 
meet at exterior angle less than $\pi -\epsilon$. This is clearly 
invariant under M{\"o}bius transformations.
\end{proof}

Suppose we are given a mapping between the boundaries of two 
crescents with interior angles $\alpha_1, \alpha_2$
 that agrees with a M{\"o}bius transformation on 
each boundary arc (but possible different transformations on 
each arc). Normalizing so the vertices are $0$ and $\infty$, 
the boundary maps must be of the form $z \to \lambda_i z$ 
for $i =1,2$. Mapping the crescents to  strips $S_i = \{z= 
x+ i y: 0\leq y  \leq \alpha_i \}$ by a logarithm, these maps 
become $z \to z+t_i$. The boundary map can be extended 
to the interior by a unique affine map
 $T:(x,y) \to  (x+t_1 +(t_2-t_1)y/\alpha_1, y 
 \frac {\alpha_2}{\alpha_1} )$.
When this map is conjugated back to a map between the crescents, it 
defines a quasi-conformal map with minimal possible dilatation 
extending the given boundary values (e.g. Theorem 3.1 of \cite{EM-log-spiral}
for a simple proof; strict equality actual holds 
\cite{Beurling-Ahlfors}, \cite{Strebel}).
 We shall call such a map an affine-crescent map.

Suppose we are given an $\epsilon$-Delaunay triangulation in 
a region $\Omega$ and a map $f: \Omega \to \Omega'$ that sends
the vertices to the vertices of another $\epsilon$-Delaunay
triangulation. On each $\tilde T$, define $g$ to be the 
M{\"o}bius transformation defined by the images of the three 
vertices. On the crescents separating two ideal triangles, 
define $g$ to be the affine crescent map extending the definition 
on the boundary of the crescent.
Thus $g$ is an approximation to $f$ that is M{\"o}bius on the
ideal triangles and quasiconformal on the crescents. If $f$ is 
M{\"o}bius on the quadrilateral formed by two adjacent triangles, 
then it is $g=f$ on the two corresponding ideal triangles and the 
crescent separating them. Otherwise the quasiconformal
constant of $g$ is bounded in terms of the quasiconformal 
constant of $f$ and the  hyperbolic size of the triangles.

We will call an (infinite) $\epsilon$-Delaunay triangulation an 
$(\epsilon,s)$-triangulation for $\Omega$ if every edge
has hyperbolic diameter $\sim s$  in $\Omega$ and the 
circumcircle of every triangle has hyperbolic diameter
$\sim s$. Next we 
want to observe that such a triangulation always exists.

The plane can be tiled by a collection 
 equilateral triangles ${\cal T_n}$ of side 
length $2^{-n}$ in such a way that the  each triangle 
of size $2^{-n}$ is a union of four triangles  in 
${\cal T}_{n+1}$. 
Given a point $ x \subset \Omega$  and $0<\lambda  < \frac 1{4}$ there is a 
triangle $T \in {\cal T_n}$ that contains $x$ and so that 
$$ (\lambda/2) \dist(T,\partial \Omega) \leq
\ell(T) \leq \lambda\dist(T,\partial \Omega)$$
and it is unique except when $x$ is on the common boundary 
of a finite number ($\leq 6$) of such triangles.  Any two triangles 
that satisfy this condition have adjacent sizes (since they
are both comparable to the same number within a factor of two).

So we can cover $\Omega$ by a union 
of triangles whose interiors are disjoint and each is
approximately size $\lambda$ in the hyperbolic metric. We 
claim that by adding 
some extra edges we can preserve this property and also get 
a $\epsilon$-Delaunay triangulation.
To see how, form a triangular mesh by taking the lattice triangles whose
size is comparable to the distance to the boundary.

We can also arrange that if two triangles meet a common triangle, then they 
must be of adjacent sizes.
To see this, suppose $T_1$ and $T_2$ are adjacent and 
$T_2$ and $T_3$ are adjacent and that $T_1$ is the largest of
the three triangles.  Let $\ell(T)$ denote the side length of 
a equilateral triangle. Then
\begin{eqnarray*} \label{est10}
 \ell(T_3) &\geq&  (\lambda /2\sqrt{2}) \dist(T_3, \partial \Omega) \\
     &\geq& (\lambda/2\sqrt{2})[\dist(T_1, \partial \Omega) - \ell(T_2) - \ell(T_3)] \\
         &\geq& (\lambda /2\sqrt{2})( \frac 1 \lambda - 1 -1) \ell(T_1) \\
         &\geq &(\frac 1{2 \sqrt{2}}  - \frac {\lambda}{\sqrt{2}}) \ell(T_1)\\
          &> &\frac 14 \ell(T_1),
\end{eqnarray*}
if $\lambda$ is small enough.
Since $\ell(T_1)/\ell(T_3)$ is a power of $2$ we must have 
$\ell(T_3) \geq \ell(T_1)/2$, as desired.

 If two adjacent triangles are different sizes then some 
interior edges must be added to the larger one to make it 
a triangulation (but the smaller one does not hit an even smaller
one by our previous calculation, so it does not need to be divided).
 There are three cases.
\begin{enumerate}
\item   If the larger one is bordered 
on all three sides by smaller ones, then we divide it into
four equilateral  triangles in the usual way.  
\item If is bounded
on exactly one side by smaller triangles, we add the bisector
of the opposite angle. 
\item  If it is bounded on exactly two 
sides by smaller triangles, we add the segment parallel to
the third side $e$ and half its length and the three segments connecting 
the midpoint of the new segment to corners of the triangle.
\end{enumerate}
See Figure \ref{border-tri}. This is clearly $\epsilon$-Delaunay.
Indeed the worse case is  in the third case above.  The bottom 
triangle has two angles of size $\alpha = \arctan(\sqrt{3}/2)
\approx .713714 \approx .22718 \pi$ and 
one of angle $\beta = (\pi-2\alpha) \approx .544 \pi$.
Since $\alpha$ is opposite an angle of size $\frac 23 \pi$ and 
$\beta$ is opposite angle of $\frac 13 \pi$, we see that 
every possible quadrilateral is at least $ .106 \pi$-Delaunay.

\begin{figure}[htbp] 
\centerline{ 
	\includegraphics[height=1.5in]{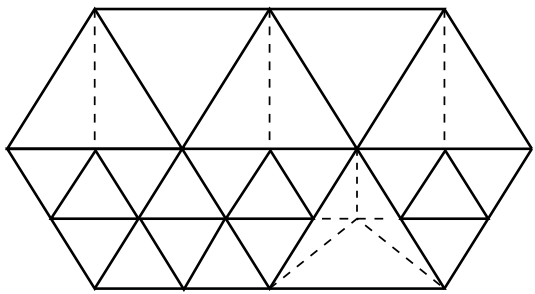}
	}
 \caption{\label{border-tri} When triangles of different size
 meet we subdivide the larger one to make a triangulation. This 
produces $\epsilon$-Delaunay triangulations for a uniform 
$\epsilon >0$}
 \end{figure}

The proof of Theorem \ref{near-Mobius} is now quite simple.
Take a  $(\epsilon_0, s)$-triangulation of $\Omega_1$, restrict 
the map $\varphi$ to the vertices and take $\psi$ to the 
piecewise M{\"o}bius extension of these values to $\Omega_1$. 
If $s$ is smaller than $C/2$ then on the union of any two 
adjacent triangles, the map $\varphi$ is $\epsilon$-close to a M{\"o}bius
transformation, and this  implies $\psi$ is hyperbolic 
biLipschitz with constant $1+O( \epsilon)$ where the constant 
depends only on $\epsilon_0$ and $s$.  This proves Theorem \ref{near-Mobius}.


\section{Computing the bending lamination in linear time} 
\label{lamin=linear}

We have now finished introducing the $\iota$ map and describing 
the relevant estimates. We now start our discussion of the 
algorithm for computing conformal maps, starting with the 
construction of the bending lamination of a finitely bent 
domain in linear time. This will lead to our
decomposition of the plane,  the representation of 
conformal maps  and the method for improving such 
representations.

Recall that $R$ denotes the nearest point retraction from 
a planar domain to its dome.
Given a finitely bent domain $\Omega$, we noted above that 
 $\varphi= \iota \circ R$ is 
a continuous map of $\Omega$ to $\disk$,
 equals $\iota$ on the boundary, is 
  a M{\"o}bius transformation on each gap,  and collapses
the crescents to a union of geodesics $\Gamma$ in $\disk$ called the 
bending lamination. 
 To each 
geodesic $\gamma \in \Gamma$ we associate the angle of the 
corresponding crescent. This is called the bending measure 
of the geodesic and is an example of a transverse 
measure on a lamination.

  A finite lamination $\Gamma$ in the disk lies in the hyperbolic 
   convex hull of its endpoints. If it triangulates the convex 
   hull we say it is complete. We shall assume that our bending 
   laminations are complete, which is always possible by adding
    at most $O(n)$ extra geodesics with bending angle $0$ (since 
     we only need $2n-3$ 
    edges to triangulate $n$ points).

\begin{figure}[htbp]
\centerline{ 
	\includegraphics[height=1.5in]{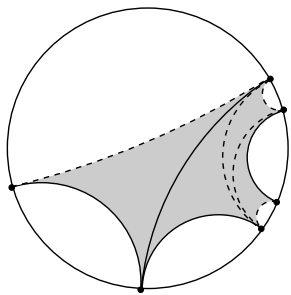}
 }
\caption{ \label{complete-lamin} The convex hull of six points, 
a lamination with these endpoints (solid lines) and a completion
of it (dashed lines).
}
\end{figure}

Next we will check that $\iota$ and the bending 
lamination of a finitely bent domain $\Omega$
can be constructed in linear time, given the 
medial axis of $\Omega$. This is fairly straightforward, 
but we record it formally with some definitions and 
a lemma.
Suppose we have a finite collection  of disks, ${\cal D}$, in the plane and
an adjacency relation between them that makes the collection
into the vertices of a tree. Suppose the disk $D_0$ has been
designated the root of the tree. Then any other disk $D$ has a unique
``parent'' $D^*$ that is adjacent to $D$ but closer to the root.
Assume that  for every (non-root) disk we are given a map
$\tau_D: D \to D^*$. Then we can define a map $\sigma_D: D \to D_0$
as follows. If $D = D_0$, the map is the identity. Otherwise,
there is a unique shortest path of disks  $D_0, \dots, D_k =D$
between $D_0$ and $D$. Note that each disk is preceded by
its parent. Thus $\sigma = \tau_{D_1} \circ \dots \circ \tau_{D_k}$
is a mapping from $D$ to $D_0$ as desired. We will refer
to this as a ``tree-of-disks'' map.

\begin{lemma} \label{tree-of-disks}
With notation as above, assume that every map $\tau_D$ is M\"obius.
Then given $n$ points $\v=\{v_1, \dots, v_n\}$,
with $v_k \in \partial D_k$ for $k =1,\dots,n $, we can compute
the $n$ image points $\sigma(\v) \subset \partial D_0$ in at most $O(n)$ steps.
\end{lemma}

\begin{proof}
 If $D \in {\cal D}$ has positive radius,
choose  three distinct reference points $z_1^D,z_2^D,z_3^D$ on $\partial D$;
otherwise let this collection be empty.
Every other point $z$ on $\partial D$ is uniquely determined
by the cross ratio $\cross(z_1^D,z_2^D,z_3^D,z)$.
Label each point $v$ in $\v$ with the minimal
$k$ so that $v$ is on the boundary of a $k$th generation
disk $D$. For $k=0$ we do nothing to the vertex.
For $k >0$,  compute $\tau_D(v) \in \partial D^*$ and record
the cross ratio $\cross(z_1^{D^*},z_2^{D^*},z_3^{D^*},\tau_D(v))$.
Also compute and record the images of the three reference
points for $D$, i.e.,
$\cross(z_1^{D^*},z_2^{D^*},z_3^{D^*},\tau_D(z_k^D))$ for
$k =1,2,3$.

If a vertex is on $\partial D_0$ then it maps to itself.
If $D$ is a first generation disk, then we just compute
$\tau_D(v) \in \partial D_0$ and compute the $\tau_D$ images of
the three reference points for $D$. For each child of
$D'$ of $D$, we can now compute $\sigma_{D'}$ for any
associated vertices using the previously recorded cross
ratios with respect to the reference points for $D$ and
we can also compute the images on $\partial D_0$
 for the references points for $D'$.  In general, if
$D$ is a disk and we have already computed where the
reference points for its parent are mapped on $\partial D_0$,
we can use the recorded cross ratio information to compute
where the associated vertices and reference points  for
$D$ map to.  This allows us to map every point of $\v$
to $\partial D_0$ in $O(n)$ steps.
\end{proof}

To construct the bending lamination of a finitely bent domain we 
apply this lemma to the collection of base disks corresponding to 
the gaps of the normal crescent decomposition and with two gaps being 
adjacent iff they are separated by a single crescent. In this case
adjacent disks either (1) intersect at exactly two points and we 
take the elliptic transformation that fixes these points and moves 
the child to the parent, or  (2) the disks coincide (if the crescent had
bending angle $0$) and we take the identity map.
We can also compute the gaps and the M{\"o}bius transformations
mapping these gaps to the ones in $\Omega$ in time $O(n)$.

In general, the disks in a ``tree-of-disks'' need not intersect.
In \cite{Bishop-fast}, this lemma is used to construct the exact $\iota$
map for a polygon.  The vertices are the medial axis disks corresponding 
to the vertices of the medial axis with adjacency inherited from the
medial axis. Adjacent disks need not intersect (e.g., consider two ends 
of a long edge-edge bisector), but we can still define an
explicit M{\"o}bius transformation between them (but not an 
elliptic transformation in this case).

\section{Covering  the bending lamination} \label{cover}

Our goal in this section is to cover   the bending lamination
 of a finitely bent domain by standard regions.  Our standard
regions will be ``Whitney  boxes'', which are approximately 
unit hyperbolic  neighborhoods of points, and ``Carleson towers'',
 which look like
unit neighborhoods of long hyperbolic geodesic segments.

The construction can carried out either in the unit disk 
or the upper half-plane. It is slightly easier to draw accurate 
figures in the upper half-plane, so we will describe it there, 
and only trivial changes are  needed to  move it to the 
unit disk.

Given an interval $I \subset \reals $, the corresponding 
 Carleson square is the region in the upper half-plane   of the form 
$ \{ z=x+iy:  x \in I, 0< y < |I| \}$. The ``top-half''
of $Q$ is 
$T(Q) = \{ z \in Q: y >|I|/2 \}.$
This will be called a Whitney box, and its Euclidean  diameter 
is comparable to its Euclidean distance from $\reals$
(abusing notation we may also call them Whitney ``squares'', 
even though they are Euclidean rectangles; the main point is
that they are approximately unit size in the hyperbolic metric).
When $I$ ranges over all  dyadic intervals (i.e., 
all intervals of the form $[j 2^{-n}, (j+1)2^{-n}]$), the corresponding 
 Whitney boxes partition the upper half-plane into 
pieces with approximately unit 
hyperbolic size. See Figure \ref{CarlesonSq}.
Carleson squares are  named after Lennart
Carleson who used them in his  solution of the 
corona problem and they are  now ubiquitous in  
function theory \cite{Carleson}, \cite{Garnett-BAF}. 

\begin{figure}[htbp]
\centerline{ 
	\includegraphics[height=2in]{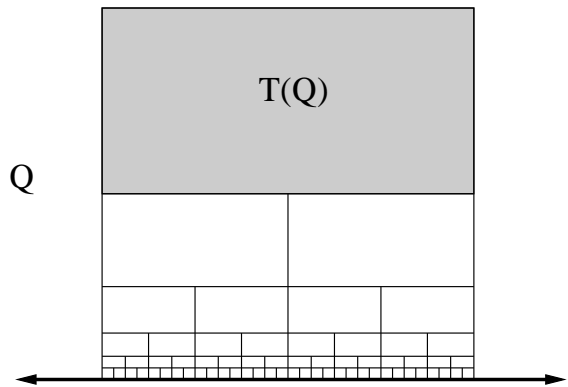}
	}
\caption{ \label{CarlesonSq} A decomposition of a Carleson square
into dyadic Whitney boxes. 
\cite{Garnett-BAF} }
\end{figure}

Dyadic  Carleson squares form a tree under intersection of the 
interiors. Each square has  a unique parent and two children.
The parent of a dyadic Carleson square $Q$ will be denoted
$Q^*$. 
This obviously also induces a tree structure on Whitney boxes.
We will say two dyadic Whitney boxes are neighbors if they are the 
same size and adjacent; each box therefore has a ``left'' and 
a ``right'' neighbor.
One of these is a ``sibling'' in the sense that it shares 
a parent, while the other does not.

Suppose  $\Gamma$ is a complete, finite geodesic lamination 
in $\uhp$. We assume that $\Gamma$ has been normalized so that
its  set of endpoints $S$  satisfies  
$\{ 0,1\} \subset S \subset [0,1]$.

 For $x \in \reals$, let $W(x) \subset \uhp$
 be the Euclidean cone of angle $\pi/2$ and vertex  $ x$ whose 
axis is vertical.
This is called the Stolz cone with vertex $x$.
Then $W = \cup_{x \in S} W(x) $ is  an infinite  polygon 
with $2n$ sides.  See Figure \ref{sawtooth}.  This type of 
region is called a sawtooth domain and is also approximately 
a unit neighborhood of $\Gamma$ (at least if we truncate it 
at height 2). Clearly we can compute $W$ from $S$ in linear time 
if we are given $S$ as an ordered set.

\begin{figure}[htbp]
\centerline{
	\includegraphics[height=1in]{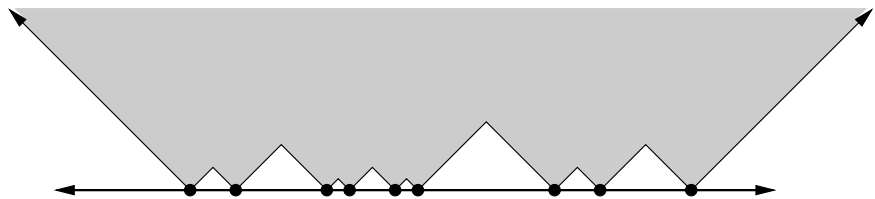}
	}
\caption{ \label{sawtooth} 
The sawtooth domain associated to the set $S$. This region is 
approximately a unit neighborhood of the bending lamination and 
has only $O(n)$ boundary arcs.
}
\end{figure}

We can compute the medial axis of $W$ in time 
$O(n)$ (we don't even need the full strength of the theorem of 
Choi-Snoeyink-Wang: $W$ is a monotone histogram so a 
modification of the simpler 
algorithm  from \cite{AGSS} for convex polygons will work 
in linear time, see \cite{CSW-99}).   In the medial axis of 
$W$,  an edge-edge bisector must be vertical. 
Suppose $A_0$ is a large number (to be fixed later, but 
$A_0 =20$ will work) and consider an edge-edge bisector  $e$
of hyperbolic length $\geq A_0$.

First, suppose $e$ has finite  hyperbolic length.
 This simply means that $e$ is bounded 
away from $\infty$ and the real line (i.e., it is 
a compact subset of $\uhp$).
Let $Q_1$, $ Q_2$ denote the  Whitney boxes 
that contain the top and bottom endpoints of $e$ respectively.
These must be distinct because we are assuming $e$ has large 
length, so its endpoints are in boxes far apart.
Let $e^*$ and $Q_1^*$ denote the vertical projections of $e$ 
and $Q_1$ onto the real axis. If $e^*$ is contained in 
the middle third of $Q_1^*$  then let $Q_3 =Q_1$. If $e^*$ is in 
the left third of $Q_1^*$ then let $Q_3$ denote the union of 
$Q_1$ and its neighbor to the left.  If $e^*$ is in 
the right third, we take the union with the neighbor to the right.
  In every case, $Q_2^*$ now hits the middle 
third of $Q_3^*$. 

The corresponding  Carleson tower is the 
trapezoid that has the bottom edge of $Q_3$ and the 
top edge of $Q_2$ as its bases. See Figure \ref{CT-defn}.
Each tower is associated to a unique edge of the medial axis 
so obviously there are at most $O(n)$ towers associated to 
a set $S$ with $n$ points.  The corresponding Carleson 
arch is $Q_4 \setminus Q_5$ where $Q_4$ is the union 
of  Carleson squares with top edge contained in the bottom 
edge of $Q_3$ and $Q_5$ is the Carleson square whose top edge 
is the top edge of $Q_2$.

\begin{figure}[htbp]
\centerline{
	\includegraphics[height=2.5in]{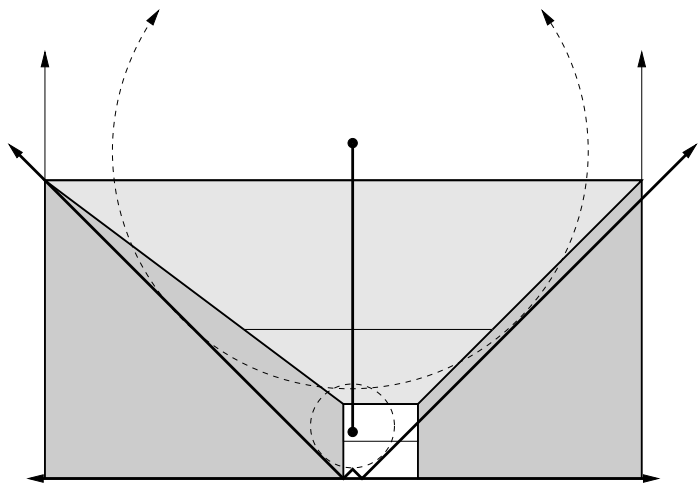}
}
\caption{ \label{CT-defn} 
The union of shaded areas is a Carleson arch.  The lighter 
shaded trapezoid is the corresponding Carleson tower.
}
\end{figure}

The decomposition piece associated to $e$ is $Q = Q_3 
\setminus Q_2$ and is called an arch. It meets the 
real line in two intervals since the base of $Q_2$ is 
much shorter that the base of $Q_3$ and hits the 
center third of $Q_3$'s base.
 If  $e$ has hyperbolic length $L$ then 
the resulting arch is said to have  hyperbolic``width'' $L$ and
the extremal distance in $\uhp$ between the two components
of $Q \cap \reals$ is  $\frac \pi L+O(1)$.
We will call this an $\epsilon$-arch if 
$\epsilon  \geq  \pi / L$. 
Note that if $\epsilon$ is small, then
the two  boundary components of a $\epsilon$-arch in $\uhp$ in $\uhp$ are very 
far apart in the hyperbolic metric ($\sim 1/\epsilon$).

If $e$ has infinite length then the corresponding Carleson tower 
is a triangle with one vertex on $\reals$ where $e$ hits the 
line, and opposite side  that is the bottom of $Q_3$ (defined as above).
The Carleson arch is ``degenerate'' in this case: it is the 
same as the union of Carleson squares $Q_4$ above, except that
it comes with a special marked point in the middle third of its 
base.

Also note that any  two 
hyperbolic geodesics connecting the top and  bottom edges of
the corresponding Carleson tower must be within $O(e^{-d})$
of each other when they are more than hyperbolic distance $d$ from 
either arch boundary.
 Thus the convex hull is 
thin in the arches, in the sense that removing a small ball 
there will disconnect it.   This is not true outside the arches.


If we are given $A > A_0$ then we 
let  $e'$ denote the ``central part'' 
of $e$, i.e., the points of $e$ that are at least hyperbolic
distance $A$ from the endpoints of $e$. Associated to 
$e'$ is a Carleson arch called the ``central arch'' or 
``$A$-arch'' associated to $e$.
Let $\{e_j\}$ be an enumeration of the edge-edge bisectors 
of length $\geq A_0$
in the medial axis of $W$, let $A_j$ be the Carleson arch associated 
to $e_j$ and $A_j'$ the arch associated to its central part $e_j'$
(possibly empty if $e$ is shorter than $2A$).
Let ${\cal K}$ be an enumeration of the connected components of $
\uhp \setminus \cup_j A_j'$ and let ${\cal N}$ be the $A_0$-arches.
(${\cal N}$ is for ``thi{\bf n}'' and ${\cal K}$ is for ``thic{\bf k}''.)
  Then the sets  in ${\cal N}$ and ${\cal K}$ 
cover the whole upper half-plane and  each set only overlaps 
sets of the other type. These overlaps are themselves 
Carleson arches corresponding to segments of hyperbolic length 
$A-A_0$. This covering of $\uhp$ will be referred to as  a
thick/thin decomposition of $\uhp$ with $\eta$-thin overlaps, 
where $\eta^{-1}  \approx A -A_0$.  We will use it later to define a
thick/thin decomposition of a polygon in Section \ref{thick-thin}.

%
%

\begin{lemma} \label{Whitney-near-MA}
 The number of Whitney boxes that 
hit $C(S)$ but do not hit any central Carleson arch is 
 at most $O(A n)$.
\end{lemma}

\begin{proof}
First suppose $A = A_0$.
Since $\Gamma$ is complete and contains $n$ geodesics,  its hyperbolic
 convex hull  can be split into $n-2$ ideal hyperbolic 
triangles and hence the hyperbolic convex hull has hyperbolic 
area bounded by $\pi(n-2)$. See Appendix \ref{background} for a 
discussion of hyperbolic area.  Thus it is enough to show that there 
is a $\delta >0$ and a $C < \infty$ so that each Whitney 
box $Q$ can be associated to a hyperbolic
ball of area $\delta$ contained in the convex hull of $\Gamma$ 
and within hyperbolic distance  $C$ of $Q$.
For, if we place $M$ balls of area $\delta$ into a set of 
area $\leq \pi n$, then  at least
$M\delta/ \pi n$ of them intersect in a common point.
 Since any point can only have a 
bounded number of distinct Whitney  boxes within hyperbolic
distance $C$ of it, we see that $M$ must be bounded by a 
multiple of $n$ (the constant depending on $\delta$ and 
$C$).
 If $Q$ is  a Whitney box that  hits 
$C(S)$, but does not hit any 
$\epsilon$-Carleson arch, then either it is about unit 
size (there are only $O(1)$ such boxes hitting $C(S)$)
or $2Q \cap C(S)$ contains a ball with radius bounded
below uniformly, which completes the proof.

If $ A > A_0$ then the gap between and $A_0$-arch and 
its central $A$-arch can be filled by $O(A)$ Whitney 
boxes.  Since there are $O(n)$ arches, this is at most 
$O(An)$ extra boxes.
\end{proof}

Thus we know that given $n$ ordered points in $\reals$ the convex hull of
these points can be covered by  $O(n)$ Carleson towers and $O(n)$ Whitney 
boxes. Next we wish to show these towers and boxes can actually 
be found in time $O(n)$.


\begin{lemma}
Suppose $\Gamma$ is a finite, complete geodesic 
lamination containing $n$ geodesics. Then in time 
$O(n)$ we can find a covering of $\Gamma$ by at
most $O(n)$ dyadic Whitney boxes and Carleson towers.
\end{lemma}

\begin{proof}
Compute the medial axis of the sawtooth domain $W$ corresponding to 
the set of endpoints $S$.  Locate all the edge-edge bisectors
of hyperbolic length $\geq A_0$ and record the associated  Carleson towers.  

Each tower has a top and bottom horizontal edge. Extend these
edges until they hit the boundary of $W$ (recall that we 
know which edges of $W$ get hit). Near a tower, $W$ looks
very close to a Stolz cone, so the length of these edges is close
to 1 if $A_0$ is large enough.  In particular,  for large $A_0$ it has 
length $<2$ and so at most three 
Whitney boxes suffice to cover each of these edges.  (This determines
our choice of $A_0$).

These edges cut $W$ into connected components.  After we remove the 
components corresponding to towers, we have at most $O(n)$ 
components remaining.  We claim that if one of these 
components hits $m$ dyadic Whitney boxes then we can 
find all these boxes in time $O(m)$.
\begin{figure}[htbp]
\centerline{
 \includegraphics[height=2in]{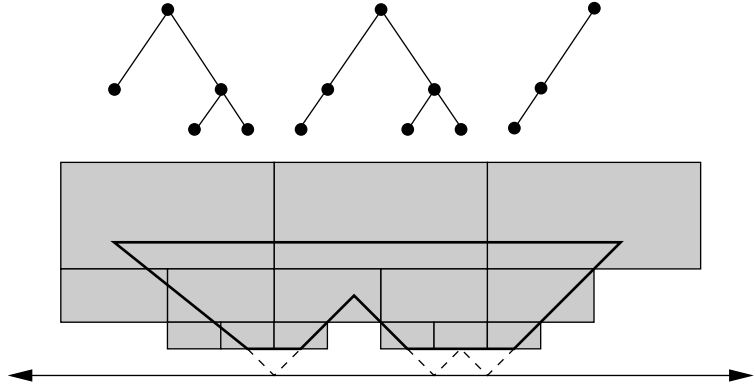}
 }
\caption{\label{Sq-Saw} A truncated  component of $W$ and the 
dyadic Whitney boxes that hit it. These boxes form vertices 
of at most three trees rooted at the topmost boxes. }
\end{figure}

To do so consider one such component $W_0$.  It 
has a single horizontal ``top'' edge of length 
$L$ and a
bottom edge that is  polygonal curve.
Project each 
segment of this curve  vertically to the real line and choose a dyadic interval 
contained in each projection and of comparable size. 
Then enumerate all the dyadic intervals which contain 
at least one of our choice and have length $\leq L$.
If there are $m$ such they can be listed in time 
$O(m)$ (see Lemma \ref{make tree} of Appendix \ref{background}).
The corresponding Whitney boxes, plus a constant number of their 
neighbors,   cover $W_0$, so we are done.
\end{proof}

As a consequence of the proof, we see that we can also record
the adjacency relations between the chosen boxes in time 
$O(n)$.


\section{Extending the convex hull cover to all of $\uhp$}
 \label{extend-decom}

Given a finite set of points $S \subset [-1,1]$ we will 
construct the corresponding ``Carleson-Whitney'' decomposition 
of the upper half-plane by modifying the cover of  the convex 
hyperbolic hull of $S$ from Section \ref{cover}. 
This  decomposition of the upper half-plane will consist 
of the outside of some fixed Carleson  square $Q_0$ and 
a finite number of pieces whose 
union is all of $Q_0$.  These pieces come in four types
(see Figure \ref{CW-pieces1}:
\begin{enumerate}
\item  Carleson squares,
\item Whitney  squares, which are obtained by dividing a 
    Whitney box into a bounded number of equal sized Euclidean squares,
\item Carleson arches, which consist of a Carleson square with a smaller Carleson 
square removed from it; we assume the base of the smaller square 
   hits  the middle third of the base of the larger square and is much smaller.
   \item Degenerate arches, where the smaller square is replaced
      by a single point of $S$, contained in the center third of the base. Moreover, 
     every point of $S$ will be associated to a degenerate arch in this way.
      Note that these pieces are really Carleson squares, but it is 
      convenient to consider them separately from squares $Q$ whose closures 
     do not contain a point of $S$.
\end{enumerate}

\begin{figure}[htbp]
\centerline{
 \includegraphics[height=1.15in]{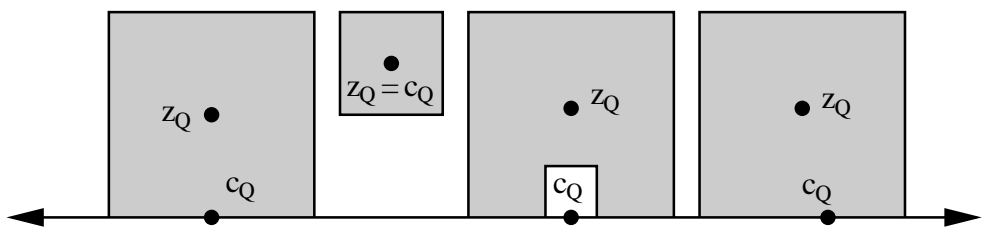}
 }
\caption{ \label{CW-pieces1}  
The shapes  of pieces used in a Carleson-Whitney decomposition:
Carleson squares,  Whitney
 squares,  Carleson arches and degenerate
Carleson arches.
}
\end{figure}

In practice, we could avoid arches by simply using more Whitney type
squares but we would lose the linear dependence on $n$.  For many 
domains, however, this might not be much of a loss, and would 
simplify much of what follows. Moreover, when using finite, 
rather than infinite, precision it may be  more practical to simply cover
any arches by a union of Whitney and Carleson boxes.

\begin{figure}[htbp]
\centerline{
 \includegraphics[height=1.5in]{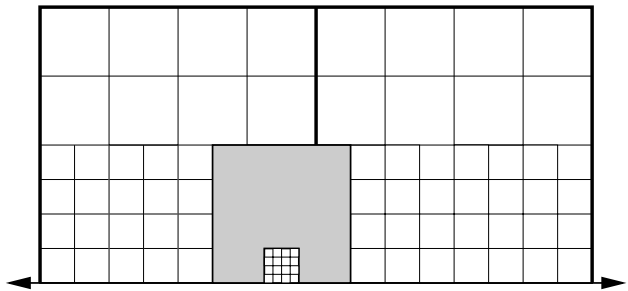}
 }
\caption{ \label{defn-arch}  
Definition of a Carleson arch. The bounding squares of an arch need not be 
dyadic, in which case we tile the region around them by Whitney boxes and 
Carleson squares to fill in a union of dyadic squares.
}
\end{figure}

To construct the covering we fix some $\epsilon >0$, 
take the hyperbolic convex hull
of $S$ and cover it by $O(n)$ $\epsilon$-Carleson arches and Whitney 
boxes as in Section \ref{cover}.
The remaining parts of 
$Q_0$ can be written as a   union of Carleson squares $Q$ whose
top edges meet the bottom edges of Whitney boxes in the cover 
above. The part of the boundary of $Q$ that hits the existing cover is 
a connected piece $E$  of the boundary that contains the top edge.
  We then  subdivide the Carleson box near $E$  into Whitney type and
Carleson squares so that any two adjacent squares have comparable 
size and that subsquares touching $E$ have size comparable to the 
distance from $\reals$. 
See Figure \ref{divide-square} to see how to do this.
 This insures that in our decomposition 
of $\uhp$ when two pieces meet, the intersecting boundary components
have comparable size (this is also true for the pieces themselves, 
except for arches, which may be much larger than adjacent pieces
underneath them; however, the lower boundary of the arch is comparable in 
size to the adjacent pieces). Moreover, the number of squares created is 
bounded by a multiple of the number of Whitney squares in the
cover, so is $O(n)$.

\begin{figure}[htbp]
\centerline{
 \includegraphics[height=2in]{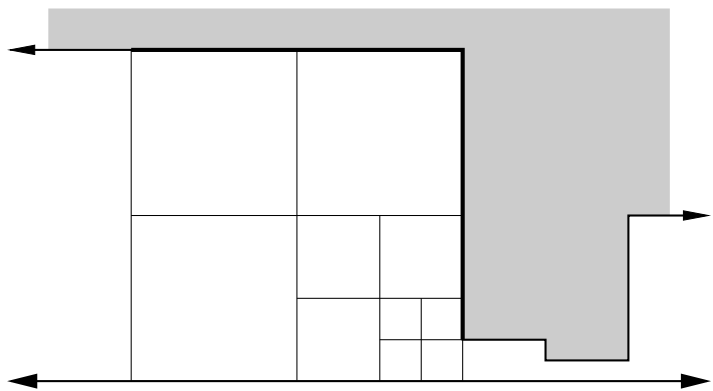}
}
\caption{ \label{divide-square}  
A Carleson square that lies below our cover of the convex hull
can be subdivided into subsquares so that any two adjacent 
squares have comparable size.  The number of squares needed
is bounded by a multiple of the number of squares in the 
covering region and hence is $O(n)$.
}
\end{figure}

The final step of the construction is to subdivide pieces as 
follows.
Each  Whitney  square  is subdivided into a bounded
number of Whitney type squares $Q$  with the property that 
$M Q$ does not hit a point of $S$ ($M$ is a fixed large number 
to be specified later and $MQ$ denotes the box concentric with 
$Q$ and $M$ times larger). For example, see 
Figure \ref{regions}. Similarly, we divide 
Carleson squares into a bounded number of Whitney type squares
and Carleson squares so that each can be expanded by a 
factor of $M$ without hitting $S$ (this can be done since each 
point of $S$ is contained in a degenerate Carleson arch, and 
hence any Carleson square in our decomposition is separated 
from $S$ by a uniform multiple of its diameter).
We will call these $M$-Whitney squares or $M$-Carleson
squares for $S$.
Similarly we replace arches by thinner arches $Q$ so that 
$M Q$ misses $S$ and tile the resulting gaps by Whitney type squares 
and Carleson boxes (for an arch $MQ$ means expanding the bigger square 
by a  factor of $M$ and shrinking the smaller one by a factor of $M$).

$M$ will be fixed later to insure various properties.
For example, it is easy to verify the following fact
we will need using hyperbolic geometry and the fact 
that M{\"o}bius transformations are hyperbolic isometries.

\begin{lemma} \label{defn-M}
 If $M$ is large enough the following holds. Suppose $Q_1$ is 
an $M$-Whitney square for $S$ and  
$Q_2$ is a $M$-Whitney square for $\tau(S)$ where $ \tau $ is a M{\"o}bius 
transformation of $\uhp$ to itself and  that $\tau(Q_1)
\cap Q_2 \ne \emptyset$. Then $2 Q_2 \subset \tau(10 Q_1)$.
\end{lemma}

A similar result holds for Carleson squares. For Whitney and
Carleson squares, 
$$ \diam(Q) \sim \frac 1M \dist(Q, S),$$
but this is not true for arches.  However, if $E$ is a 
boundary component of an arch then it is true that 
$$ \diam(E) \sim \dist(E, S).$$

\begin{figure}[htbp]
\centerline{
 \includegraphics[height=1.15in]{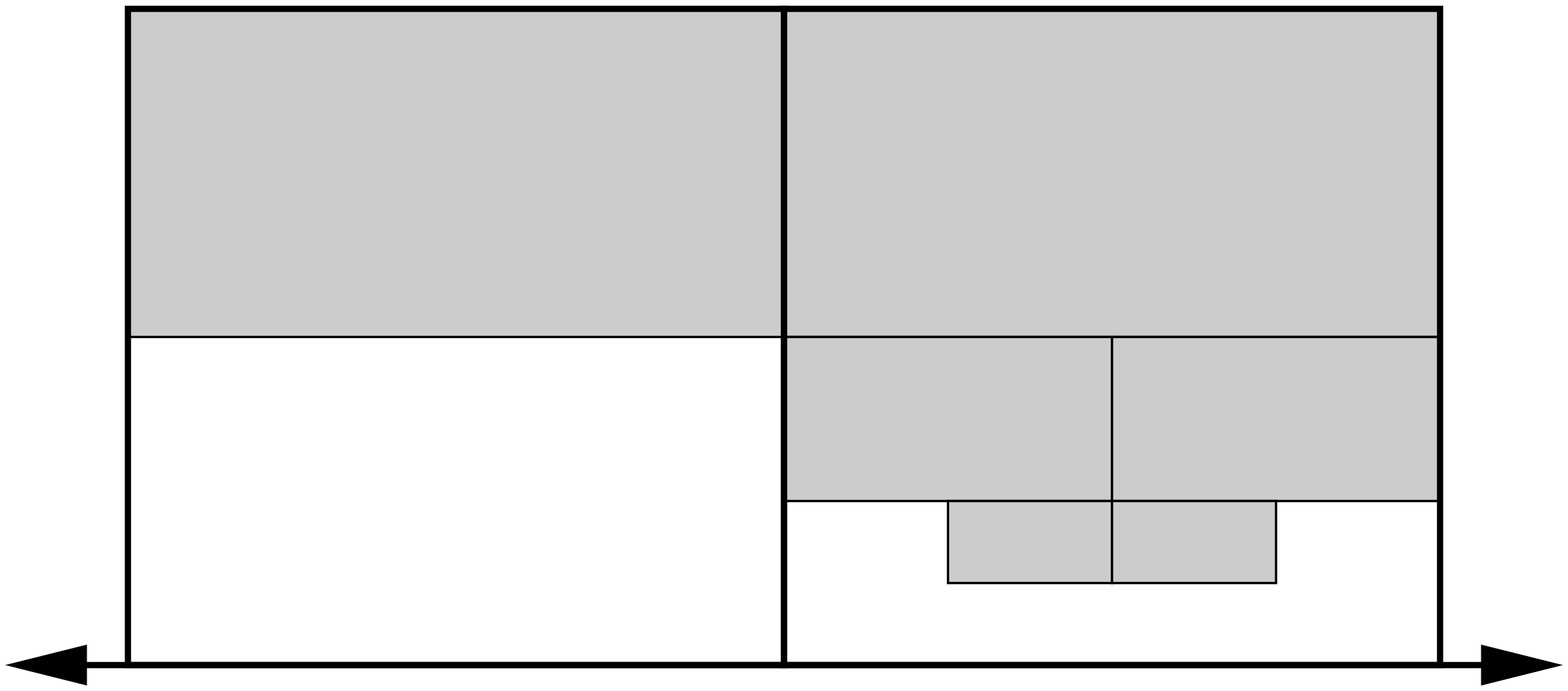}
$\hphantom{xx}$
 \includegraphics[height=1.15in]{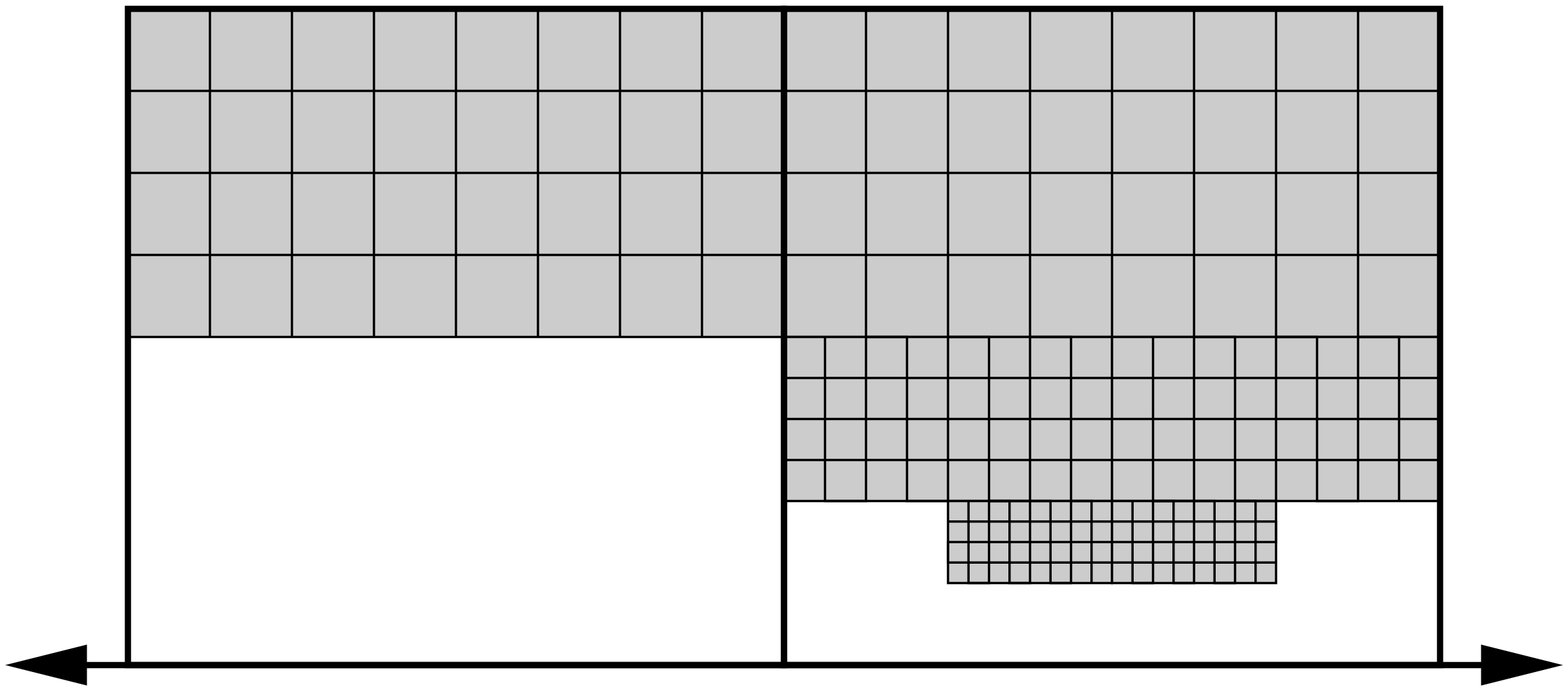}
}
\vskip.2in
\centerline{
 \includegraphics[height=1.15in]{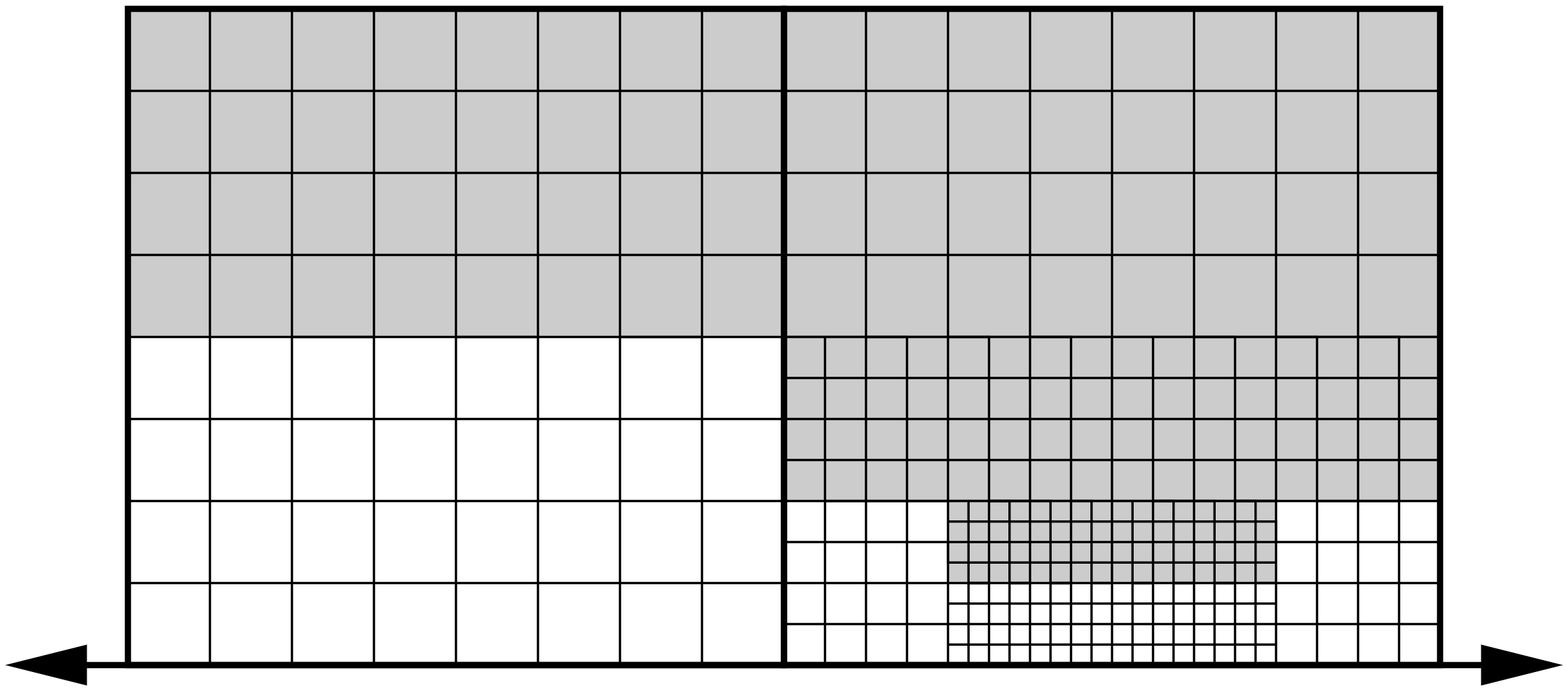}
 }
\caption{ \label{regions}  
On the upper left is a collection of Whitney squares, on the upper right 
 a division 
of these into Whitney type squares and on the bottom we subdivide the 
region below the Whitney squares into Whitney type squares and Carleson 
squares so that adjacent squares are of comparable sizes.
}
\end{figure}

The pieces of our decomposition form the vertices of a tree, where 
the parent of a piece is the piece that is adjacent and directly above it.
The leaves of the tree are either Carleson squares or degenerate Carleson arches. 
We will also consider the graph that results from adding edges corresponding 
to adjacency of pieces (left/right as well as up/down).

 Given an element of such a decomposition we 
consider the part of its boundary in the open upper half-plane.
For Whitney type squares this is a square, for 
Carleson squares it is an arc (the top and sides of the rectangle)
and for arches there are two components, each of which is the top
and sides of a square.
Given a boundary component $E$ of one of these types, 
we let $\diam(E)$ denote its Euclidean diameter and let $N_s(E)$ be 
the $s \cdot \diam(E)$  Euclidean neighborhood of $E$. 
Throughout the paper we only need one fixed value of $s < \frac 14$, 
say $s = 1/10$.
Let $N_s$ be the union of these 
sets over all boundary components of all pieces of our 
decomposition.
Similarly we define $N_s(\partial Q)$ to be the union of $N_s(E)$ 
over all boundary components of $Q$ (one such for boxes, two for arches)
and let $N_s(Q)$ be the union of this with $Q$.

\begin{figure}[htbp]
\centerline{
 \includegraphics[height=1.5in]{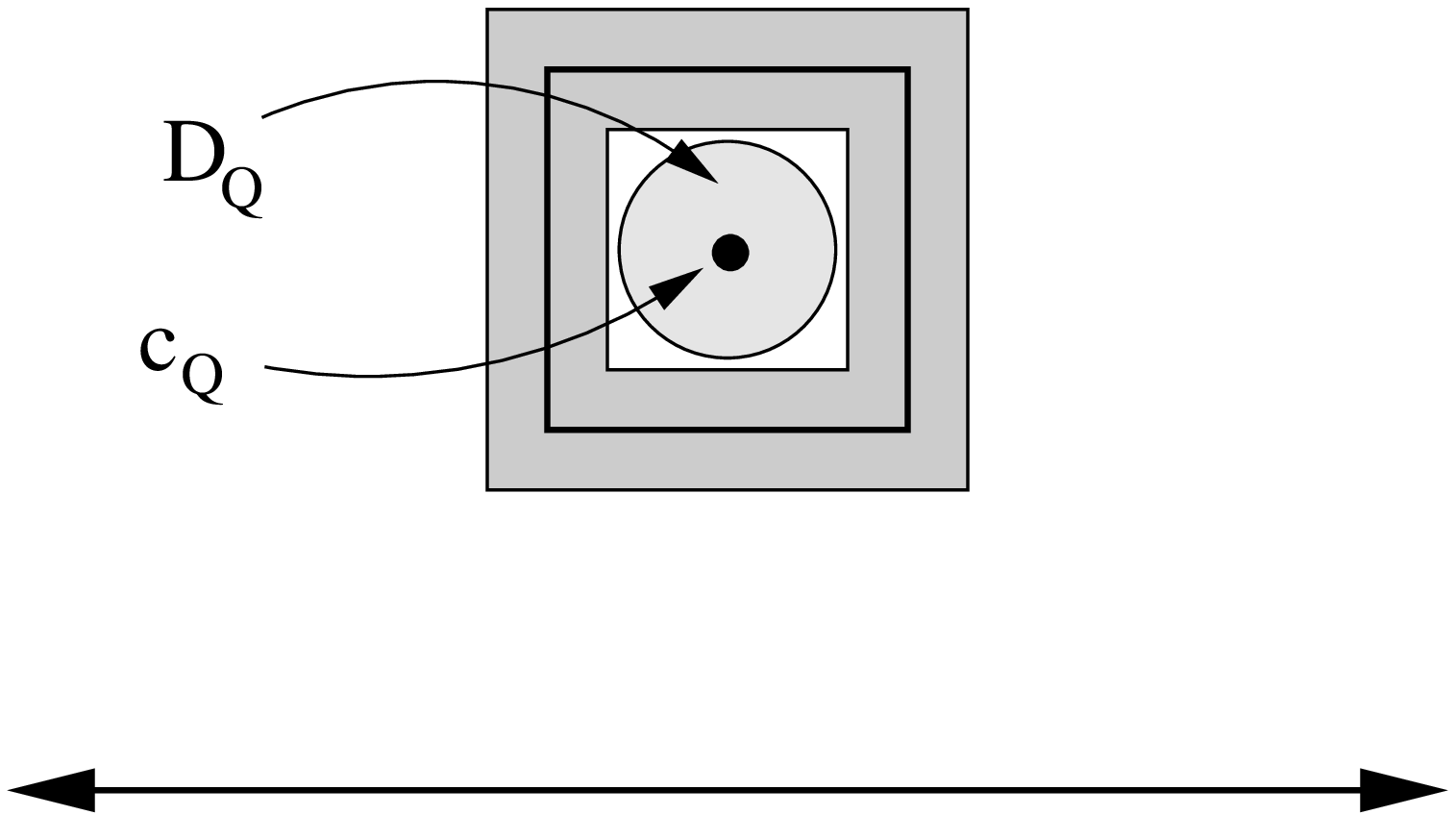}
$\hphantom{xxxx}$
 \includegraphics[height=1.6in]{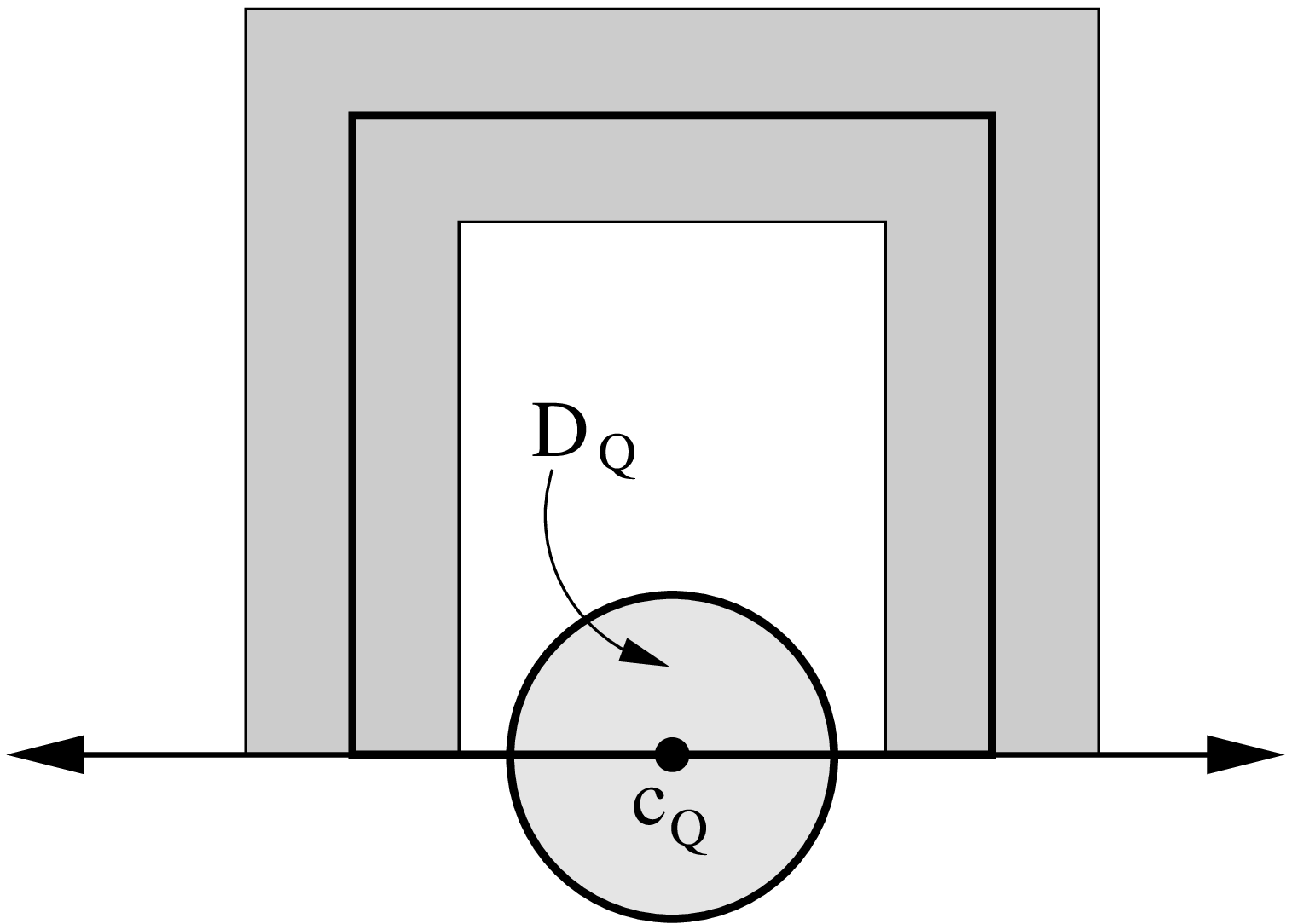}
}
\vskip.3in
\centerline{
 \includegraphics[height=2.5in]{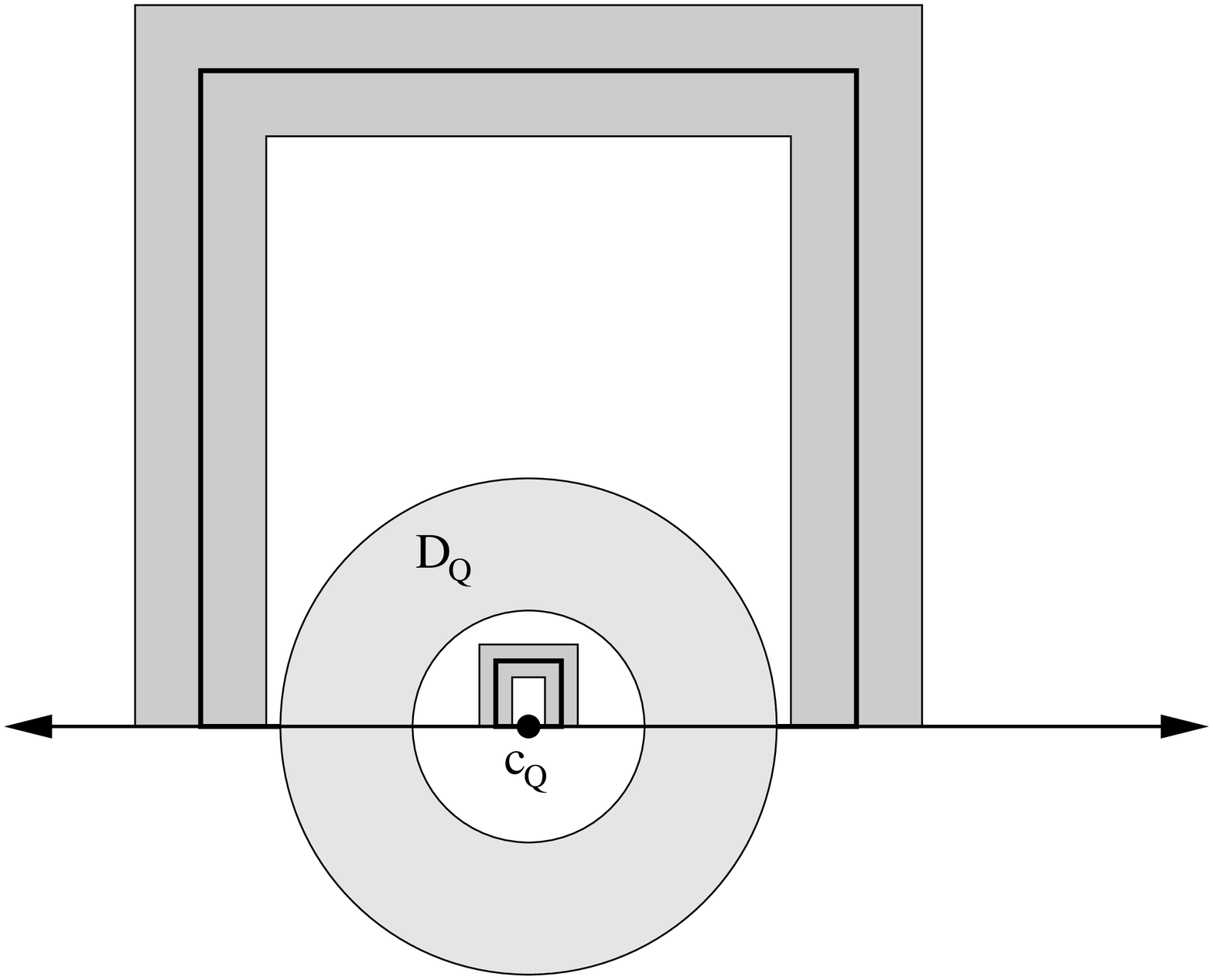}
 }
\caption{ \label{nbhds}
This shows the boundary neighborhoods (darker), centers and empty regions 
(lighter) for Whitney squares, Carleson squares and Carleson arches.
For a  degenerate arch these are the same as for the Carleson  square.
The picture of the arch is deceptive since the scales for the 
two boundary components should be much farther apart, and the empty 
region much ``thicker''.
}
\end{figure}

If $Q$ is a piece  of our decomposition, we let $z_Q$ denote its center
(for squares) or the center of the larger square (for arches).
Similarly,  we let $c_Q$ denote its 
center if $Q$  is a Whitney type square; the center of its base if 
$Q$ is a Carleson square; and the center of the base of the smaller
square if $Q$ is a Carleson arch and the associated point of $S$ 
for degenerate arches. See Figures \ref{CW-pieces1} and  \ref{nbhds}.

Given a disk $D = D(x,s)$  and a number $r >0$ we let 
$r D=D(x, r s)$.  Similarly, given 
an annulus $ A=\{ s< |z-x| < t\}$ we let 
$r A = \{ s/r< |z-c_Q| < r t\}$.  With this 
notation we see that each piece $Q$ of our decomposition contains 
a disk or half-annulus $D_Q$ such that 
$$D_Q \subset Q  \setminus N_s(\partial Q)
\subset N_s(Q) \subset  \lambda_s D_Q, $$
where a little arithmetic shows
\begin{eqnarray} \label{lambda-defn}
  \lambda_s = \frac {\sqrt{(1+s)^2+(\frac 12 +s)^2}}{\frac 12 -s},
\end{eqnarray}
and where $D_Q$ denotes a disk for square pieces and an annulus for arch
pieces; we use the same letter in both cases to simplify notation.  
We call these the ``empty regions'' associated to each piece of the 
decomposition; ``empty'' because  when we define quasiconformal mappings
from $\uhp$ to $\Omega$, the 
 Beltrami dilatation $\mu$  will be supported 
in $N_s$,  which is disjoint from $D_Q$. Thus our maps 
will be conformal in the empty regions and so will have power series 
or Laurent series expansions there. These series will be how
we record our quasiconformal maps.

Given the decomposition ${\cal W}$ we let ${\cal B}$ (for ``boundary'')
be collection 
of $O(n)$ closed squares with disjoint interiors that lie 
in $N_s$ and whose union cover $N_s$. See Figure \ref{bdy-cover}. 
Only a uniformly bounded number of squares are used for each boundary 
component and each such  square is associated to a disk in the arch 
(and disjoint from the covered neighborhood of the boundary) whose 
triple covers the square.  This disk is the ``empty region'' 
associated to the square. 

\begin{figure}[htbp]
\centerline{
 \includegraphics[height=2.25in]{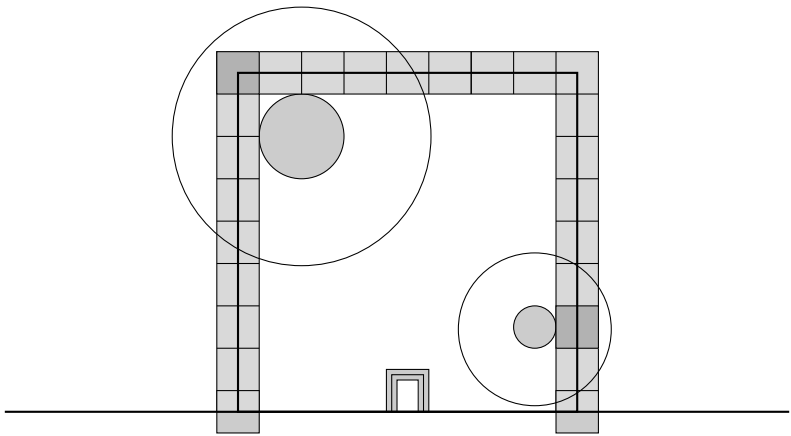}
 }
\caption{ \label{bdy-cover}
A neighborhood of an arch boundary can be covered by squares, and 
each square is associated to a disk inside the arch whose triple
covers the square.  These disks are the ``empty regions'' associated 
to the squares and will be used to define ``partial representations'' 
in Section \ref{epsilon-reps}.
}
\end{figure} 



\section{$\epsilon$-representations of polygonal domains}
\label{epsilon-reps}  

One problem in approximating a conformal map $f:\disk \to \Omega$ 
(or in the reverse direction) is to decide how to represent 
the function.  One obvious approach would be to use a truncation
 of 
the  power series of $f$, $f_n(z) = \sum_{k=0}^n a_k z^k$, 
but this converges slowly (since $f'$ is discontinuous at the 
prevertices) and the number of terms needed for a given 
accuracy depends on the
geometry of the domain.
We will use a representation using $O(n)$ different 
series that avoids these problems.

 Given 
a Carleson-Whitney decomposition ${\cal W} $, 
it is easy to see that there
is a corresponding piecewise polynomial partition of unity $\{ \varphi_k\}$ 
whose   gradients are supported in $N_s$. More precisely, there is 
 a collection of piecewise polynomial  functions  (uniformly 
bounded degree) such that
\begin{enumerate}
\item $0 \leq \varphi_k \leq 1$,
\item  $\supp(\varphi_k) \subset N_s( Q_k)$,
\item $\sum_k \varphi_k(z) =1$ for all $z$, 
\item $\supp (\nabla \varphi_k ) \subset N_s(\partial Q_k)$,
 \item $ |\nabla \varphi_k (z) | \leq C/(s \cdot \diam(E))$ for $z \in N_s(E)$, 
when $E$ is a component of $\partial Q_k$.
\item  $|\nabla^2 \varphi_k (z) | \leq C/(s \cdot \diam(E))^2$ for $z \in N_s(E)$, 
when $E$ is a component of $\partial Q_k$.
\end{enumerate}

To prove this, consider Figure \ref{PartitionTypes}.  On the top 
it  shows part of the decomposition and a covering of the boundary arcs by 
small shaded squares and trapezoids (which allow the squares to shrink as we
approach the boundary). Outside these shaded  regions our functions 
are constant; either $0$ or $1$. Within the shaded squares they 
interpolate between $0$ and $1$. There are only five types of 
regions, and we need only show how to build partition functions for each 
type.
\begin{figure}[htbp]
\centerline{
 \includegraphics[height=2.75in]{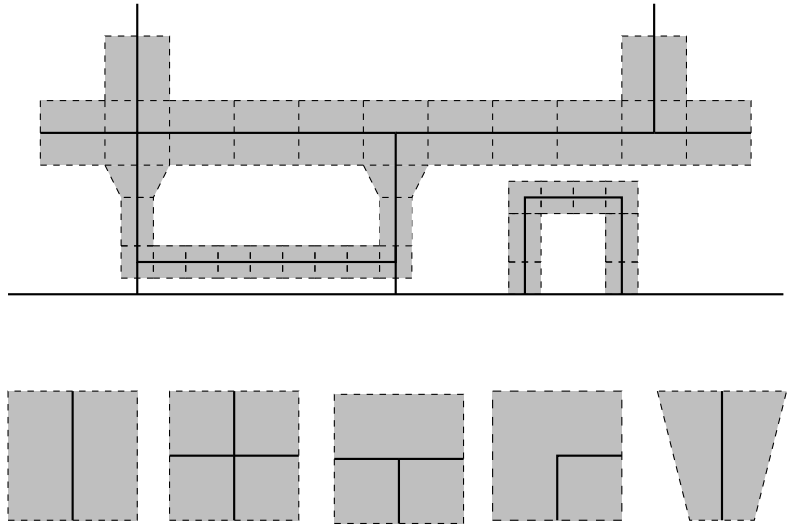}
 }
\caption{ \label{PartitionTypes}
The partition of unity is non-constant only of these types of regions. 
}
\end{figure} 

Let $f(x) =   (1-x^2)$  and 
$$g(x) = \int_{-1}^x f(t) dt/\int_{-1}^1 f(t) dt  
      = \frac 12 + \frac 34 ( x - \frac 13 x^3).$$
Then $g(-1) = 0, g(1) =1$ and $g' =0$ at these two points (by starting 
with $ f(x) =(1-x^2)^k$ we could get higher derivatives to also 
vanish, if needed).  Let $h(x)= g(x)$ if $-1 \leq x \leq 1$, $h(x) =0 $
for $x <-1$ and $h(x) =1$ for $x >1$. For the leftmost square 
on the bottom of Figure \ref{PartitionTypes}  we can take  the functions
$\varphi(x,y) = h(2x)$ and $1-\varphi$ (if $Q = [-1,1]^2$; otherwise we apply 
a Euclidean similarity to make $\varphi$ fit into $Q$).
These are illustrated in the top row of  Figure \ref{PartitionFunctions}.
The next three squares in  Figure \ref{PartitionTypes}  are handled by 
the functions 
$$ h(2x) h(2y), (1-h(2x))h(2y), (1-h(2x))(1-h(2y)), h(2x)(1-h(2y)).$$
$$  h(2y), h(2x)(1-h(2y)), (1-h(2x))(1-h(2y))$$
$$ h(2x) (1- h(2y)) ,  1- h(2x)(1-h(2y)),$$
respectively. The trapezoid uses 
$$ h(2x \cdot (2- h(2y)) ),  1- h(2x \cdot (2- h(2y)) ) .$$
It is easy to check that these function have the desired properties.

\begin{figure}[htbp]
\centerline{
 \includegraphics[height=1.3in]{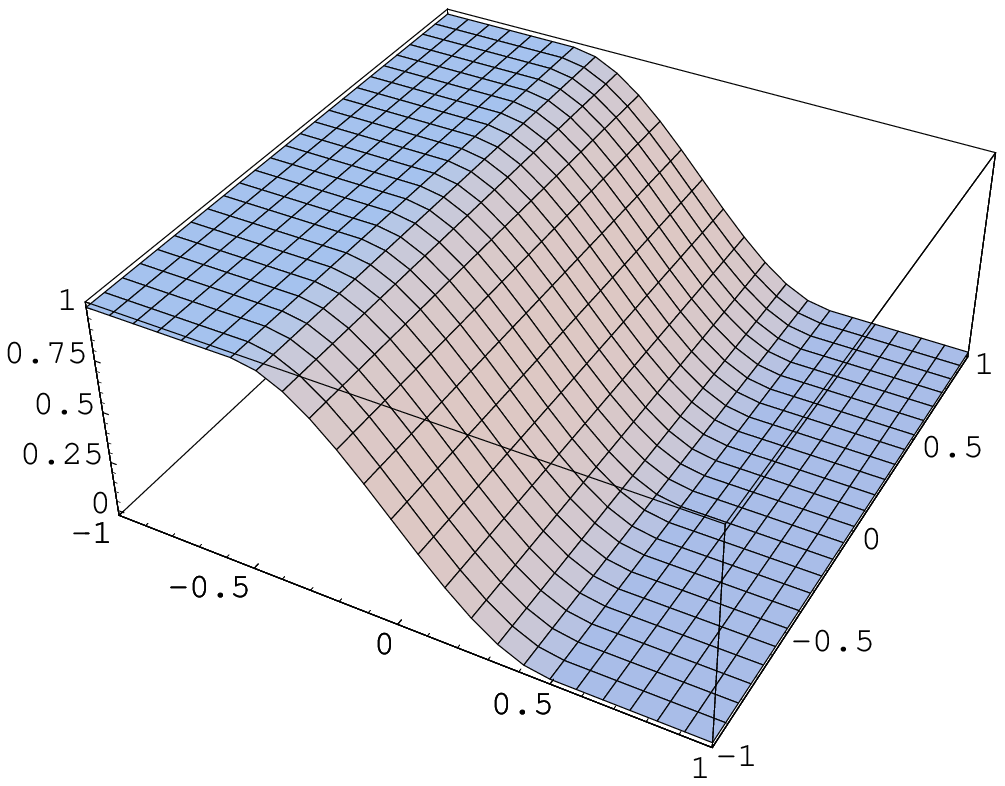}
$\hphantom{xxxx}$ 
 \includegraphics[height=1.3in]{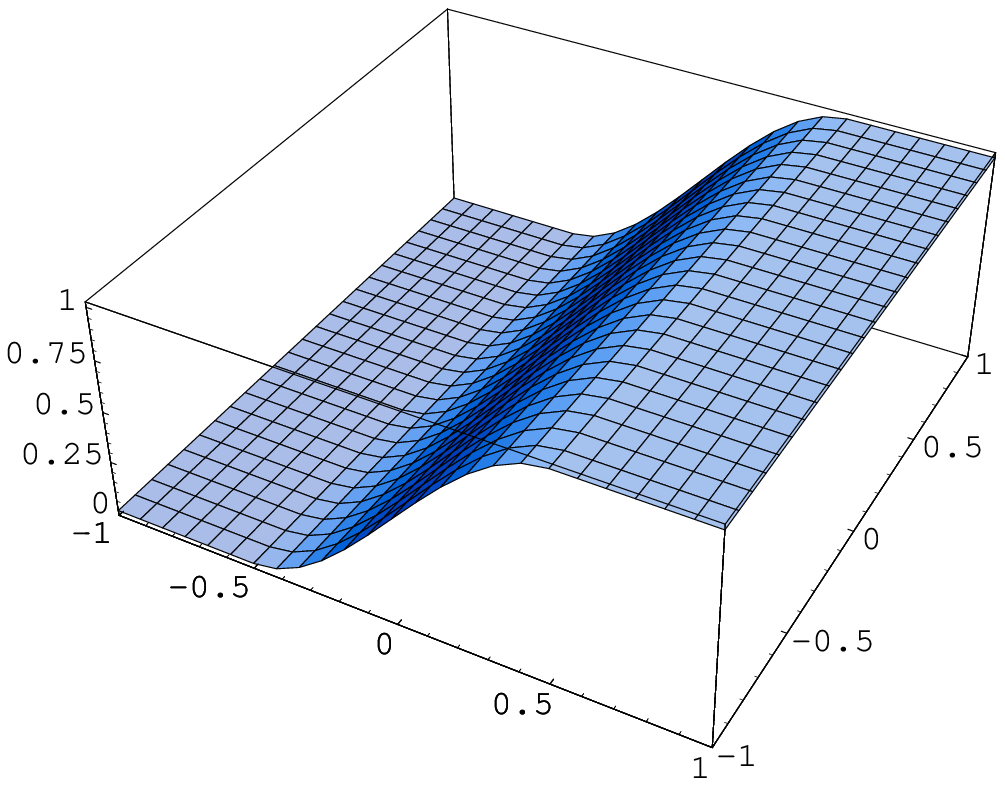}
 }
\centerline{
 \includegraphics[height=1.3in]{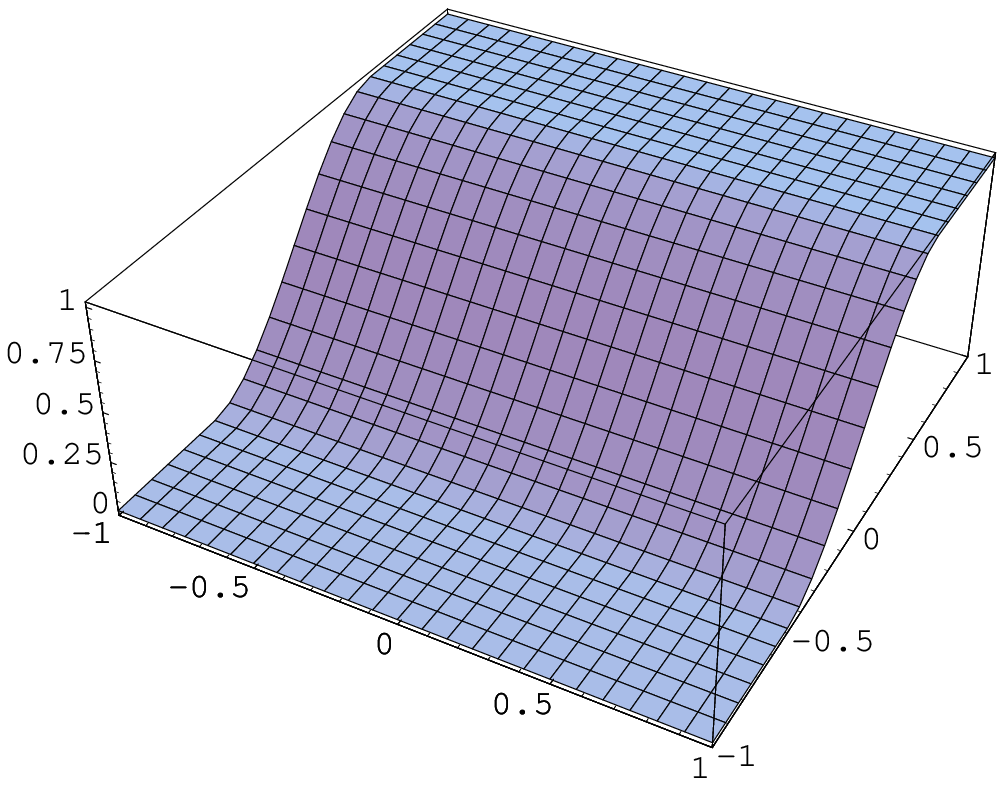}
$\hphantom{xxxx}$ 
 \includegraphics[height=1.3in]{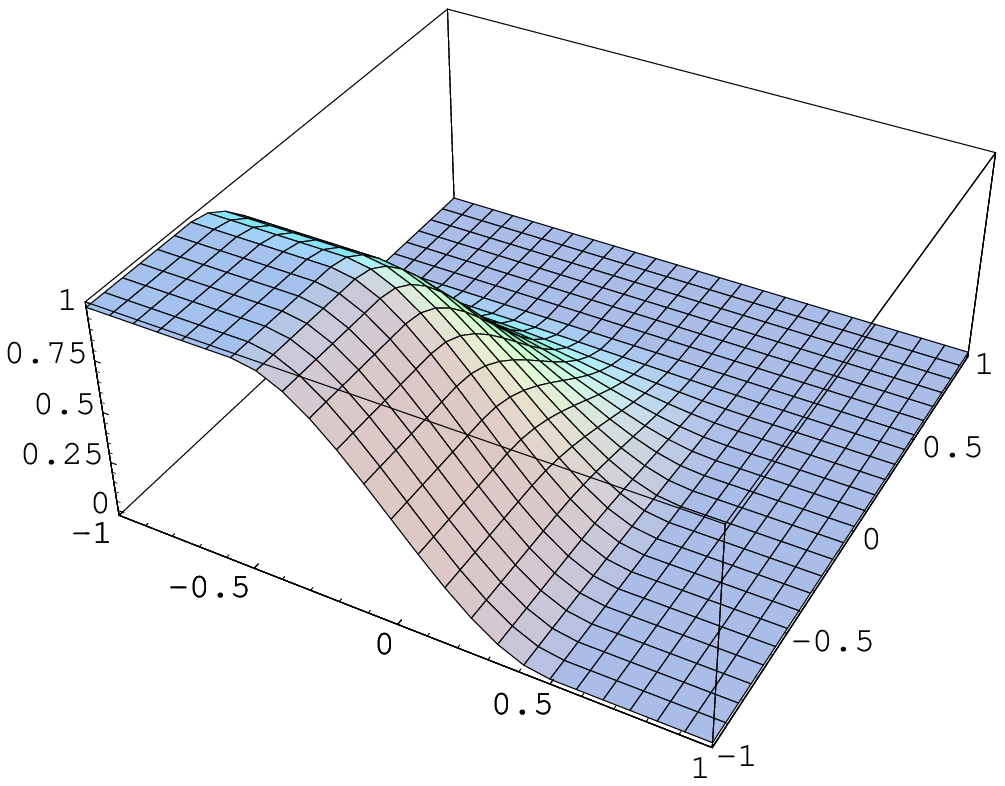}
$\hphantom{xxxx}$ 
 \includegraphics[height=1.3in]{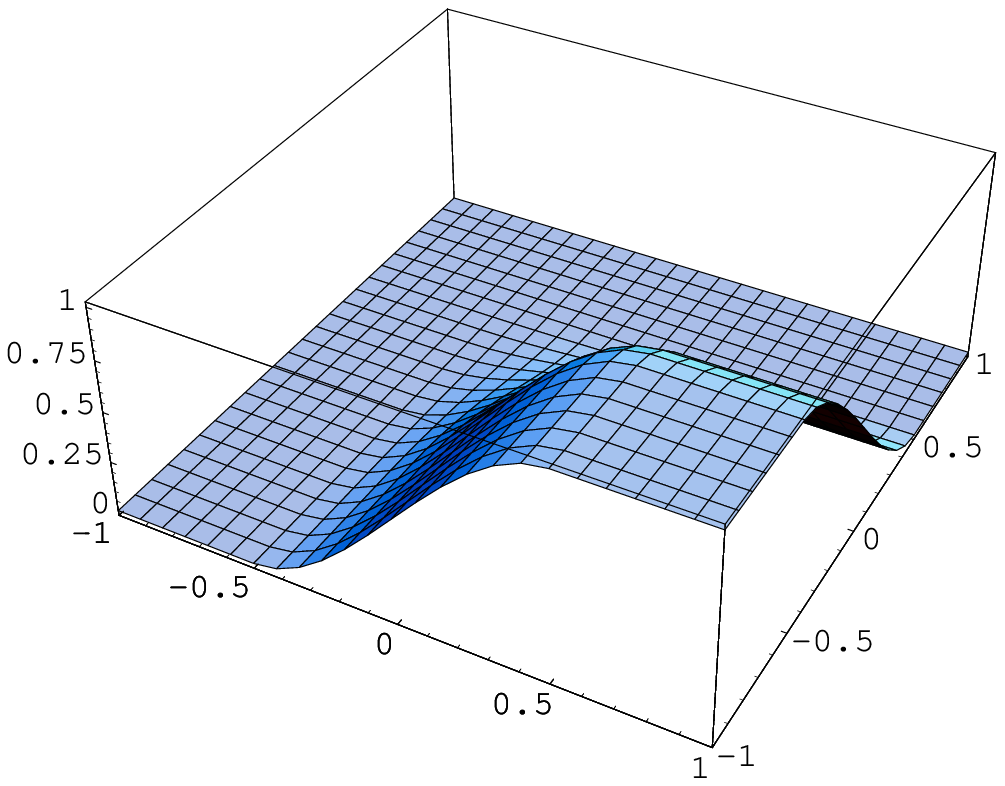}
 }
\centerline{
 \includegraphics[height=1.3in]{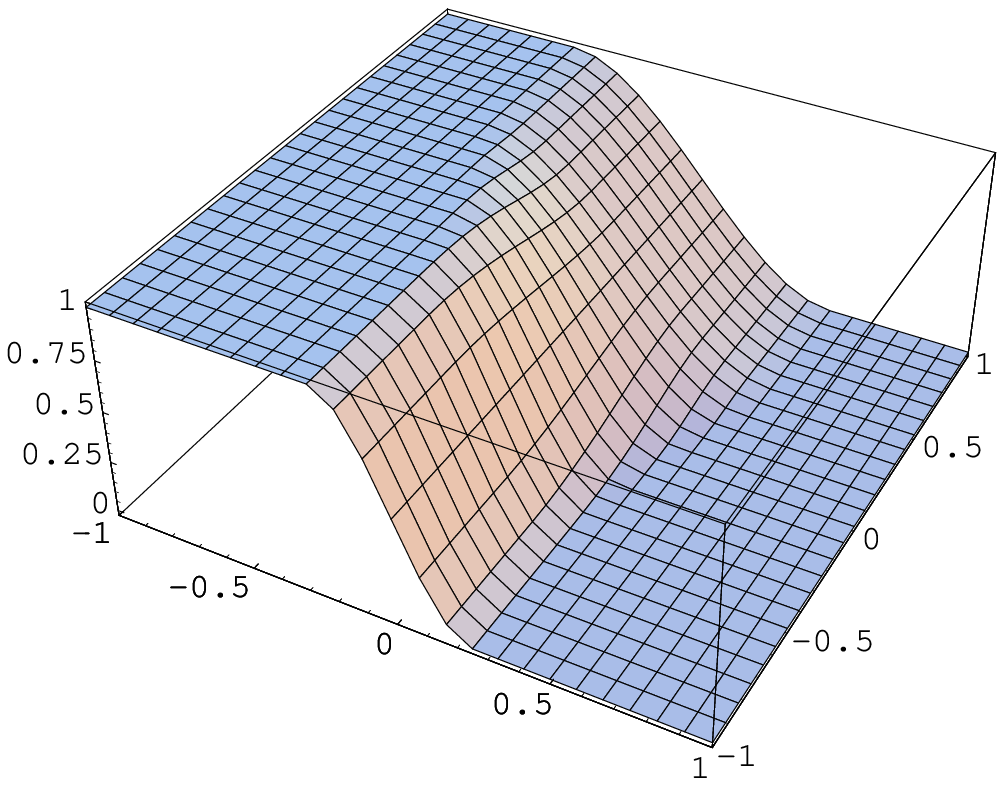}
$\hphantom{xxxx}$ 
 \includegraphics[height=1.3in]{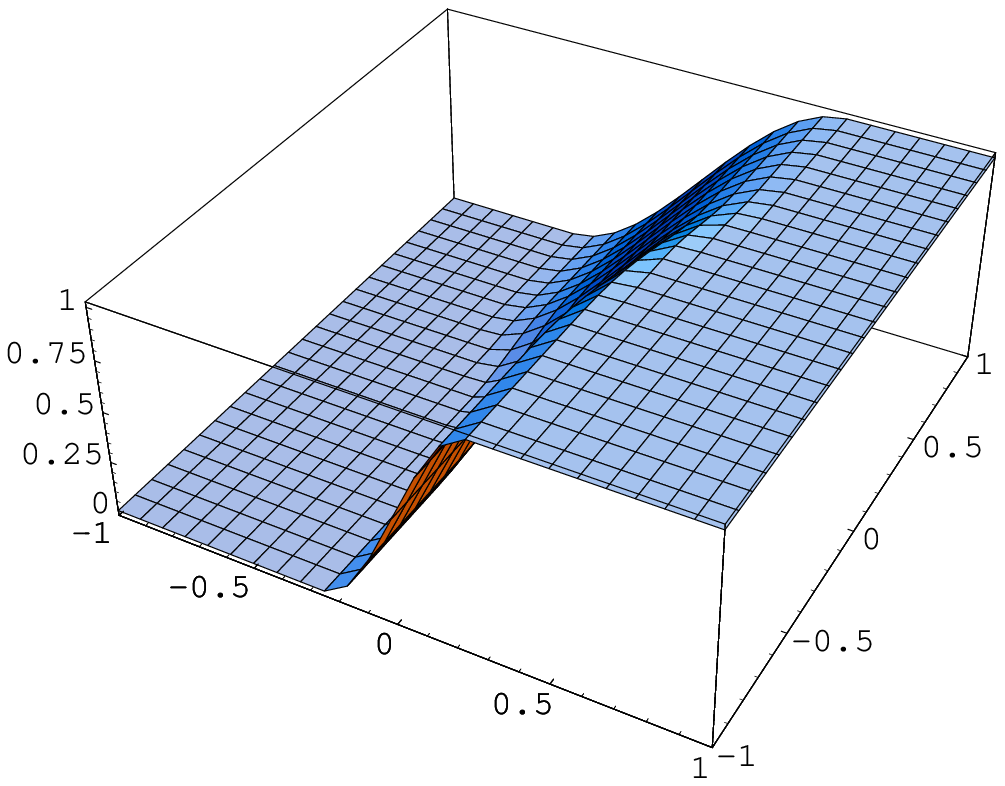}
 }
\caption{ \label{PartitionFunctions}
Some of the functions used  in our piecewise polynomial partition of unity.
The three rows correspond to the leftmost, center and rightmost regions at the 
bottom of Figure \ref{PartitionTypes}.
}
\end{figure}

Now suppose that for each piece $Q_k \in {\cal W}$ 
of our decomposition we have a
conformal map $f_k$ defined on $N_s(Q)$ (we think  of this as 
a local approximation to the globally defined conformal map $f$).
This collection of functions is denoted ${\cal F}$   and 
has the same index set, $I$,  as ${\cal W}$.
We say that ${\cal F}$  represents $\Omega$ if each
  point of $S$ is mapped to the corresponding vertex 
of $\partial \Omega$, and each component (if there is more than one)
of $\partial Q_k \cap \reals$ is mapped into the corresponding   edge of $\Omega$.
 Such a ${\cal F}$ is an $\epsilon$-representation if 
each $f_k$ is 1-biLipschitz with respect to the hyperbolic metrics 
and whenever
 $E_1$ and $E_2$ are 
boundary components of pieces  such that $N_s(E_1)$ and 
$N_s(E_2)$ overlap, then 
\begin{eqnarray} \label{close}
 |f_1(z)- f_2(z) | \leq \epsilon   \cdot \diam(f_1(E_1)),\qquad
 z \in N_s(E_1) \cap N_s(E_2).
\end{eqnarray}
We define the norm of our collection, $\| {\cal F} \|$ to be the 
smallest $\epsilon$ for which (\ref{close}) holds for every 
pair of adjacent elements of ${\cal W}$.

If $E$ is a boundary component of piece $Q_k$, define
$$ \partial F(E) = \frac{\diam (f_k(E))}{\diam(E)}.$$
If $Q_k $ is a Whitney square then $f_k$ is 
 conformal on a neighborhood of $Q_k$. For other pieces,
we can apply Schwarz reflection to see  $f_k$ can be 
extended to be conformal 
on a neighborhood $N_s(E)$ for a uniform $s$. So by Koebe's
distortion theorem, $|f'|$ is comparable at any two points of $E$ 
with uniform bounds. Thus given boundary components $E_1$, $E_2$ of
two adjacent pieces, the following are all comparable to each
other and any of them could be used in (\ref{close}): 
$$\partial F(E_1) \diam(E_1),  \quad
|\partial f_1(z)| \diam(E_1), \quad
|\partial f_2(z)|\diam(E_2), \quad
z \in E_1 \cup E_2.$$

It will  be convenient to express (\ref{close}) in a 
way that suppresses the $\diam(f_1(E_1))$ term. Roughly, 
(\ref{close}) says that adjacent functions are close to 
agreeing in the hyperbolic metric of the image domain.
This is precisely 
true for Whitney boxes where $d\rho_\Omega \sim ds/\dist(z, \partial 
\Omega) \sim ds /\diam(f_1(E_1))$, but not for boundary pieces.  We can 
make it true for all pieces by replacing the hyperbolic metric 
$\rho_\Omega$  by a related metric $\tilde \rho_\Omega$. Suppose 
we have a finite set $S \subset \circle$ that divides the circle 
into disjoint arcs $\{I_j\}$. For each arc, take the disk $D_j$ 
(or possibly a half-plane or disk complement) 
that intersects $\disk$ along this arc and is orthogonal to $\circle$.
Take the union of these disks and $\disk$. The result is a 
simply connected domain $\Omega_S$ that contains $\disk$ and
so that $S \subset \partial \Omega_S$.  Let $\tilde \rho_S$ be the 
hyperbolic metric on $\Omega_S$.  If $f: \disk \to \Omega$ is a 
conformal map to a polygonal domain and $S$ is the set of conformal 
prevertices, define $\tilde \rho_\Omega$ on $\Omega$ by pushing $\tilde 
\rho_S$ forward by $f$. Clearly for a boundary component of a 
decomposition piece in $\disk$, 
$\diam(E) \sim \dist(E, \partial \Omega_S)$, which means 
$d \tilde \rho_S \sim  ds /\diam(E)$ on $E$. Therefore by Koebe's 
distortion theorem, 
(\ref{close}) is equivalent to 
\begin{eqnarray} \label{equiv-close}
\tilde \rho_\Omega(f_1(z), f_2(z)) =O(\epsilon) \text{ for all } z \in N_s(E).
\end{eqnarray}
This is a little cleaner looking than (\ref{close}) and also is 
more clearly conformally invariant: if $g:\Omega_1 \to \Omega_2$ is a conformal map 
of a polygonal domain to a circular arc domain that maps 
edges into edges  then 
$$ \tilde \rho_{\Omega_1}(z,w) = \tilde \rho_{\Omega_2} (g(z), g(w)) .$$
In particular, this works for  the conformal map to the disk.
 This will allow us to estimate 
(\ref{equiv-close}) assuming we have maps into the unit disk; this will 
be convenient in certain proofs, while the more concrete version 
(\ref{close}) is more appropriate for certain explicit calculations.
We could also have used the hyperbolic metric on the plane punctured 
at the points of $E$; away from these points this is approximately the 
same size as $\tilde \rho$, but near each point the asymptotics are 
different (but we never use the metric there, so either choice would be fine).


Occasionally we will want to measure the distance between a collection 
${\cal F}$ and a single function $f$ defined on $\uhp$.  We write
$$\tilde \rho_\Omega(f, {\cal F}) = \sup_j \sup_{z \in Q_j} \tilde
\rho_\Omega(f(z), f_j(z)).$$

Given a collection of functions satisfying (\ref{close}), define 
\begin{eqnarray} \label{defn-F}
 F = \sum_{k \in I} f_k \varphi_k,
\end{eqnarray}
Note that whenever we take a non-trivial combination of $f_k$'s 
along the boundary, the values must lie on the same edge 
of $\partial \Omega$. Thus any convex combination also lies 
on this edge, and hence $F$ maps $\reals$ to $\partial \Omega$.
If we try to define representations of domains with curved boundaries, we
lose this property (but we shall see how to deal with this 
in Section \ref{reps of FB}).

\begin{lemma} \label{F=QC}
 $F$ is quasiconformal with constant $1+O(\|{\cal F}\|)$, 
 if $\| {\cal F} \|$ is small enough.
\end{lemma}

\begin{proof}
 First note that $F$ is a continuous mapping of $\uhp$ into 
$\Omega$ and that it maps $\reals$  onto $\partial \Omega$.
 If two functions $f_1, f_2$ are close 
in the sense of (\ref{close}),
then the Cauchy estimates imply their derivatives are 
also close in the sense 
\begin{eqnarray}  \label{close2}
|\partial f_1(z)-\partial f_2(z) | \leq C \epsilon |\partial F( E_j)| 
\end{eqnarray}
for $j=1$ or $2$. 
Note that  (all sums are over $k \in I$, the index set of ${\cal W}$),
$$ \partial F  =\sum_k (\partial f_k \cdot\varphi_k + f_k\cdot \partial \varphi_k  ),$$
$$ \overline {\partial}F=  \sum_k f_k  \cdot \Dbar \varphi_k,$$
because $\Dbar f_k =0$.
Since only a bounded number of terms of our partition of unity are 
non-zero at any point we have 
\begin{eqnarray}\label{sharp2}
 \sum_k |\nabla \varphi_k(z)| \leq \frac C{\diam(E)}, \quad z \in N_s(E).
\end{eqnarray}
Also, since the $\sum_k \varphi_k \equiv 1$  we know $\sum_k \partial \varphi_k=
\sum_k \Dbar \varphi_k =0$.  Note that if $\sum_k a_k =0$ 
 and $|b_k - b| \leq \epsilon$, 
then 
\begin{eqnarray} \label{sharp1}
 |\sum_k a_k b_k| = |\sum_k a_k (b_k -b)| 
\leq \epsilon \sum_k | a_k|.
\end{eqnarray}
Hence by (\ref{close}), (\ref{sharp2}) and (\ref{sharp1}), 
$$ |\overline{\partial} F| \leq \frac {C \epsilon  \cdot 
              \diam(f(E))}{ \diam(E)} \leq 
                 C \epsilon |\partial f(E)|.$$
Because $\sum \varphi_k \equiv 1$, 
$$ \partial F = \partial f + \sum ((f_k)_z  - \partial f) \varphi_k 
+ f_k  \cdot (\varphi_k)_z = \partial f + I + II.$$
Thus $|\partial F - \partial f| \leq I + II$ and
we can estimate these terms as 
$$ I \leq \sum_k \frac {|f_k(z) - f(z)||\partial F(E)|}{\diam(E)}
                  \leq C \epsilon |\partial f(E)|, $$
$$ II \leq   \frac { C \epsilon |\diam  f(E)|}{\diam(E)} \leq 
                           C \epsilon |\partial F(E)|.$$

Thus, 
\begin{eqnarray*}
|\mu_F(z)|= |\frac {\Dbar F}{\partial F}(z) |
&\leq&
|\frac{\sum f_j(z)\Dbar \varphi_j(z)}{\sum \partial f_j(z) \cdot \varphi_j (z)
               + f_j(z) \cdot \partial \varphi_j(z) } | \\
&\leq&
         \frac {C \epsilon |\partial F (E)|}{ |\partial F(E)| - C 
                   \epsilon |\partial F(E)|}  \\
   &\leq& C \epsilon.
\end{eqnarray*}
In particular, if $\| {\cal F} \| \leq \epsilon$, then 
$F$ is a $(1+C\epsilon)$-quasiconformal map of $\uhp$
to $\Omega$.
\end{proof}

Next we define a collection  ${\cal E}$ of elementary mapping functions.  All 
of these can be considered as conformal maps of the upper half-plane 
to regions bounded by at most three straight edges 
(segments, rays or lines) with 
at most one bounded segment used.
 As special cases  we take the 
functions $f(z) = a z +b $ (one boundary line), 
 $f(z)= az^\alpha +b$  (two boundary rays) and
$f(z) = a \log z + b$ (two parallel boundary lines).
When there are three sides there must be two rays and a finite 
segment and  such a map is given by the Schwarz-Christoffel
formula; it is elementary in the sense that there is no parameter
problem to solve; when there are three sides we can take $\infty$ to map to itself 
and $\pm 1$ to map to the finite vertices.  See Appendix \ref{background}

Suppose $f: \uhp \to \Omega$ is conformal and $S$ consists 
of the preimages of the vertices of $P = \partial \Omega$.
The main observation we need to make is that on each piece of our 
Carleson-Whitney  decomposition  $Q_k$ we can write
$f = f_k \circ g_k$ where $f_k \in {\cal E}$ and $g_k$
is conformal on $M Q_k$ for Whitney squares  and 
on $M(Q_k \cup Q_k^*)$ ($Q^*$ is the reflection of $Q$ 
across $\reals$) for other pieces, and hence $g_k$ has a uniformly 
convergent power or Laurent  series on $D_k$.
In the case of Whitney type squares or 
Carleson squares with no vertex, this is clear and we can take 
$f_k $ to be the identity. For degenerate arches this is also clear with 
$f_k$ being a power function, for if  the vertex $v$ has interior angle 
$\alpha$ then $f(z)^{\pi/\alpha}$  can be extended conformally 
by Schwarz reflection.

The other elements of ${\cal E}$ are only needed for arches.
In this case the two components of $\partial M Q \cap \reals$  map to 
line segments, and  we can choose $h \in {\cal E}$ so that $h^{-1} \circ f$
 will map both  these to disjoint line segments 
into $  \reals$. Hence  $ h^{-1} \circ f$ will have
a conformal extension to $M Q \cup  M Q^*$.
 
\begin{figure}[htbp]
\centerline{ 
	\includegraphics[height=2in]{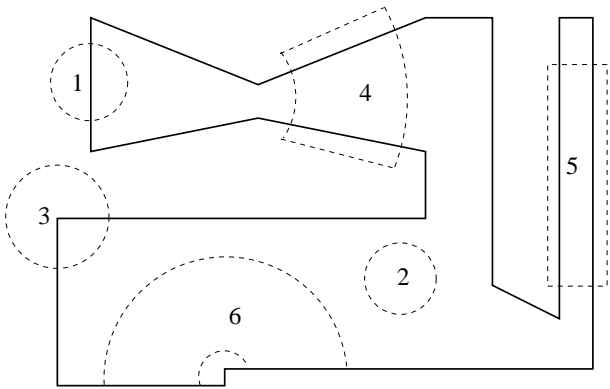}
	}
\caption{ \label{patches}  
This shows parts of a polygon that correspond to using 
different elements of ${\cal E}$. In regions 1 and 2 we would
use the identity, in regions 3 and 4 a power function, a logarithm in 
region 5 and a 3-sided Schwarz-Christoffel map in region 6.
}
\end{figure}

One of the main difficulties with extending our methods to non-polygons
is to define the analogous class of maps ${\cal E}$.
We shall see in Section \ref{reps of FB} how to do this 
for some circular arc polygons, but even in this case problems 
arise if we need to map two circular arcs into a line 
simultaneously, so arches can't be used in general. For
more general curved boundaries, it may be difficult even to 
define maps at the vertices.

Let $ g_{k,p} = \sum_{j=-p}^p  a_n (z-z_k)^j$ be the truncation
of the  power or Laurent series 
for $g_k$  around the point $z_k = z_{Q_k}$. Then since $g_k$ is 
conformal on $10Q_k$, we have 
$$ |g_k(z) - g_{k,p}(z) | \leq C \diam(g_k(E)) \lambda^p,$$
for a boundary component $E$ of $Q_k$ and for 
 $z \in N_s(E)$  and some fixed $\lambda < 1$. Thus 
$$ |f(z) - f_k(g_{k,p}(z)) | \leq C \diam(f(E)) \lambda^p,$$
for $z \in N_s(E)$  and some fixed $\lambda < 1$. 
We can formalize this 
a bit with some notation.

An $\epsilon$-representation of $ \Omega$ 
is a triple $(S, {\cal W}, {\cal F})$ consisting  of:
\begin{enumerate}
\item An ordered set $S \subset \reals$ of $n$ points,
\item A Carleson-Whitney decomposition ${\cal W}$  of $\uhp$ with $O(n)$ pieces
     that extends a covering of the hyperbolic convex hull of $S$,  
     together with the adjacency structure of the pieces and
    with a piecewise polynomial partition of unity $\{ \varphi_k\}$
           as described above,
\item  A collection of functions ${\cal F}$ of norm $\epsilon$, and each 
         function has the form  $f_k(z)= h_k(g_k(z))$ 
               where $g_k$  consists of 
       $p = O(|\log \epsilon|) $ terms of a  power series (for the square pieces) 
            or Laurent series 
          (for the arches) and $h_k \in {\cal E}$,
 \end{enumerate}

It may be a good idea to use Laurent series to implement the algorithm, 
but
we would like to avoid the use of Laurent series  in the proof  in order 
to make use of known results about fast manipulations of 
 power series, without having to extend them to  allow negative powers.
 It turns out that our iteration to improve representations only 
 needs the series expansions near the boundary of 
 our  decomposition regions and we can always convert a Laurent series to power
 series on a collection of disks that cover each boundary component. 
 Therefore, we define a
 partial $\epsilon$-representation as a triple $(S, {\cal B}, {\cal G})$
where ${\cal B}$ is the covering of the boundary of the decomposition 
introduced at the end of Section \ref{extend-decom} and ${\cal G}$ consists of 
power series on  the corresponding empty regions that 
satisfy (\ref{close}).
For Whitney or Carleson squares, this is the same as before, but for arches 
we have ``thrown away'' the representation in the middle of the arch, keeping 
power series representations that are valid only near the boundary
of the arch.  

Clearly a representation can be turned into a partial representation by 
restricting an element of ${\cal F}$ to  a disk and computing its power
series there. Conversely, a partial representation can be converted 
back to a representation by using the power series near the boundary to 
compute integrals of the form $\int f(z) z^{k} dz$ around the boundary of 
an arch and  thus 
recover the coefficients of the Laurent series.  Moreover, we shall see 
in Appendix \ref{fast-power-series} that these conversions can 
be done quickly.



Next,  we deal with a technical point concerning 
$\epsilon$-approximations. 
Later we will want to take power 
series expansions that are  good 
approximations to the desired conformal map on the 
empty disks of our decomposition and deduce that they are 
still good approximations on a larger disk that includes 
the entire piece.

More generally, suppose $f$ is an  conformal map
defined on $D(0,2)$.
Suppose we  know the power series of an  analytic function $h$, 
defined on the 
smaller disk $D(0,1)$ so that $|h-f| \leq \epsilon$ on 
the smaller disk. Can we use $h$ to get an approximation 
to $f$ on a larger disk, say $D(0,r)$ for some $1< r< 2$?
Note that $h$ itself may not work.  For example, if 
 $f(z)=z$ and $h(z) =
z+(\frac 89 z)^{1000}$ then $h$ is a good approximation 
for $|z|< 1$ but not for  $|z| > 9/8$.  However, 
if we truncate the power series for $h$ appropriately, then
we obtain a uniformly good 
approximation for $f$ on a uniformly larger disk.
More precisely:

\begin{lemma} \label{expand-est}
There is a $C < \infty$ so that given $0 < \beta < 1$ and 
$ 1 < r=R^\beta < R < \infty$, the following holds.
Suppose $D = \{ z:  |z| <t \}$ 
and suppose 
  $f(z) = \sum_{n=-\infty}^\infty a_n z^n $ 
 is a conformal on 
$2R\cdot D =\{ z:  |z| \leq 2Rt\}$.
Suppose also that $h(z) = \sum_{n=0}^\infty b_n z^n$ is holomorphic  on 
$D$ and that  
$|f-h| \leq \epsilon$  on $D$.
 Let $ g(z) = \sum_{n=0}^q b_n z^n $ be the 
truncated power series for $h$  where
 $q = \lfloor c \log \frac 1 \epsilon \rfloor$, 
 $c =1/\log R$.
Then 
$$ |f(z) - g(z)  | \leq \frac{C \epsilon^{1-\beta}}{1-R^{\beta-1}}  |f'(0)|,$$
for all $z \in  rD$. 
\end{lemma}

\begin{proof}
Without loss of generality we assume $t=1$.
Also, normalize $f$ so that 
$a_0=0$ and $|a_1| =1$. Because $f$ is conformal on $2RD$, 
this implies $f(RD)$ has uniformly bounded diameter
(Koebe distortion theorem, Appendix \ref{background}).
Thus  $ |a_n| \leq C R^{-n}$,  $n \geq 0$
since  the usual formula 
$$ a_n = \frac 1{2 \pi i} \int_{|z|=R} \frac { f(w)dw}
                {w^{n+1}},$$
bounds  the coefficients in terms of the maximum of $|f|$.
Also, if 
$\|f -g\|_\infty <  \epsilon$ for $z \in D$,
then 
 $|a_n - b_n| \leq C \epsilon$ for $n \geq 0$.
The rest of the proof is a simple exercise summing series, i.e., 
 for $1 < |z|\leq r$, 
\begin{eqnarray*}
|f(z) - \sum_{k=0}^q b_k z^k | 
& \leq &  |\sum_{|k| >q+1}a_k z^k |+ \sum_{|k|\leq  q}  |a_k-b_k| |z|^k \\
& \leq &  C\sum_{k=q+1}^\infty  (\frac rR )^k + C \epsilon \sum_{k=0}^q  r^k \\
& \leq &  \frac C{1-(r/R)} (\frac rR)^q  + C \epsilon r^q\\
& \leq & \frac C{1-R^{\beta-1}} (\frac rR)^{c |\log \epsilon|} 
             + C \epsilon r^{c|\log \epsilon|}\\
& \leq & C \epsilon^{c \log \frac Rr} + C \epsilon^{1-c \log r}.
 \end{eqnarray*}
Taking $c = 1/\log R$ and $\beta = \log r / \log R$, the 
last line becomes $O(\epsilon^{1-\beta})$, as desired.
\end{proof}




\section{The thick/thin decomposition of a polygon}
\label{thick-thin}

Suppose $\Omega$ is simply connected domain with $n$ specified boundary 
points $V$, let $f$ be a conformal map of $\uhp$ onto $\Omega$
and let $S = f^{-1}(V)$. Construct the covering of the hyperbolic
convex hull of $S$ as in  Section  \ref{cover}. 
If no $\epsilon$-Carleson arches are needed to cover 
$C(S)$ then we say $(\Omega, V)$ is $\epsilon$-thick 
(or just ``thick'', if $\epsilon$ is understood from context).
If $\Omega$ has a polygonal boundary we let $V$ be the 
vertices and this defines a ``thick polygon''.

A more geometric way to think of  thick polygons  uses 
extremal distance. Suppose $e_1$, $e_2$ are non-adjacent
edges of $\Omega$. We say the pair of edges  is $\epsilon$-thin
in $\Omega$  if 
the extremal distance between them is $\leq \epsilon$, i.e., 
if the modulus of the path family connecting these 
edges inside $\Omega$ is $\geq \frac \pi{\epsilon}$.
By our remarks in Section Lemma \ref{cover}, such a pair 
corresponds to an $\delta$-Carleson arch with $\delta
\sim \epsilon$.  Indeed, thin
pairs of non-adjacent sides are in 1-1 correspondence with 
Carleson arches (except for the possibility that one side
has unbounded preimage under $f$, and gives two arches 
in the decomposition of $\uhp$).  Thus if a polygonal domain is $\delta$-thick in 
the sense of the previous  paragraph, then 
every non-adjacent pair of edges is $\epsilon$-thick for some 
$\epsilon \simeq \delta$, and conversely.
Some examples of thick  and not thick polygons  are given 
in Figures \ref{thick-parts} and \ref{thin-parts-fig}.  
  
\begin{figure}[htbp]
\centerline{
 \includegraphics[height=1.25in]{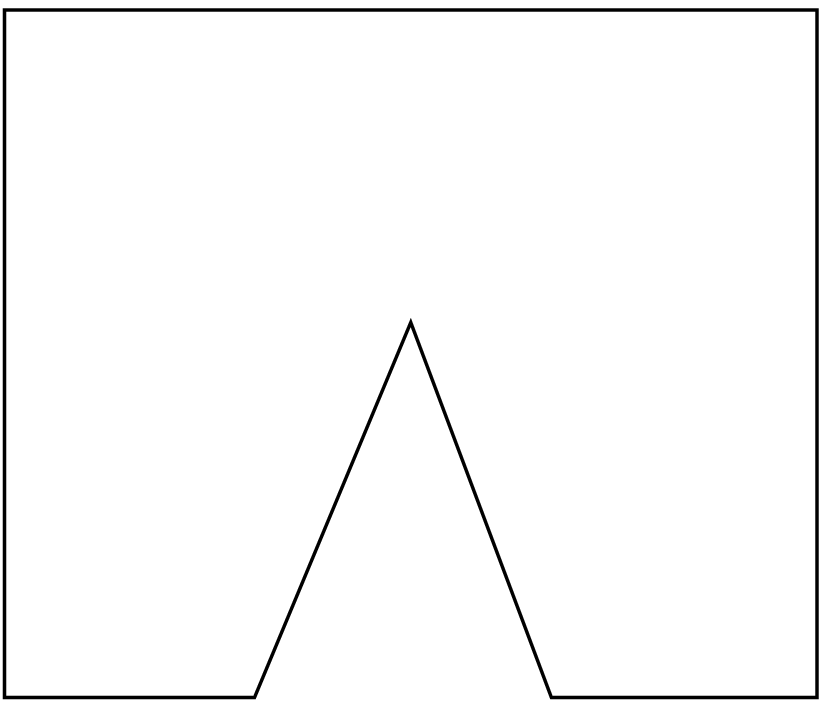}
$\hphantom{xxx}$
 \includegraphics[height=1.25in]{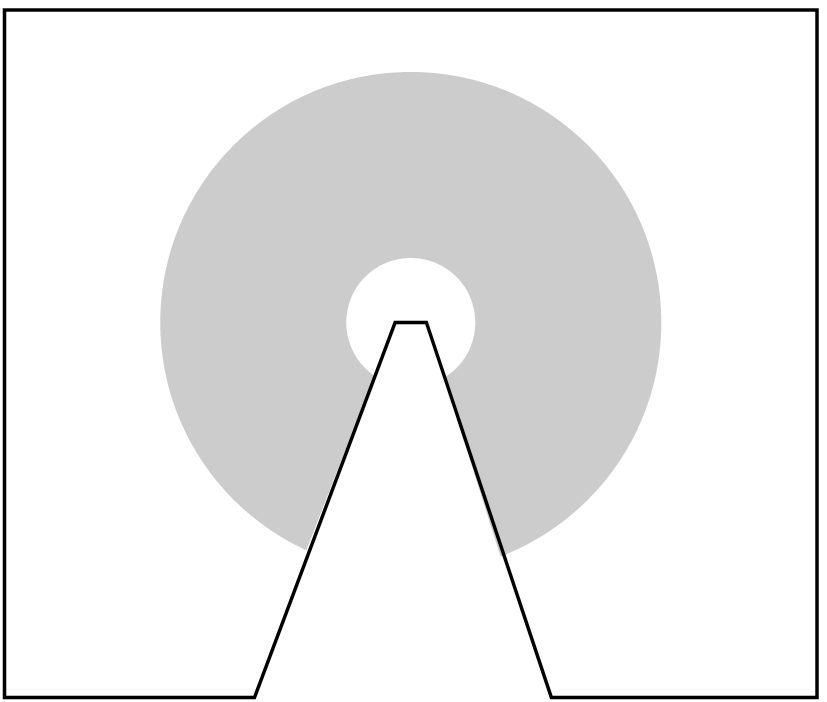}
$\hphantom{xxx}$
 \includegraphics[height=1.25in]{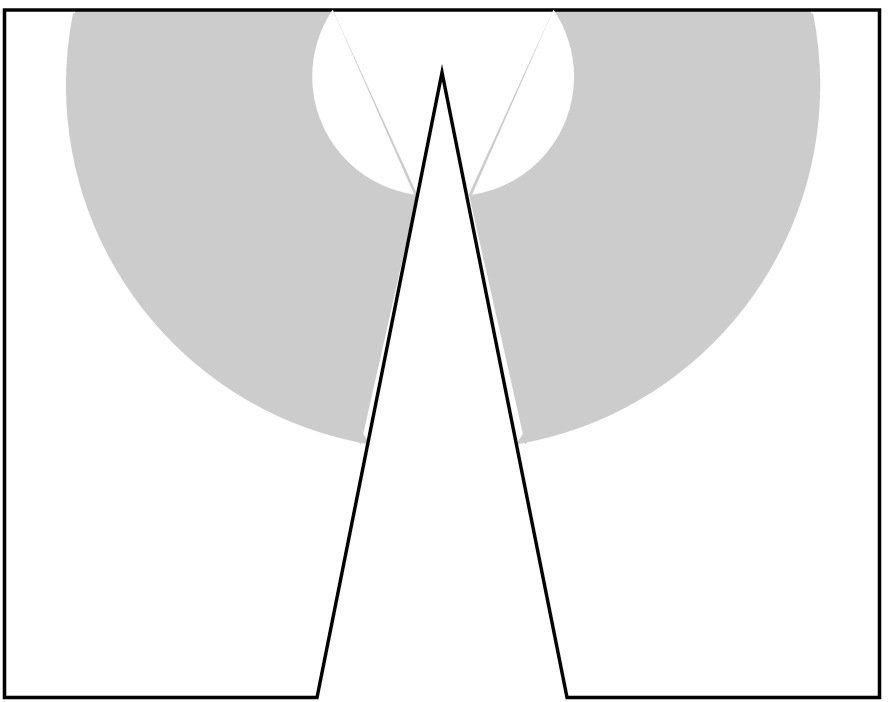}
 }
\caption{\label{thick-parts} A thick polygon and two 
non-thick variations. The shaded areas join non-adjacent edges
with small extremal distance and  correspond to 
Carleson arches.}
\end{figure}

\begin{figure}[htbp]
\centerline{
	\includegraphics[height=.75in]{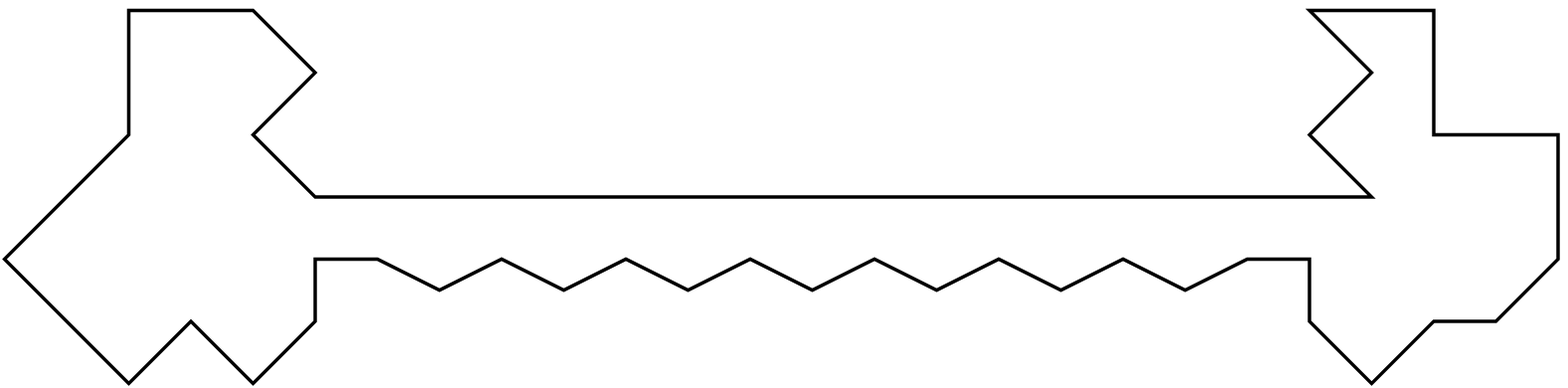}
	}
\vskip.2in
\centerline{
 \includegraphics[height=.75in]{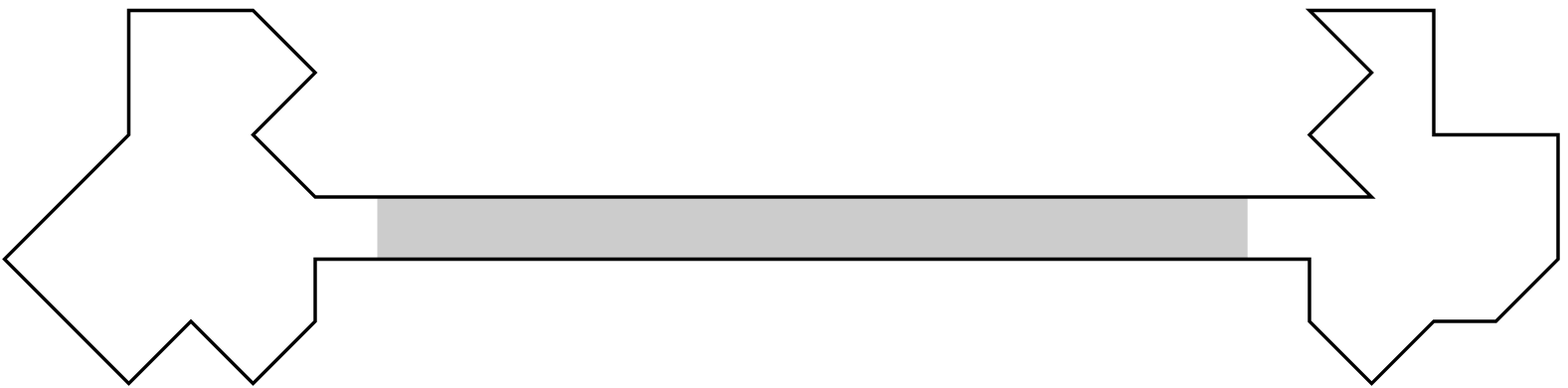}
 }
\caption{\label{thin-parts-fig} The polygon on top is 
thick, but not the one on the bottom. }
\end{figure}

We can always add more vertices to the boundary to make a polygon 
thick, although there is no bound on how many we might have to add.
Also note that adding vertices can convert a thick polygon to a 
non-thick one, i.e., add two vertices $\epsilon$ apart to the 
middle of an edge of length $1$. 

Another way to approach thick and thin parts is to consider 
the Riemann surface obtained by reflecting the dome of $\Omega$
across $\reals^2$ and removing the vertices of $\partial \Omega$.
This is a punctured Riemann surface (isometric to the unit 
sphere with the $\iota$-preimages removed) and the thin parts 
of $\Omega$ correspond to the thin parts of this surface, and 
the polygon is thick if this surface has no hyperbolic thin parts.
We will not pursue this connection between polygons and surfaces
further here, but it may prove interesting (e.g., what does 
Mumford's compactness theorem for surfaces say about polygons?).

\begin{lemma}
If $(\Omega, V)$ is $\epsilon$-thick then any $K$-quasiconformal 
image 
is $\epsilon/K$-thick.
\end{lemma}

\begin{proof}
Obvious since conformal modulus can be changed by at most a 
factor of $K$.
\end{proof}

Because of this lemma, we can deduce that a polygon is 
$\epsilon$-thick by computing the $\iota$ preimages of the 
vertices and making sure there are no $\epsilon/K$-arches 
in its covering, where $K$ is the universal quasiconformal 
bound on the extension of the $\iota$ map to the interior.
Thus in time $O(n)$ we can  determine (up to a bounded  factor of $K$) the 
degree of thickness of a polygon.

It will also be convenient to define a class of polygons called 
``thin''.
 A thin polygon has two special edges that we 
call the ``long edges'', and  which may be either adjacent or 
non-adjacent.
In the first case, (which 
we call parabolic) assume both long edges start at $0$ and  end 
at points $\{a,b\} \in A_R= \{z: R \leq |z| \leq 2 R \}$. 
We assume the other edges  of the polygon
all lie in $A_R$, have lengths comparable to $|a-b|$ and interior angles 
bounded away from $0$ and $2 \pi$.

In the second case (which we call hyperbolic), the long edges  
start at points $a,b$ in $A_1=\{z: 1 \leq |z| \leq 2\}$  and 
end at points $c,d$ 
in $A_R=\{z: R \leq |z| \leq 2R\}$, $R  \geq 4$. 
The remaining sides of a thin polygon
all lie in either $A_1$ or $A_R$, have length comparable to $|a-b|$  and  
$|c-d|$ respectively, and all interior angles  are bounded away from 
$0$ and $2 \pi$ (the exact bound is unimportant, but we can take
$\pi/6$ to be specific).  We say a polygon is $\epsilon$-thin if there
is a Euclidean similarity to shape as described and if the extremal 
length between the two long sides in the polygon is less than 
$\epsilon$ and there is no other pair of non-adjacent sides 
that are $\frac{1}{100}$-thin (again, the exact constant is 
unimportant).
 Although the description is a bit wordy, a picture makes 
it much clearer, see Figure \ref{thin-polys}.

\begin{figure}[htbp]
\centerline{
 \includegraphics[height=1.25in]{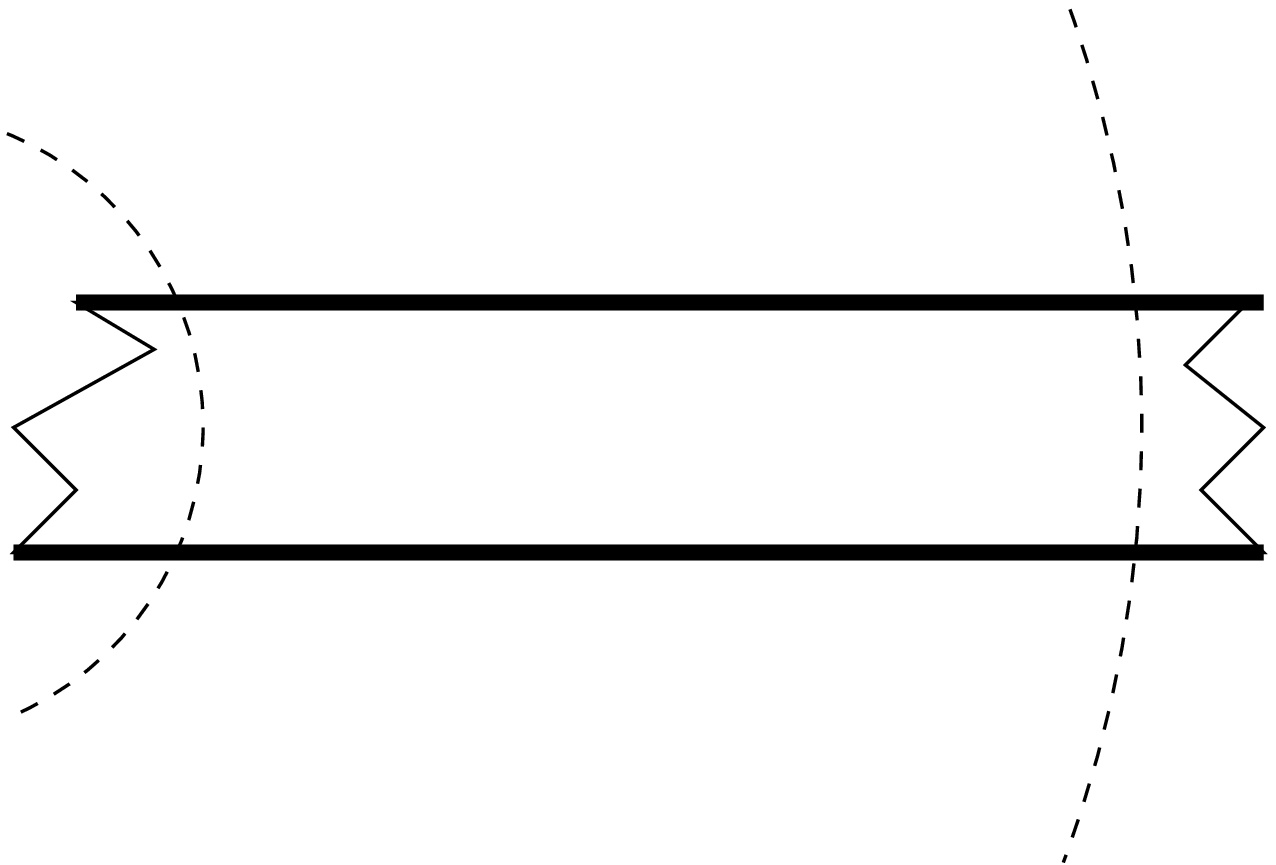}
$\hphantom{xxx}$
 \includegraphics[height=1.25in]{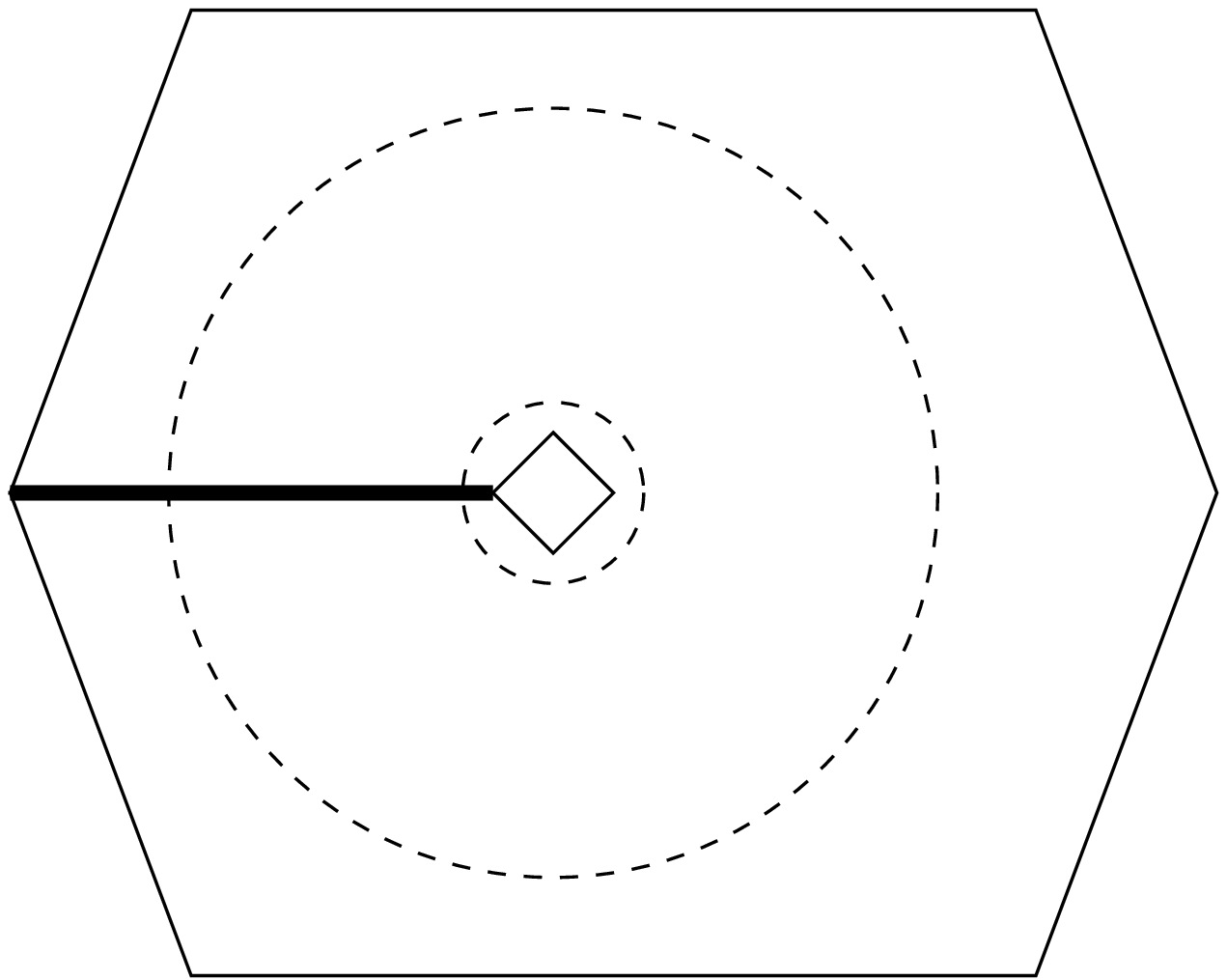}
$\hphantom{xxx}$
 \includegraphics[height=1.25in]{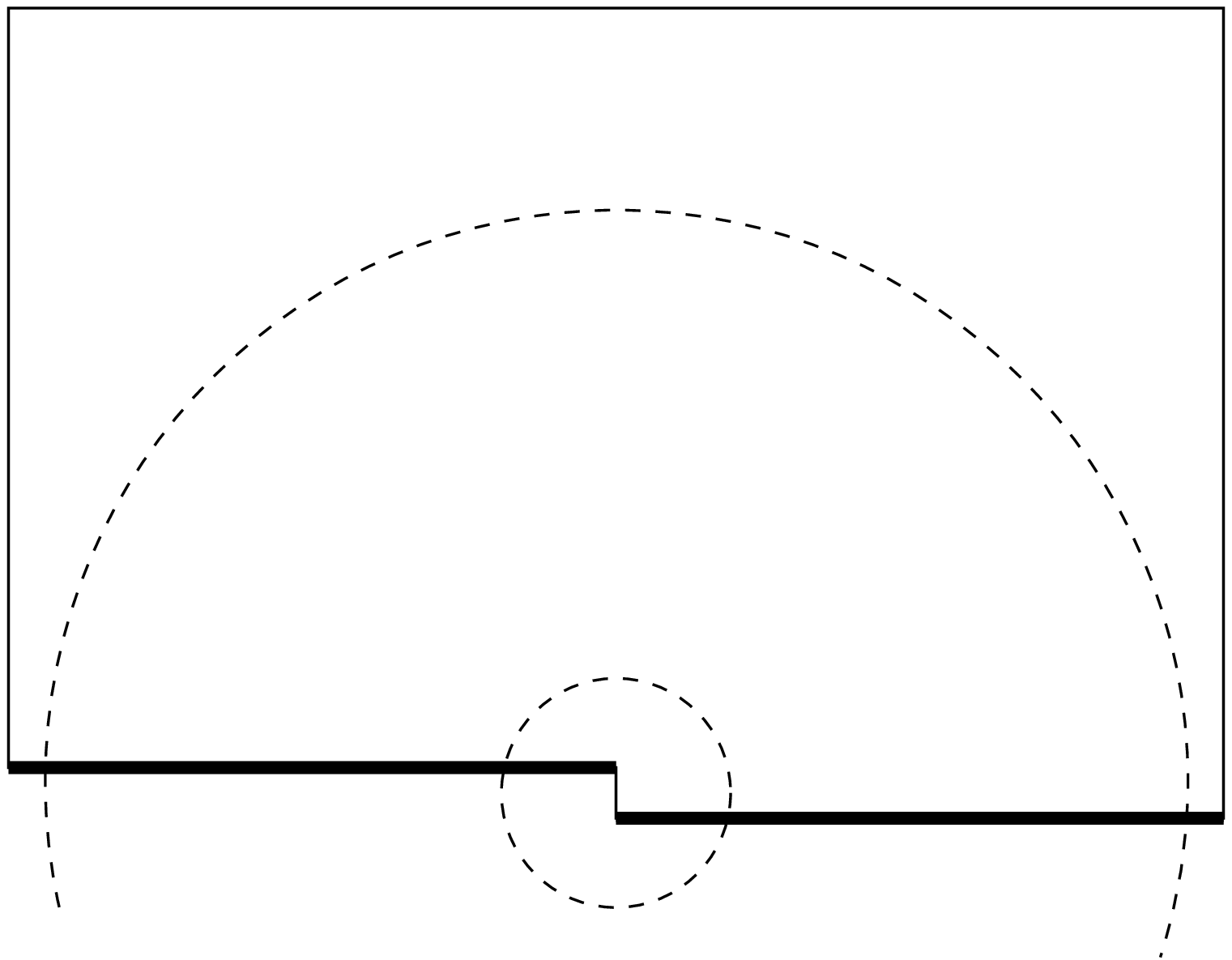}
 }
\caption{\label{thin-polys}  Examples of hyperbolic thin polygons. The long sides 
   are emphasized and the dashed circles show outer boundary of $A_1$ 
  and the inner boundary of $A_R$. }
\end{figure}

The names of the two cases correspond to the 
thin parts that occur in Riemann surfaces. An
$\epsilon$-thin parts  of a hyperbolic manifold 
are the connected components of the set where the injectivity 
radius of the surface is $< \epsilon$. In the surface case,  parabolic 
thin parts (also called ``cusps'') are non-compact and have 
a single (finite) boundary component, 
whereas hyperbolic thin parts are compact and have two boundary 
components.

\begin{lemma} \label{thick-thin-lemma}
There is an $\epsilon_0 >0$  and $0< C < \infty$ 
so that if $ \epsilon < \epsilon_0$
then the following holds.
Given a simply connected, polygonal domain $\Omega$ we can 
write $\Omega$ is a union of subdomains $\{ \Omega_j\}$ 
belonging to two families ${\cal N }$ and $ {\cal K} $.
The elements of ${\cal N}$  are $O(\epsilon)$-thin polygons
and the elements of 
${\cal K}$  are $\epsilon$-thick.
The number of edges in all the pieces put together is $O(n)$
and all the pieces can be computed in time $O(n)$ (constant
depends on $\epsilon$).
We can either choose the pieces to be disjoint except for 
common boundaries, or given $\epsilon_0 > \delta >4 \epsilon$, 
we can choose them so that pieces of 
each type only intersect pieces of the other type and that 
 the intersection  is a $\delta$-thin polygon.
\end{lemma}

\begin{proof}
All we have to do is choose polygons that approximate the 
images of the regions in ${\cal N}$ and ${\cal K}$ associated
to the Carleson-Whitney decomposition in Section \ref{extend-decom}
(see the discussion preceding Lemma \ref{Whitney-near-MA}).
  Suppose $E$ is a boundary component
of a Carleson arch. Then the Euclidean distance of $E$ to the nearest
prevertex point is at least $\frac 12 \diam(E)$ and hence we can use 
Schwarz reflection to show that $f$ extends to be conformal in 
a ball of radius $\frac 12 \diam(E)$ around every point of $E$.
Then the Koebe distortion theorem implies there is a $P< \infty$
so that if we take $P$ equally spaced points along $E$ (including its
endpoints on $\reals$) and connect the $f$ images of these by 
line segments, the resulting polygonal curve is simple and 
hence divides $\Omega$ into two polygonal subdomains.

\begin{figure}[htbp]
\centerline{
 \includegraphics[height=1.5in]{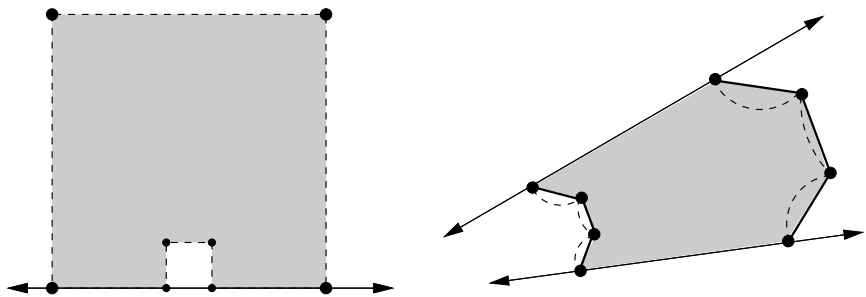}
 }
\caption{\label{intersections} The intersection of a 
thin and thick piece is approximately the image of 
a Carleson arch whose width may be chosen. }
\end{figure}

When we do this for every 
domain in ${\cal N}$ and ${\cal K}$ we get the desired 
decomposition with the desired overlaps (if the parameter 
$A$ is chosen correctly). The thin pieces have 
at most $2P$ edges. To check that the  thick pieces are really 
thick, we have to show that no 
two  non-adjacent edges have modulus in the piece less than $1/\epsilon$.
If both edges were new ones introduced (e.g., subarcs of a  crosscut
or a subarc of an original edge created by  an endpoint 
of a crosscut) then this is clear from construction.  A similar 
proof works  if one edge is 
new and the other was an original edge.  Finally, if both edges 
are original, then the modulus in the piece is smaller than it is 
in the whole polygon, and hence the pair is still thick, since it 
was thick before.
\end{proof}

The division of $\Omega$ into thick and thin pieces is not unique, 
because parabolic thin pieces at the vertices either may or may not
be included. However,  
${\cal N}$ must contain all the thin parts corresponding to all the Carleson arches
in the decomposition corresponding to $\Omega$.  
It is easy to estimate 
the conformal map to thin polygons, at least if we stay near the long edges 
and away from the other edges.

\begin{lemma} \label{rect-lemma}
Suppose $R$ is  the rectangle $[0,L] \times [0,1]$ and 
$\Omega$ is a simply connected domain containing $R$ 
that is formed by replacing the two vertical sides 
of $R$ by curves. Let $f: R \to \Omega$ be the 
conformal map such that $f(c)=c$ and $f'(c) >0$, where
$c = \frac L2 + i \frac 12 $ denotes the center of $R$.
Then $|f(z) - z| \leq O(e^{-L/2\pi})$ for $|z-c|< 2$.
\end{lemma}

\begin{proof}
Suppose $g: \disk \to R$ and $h: \disk \to \Omega$ are 
conformal, taking $0$ to $c$ with positive derivative 
at $c$. Then $F=h^{-1} \circ g$ is a conformal map 
from the disk to  $W = h^{-1} (R) $ that fixes the origin 
and has positive derivative there. Since  the extremal 
length of the path family connecting $B(c,2) \cap R$ to 
$\partial \Omega \setminus \partial R$ in $\Omega$ 
is $\geq \frac 12 L +O(1)$, we get that $\partial W 
\setminus \circle $ has two components, each of Euclidean 
diameter $\leq \delta$ where $\delta = O(e^{-L/2\pi})$
by Lemma \ref{EL-diam}.
See Figure \ref{dogbone}.
 Thus $|F(z)| \geq 1- O(e^{-L/2})$. So if 
$\log \frac {F(z) }{z} = u(z)+ iv(z)$, then 
$-\delta \leq u \leq 0$ on $\disk$ where $ \delta 
= O(\exp(-L/2))$. By reflection $u$ extends to be harmonic 
with $|u|\leq \delta$ on $\reals^2 \setminus E$
where $E =h^{-1}(\partial \Omega \setminus \partial R) \subset
\circle$ has two components. Thus
$$ |\nabla v | = |\nabla u| = O(\delta/\dist(z,E)),$$
which together with $v(0) =0$, implies $|u|+|v|=O(\delta)$
on $\{z: \dist(z, E) \geq \frac 1{10} \}$, which further implies 
$|F(z) - z   | =O(\delta)$ on the same set.
Back on $R$, this means that 
$|f(z)-z|=O(e^{-L/2\pi})$, as desired.
\end{proof}

\begin{figure}[htbp]
\centerline{ 
	\includegraphics[height=1.25in]{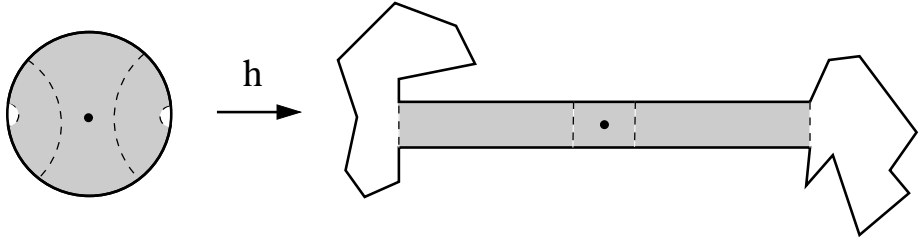}
	}
\caption{\label{dogbone}
In the center of a thin polygon the conformal map is 
essentially independent of the region outside the 
center. }
\end{figure}

Given an $\epsilon$-thin polygonal domain  $\Omega$ with 
long edges connecting $A_1$ and $A_R$ as above,  let 
$\Omega' = \Omega \cap \{ z:  |z| < R\}$ if $\Omega$ is 
parabolic and
$\Omega' = \Omega \cap \{ z: 1 < |z| < R\}$ if $\Omega$ 
is hyperbolic. 
We denote the ``center'' of  $\Omega$ as those points $z$ 
so that the extremal distance  of $E_z = \{w  \in \Omega : |w| =|z| \}$
to both $A_1 \cap \partial \Omega$ and $A_R \cap \partial 
\Omega$ is greater than $\epsilon/2 + O(1)$ (in the hyperbolic case)
or to just $A_1 \cap \partial \Omega$ (in the parabolic case).
Lemma \ref{rect-lemma}
says that in the center of $\Omega$, the conformal map from 
$\uhp$ to $\Omega$ is essentially independent of the part of 
the boundary outside $\Omega'$.  In particular, we can chose 
an elementary mapping from ${\cal E}$ (see Section \ref{epsilon-reps}) 
that agrees with the conformal mapping onto $\Omega$ up to 
order $O(e^{-2\pi/\epsilon})$ in the center of $\Omega$.


In Lemma \ref{thick-thin-lemma} we saw how to break a 
polygon into thick and thin pieces. 
Next we observe that we can reconstruct a representation for 
the whole domain from representations of the thick pieces (since 
we know maps onto the thin pieces automatically).


\begin{lemma} \label{combine-lemma}
Suppose we are given a decomposition of a $n$-gon 
into thick and thin pieces (of either 
parabolic or hyperbolic type) with $\eta$-thin 
overlaps and we are given a  normalized
$\epsilon$-representation of each thick piece. Then 
in time $O(n)$ (constant depending only on $\eta$ and $\epsilon$)
 we can compute a $\delta$-representation 
of $\Omega$ where $\delta =O(\epsilon +e^{- c\eta})$
for some $c >0$.
\end{lemma}

\begin{proof}
Choose a thick part $\Omega_0$ and think of it as the ``root'' 
of a tree where the thick  pieces are the 
vertices and are adjacent iff they are connected by 
a thin part. In a representation, a  thin part corresponds to 
an arch. This arch has two complementary components; a bounded one 
and an unbounded one. A normalized representation for each 
thick part is one where the root lies in the unbounded complementary 
 components
of its adjacent  hyperbolic thin arches,
 and the other thick parts   lie in the bounded 
complementary components of the hyperbolic 
thin part  that connects each to its parent.

By choice, the overlap of a thick  piece $\Omega_0$ and 
a thin piece $\Omega_1$  corresponds to 
an $\eta$-thin piece.  Thus in the decomposition of 
$\uhp$ for $\Omega_0$, we can find a collection of
Whitney boxes and Carleson squares whose union is 
an $\eta$-Carleson arch.  Similarly, if $\Omega_1$ is 
$\epsilon$-thin  we have a map of  an $\epsilon$-arch
to $\Omega_1$. Moreover, we can choose a $\eta$-subarch 
that  maps to the same region. We can renormalize 
the $\epsilon$-arch by linear transformations so that 
the two $\eta$-arches agree. Then remove the part of the 
decomposition for $\Omega_0$ that lies below the 
$\eta$-arch corresponding to $\Omega_1$ and replace 
it  by an $\epsilon$-arch and the representing function 
for $\Omega_1$. See Figure \ref{combined}.

\begin{figure}[htbp]
\centerline{
 \includegraphics[height=1.25in]{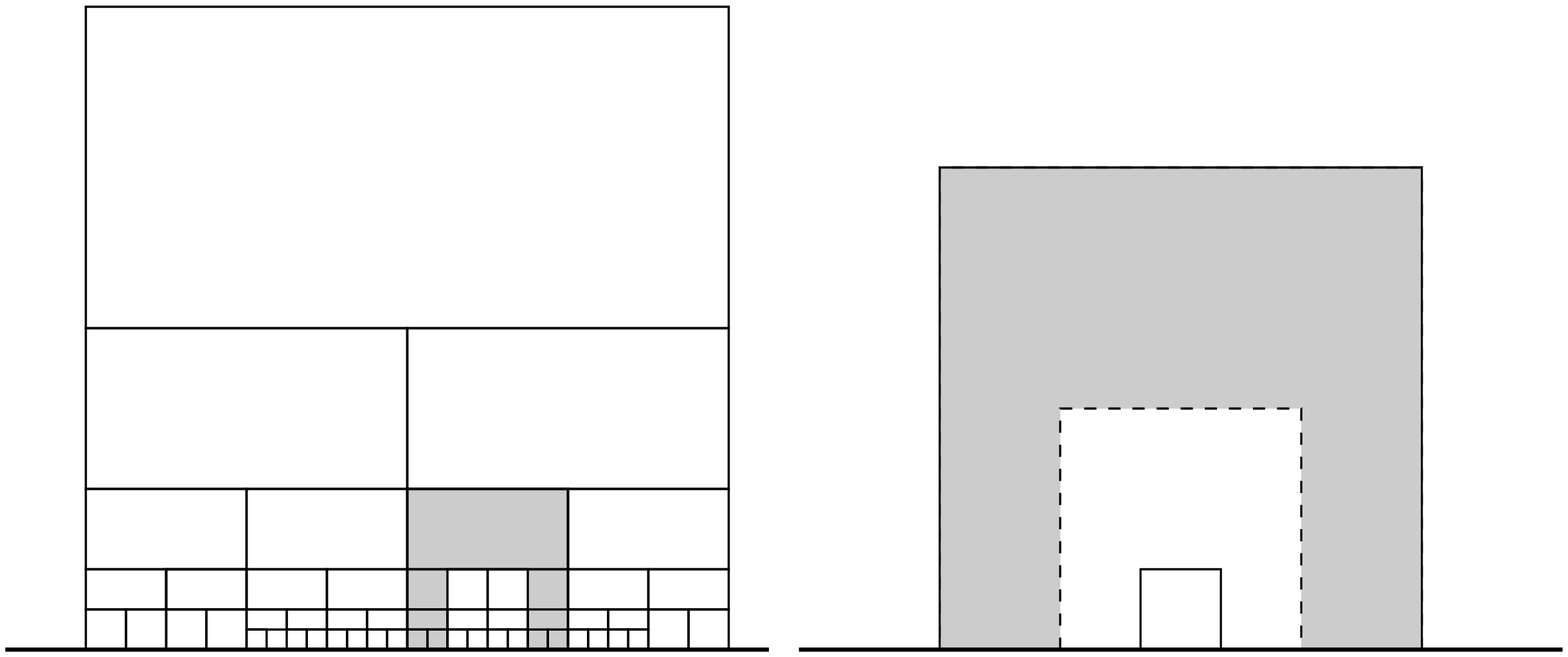}
$\hphantom{xx}$
 \includegraphics[height=1.25in]{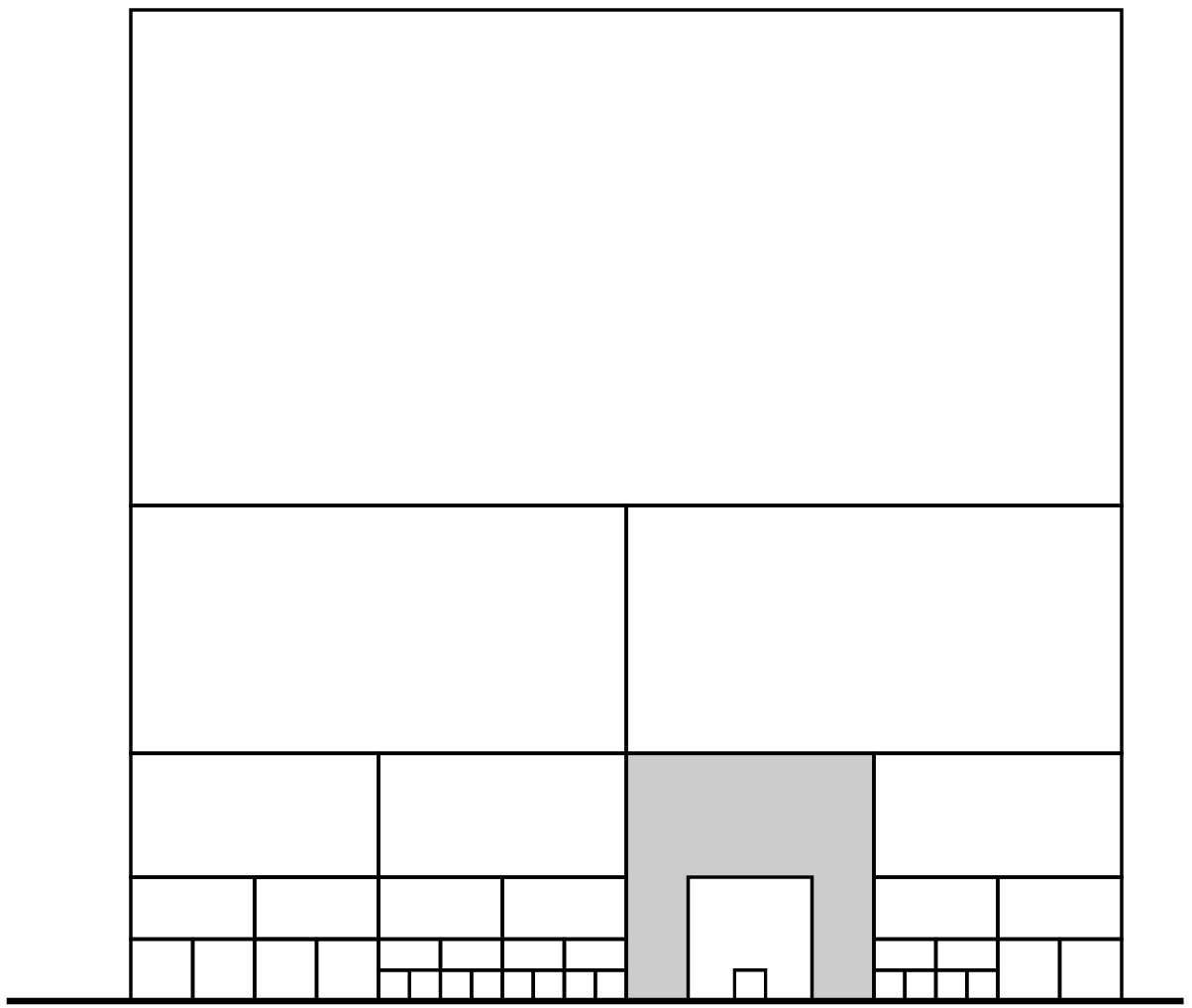}
 }
\caption{\label{combined} 
 Combining a representation of a thick piece with 
an arch from a thin piece.
   }
\end{figure}

Now repeat this for the other thin parts adjacent to 
$\Omega$.  If any of these thin parts is adjacent to 
a second thick part $\Omega_2$, then we renormalize 
the decomposition for $\Omega_2$
and   insert a copy of
the renormalized decomposition for $\Omega_2$ below
the $\epsilon$-arch of $\Omega_1$.  Continuing in this way, 
we eventually insert all the pieces and obtain a 
$O(\epsilon +e^{-2 \pi/\eta})$-representation 
of the whole polygon.
\end{proof}


\section{Representations of finitely bent domains}
\label{reps of FB}

One of the main ideas of the algorithm is to connect the given 
polygonal domain $\Omega$ to the unit disk by a chain of 
intermediate domains $\Omega_0 = \disk, \Omega_1, \dots 
, \Omega_N = \Omega$ so that each is mapped to the next 
by an explicit quasiconformal map with small constant. 
For each domain $\Omega_k$ we have to know how to improve 
a $\epsilon$-representation to any desired accuracy and 
we have to know how to take a representation for 
$\Omega_k $ and create one for $\Omega_{k+1}$. The 
improvement step will be discussed in later sections; 
in this section we discuss the the creation of chain and passing 
representations from one step  to the next.

The conformal maps onto the finitely bent elements on our chain are
only computed to a fixed accuracy $\epsilon_0$; 
just enough to allow us to compute the map onto the next 
element. Thus the precise timing of these steps is not 
crucial (as long as it is linear in $n$ with a constant 
depending on $\epsilon_0$). This gives us an $\epsilon_0$-representation
of $\Omega$ in time $O(n)$.  The time   dependence on $\epsilon$
in Theorems \ref{thmQC} and \ref{main} comes from only from 
iterating Lemma \ref{Newton-radius} for the domain $\Omega$ in 
order to improve $\epsilon_0$ to $\epsilon$.

If $\Omega$ is finitely bent, we have seen how to create the chain  
with angle scaling. For a polygonal domain, we will 
approximate by a finitely bent domain and then use
angle scaling on this. At this point we have a choice.  One
possibility is to 
replace the chain of  finitely bent domains by a chain of 
inscribed polygons 
and describe how to pass representations of these polygons from 
one to the next. The other possibility is to deal directly with the finitely 
bent chain, but this requires extending the definition of 
$\epsilon$-representations to such domains (which causes some 
difficulties as was pointed out earlier).
Both approaches involve a similar number of technicalities, 
but we will take the second one since it is seems cleaner and 
shows how our method can be adapted to certain domains with 
curved boundaries. 


Suppose $\Omega$ is a finitely bent domain with $n$ vertices.
 As before, an
$\epsilon$ representation is a triple $(S, {\cal W}, {\cal F})$
where $S$ is a set of $n$ points, ${\cal W}$ is the Carleson-Whitney 
decomposition  associated to $S$ and ${\cal F}$ consists of a 
power or Laurent series  for each piece and a boundary map 
for each boundary piece. Each boundary map sends an interval 
on $\reals$ to an subarc of $\partial \Omega$ and each is
power function (possibly the identity) followed by a M{\"o}bius 
transformation. See Figure \ref{FBboundarymaps}. These maps 
clearly suffice for Carleson squares  and degenerate Carleson  arches.
However, for such maps to suffice for Carleson arches we have to 
make an assumption about $\Omega$:  there is a $\delta_0>0$, 
so that whenever two sides of $\partial 
\Omega$ have  extremal distance $< \delta_0$  in $\Omega$, the sides 
lie on intersecting circles. In this case, a power function, followed 
by a M{\"o}bius transformation can be used to map the base intervals 
of the arch into the the sides of $\Omega$. We will call a 
finitely bent domain with this property simple. It is 
straight forward to check that this property is preserved 
by any angle scaling which decreases angles.

\begin{figure}[htbp]
\centerline{
 \includegraphics[height=2.00in]{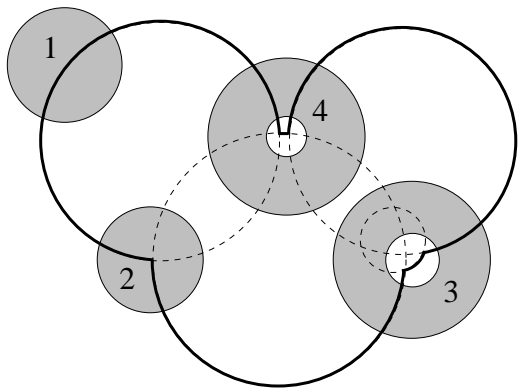}
 }
\caption{\label{FBboundarymaps} 
 In region 1 we use a M{\"o}bius transform and in region
  2  we can use a power function followed by a M{\"o}bius 
  transformation. The same type of function works in the arch 
  3, since the corresponding circles intersect.
  However, arches of type 4 are not allowed since it is 
   not clear how to explicitly map the boundary arcs into 
   a single line by a conformal map.
   }
\end{figure}

Given an $\epsilon$-representation of a simple finitely bent domain, 
we can define a quasiconformal map $F$ from $\uhp$ into $\Omega$ 
using a partition of unity as in formula (5). However, this map 
is not onto $\Omega$. The problem is that when we take a
convex combination of two points on a circular arc, the result 
is not on the circle (unless the points coincide). However, if the 
circle has radius $1$ and the points are only $\epsilon$ apart, 
then it is clear that any convex combination is within $O(\epsilon^2)$ 
of the boundary. We claim that the domain $F(\uhp) \subset \Omega$ 
can be  mapped to $\Omega$  by a $1+O(\epsilon^2)$ quasiconformal 
map whose boundary values agree with radial projection on each arc of
$\partial \Omega$. However, to prove this, it is not enough to 
know that the two boundaries are within $O(\epsilon^2)$ of each 
other; we also need to know that the tangent directions agree to 
order $O(\epsilon^2)$. 

Lemmas \ref{comp-K} and \ref{radial map} in Appendix \ref{background}
give precise estimates bounding the quasiconformal ``cost'' of 
pushing the boundary in this way. 
We will use the latter to show that $F(\uhp)$ can be mapped to 
$\Omega$ by a $(1+O(\epsilon^2))$-quasiconformal map.
 After rescaling by linear maps, we
may assume the base interval of the boundary piece is
$[0,1]$ and the image arc lies on the unit circle.
Suppose $x \to (u,v)$
and $x \to (a,b)$ are two maps  sending $[0,1]$ to an arc of the
unit circle  with $|u-a|, |v-b|, |u'-a'|, |v'-b'| = O(\epsilon)$ and 
suppose $\varphi$ is a partition of unity. Write  
$F= (\varphi)(u+iv) + (1-\varphi)(a+ib)$ and 
$g(z) = |F|^2= F \cdot \bar  F=|\varphi (u+iv) + (1-\varphi)(a+ib)|^2$.
So to apply Lemma  \ref{radial map} to $F$ it is enough 
to estimate the derivative of $g$. We begin by rewriting $g$ 
as follows:
 \begin{eqnarray*} 
 g(x) 
&=& (u \varphi + a (1-\varphi))^2 + 
   (v \varphi + b (1-\varphi))^2   \\ 
&=& ((u-a)\varphi  + a)^2 + 
   ((v-b)\varphi  + b)^2   \\ 
&=& a^2 + b^2   + 2(u-a)a  \varphi +(u-a)^2\varphi^2  
                + 2(v-b)b  \varphi +(v-b)^2\varphi^2   \\
&=& 1   + 2(u-a)a  \varphi +(u-a)^2\varphi^2  
                + 2(v-b)b  \varphi +(v-b)^2\varphi^2   \\
&=& 1   + 2(ua-a^2+ vb - b^2 )  \varphi +(u-a)^2\varphi^2  
                 +(v-b)^2\varphi^2   \\
&=& 1   + 2((u,v) \cdot(a,b)  - 1 )  \varphi +(u-a)^2\varphi^2  
                 +(v-b)^2\varphi^2   \\
&=& 1   + 2(\cos\theta  - 1 )  \varphi +(u-a)^2\varphi^2  
                 +(v-b)^2\varphi^2, 
\end{eqnarray*}
where $\theta$ is the angle between the vectors $(u,v)$ and 
$(a,b)$, which by assumption is $O(\epsilon)$, as is its derivative
$\theta'$ (by the Cauchy estimates). If we differentiate
$g$, the constant term drops out. The second term is bounded 
by 
$$ 2 \sin \theta \cdot \theta' \varphi + O(\theta^2) \varphi'
               = O(\epsilon^2),$$
and the third term is bounded  by 
$$ [(u-a)^2 \varphi^2]' = 2(u-a)(u'-a') \varphi^2 + (u-a) 2 \varphi \varphi'
    ,$$
and each term is $O(\epsilon^2)$. The same estimate holds  for 
the last term. 
Thus the map $F$ sends $\uhp$ onto a small quasiconformal perturbation 
of $\Omega$ (a subset, because of the convexity of disks). Later, 
in Lemma \ref{Newton-radius}, when we solve the Beltrami equation using the 
dilatation of $F$ as data, this data differs by $O(\epsilon^2)$ from 
the dilatation of a quasiconformal map $\tilde F$
  that sends $\uhp$ onto $\Omega$. However, 
the error given by that lemma is also $O(\epsilon^2)$, so using the 
dilatation of $F$ in place of the dilatation of $\tilde F$ gives an 
error which can be absorbed into this estimate.

Next we wish to approximate a thick polygonal domain $\Omega$ 
with $n$ sides by a finitely bent
domain $\OmegaFB \subset \Omega$ that has $O(n)$ sides.
We assume $\Omega$ is thick for otherwise 
$\OmegaFB $ may require $\gg n$ disks; 
consider a $1 \times r$ rectangle, $r \gg 1$. 
The problem is that the medial axis of this 
polygon has an edge-edge bisector that is very long 
in the hyperbolic metric. However, if we assume $\Omega$ 
is $\delta$-thick, then there is an upper bound $A$
on the hyperbolic length of any edge-edge bisector
(and $A \simeq  \delta^{-1}$). This is the only 
property of thickness that we will use at present.

Fix some $\eta>0$. We say $\OmegaFB$ is an $\eta$-approximation 
to $\Omega$ if it is the union of the medial axis disks 
corresponding to all the vertices of the medial axis 
and  also corresponding to certain points chosen along 
the edges of the medial axis as follows:

{\bf Case 1:} First suppose $e$ is a point-point bisector.
Then we simply choose the two endpoints of $e$.  The union of
the two corresponding maximal disks is the same as the union 
of all the maximal disks corresponding to  $e$.

{\bf Case 2a:} Next suppose $e$ is an edge-edge bisector  connecting 
    two vertices of the medial axis in the interior of $\Omega$).
 Then 
by the thickness assumption its hyperbolic length is $\leq A$.
Thus given  $\eta$,
we can choose a collection of points on that edge that include
the two endpoints and has the property that no point of the 
edge is more than hyperbolic distance
 $\eta$ from some point of the collection. We call 
  this being $\eta$-dense.
Clearly we need at most $1+A/\eta$ points per edge. 
See left side of Figure \ref{pt-edge-app}.

{\bf Case 2b:} Now suppose  $e$ is an edge-edge bisector with one endpoint 
     a convex vertex of $\Omega$. Choose the segment of hyperbolic length 
      $A$ starting at the interior vertex and place points that 
      are $\eta$-dense along it.
     Clearly we need at most $1+A/\eta$ points per edge.  Note that 
     the finite union of disks has one, $D$,  closest to the corresponding 
     convex vertex of $\Omega$. The arc $\partial D \cap \Omega$ closest 
     to the vertex will also be a boundary arc of $\OmegaFB$, and we will 
     call it a ``ending arc'' of $\OmegaFB$.  We place $\eta^{-1}$ 
     equally spaced vertices along the arc to split it into smaller 
     ending arcs. These have the property that the extremal distance to 
     any non-adjacent arc is bounded away from zero (independent of 
     $\eta$), a property we will use later.

\begin{figure}[htbp]
\centerline{
 \includegraphics[height=1.1in]{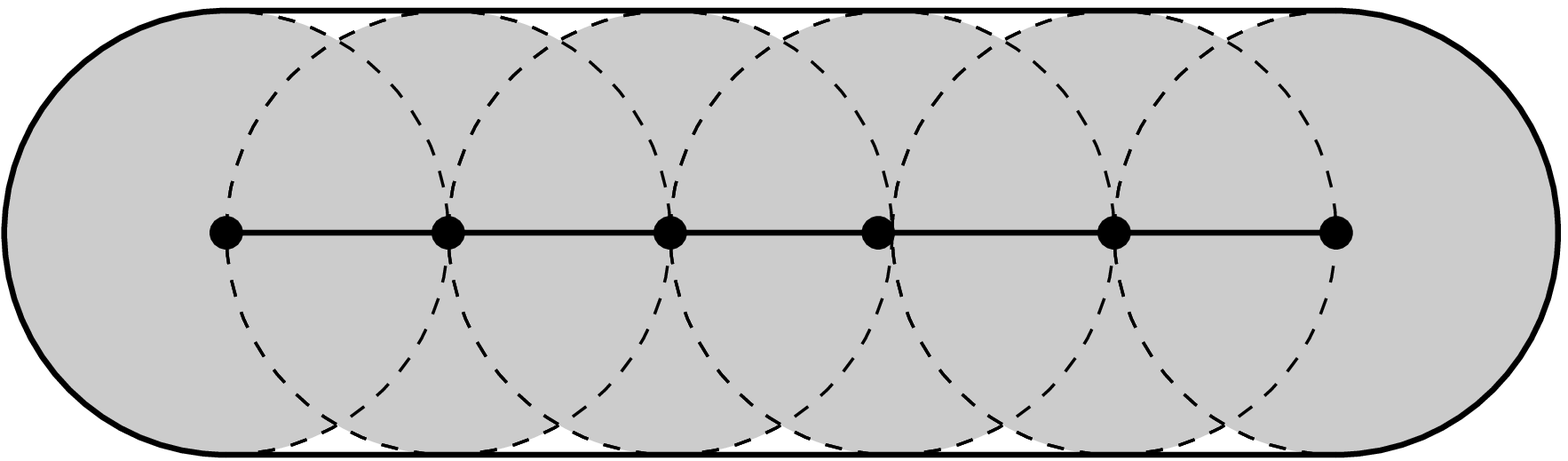}
$\hphantom{xxxxx}$
 \includegraphics[height=1.1in]{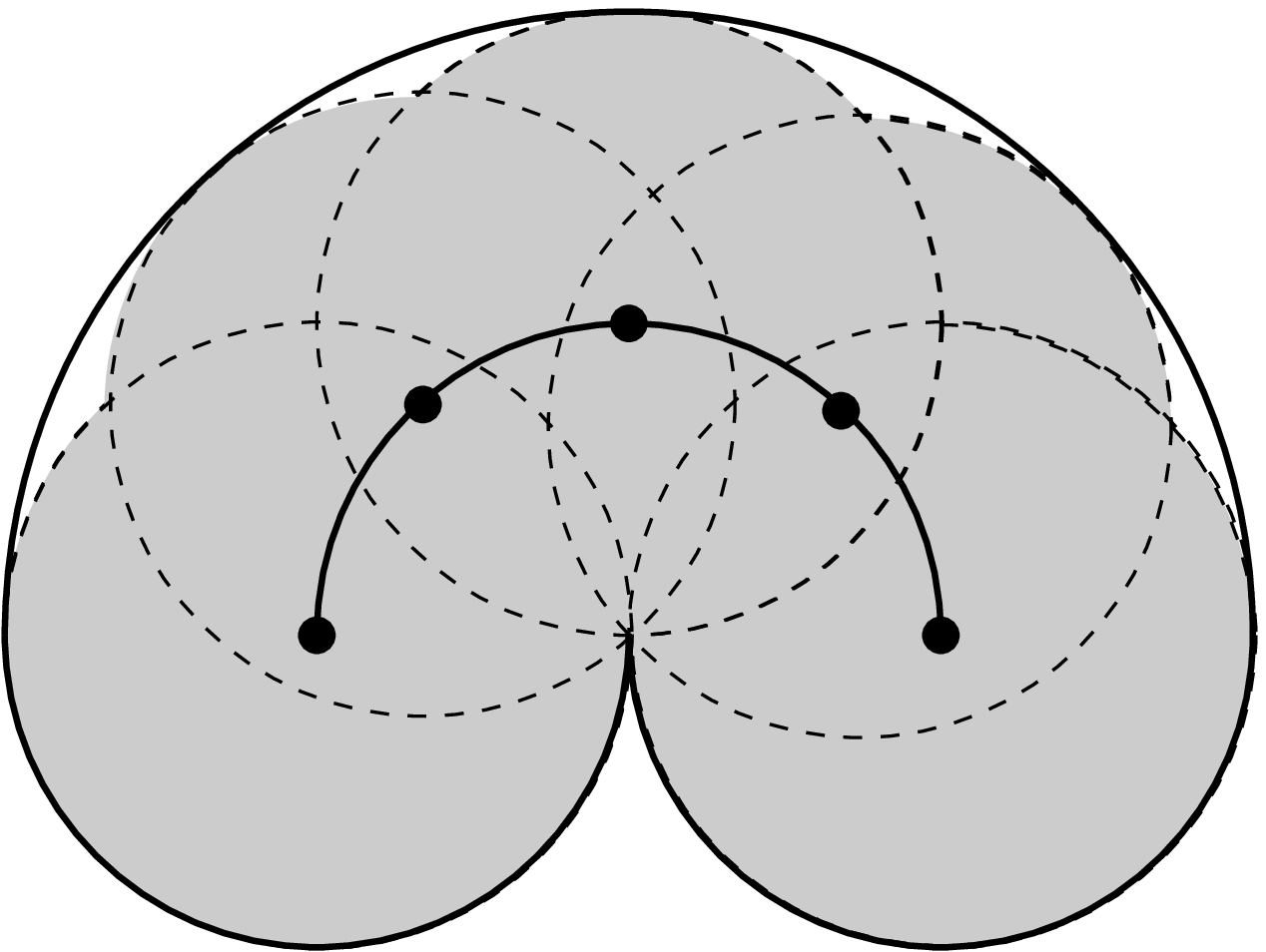}
 }
\caption{\label{pt-edge-app} The finitely bent approximation  for a 
     edge-edge bisector and a point-edge bisector  (the edge has been
     mapped to a circular arc).}
\end{figure}

{\bf Case 3:} Finally suppose $e$ is a point-edge bisector.
Renormalizing by a M{\"obius} transformation, assume the point
is $0$ and the edge is a subarc of the unit circle (such a 
transformation  preserves the collection of medial axis disks, 
but does not preserve the medial axis itself since Euclidean centers
of a circle need not be preserved by M{\"o}bius transformations).
In this new region the medial axis is an arc of the circle of 
radius $1/2$ and we take a collection of points that includes the 
endpoints, and so that the angular separation between points (when 
viewed from $0$) is $\leq \eta$.  Clearly at most $1+2\pi/\eta$
suffice. Now map this collection of maximal disks back to the 
original domain to get a collection of disks centered on $e$.  
See the right side of Figure \ref{pt-edge-app}.

We have shown:  

\begin{lemma} \label{FB-app}
Suppose $\Omega$ is a $\delta$-thick polygonal domain  with $n$
sides and 
$\OmegaFB$ is an $\eta$-approximation of it. Then 
$\OmegaFB$ has $O( n /(\delta \eta))$ boundary arcs.
\end{lemma}

Also note that by construction, each non-ending arc of $\partial \OmegaFB$ 
is tangent to an edge of $\partial \Omega$, projects orthogonally onto
a subsegment of that edge, and makes at most angle $O(\eta)$ with lines
parallel to that edge. Using Lemma \ref{comp-K}
 this means that there is a $1+O(\eta)$ quasiconformal 
map from $\OmegaFB$ into a subdomain $\OmegaRC$ 
($rc$ is for ``rounded corners'')
of $\Omega$ bounded by the the ending edges of $\OmegaFB$ 
and the projections  onto $\partial \Omega$ of the non-ending edges.

\begin{figure}[htbp]
\centerline{
 \includegraphics[height=1.50in]{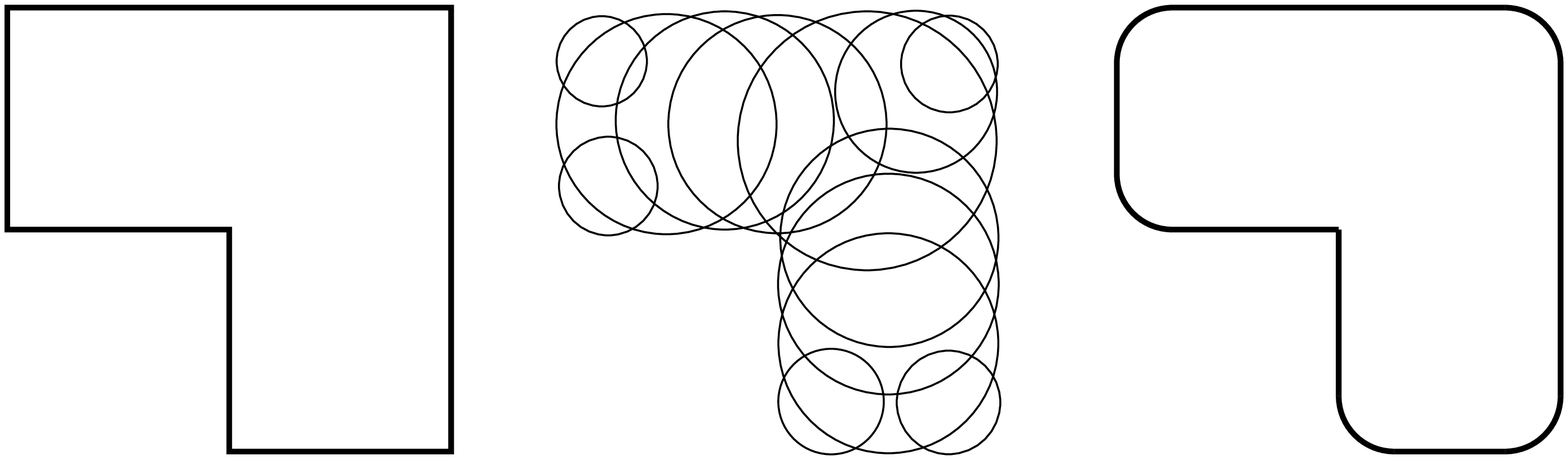}
 }
\caption{\label{RoundedOmega} 
   A polygon $\Omega$, a union of medial axis disks defining a 
   finitely bent approximation $\OmegaFB$ and the corresponding 
   domain with rounded corners, $\OmegaRC$.
   }
\end{figure}

\section{Really thick and almost thick pieces} \label{really thick}

 It would be very convenient if 
the finitely bent approximation of a thick polygon was also 
thick, for then the representation of $\OmegaFB$ would not 
require any arches. This is not the case (see  
Figure \ref{FB approx}), but there is 
a substitute that will be sufficient for 
our purposes. We shall show that if a polygon 
satisfies a stronger version of thickness, then 
its finitely bent approximation satisfies a weaker 
version of thickness. We make this precise with 
two new definitions. 

\begin{figure}[htbp]
\centerline{
 \includegraphics[height=1.50in]{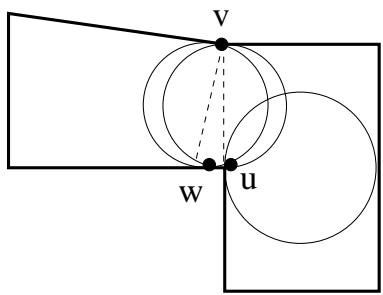}
 }
\caption{\label{FB approx} 
   The polygon is clearly thick, but has interior angle $\pi+\epsilon$
    at the vertex  $v$. This creates two medial axis 
    vertices whose circles intersect at $w$ which is only 
    $O(\epsilon)$ away from $u$, which must also be a vertex
    of the finitely bent approximation.
    Since $ \OmegaFB$ has $O(1)$ sides and at least one side shorter than 
    $O(\epsilon)$, it can't be $\delta$-thick if $\epsilon \ll \delta$.
   }
\end{figure}

Suppose $e$ is an edge of  a polygonal domain  $ \Omega$. 
 We call a point $x \in  e \subset \partial \Omega$ ``exposed'' if there is 
a medial axis disk $D$ with $x \in \partial D$ and so that 
$\partial D$ also contains a point of  an edge that is 
not adjacent to $e$.  See Figure \ref{Exposed}. 
We say the edge $e$ is exposed if it contains an exposed 
point in its interior.
We will say that $\Omega$ 
is ``really $\delta$-thick'' if it is $\delta$-thick and 
vertices  with interior angle $> \pi/2$ are never endpoints 
of exposed edges.

\begin{figure}[htbp]
\centerline{
 \includegraphics[height=1.25in]{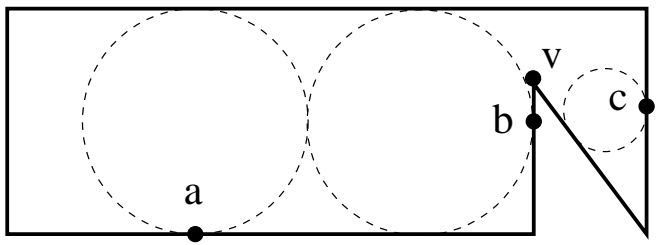}
 }
\caption{\label{Exposed} 
    The points $a,b$ are exposed but the point $c$ is not (nor is the 
     edge containing it. The polygon is thick, but not really thick, 
     because the vertex $v$ has large interior angle and bounds an 
     exposed edge (the one containing $b$).
   }
\end{figure}



Given a $\delta$ 
thick polygonal domain $\Omega$, we can always modify it to be really 
$\delta$-thick by adding a bounded number of edges per  
vertex (thus $O(n)$ overall). Compute the iota map, and the
corresponding Carleson-Whitney  decomposition. At each 
vertex with angle $\theta > \pi/2$ compute the diameter, $d$,  of the 
image of the corresponding degenerate arch, and let 
$r = \beta d $. The constant
$\beta$ is chosen small enough so that $\Omega$ contains 
a sector of radius $4r$ and angle $\theta$. On the other hand, 
$\Omega$ does not contain a sector of angle $\theta$ and 
radius $> C(\beta) r$, for this would violate the choice 
of the degenerate arch.

  Draw circles 
of radius $r, 2r$ around the vertex and replace the 
part of the edges adjacent to $v$ inside the smaller 
circle by the polygonal arc as shown in Figure 
\ref{ReallyThick}.  The new arcs are chosen so that any 
medial axis disk that hits an interior point has 
radius $< 2r$, and hence does not hit a third 
edge of $\partial \Omega$. Thus all these new edges 
are unexposed. 
When we finish making this construction
at every  vertex, the resulting domain is really 
$\delta$-thick, for all the new edges are unexposed and 
none of the old adjacent pairs of edges have small extremal 
distance (if they were separated by a modified vertex).

\begin{figure}[htbp]
\centerline{
 \includegraphics[height=2.00in]{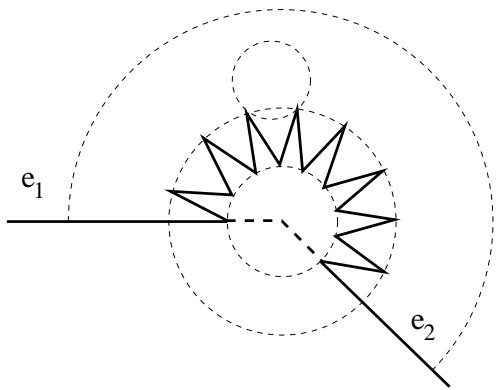}
 }
\caption{\label{ReallyThick} 
    We replace the boundary of $\partial \Omega$ by a 
    polygonal arc near each concave vertex and at a 
    scaled comparable to the largest sector in $\Omega$
    around this point.  If the original domain is thick, 
    then the new domain is really thick.
   }
\end{figure}

In particular, given a $\delta$-thick polygonal domain $\Omega$, 
we can perform a thick/thin decomposition to remove a 
thin neighborhood of each vertex. Using the construction 
above we can arrange for the thick component (there is only 
one since $\Omega$ is thick to being with) to be really thick.
If we can construct an $\epsilon$-representation for this 
really thick domain, then we can use Lemma \ref{combine-lemma} to obtain 
a $\delta$-representation of $\Omega$. Thus it suffices 
to assume that $\Omega$ is really thick.

We say that a finitely bent approximation  $\OmegaFB$  of
a polygonal domain $\Omega$ 
is ``almost $\delta$-thick'' if 
whenever two non-adjacent arcs $I_1, I_2$  have  extremal 
distance  $\leq \delta$ in $\OmegaFB$, then these arcs 
are not ending arcs and they both project onto the same edge of $\Omega$. 
Thus any arch that occurs in a Carleson-Whitney decomposition 
corresponding to $\OmegaFB$ 
has  base intervals that map to arcs projecting onto the same 
edge of $\Omega$.

We defined the normalized separation between $I_1$ and
$I_2$ as 
$$ \alpha =\dist(I_1, I_2)/\min(\diam(I_1),\diam(I_2)).$$
Since both arcs are tangent to the same line and project 
orthogonally onto disjoint segments, its easy to see 
that the circles containing these arcs intersect at 
angle $O(\alpha)$. See Figure \ref{almostthick}. 

\begin{figure}[htbp]
\centerline{
 \includegraphics[height=1.25in]{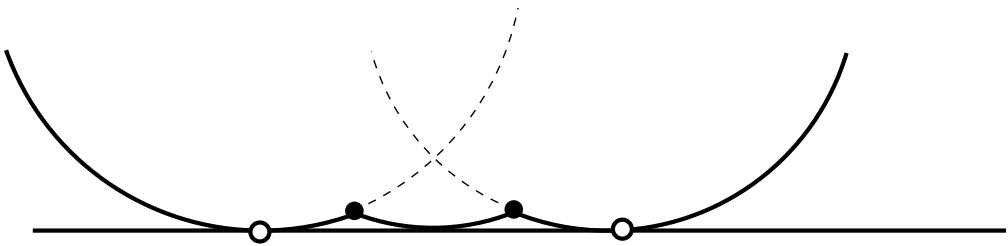}
 }
\caption{\label{almostthick} 
   We have two circular arcs of diameter $\geq 1$, which  
   are distance $\alpha$ apart and both tangent to the same 
   line. Moreover, the endpoints (shown black) are connected 
   by a path of circular arcs of similar radius and tangent 
   to the same line, which implies the distance from the 
   line is $O(\alpha^2)$. Thus the tangent points of the 
   arcs on the line (shown white) are only $O(\alpha)$ apart
  and the corresponding circles intersect with angle $O(\alpha)$.
     }
\end{figure}

 The main fact we will use concerns the 
conformal map $f$  from $D_1$ to $D_1 \cup D_2$.

\begin{lemma} \label{eta power} 
 Suppose 
$D_1, D_2$ are disks of comparable size  which intersect 
at angle $\alpha$ and $f: D_1 \to D_1 \cup D_2$ is  conformal 
and chosen to fix the two intersection points, $a,b$,  and 
 and  the point $c \in \partial D_1 \setminus D_2$ on the 
bisector of $a,b$.
Then $|f(z) -z| = O(\alpha) \diam(D_1)$
for every $z \in D_1$.
\end{lemma}

\begin{proof} 
  This can be proven by an 
explicit calculation.  Note that 
 $f= \tau^{-1}( \lambda \tau(z)^{1+ \alpha}  )$ 
where $\tau(z) = (a-z)/(z-b)$ and $\lambda $ is an 
appropriately chosen constant of modulus $1$ so that 
$\tau(c) = i$. 
See Figure \ref{EtaPower}.
Suppose $z \in D_1$,  $w = \tau(z)$ and $x = |w|$. If $ 0 \leq x \leq 1$, 
then $x- x^{1+\alpha}$ attains its maximum at its 
critical points, i.e., at $x = (1+\alpha)^{1/\alpha}$ which  
close to $1/e$ for small $\alpha$. Thus the maximum is 
$O(\alpha)$ which proves the estimate for such $z$ (because 
$\tau$ and $\tau^{-1}$ have bounded derivative at such points).
The case of $x> 1$ is similar using $1/\tau$. 
\end{proof}
\begin{figure}[htbp]
\centerline{
 \includegraphics[height=1.75in]{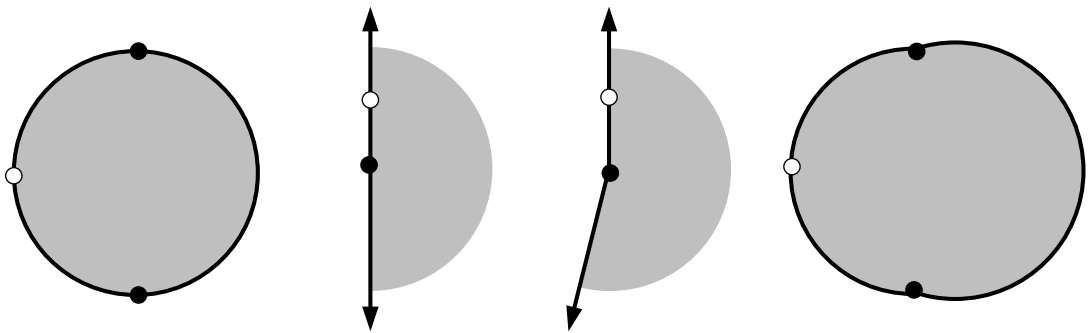}
 }
\caption{\label{EtaPower} 
    The disk $D_1$ mapped by $\tau$, $z^{1+\alpha}$ and $\tau^{-1}$.
    The M{\"o}bius transformation $\tau$ sends $a,b,c$ to $0,\infty, i$.
   }
\end{figure}

\begin{lemma}
If $\delta$ and $\eta$ are  small enough the following holds.
If $\Omega$ is a really $\delta$-thick polygon,
then the finitely bent $\eta$-approximation, $\OmegaFB$, 
is almost $\delta/2$-thick.
\end{lemma}

\begin{proof}
Suppose we have two non-adjacent  arcs, 
$I_1, I_2$  in the boundary of  $\OmegaFB$.
 By construction, 
the extremal distance of an ending arc to any non-adjacent 
arc is bounded away from zero (independent of $\eta$). So we
may assume neither of these is an ending arc.
 Thus they are 
mapped to line segments $J_1, J_2 \subset \partial \Omega$,  
under the $(1+O(\eta))$- quasiconformal 
map $\OmegaFB \to \OmegaRC$ discussed above. Since $K$-quasiconformal maps 
change extremal length by a factor of at most $K$, the extremal 
distance between $J_1$ and $J_2$ is  $  \geq (  1+O(\eta))(\delta/2) \leq 
 \delta $ if $\eta$ is small enough.
Hence the extremal 
distance between the edges of $\Omega$ containing
 them is also  $\leq \delta$ . Since 
$\Omega$ is $\delta$-thick, the edges containing $J_1$ and $J_2$ must be 
the same or adjacent.
 If they are adjacent 
edges, then  the common endpoint $v$  must   be $\leq \pi/2$ 
by the really thick condition.  But the arc $I_1$ has 
a tangent point within distance $\alpha \cdot \diam(I_1)$  of the 
common vertex for a small 
$\alpha$ (depending on $\eta$), which means the circle containing $I_1$ can't 
bound a disk in $\Omega$ (it would hit the other edge  hitting
$v$). This contradiction means 
$I_1, I_2$ project onto the same edge of $\partial \Omega$, 
a.s desired. 
\end{proof}

\begin{lemma} \label{compose rep FB}
For any $\epsilon >0$ there is an $N < \infty$ so that the 
following holds. 
Suppose $\{\Omega_k\}_0^N$ is the angle scaling family constructed from 
an almost $\delta$-thick finitely bent domain $\OmegaFB=\Omega_N$ and suppose 
$0 \leq k < N$. Then given an 
$\epsilon$-representation of $\Omega_k$, we can construct an 
$2\epsilon$-representation of $\Omega_{k+1}$ in time $O(n)$ (constant
depends on $\epsilon, N$).
\end{lemma} 

\begin{proof}
The representation of $\Omega_{k+1}$ uses the same set $S \subset
\reals $ and the same
Carleson-Whitney decomposition ${\cal W}$  as is given for $\Omega_k$; 
only the functions will change by composing with a conformal map that 
sends the image from $\Omega_k$ to $\Omega_{k+1}$.
(This simplicity is why we have gone through the effort of defining almost 
thick domains; if we allowed more general arches, we would have 
to update $S$ and ${\cal W}$ as well.)

For each Whitney box $Q$ let $f$ be the map defined on $Q$ and 
  choose a  gap $G$  in the gap/crescent 
decomposition of $\Omega_k$ which hits $f(Q)$ (or so that an adjacent 
crescent hits $f(Q)$).
This gap  is 
the image under  M{\"o}bius transformation $\sigma_k$,
of an ideal triangle in the complement of the bending lamination.
The corresponding gap in $\Omega_{k+1}$ is the image of the same 
triangle under a transformation $\sigma_{k+1}$.
We modify the function  $f$
associated to $Q$ by post composing by $\sigma_{k+1} \circ \sigma_k^{-1}$.

In the previous paragraph we claimed that for each 
box $Q$ we could find a gap or crescent that intersects
$f(Q)$. We should verify that we can do this for all 
$O(n)$ boxes in time $O(n)$.  This is not immediately 
clear since arbitrarily many gaps and crescents may hit 
a particular $f(Q)$, so there is not a constant amount of
work to do per box. However, both the Whitney boxes and the 
gap/crescent decomposition come with tree structures. In 
time $O(n)$ we can search both trees to find some box and 
some gap/crescent that intersect. Treating these as roots 
of their respective trees, we can then search outwards 
finding the desired intersecting pieces and visiting each 
vertex of each tree only once (the process is analogous 
to merging two sorted lists in linear time). The details are given at 
the end of Appendix \ref{background} following Lemma \ref{make cover}.

For each of the three kinds of boundary pieces, $Q$,
we associate conformal  maps  $\tau_k, \tau_{k+1}$ onto  
corresponding part of $\Omega_k, \Omega_{k+1}$ and form 
the updated map by composing with $ \sigma_k = \tau_{k+1} 
\circ \tau_k^{-1}$. We just have to specify the maps 
 $\tau_k$ for each type of piece.

If $Q$ is a Carleson square, then it hits a boundary 
crescent and $\tau_k$ is the associated M{\"o}bius transformation.
 For a degenerate 
arch, the base is an interval containing a point $v$ of $S$ which
divides it into two subintervals whose images lie on intersecting 
circles. We take the map $\tau_{k}$ to be  the composition 
of a power function (determined by the angle of intersection 
of the circles)  and  
M{\"o}bius transformation chosen send  $v, \infty$ to 
the two intersection points of the circles. 
 For an arch $Q$, the base 
consists of two intervals which are mapped onto two arcs in 
boundary of the corresponding domain. 
By construction, $\OmegaFB$ is almost thick, so these two arcs
lie on circles which intersect at angle $O(\eta)$ where $\eta$ is the 
normalized separation between them. The map $\tau_k$ is 
a M{\"o}bius transformation of the upper half-plane to $D_1$, 
followed by the conformal map given in Lemma \ref{eta power}.

To estimate the norm of the new collection, we use 
Lemma  \ref{varphi-est} (for Whitney boxes and Carleson squares),  
Lemma \ref{elliptic-est} (for degenerate arches)  and 
Lemma \ref{eta power} (for arches). Each result says that the 
maps $\tau_{k+1} \circ \tau_k^{-1}$ are close to the 
angle scaling map 
$\psi_k : \Omega_k \to \Omega_{k+1}$  which is a
quasi-isometry with constants as close to $0$ as we wish
by taking $N$ large. Thus for any point in overlapping 
neighborhoods of adjacent pieces  $Q_i, Q_j$
\begin{eqnarray*} 
 |f_i^{k+1} (z) - f_j^{k+1}(z) | &\leq&   |f_i^{k+1}(z) - 
          \psi_k(f_i{k}(z)) |
+ |\psi_k(f_i^k(z)) - \psi_k(f_j^k(z))  | \\
 && \qquad \quad | \psi_k(f_j^k(z)) - f_j^{k+1}(z) |. 
\end{eqnarray*}
The middle term is small if ${\cal F} = \{ f^k\}$ has small norm and 
$\psi_k$ has quasi-isometry constant close to $0$. The 
first and last terms are small since $\sigma^k_j$ is close 
to $\psi_k$. 
\end{proof}

\begin{lemma} \label{final rep}
Given an $\epsilon$-representation of $\OmegaFB$ (the 
finitely bent $\eta$-approximation 
to a really thick polygonal domain  $\Omega$)  we can construct 
an $2\epsilon$-representation of $\Omega$ in time $O(n)$,
assuming $\eta$ 
is chosen small enough (depending 
on $\epsilon$, but not on $n$ or the geometry of $\Omega$).
\end{lemma} 

\begin{proof}
The proof is similar to the previous proof. As above, we leave 
the set $S$ and the decomposition the same, and only change
the maps. For Whitney boxes, we compose with the identity (i.e., 
use the same map).  We do the same for boundary pieces
corresponding to ending arcs.  For Carleson squares  corresponding to 
non ending arcs, we compose by the M{\"o}bius transformations 
that maps the boundary arc onto its orthogonal projection on 
$\partial \Omega$ (with the tangent point going to itself). 
For degenerate arches we use the power function composed with 
M{\"o}bius transformation that sends the boundary arc into 
its orthogonal projection with the vertex going to its 
projection. For arches we use the conformal map of $D_1 
\cup D_2$ onto the half-plane bounded by the line containing 
the corresponding boundary edge. Since the diameter of the 
boundary pieces $W$  are  comparable to the lengths of the boundary 
arcs they hit, each type of maps moves points by less than
$O(\delta) \diam (W)$. Thus an $\epsilon$-representation is 
converted to a $(\epsilon+ O(\delta))$-representation by taking 
$\delta$ small enough. 
\end{proof}

So given any polygonal domain, we perform a thick/thin decomposition 
to remove the hyperbolic thin parts. For each resulting thick part
we decompose it into a really thick part and its parabolic 
thin parts. We then take a finitely bent approximation of each  
really thick components, and using angle scaling chains to 
inductively construct  representations of the really thick components.
We then combine these with the explicit representations of the 
thin components to get a representation of the original 
domain using Lemma \ref{combine-lemma}. 
This completes proof of Theorems
\ref{thmQC} and \ref{main}  except for the 
 proof of Lemma \ref{Newton-radius}.


\section{Iterating to the solution} \label{iterate}

In this section we will show how to improve a 
partial $\epsilon$-representation quickly  assuming a certain 
$\overline{\partial}$-problem can be solved quickly.
The following is our Newton type iteration for improving 
an $\epsilon$-representation. The fact that $\epsilon_0$ does 
not depend on $\Omega$ is one of the pillars upon which the 
whole proof rests.

\begin{lemma} \label{Newton-radius}
There is an $\epsilon_0 >0$ so that if $0<  \epsilon < \epsilon_0$
then the following holds.
Given a partial  $\epsilon$-representation of a polygonal region $
\Omega$ we can construct a partial  $\epsilon^2$-representation 
in time $O(n \log \frac 1 \epsilon \log \log \frac 1
\epsilon ) $.
\end{lemma}

Recall the notation $\Dbar f = f_{\bar z} = \frac 12 ( f_x + i f_y)$, 
$\partial f = f_z = \frac 12 ( f_x - i f_y)$. The Beltrami 
dilatation of a map is given by $\mu_f = f_{\bar z} / f_z = 
\Dbar f /\partial f$.

Suppose $(S_0, {\cal B}_0, {\cal G}_0)$ is a partial $\epsilon$-representation 
of the domain (recall this means $S$ is a set of $n$ points giving 
our current guess of the prevertices, ${\cal B}$ is the covering of
the corresponding decomposition, and ${\cal G}$ is a function defined 
on each decomposition piece). 
 Let $F$ denote the quasiconformal map associated to this 
 $\epsilon$ approximation by (\ref{defn-F}). Then $\mu_F$ is a  piecewise
rational function  bounded by $\epsilon$ and we can estimate 
it to within $\epsilon^2$ by a polynomial on each piece with 
at most $O(p)$ terms (using the geometric sum formula).
More precisely, if $Q_k$ is a decomposition piece and  $z \in N_s(Q_k)$,  then
\begin{eqnarray*}
\mu_F(z) & =&
  \frac{\sum f_j (z)\Dbar \varphi_j(z)}{\sum \partial f_j(z) \cdot \varphi_j (z)
               + f_j(z) \cdot \partial \varphi_j(z) }  .
\end{eqnarray*}
Let $\mu$ be the symmetrized 
version (with respect to the real axis) of this approximation,
i.e. $\mu$ is extended to the lower half-plane by $\mu(\bar z) =
\overline{\mu(z)}$. Any solution of the Beltrami equation with this 
data is also symmetric, so maps the real line to itself.  Note that $\| \mu \|_\infty =O(\epsilon)$ 
and  is supported on the $O(n)$ squares that cover the boundary of 
our Carleson-Whitney decomposition (and its reflection in the lower 
half-plane).
If $\|\mu\|_\infty \leq \epsilon$, we
 wish to  find a quasiconformal map $H$ of $\uhp$ to itself
that satisfies 
\begin{eqnarray} \label{BE+error}
\mu_H = \mu + O(\epsilon^2).
\end{eqnarray}
Then $G=F \circ H^{-1}$ is   a quasiconformal map of $\uhp$
to $\Omega$ with quasiconformal constant $1+ O(\epsilon^2)$.

Let $p = O(|\log \epsilon|)$. Suppose we can find 
a  $p$-term 
series expansion for $H$ in each empty piece of the decomposition.
Then we could compose $H^{-1}$ with our existing representation 
for $F$ to get the $p$ first terms of series approximating  $F \circ H^{-1}$
on the empty pieces of the Carleson-Whitney decomposition associated 
to $H(S_0)$, and this
series is accurate to  within $\tilde \rho_\Omega$-distance $O(\epsilon^2)$. 
We can then  apply 
Lemma \ref{expand-est} to obtain a $O(\epsilon^\alpha)$
representation for some  $\alpha  >1$. Iterating this a 
fixed number of times (depending on $\alpha$ and the constant 
in the ``$O$'') gives a $\epsilon^2$-representation.

In the remainder of this section we show how to define the 
new representation of $\Omega$, given the expansions for  $H$.
 In the next section we show how to 
define $H$, assuming we can solve a certain $\Dbar $ problem,
and in Section \ref{multipole} we show how to solve this
$\Dbar$  problem.


How do we define the new representation 
for $\Omega$? 
Suppose $(S_0, {\cal B}_0, {\cal G}_0)$ is the previous 
partial $\epsilon$-representation
with dilatation $\mu$ and that $H$ is a $(1+O(\epsilon))$-quasiconformal map 
of $\uhp$ to itself that solves $\mu_H = \mu + O(\epsilon^2)$.
Moreover, $H$ is conformal on the empty pieces of our 
decomposition and we have power  series on each empty 
piece $D$ that agree with $H$ to within $O(\epsilon^2 \diam(f_k(E)))$
on $N_s(E)$ where $E$ is a boundary component of $Q_k$.

The points $S=H(S_0)$ are known, and we take these to be the first 
part of our new triple.  In time $O(n)$ we can compute a covering 
of the hyperbolic convex hull of $S$ and extend this to 
a decomposition of $\uhp$. This will be our new ${\cal B}$.
For each Whitney square $Q$ in ${\cal B}$, choose a Whitney square 
$Q_0 \in {\cal B}_0$ so that $H(Q_0)$ hits $Q$ (we can do this in 
bounded time since $H$ is almost an isometry and we need only search 
a uniformly bounded number of possible squares).
Then there is a ball $B$ so that $Q \subset B \subset H(5 Q_0)$ and 
we can compute the inverse $h_Q$ of the power series of $H$ on $B$ in 
and then compute the composition of that series 
followed by the series for $Q_0$ (which converges on $10 Q$).
This gives the element of ${\cal F}$ corresponding to $Q$.
This computation takes time $O(p \log p)$, $p=\log \frac 1\epsilon$
(see Lemma \ref{inexact-comp} of Appendix \ref{fast-power-series}).
A truncation of this series, as given by Lemma \ref{expand-est}, 
will give the desired element of ${\cal F}$.  To see how, 
we use the following result.

\begin{lemma}
Suppose $f: \disk \to \disk$ is $(1+\epsilon^2)$-quasiconformal 
and is conformal on the disk $D(0,1/M)$ and maps $0$ to $0$. Then there is a 
truncation $g$ of the power series of $f$ at $0$, such that 
$|g(z)-f(z)| \leq O(\epsilon^{2(1-\beta)})$ on $D(0,M^{\beta-1})$.
\end{lemma}

\begin{proof}
Since $f$ is $(1+\epsilon^2)$-quasiconformal and fixes the
origin, it agrees with 
a rotation $\tau$  to within $O(\epsilon)$.  Apply Lemma \ref{expand-est}
to deduce the result.
\end{proof}

Our previous remarks prove:

\begin{lemma} \label{small-norm}
If $\| \mu\|_\infty = \epsilon < \epsilon_0$ then 
$\| {\cal G}\| \leq O( \epsilon^{2(1-\beta)})$.
\end{lemma}

Thus if $\beta < 1/2$ and $\epsilon_0$ is small enough we get
a definite improvement, and a bounded number of iterations will 
improve it below $\epsilon^2$.  This choice of $\beta$ determines
the choice of $M$ in the construction of the decomposition in 
Section \ref{extend-decom}.

\section{Reducing the Beltrami problem to a $\Dbar$ problem}
\label{reduce-to-Dbar}

In Appendix \ref{background} we recall that the 
Beltrami equation $\Dbar f = \mu \partial f$ can be 
solved using an infinite series of the form 
$T \mu + T \mu T \mu + \dots $, where $T$ is the 
Beurling transform. The method we describe below 
might be adapted to use several terms of this series, 
but we shall use only the leading term.

We will replace the Beltrami equation 
$ \Dbar H  = \mu  \partial H$ with the  easier equation 
$\Dbar H =  \mu$. 
For $\epsilon$ small, we will show $\partial H$ is close to a 
constant on each piece of our decomposition, which means the 
solution of the Beltrami problem is close to a constant multiple 
of the solution of the $\Dbar$-problem.  
  Since $\mu$ has compact support, 
the $\Dbar$-problem can be  exactly solved by the convolution 
$$ G_1 (z) = \frac 1{2 \pi i} \iint \frac {\mu(w) dxdy}{z-w}.$$
It will actually be slightly more convenient to deal 
with $ \partial  G_1$, which  is given by the 
Beurling transform $T\mu$, 
$$ \partial G_1(z) = 
 T \mu (w) = \lim_{r \to 0} \frac 1 {2 \pi i} \iint_{|z-w| > r} \frac {\mu (z)}
{ (z-w)^2} dxdy.$$
This convolution gives a solution of $\Dbar G_1 = \mu$, but 
we cannot compute $\partial G_1$ exactly in finite time. However,  we
can  compute  
$p$ terms of a power or Laurent series that approximates $\partial G_1$ 
in each of the empty 
pieces of our decomposition (where $G_1$ is holomorphic).
 We will refer to the expansion on a piece
$Q$ as $G_Q$. We will see later that for each piece $Q$ 
of our decomposition we can compute an expansion $G_Q$ so that 
\begin{eqnarray} \label{GQ}
 |\partial^2 G_1(z)  - \partial^2 G_Q(z)| \leq \frac {C \epsilon^2}
{ \diam (D_Q) }, z \in D,
\end{eqnarray}
where $D$ denotes the empty region of $Q$.
(The $\epsilon^2$ could be 
replaced by a higher power of $\epsilon$, if necessary, by simply 
taking more terms in $G_Q$). 

 It will also be convenient to consider a function 
$G_2$ on $\uhp$ so that (1) $G_2=G_1$ on $N_s$ (recall this is 
the union of  $N_s(E)$ over all boundary  components of all 
pieces) , (2)  $G_2=G_Q$ on  $D_Q$
and  (3) $\Dbar G_2 = \Dbar G_1 + O(\epsilon^2)$.
This can easily be obtained by  combining $G_1$ and the $G_Q$'s 
using a partition of unity 
whose  gradient is supported in $N_{s}(D)  \setminus D$,
 whose gradient in bounded by $O(\diam(E)^{-1})$ and whose second 
gradient is bounded by $O(\diam(E)^{-2})$ on $N_s(E)$, when $E$ 
a boundary component of $D$. Note that this implies 
\begin{eqnarray} \label{G1-G2}
  | \partial G_1 (z) - \partial G_2(z) | \leq C \epsilon^2.
\end{eqnarray}
The algorithm only 
requires the computation of $ \partial G_Q$, not of $ \partial G_2$;
 the latter function will 
only be used in the proof that the algorithm gives the desired 
accuracy.

\begin{lemma}\label{H''-est}
Suppose $E$ is a boundary component of a piece $Q$. Then
$| \partial^2 G_1(z)| \leq  C \epsilon / \diam(E)$, $z \in N_s(E)$.
The same estimate also holds for $G_2$.
\end{lemma}

\begin{proof}
Assume $ z \in N_s(E)$.
Let $D$ be a Euclidean disk of radius $r \simeq \diam(E)$ around $z$
and let $\chi_D$ denote its characteristic function ($=1$ on $D$ 
and $0$ off $D$).
Let $L(w)$ be the linear function such that $|\mu(w) - L(w)| \leq C
\epsilon |z-w|^2/ r^2$  for $w \in D$ (there is such a function 
because we constructed
$\mu$ to have second derivative bounded by $C \epsilon /r^2$).
Then,
\begin{eqnarray*}
| \partial^2 G_1(z) | &\simeq& |\iint \frac {\mu(w) dxdy}{ (z-w)^3} |\\
        & \lesssim& \iint  \frac {|\mu(w)-L(w) \chi_D(w)| dxdy}{ |z-w|^3}
            + |\iint_D  \frac { L(w)  dxdy}{ (z-w)^3}|
             + \iint_{D^c}  \frac {\epsilon dxdy}{ |z-w|^3}\\
       &\lesssim& C \epsilon \iint_D \frac {dxdy}{r^2|z-w|}+ 0 + C \epsilon/r \\
        & \leq & C \epsilon/r.
\end{eqnarray*}
The final claim holds because $G_1=G_2$ on $N_s(E)$.
\end{proof}

\begin{cor} \label{H''-cor}
Suppose $Q$ is a piece of our decomposition. Then 
$ |\partial G_2(z) - \partial G_2(z_Q)| \leq C \epsilon$
for every $z \in Q$.
\end{cor} 

\begin{proof}
Recall the definition of $z_Q$ (see Figure \ref{CW-pieces1}).
For Whitney type squares and Carleson squares, there is only one
boundary component and the proof of Lemma 
\ref{H''-est} actually applies to any point in $Q$ and then we 
 integrate along a segment connecting 
$z$ and $z_Q$ and note that the length of the segment at most $O(\diam(E))$.

For arches, it is slightly more involved since there is a boundary component 
that is much smaller than the diameter of the
whole piece. Since $\partial G_2$ is holomorphic 
in the empty part of the piece, the maximum principle implies it is enough 
to check the inequality near the boundary components.  The proof for the 
``big'' boundary component is just like the case of Whitney squares and 
Carleson boxes.  For the `small' boundary component it suffices to check the 
estimate at one point near  that component, say the one directly underneath 
$z_Q$ and distance $2s \cdot  \diam(E)$ above $E$. Then 
$$ \partial^2 G_1 = O( \iint \frac {|\mu(w)|}{|z-w|^3} dxdy),$$
and we can break the integral into two parts corresponding to the 
two complementary components of the annulus.
The unbounded  part is bounded by $O(\epsilon \diam(Q)^{-1})$ and the other 
is bounded by $O(\epsilon t^{-3} \diam(E)^2)$ for a point at height $t$.
Integrating both estimates from $t=\diam(E)$  to   $t=\diam (Q)$, gives
$ O(\epsilon ) + O(\epsilon)$ and is an upper bound for the variation of 
$\partial G_1$ along the vertical line segment connecting the two boundary 
components.  This proves the desired estimate for $G_1$. It follows for 
$G_2$ by (\ref{G1-G2}).
\end{proof}

Given the function $G_2$, which  is an approximate solution of the 
$\Dbar$-problem, we can define a quasiconformal mapping on each 
piece of our decomposition.  Assume for the moment that  
$z_Q$ maps to  $w_Q$ and let $d_Q = 
\text{Im}(w_Q) /\text{Im}(z_Q)$. Define 
$$ H_Q(z) = L_0(z) + d_Q (G_2(z) -L_1(z)),$$
where
$L_0$ is the unique conformal linear map of $\uhp$ to itself 
that maps $z_Q$ to $w_Q$, and $L_1$ is the unique conformal linear map
that agrees with $G_2$ at $z_Q$. Note that $H_Q(z_Q) = w_Q$ and 
$\partial L_0 = d_Q$ and that we only need to know the 
difference $G_Q(z) - G_Q(z_Q)$ in order to define $H_Q(z)$.

By Corollary \ref{H''-cor} $|\partial G_2 - \partial L_1| \leq O(\epsilon)$
on $N_s(Q)$. Thus for $z \in N_s(Q)$, 
$$ \partial H_Q = d_Q + d_Q O(\epsilon) = d_Q (1+ O(\epsilon)),$$
and  using this we get
\begin{eqnarray*}
\Dbar H_Q = d_Q \Dbar G_2 =  d_Q \mu + O(d_Q \epsilon^2)   
= (\mu +O(\epsilon^2)) \partial H_Q 
\end{eqnarray*}
\begin{eqnarray} \label{H-mu}
\mu_{H_Q} &=& (\mu +  \epsilon^2)
\end{eqnarray}
 for $z \in Q$. Thus  $H_Q$ is a quasiconformal 
map on $N_s(Q)$ with dilatation  $\mu + O(\epsilon^2)$.

So far we have assumed that we know the images 
$w_Q$ of 
$z_Q$. We now describe how to find $w_Q$. We proceed
in a ``top-down'' manner,  by starting 
at the root piece $Q_0$ and assuming $H(z_{Q_0}) = z_{Q_0}$ is fixed. Then 
if $Q$ is a child of $Q_0$ let $w_Q= H_{Q_0}(z_Q)$, i.e., use the 
quasiconformal map for the parent piece to define where the center point 
of the child maps to.  In this way we can proceed inductively down
the tree of pieces and define $H_Q$ for every piece.

In each step of this procedure, we might introduce an error of 
size  $O(\epsilon^2)$ in the definition of $w_Q$. This is not a 
problem by itself since we are only trying to compute a map 
with this accuracy. However, it is possible to have to adjacent 
pieces $Q$, $Q'$ of our decomposition that have no common ancestor for 
a very large number of generations, and in such a situation the 
definitions of $w_Q$ and $w_{Q'}$ might begin to diverge to an 
unacceptable degree. 
In order to avoid  this problem,  we modify  the ``top-down'' 
induction described above for Whitney type pieces. If the current 
piece is a Whitney type piece $Q$ and it is adjacent to a 
Whitney type piece $Q'$ of the same height and they have  $3$rd 
generation descendants that are Whitney type and 
 adjacent, then we modify the definitions 
of $w$ for the $3$rd generation descendants as follows. 
For the descendants of $Q$ and $Q'$ that are adjacent we 
define $w$ using the average of the values we would get using 
$H_Q$ and $H_{Q'}$. For the descendant next to the center of $Q$  
we use the value of $Q$ alone and for the intermediate values
we take weighted averages that linearly interpolate between the 
endpoint cases.  This step insures that adjacent pieces 
have $d_Q$ values that are within $1+O(\epsilon)$ of each other, i.e.,

\begin{cor} \label{cor43}
If  $Q$ and $Q'$ are adjacent pieces of the decomposition then 
$   d_Q / d_{Q'} = 1 + O(\epsilon)$.
\end{cor}

Suppose $Q$ and $Q'$ are adjacent pieces. How close are 
the functions $H_Q$ and $H_{Q'}$ along the common boundary?
Note that both maps are  quasiconformal on   a Carleson 
square containing both $Q$ and $Q'$ and that $H_{Q'}$ was 
chosen to agree with $H_{Q}$ at the point $z_{Q'}$. Thus the 
maps also agree at the reflection of this point in the 
lower half-plane and both maps take $\infty$ to $\infty$.
Moreover, by Corollary  \ref{cor43} and (\ref{H-mu})
the dilatations of these maps agree to within $\epsilon^2$
and hence it follows from Lemma \ref{QC-near-Id} of
Appendix \ref{background}
that the maps agree to within $d_Q \epsilon^2$, i.e., 
\begin{eqnarray} \label{H-close}
   | H_Q(z) - H_{Q'}(z)| = O(d_Q \epsilon^2),
\end{eqnarray}
for $z$ in $N_s(Q) \cap N_s(Q')$.

Given the approximate solutions $H_Q$ on each piece we can combine
them using a partition of unity
to get a single approximate solution on the whole 
upper half-plane (again, we only do this to estimate the error; 
it is not necessary to do so as part of the algorithm).
Define a mapping of $\uhp$ to itself by 
$$ H = \sum_Q \varphi_Q H_Q$$
where $\{ \varphi \}$ is a partition of unity as in 
Section \ref{epsilon-reps}.
If $Q$ and $Q'$ are adjacent squares and $z$ is in the $s$ neighborhood
of the common boundary, then 
$$ |d_H(z_Q) - d_H(z_{Q'}) | \leq C d_H(z_Q)  \epsilon.$$
$$|H_Q(z)-H_{Q'}(z) | \leq \epsilon d_H(z_Q) \text{Im}(z_Q)$$
 Also,
$$ \Dbar H = \sum_Q \Dbar \varphi_Q 
\cdot  H_Q + \sum_Q \varphi_Q \cdot  \Dbar H_Q =I +II .$$

Since $\sum_Q \Dbar \varphi_Q = 0$,
$\sum |\Dbar \varphi_Q| \leq C/\text{Im}(z)$, 
the estimate (\ref{H-close}) implies that I is 
bounded by $O(d_Q \epsilon^2)$. 
Since $\sum_Q \varphi_Q =1$, the estimate 
(\ref{H-mu}) 
 shows that second sum is $d_Q \mu+  O( d_Q\epsilon^2) $. Similarly, 
$$ \partial  H = \sum_Q \partial \varphi_Q H_Q + 
          \sum_Q \varphi_Q \partial H_Q.$$
As above, the first sum is $O(\epsilon^2  d_H(z) )$ and the 
second sum is 
$$d_H(z_Q) (1+ O( \max|d_H(z_Q)-d_H(z_{Q'})|+ \epsilon))$$
 where the 
maximum is over all pieces $Q'$ that are adjacent to $Q$.  This 
maximum is $O(\epsilon)$, so we have 
$\partial H = d_H(z_Q)(1 + O(\epsilon))$ if $z \in Q$. Thus 
$$ \mu_{H}(z)= \frac {\Dbar H(z)}{\partial H(z)}
= \frac {d_H(z_Q)( \mu+O(\epsilon^2))}{ d_H(z_Q) (1+ O(\epsilon))} = 
    \mu+ O(\epsilon^2).
$$
Thus in the empty regions (where only one partition of unity function 
is non-zero) our piecewise solutions $H_Q$ agree with a global
solution $H$. On the empty regions of the decomposition $H_Q$ agrees 
with a function defined using only $G_Q$, which we can compute.  Thus the 
truncated expansions that we can actually compute,  
agree (on the empty regions)
with a globally defined quasiconformal 
map whose dilatation is $\mu+ O(\epsilon^2)$.





\section{Fast computation of the Beurling transform} \label{multipole}

 Now it only remains to 
compute $  \partial G_Q(z)-G_Q(z_Q)  \approx T\mu$ as quickly and as 
accurately  as we claimed. Since we only need to compute this 
difference in regions where $T \mu$ is holomorphic, we will compute 
a series expansion for 
$$ \partial T\mu(w) = \lim_{r \to 0} \frac {-1} { \pi i} \iint_{|z-w| > r} \frac {\mu (z)}
{ (z-w)^3} dxdy$$
 and then simply integrate the series term-by-term.
We need to compute this with error at most 
$$ \frac {C \epsilon^2}{ \diam (Q_j)},$$
on a Whitney square $Q_j$.

We shall use  the 
fast multipole algorithm of Rokhlin and
Greengard \cite{GR87} (named one of the top ten 
algorithms of the 20th century in   \cite{SIAM-top10}).
The basic idea is that we have  $n$ empty regions 
where want to compute a series expansion, each of which is influenced
by the $n$ regions where the data is supported. This is $n^2$ 
interactions to compute in only  time $O(n)$. The 
multipole method takes advantage of the fact that pieces of data
that are close together  affect distant outputs in similar ways.
Thus the data can be grouped together and the combined effects 
computed simultaneously. 
By way of review, we first describe how the method works in an 
easier setting: binary trees. 

Suppose $T$ is a binary tree with vertex set $V$ of
size $n$, 
 $f: V \to \reals$ is given, and we want to evaluate
 the sum
$$ F(v) = \sum_{w \in V\setminus \{v \}} K(v,w) f(w),$$
at every $v \in V$ where $K(v,w) = a^{\rho_T(v,w)}$ and $\rho_T$ is the 
path distance in the tree. There are $n$ inputs and $n$ outputs 
and each input affects the evaluation of every output, so 
naively it seems that $n^2$ operations are 
required.
 However, all $n$ values of $F$ can be computed in $O(n)$ 
steps as follows.   Choose a root $v_0 \in V$ and for any $v \in V$ 
let $D(v) \subset V$ be the  vertices that are separated from the root
by $v$, not including $v$. (i.e., its descendants). 
Let $\tilde D(v) = V \setminus (\{v \} \cup D(v)$.
Let
$$ F_1(v) = \sum_{w \in D(v)} K(v,w) f(w), \quad 
 F_2(v) = \sum_{w \in \tilde D(v)} K(v,w) f(w).$$
We can compute each of these functions in one pass through the tree. 
For $F_1$ start by setting $F_1(v) =0$ for each leaf of $T$ and 
proceed from the leaves to the root by setting
$$ F_1(v) = a \sum_{w \in C(v)} F_1(w),$$
where the sum is over the children of $v$. This is called 
the ``up-pass'' since we start at the leaves and work 
towards the root.
Next, we transfer values from $F_1$ to $F_2$ using 
an ``across-pass''. For each vertex where $F_1$ has 
already been evaluated by the up-pass, add 
 $a^2 \cdot (f(v) +F_1(v))$ to $F_2(w)$, for each sibling $w \in S(v)$ of
$v$ ($w \ne v$ is a sibling of $v$ if it has the same parent as $v$).
Lastly,   we compute $F_2$ using a ``down-pass'', by 
 setting $F_2(v) =0$  when $v$ is the root (which has no 
sibling, so was not affected by the across-pass), 
and in general if  $F_2(v)$ has already been computed,
 then for each of its children $w$ we set 
$$ F_2(w) = F_2(w) +a (F_2(v) +  f(v)). $$
We get the desired output by noting
$$ F(v) = (F_1(v) + F_2(v)),$$
for every $v \in V$ (we can evaluate vertices in any order).
Thus $F$ has been evaluated at all $n$ points in $O(n)$ steps.

The same method works more generally. If we are   given 
a rooted tree $T$ with vertex set $V$ we turn it into a 
directed graph $G$ by taking two copies  $V_1, V_2$ of $V$, point all 
edges towards the root in $V_1$ and  away from the root 
in $V_2$ and connect each vertex in $V_1$  to 
 the copies of its siblings in $V_2$. 
Assume we have linear space $X_v^1, X_v^2$
for each vertex $v  \in V$  and a linear map  from  each 
$X_v^1$ to its parent and from each $X_v^2$ to each of its 
children.
Assume we also have an ``across map'' $A_v: X_v^1 \to X^2_v$ . We define a 
linear map $L(w,v): X_w \to X_v$ by composing maps along the
path from $w$ to $v$. Given the $n$ values $x_v \in X_v^1$, 
$v \in V$ the method above   evaluates all $n$ values of 
$$ F(x_v) = \sum_{w \in V} L(w,v) x_w $$
in only $O(n)$ steps. 

 In the previous example, the linear spaces
were one dimensional and the edge maps were multiplication by 
$a$.  For our application to computing a Beurling transform, the
tree will be the tree of dyadic Whitney boxes that intersect the 
support of the dilatation $\mu$. To each Whitney box,  $Q$, we will associate 
a finite set of regions $\{ W_j\}$; 
each will be either a disk or a disk complement (including $\infty$).
The linear spaces will be spaces of analytic 
functions on these regions.  We will actually consider  two situations: 
an infinite dimensional  ideal model
 and a finite dimensional approximation
that we actually compute.

In the idealized version we consider the 
space $X_Q$  of all analytic functions  on  a region $W$.
If $W$ is a disk then every such function has a 
power series $\sum_{k=0}^\infty a_n (z-a)^k$ converging in the 
disk and for the disk complements there is a Laurent 
series $\sum_{k=0}^\infty a_n (z-a)^{-k}$.  
The finite dimensional version of these spaces are  
the spaces $X^p_Q$. These consist 
 of $p$ term power series  $\sum_{k=0}^p a_n (z-a)^k$ (for disks) 
or Laurent series  $\sum_{k=0}^p a_n (z-a)^{-k}$ (for disk complements).
There is an obvious  truncation map $T: X_Q \to X_Q^p$ and an inclusion 
map $I: X_Q^p \to X_Q$.

Given  two regions, one of which is contained 
 the other, we can restrict a function from the larger region to the 
smaller.
This defines restriction maps  $R$ between the infinite dimensional spaces 
$X_Q$.  For the finite dimensional analog, we
define maps $R^p$ between the  spaces 
$X_Q^p$ by restricting and then truncating the series expansion.
We will see how to compute the Beurling 
transform exactly using the restriction maps, and then 
check how much error is introduced when we replace these
by the finite dimensional restriction/truncation maps.

When we allow infinite expansions, then restricting an 
analytic function to a subdomain introduces no errors
and  the method described 
above allows us to compute series expansion for the Beurling 
transform in time $O(n)$ with no errors (except for filling
in the initial values of the arrays). Similarly, if
 we restrict a power series to a smaller disk, there
is no error introduced,
since  the restriction of a degree $p$ polynomial is still a degree $p$ polynomial.
However, if we change the center of a Laurent expansion, 
then a finite expansion may become infinite and truncating 
to $p$ terms causes an error (depending on $p$ and the geometry 
of the regions).  In this case, performing the restriction-truncation 
along a series of nested regions might not give the same 
result as restricting to the smallest domain is single step.
Because of this, we have to estimate the errors at each step 
and show the total accumulated error along the whole path is 
still small. 

We will now introduce the elements needed to apply these general 
ideas to the specific problem of computing $\partial T \mu$.

We start with a partial  $\epsilon$-representation
of a polygonal domain. As in Section \ref{iterate}, we use this 
to construct a dilatation $\mu $ that is a sum  
$\sum \mu_k$ of terms, each of  which are 
supported in small squares  whose union covers a neighborhood, $N_s$, of the 
boundary of our decomposition.  We assume that $\mu $ is defined by 
reflection on the lower half-plane, so that solutions of the Beltrami 
equation will be real on the real line. Each $\mu_n$ is a polynomial 
in $x$ and $y$  of degree at most $O(n)$ restricted to a small square.
The terms of this polynomial are of the form $z^k x^a y^b = (x+iy)^k x^a y^b$
with $0\leq k \leq p$ and $0 \leq a,b \leq C$ where $p$ grows
depending of the desired accuracy, but $C$ is fixed, depending 
only on the degrees of the piecewise polynomials used in 
our partition of unity associated to the decomposition ${\cal W}$
of our representation. Thus there are only $O(n)$ terms to 
consider, not $O(n^2)$ as would be the case if all powers of 
$x$ and $y$ less than $n$ had to be considered.

We take as our tree the collection of all Whitney squares in the 
upper half-plane that hit the support of $\mu$, i.e., which hit
$N_s$.  There are $O(n)$ such, since there are $O(n)$  boundary 
components in our decomposition and each only hits a bounded number 
of Whitney boxes (for arches we only need to cover the edges of the 
arch, not the interior).  The adjacency relation is the usual one; 
$Q$ is a child of $Q^*$ if the base of $Q^*$ contains the base of $Q$ 
and $Q$ is maximal with this property.  A given Whitney square can 
have zero, one or two children.  Those with no children are called 
``leaves'' of the tree. Often a child is half the size of its parent 
and the top edge of the child is half the bottom edge of the parent, 
but because of arches, there are some cases where a child is much 
smaller than its parent.
A neighbor of a dyadic Whitney square is a distinct  dyadic Whitney 
square of the same size that touches along the boundary. 
The terms ``descendant'' and ``ancestor'' have the usual meanings 
for a rooted tree.

For any Whitney box $Q$ in $\uhp$ with base interval $I$ (its vertical 
projection on $\reals$) let $c_Q$ denote the center of this base
and let $c_Q^j$, $j=1,\dots, 8$ be $8$ equally spaced points in 
$I$ (including the right, but not the left endpoint of I).
 Let $A_Q = \{z: |z-c_Q| \geq \lambda |I| \}$
where we choose $ \frac 12 \sqrt{5} < \lambda < \frac 54$,
and let $D_Q^j = \{ z: |z-c_Q^j| \leq \frac 14 |I| \}$. These 
will be called the type I and type II regions associated to $Q$
respectively.  See Figure \ref{types}.  The 
number $\lambda $ is chosen in this range so that the type I 
region does not intersect $Q$, but it does contain the 
type II regions of any $Q'$ that is the same size as 
$Q$, but not adjacent to it. See Figure \ref{type-convert}.
A series expansion in terms of $(z-c_Q)^{-1}$ or
$(z-c_Q^j)$ will be called type I and type II expansions 
respectively.

\begin{figure}[htbp]
\centerline{ 
	\includegraphics[height=2.0in]{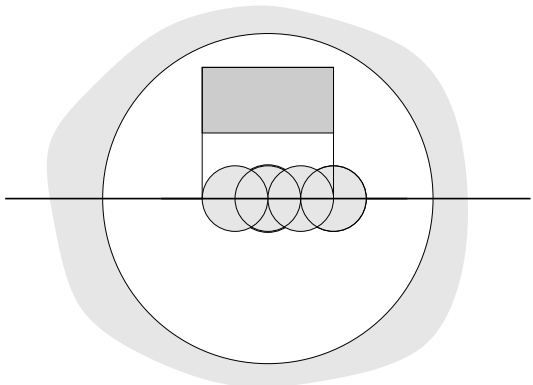}
	}
\caption{ \label{types}  
A Whitney box and its type I and type II regions.
}
\end{figure}

\begin{figure}[htbp]
\centerline{ 
	\includegraphics[height=1.5in]{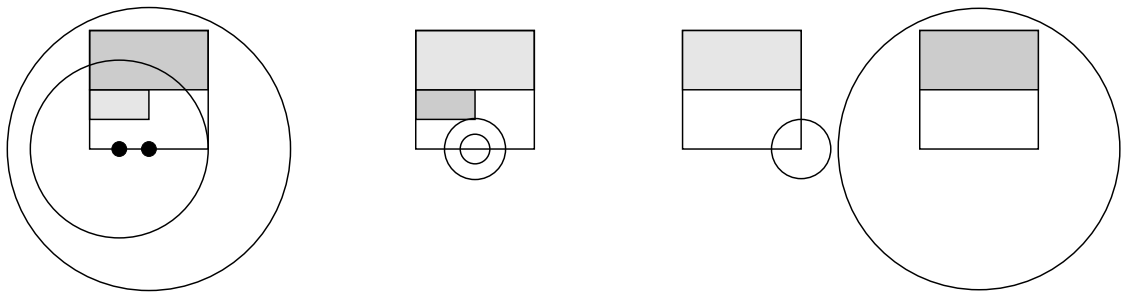}
	}
\caption{ \label{type-convert}  
The three kinds of conversions: 
multipole-to-multipole, local-to-local and   multipole-to-local.
  In each case the first 
Whitney square is shaded lighter than the second.
}
\end{figure}

Given a Whitney box $Q$ we can restrict $\mu$ to $Q$ and 
compute type I and type II expansions for  $\partial T \mu|_Q$.
Since $\mu$ is a piecewise polynomial of degree $O(p)$, and there
is an explicit formula for the expansion of each monomial, 
this can be done in time $O(p \log p)$ by the  remarks 
in Appendix \ref{fast-power-series}.

Given a Whitney box $Q$ and its parent $Q^*$, the type I region 
for $Q$ contains that for $Q^*$ and so we can take the analytic 
function $f$ defined by the type I expansion for $Q$ and compute its
Laurent expansion in the type I region for $Q^*$. Then truncate this
(infinite) series to get a type I expansion for $Q^*$. This is 
called a I-to-I conversion or a multipole-to-multipole conversion.
See Figure \ref{type-convert}.
This introduces an error of $\epsilon M_f$, where
$\epsilon = \lambda^p$, where $\lambda < 1$ 
 and $M_f$ is the maximum of $f$ on the type 
I region of $Q$.

Similarly, we can take a type II expansion for $Q$ and restrict it to 
one of the two type II regions for a child of $Q$ whose center 
agrees with the first center or is immediately to the left 
of it. Changing the center of the expansion just gives another 
degree $p$ polynomial and there is no error introduced, 
i.e., $R = R^p$. This is a local-to-local conversion.

Finally, the type I region of  a Whitney box $Q$  contains the 
type II regions of a box $Q'$ of the same size if $Q'$ is 
in $3Q^*$ but not in $3Q$ (here $Q^*$ denotes the parent of $Q$).
Therefore, we can restrict the type I expansion of $Q$ to 
a type II region of $Q'$ and do a I-to-II conversion (or multipole-to-local
conversion), with an 
error of $\epsilon M_f$, as above.

For a given box $Q$,  the associated 
 regions cover the whole upper half plane, except for 
a region of bounded hyperbolic diameter around $Q$.

  We think of the type I and type II 
expansions associated to each $Q$ as defining  two arrays 
indexed by the Whitney boxes.  We next describe 
how to initialize  and update these arrays:

\begin{description}
\item [Initialize Type I  array] For each $Q$ compute the initial type I expansion.
\item [Initialize Type II  array] For each $Q$ compute the initial type II expansion 
                              for each type II disk. 
\item [Modify Type II array]  Compute type II expansions for two neighbors and 
                           add to their initial expansions.  Every type II 
                            expansion now has contributions from at most  three 
                            boxes (its parent and the parent's neighbors)     
\item [Perform  the up-pass]  Starting with leaves of the tree, do I-to-I conversion
                          of the current type I expansion and add it to the type I expansion 
                          of the parent. Continue until we reach the root.
\item [Perform the across-pass] For each square $Q$, do  I-to-II conversions 
                          taking the current type I expansion and obtaining 
                          type II conversions for regions corresponding to 
                          centers in $3 I^* \setminus 3 I$  where $I$ 
                           is the base of $Q$ and $I^*$ is the base of $Q$'s parent. 
\item [Perform the down-pass] Starting with root square, do  II-to-II conversions, 
                          taking each type II expansion and restricting it to 
                          the two type II expansions of the children. Continue 
                          downward until we reach the leaves of the tree.
\end{description}

This is clearly $O(n)$ steps and when we are finished, the type I expansion of 
a square $Q$ contains the contribution of $Q$ and every descendant of $Q$ and 
the type II expansions of $Q$ contain the contributions of 
 every  square that is  not a (strict) descendant of $Q$  or its 
two neighbors.

The third step (Modify the type II array) is necessary because the tree structure 
on Whitney squares does not completely reflect their actual placement
in $\uhp$; two squares 
that are far apart in the tree could be adjacent in $\uhp$. The type I region 
of a square does not contain the type II disks of its neighbors (they are too close),
so we need this special step to pass the information to these regions (for other 
regions it is passed in the  across-step).

Now suppose $D$ is an empty  piece of our decomposition and $Q$ is the Whitney
box containing $D$. If $D$ is a disk, it is contained in 
a type II region of a grandparent of $Q$ and is contained in the 
type I region of all grandchildren of $Q$ and its two neighbors.
Therefore we can do series conversions and compute the expansion 
in $D$ due to these expansions.  There are only a finite 
number of Whitney boxes whose contributions have not been 
accounted for and all these lie within a uniformly bounded 
distance of $Q$.  For each piece of $\mu$ supported in one 
of these boxes, we compute the contribution to $D$ directly.


Given any Whitney boxes $Q'$ and $Q$ there is a  path in our 
directed graph that starts from an initial expansion for $Q$ 
and goes to a terminal expansion for $Q'$ or one of its neighbors.
 To see this we consider
several cases. 
\begin{enumerate}
 \item If $Q = Q'$ there is nothing to do.
 \item If $Q=Q_2$ is an ancestor of $Q_1$ then the all ``down'' path works.
 \item If $Q_3$ is a neighbor of square $Q_2$ in (1), then start with 
          the special ``Modify type II'' step and follow by all downs.
       \item If $Q_4$ is a descendant of a case (2) square, then follow
``up'' paths until we hit  a child of $Q_3$ and then use a ``across'' step 
  to bring us to an ancestor of $Q'$ (this works because by the definition of 
   the across step).
 \item If  $Q$ is  a descendant of $Q'$, then use all ``up'''s.
  \item The only remaining case is that $Q$ is a neighbor of $Q'$ or 
   a descendant of  a neighbor. Using an all ``up'' path works.
\end{enumerate}


If $D$ is an empty piece of our decomposition, we want to 
show that the desired expansion for $\partial T \mu$ can 
be computed using a bounded number or  type I and
type II expansions, plus a bounded number of direct expansions 
of nearby squares. Fix such a $D$ and suppose $Q$ is any 
Whitney box, then $D $ is a subset of one of the type I or II regions associated 
to $Q$ or one of its neighbors, unless $Q$ is within a uniformly bounded hyperbolic distance 
$M$ of $D$. If $D$ is a subset of one of these regions, then we can 
convert the series expansion on the region to one on $D$, the 
conversion being one of three types. First, we might have 
to convert  a power series in $(z-a)$ to 
one in $(z-b)$; this happens when $D$ is a disk or arch  contained in a 
disk and  involves no loss of accuracy. Second, converting an 
expansion in $(z-a)^{-1}$ to one in $(z-b)^{-1}$; this happens when 
$D$ is an arch that contains $Q$ in its bounded complementary 
component and there is a loss of accuracy (described in Lemma \ref{bound-error}).
Finally, in all other cases we must convert 
an expansion in $(z-a)^{-1}$ to one in $(z-b)$;
this also involves a small loss of accuracy.

\begin{lemma} \label{bound-error}
Suppose $|a| \leq \frac 14$, $|b| \leq \frac 12$ and 
 $f$ is analytic on $\{z:|z-a|>1 \}$ and $|f| \leq 1$ there.
Also assume $f(z) z^3$ is bounded as $ z \to \infty$.
Suppose $A= \{z:|z-b| > 2\}$ and let 
$f(z) = \sum_{j=0}^\infty a_j (z-b)^{-j} $ be the Laurent expansion 
for $f$ in $A$ and let $g(z)= \sum_{j=0}^p a_j (z-b)^{-j} $. 
Then there is $0 < \lambda < 1$ so that for $\epsilon = \lambda^p$, 
 \begin{enumerate}
\item $a_0=a_1=a_2 =0$,
\item  $|g(z)| \leq (1+\epsilon) |z|^{-3} \leq |z|^{5/2} ,$ if $p$ is
         large enough.
\item $|f(z) - g(z)| \leq  \epsilon |z|^{-3} $.
\end{enumerate}
\end{lemma}

The proof is just the standard estimates for Taylor series and left
to the reader. If we start with an expansion $f_0$
on the type I region of a box $Q$ of size $1$ 
and then restrict the expansion to get 
an expansion $f_1$ for the 
type I region of its parent, then repeat this over and over, we accumulate
an error each time. Suppose $f_k$ is $k$th expansion on the 
$k$th region $A_k$. Then $\sup_{A_k} |f_k| \leq |z|^{-5/2}$,
which means the maximum error between $f_k(z)$ and $f_{k+1}(z)$  on $A_{k+1}$
is bounded by $\epsilon |z|^{-5/2}$.  So the total error that is ever 
possible inside $A_N$ is 
\begin{eqnarray*}
&\leq&  \sum_{k=1}^N |f_k(z) - f_{k+1}(z)|\\
& \leq &
\sum_{k=1}^N  \epsilon \diam(\partial A_k)^{-5/2} (\diam(\partial A_n)/
           \diam(\partial A_k))^{-3}\\
&\leq &
O(\epsilon) \diam(\partial A_N)^{-3}  \sum_{k=1}^N  \diam(\partial A_k)^{1/2}. 
\end{eqnarray*}
Since the regions grow by at least a factor of two at each stage, the final 
sum is dominated by its final term, and so the total error is 
less than $O(\epsilon \diam(\partial A_N)^{-2.5})$.

Summarizing this argument gives:

\begin{lemma} \label{error-on-path}
Suppose $f_0(z)= \sum_{k=3}^p a_k (z-c_Q)^{-k}$ is a type I expansion 
associated to a dyadic Whitney square $Q_0$ and $|f_0|$ is bounded 
by $M$ on the type I region of $Q_0$. Suppose $ Q_1, \dots
, Q_N$ are ancestors of $Q$ and  $f_k$ is the result of
applying a   I-to-I conversion to $f_{k-1}  $ for $k=1, \dots N$. 
Then on the type I region of $Q_N$, $|f_0-f_N| \leq O(\epsilon M 
(\diam(Q_0)/\diam(Q_N))^{2.5} )$ with  $\epsilon = \lambda^p$ and 
a constant that is independent of $N$.
\end{lemma}

\begin{cor}
Suppose $Q_0$ is a Whitney square and for each $Q$ that is a descendant 
of $Q_0$, let $f_Q$ be the type I expansion of $\partial T \mu$ for $\mu$ 
restricted to $Q$. Let $f_0$ be the type I expansion for $Q_0$ obtained
by running the up-pass over all descendants of $Q$ with initial data 
$\{ f_Q\}$. Let $F_0= \sum_{Q \in D(Q_0)} f_Q$ be the exact sum of these
initial expansion restricted to the type I region of $Q_0$. Then 
$$ |f_0(z) - F_0(z) |    
                =O( \epsilon \frac{\diam(Q_0)^2}{|z-c_{Q_0}|^3}
		\| \mu \|_\infty  ).$$
\end{cor}

\begin{proof}
The maximum of $f_Q$ on the type I region of $Q$ is clearly $O(\|\mu \|_\infty /
\diam(Q))$. Therefore, by the Lemma \ref{error-on-path} the error 
of applying I-to-I conversions until we reach $Q_0$ is 
$$O(\epsilon \|\mu\|_\infty 
\diam(Q)^{1.5} /\diam(Q_0)^{2.5}) .$$
  There are at most $2^k$ descendants
of $Q_0$ with $\diam(Q) = 2^{-k} \diam(Q_0)$, so the total error of all 
of these is 
$$O( \epsilon \|\mu\|_\infty 
\diam(Q)^{.5} /\diam(Q_0)^{1.5}) = O(\epsilon \|\mu\|_\infty 
 2^{-k/2} /\diam(Q_0)).$$
 We now sum $k=1,2, \dots$ and see the 
total error over all descendants of $Q_0$ is at most 
$O(\epsilon \|\mu\|_\infty /\diam(Q_0))$.   The error is an analytic function 
on the type I region of $Q_0$ that  decays
like $|z-c_{Q_0}|^{-3}$ near infinity (since it is a difference of functions 
which do), and this gives the estimate in the corollary.
\end{proof}

\begin{cor}
Let $f$ be a type II expansion for a Whitney square $Q$ that is obtained
by starting with expansions of $\partial T \mu$ and applying
the up-pass, across-pass and down-pass. Then the total error
between $f$ and simply adding all the initial expansions that 
contribute to $f$ is $O(\epsilon \|\mu\|_\infty /\diam(Q))$.
\end{cor}

\begin{proof}
The contribution of the ancestors of $Q$ are through II-to-II 
conversions, which introduce no error.  The neighbors of ancestors
contribute are direct computation of a type II expansion, followed 
by II-to-II expansions, so also contribute no error. Every other 
contribution comes from a sequence of I-to-I conversions (the up-pass), 
followed by a I-to-II conversion (the across-pass) and then II-to-II 
conversions (the down-pass). Fix a square $Q_j$ to which the across-pass
is applied. The errors due to all the descendants of $Q_j$ is bounded 
by $O(\epsilon \|\mu\|_\infty /\diam(Q_j)$. If $\diam(Q_j) = 2^j \diam(Q)$, 
then this is $O( 2^{-j} \epsilon \|\mu\|_\infty /\diam(Q))$.  The across-pass
adds an error with the same bound and the following down-pass adds no 
new error. Thus the total contribution of $Q_j$ to the error is 
$O( 2^{-j} \epsilon \|\mu\|_\infty /\diam(Q))$. There are only a bounded 
number (at most 4) of such $Q_j$'s of a given size, so summing over all possible 
$j$'s shows the total error is at most $O( \epsilon \|\mu\|_\infty /\diam(Q))$.
\end{proof}

Now every expansion that contributes to the final expansion on an 
empty region is either directly computed from the data, or comes 
from an expansion created by the up-pass, across-pass and down-pass.
By the corollary, the up-pass creates a small error, and we already 
know the across-pass creates a small error and the down-pass 
creates no errors.  Thus the total error  comes from a 
uniformly bounded number of terms, each of which has 
error bounded by $O(\epsilon \|\mu\|_\infty /\diam(Q))$.

This completes the proof of Lemma \ref{Newton-radius} and hence of the theorem.


\appendix 

\section{Background in analysis} \label{background}

In this section will give various definitions and results 
from analysis. It is intended as a review or (very brief)
introduction to ideas used in the paper.

\subsection{M{\"o}bius transformations}
A linear fractional (or M{\"o}bius)  transformation is
a map of the form $z \to (a z + b )/ ( c z + d)$. This is a 1-1, onto,
holomorphic map of the Riemann sphere $\sphere =
\complexes \cup \{ \infty \}$ to
itself.  Such maps form a group under composition and
are well known to map circles to circles (if we
count straight lines as circles that pass through $\infty$).
M{\"o}bius transforms are conformal, so they preserve angles.

The non-identity M{\"o}bius transformations
 are divided into three classes. Parabolic transformations
have a single fixed point on $\sphere$ and are conjugate to the
translation map $z \to z+1$.  Elliptic maps have two fixed points and
are conjugate to the rotation $ z \to \lambda z $ for some $|\lambda|=1$.
The loxodromic transformations also have two fixed points and are
conjugate to $ z \to  \lambda z$ for some $|\lambda  | < 1$. If, in
addition, $\lambda$ is real, then the map is called hyperbolic.

Given two sets of three distinct points $\{ z_1, z_2, z_3\}$ and
$\{ w_1, w_2, w_3 \}$ there is a unique M{\"o}bius transformation that
sends $w_k \to z_k$ for $k=1,2,3$. 
A M{\"o}bius transformation sends the unit disk 1-1, onto itself iff it is
of the form
$$ z \to \lambda \frac {z-a}{1-\bar a z },$$
for some $a \in \disk$ and $|\lambda| =1$. Any loxodromic
transformation of this form must actually be hyperbolic.

\subsection{Conformal maps:}
A conformal mapping is a diffeomorphism that preserves angles.
We will only consider orientation preserving maps here, in 
which case a conformal map between planar domains is the same as
a 1-1 holomorphic mapping.
The Riemann mapping theorem states that given any simply 
connected, proper subdomain $\Omega$
 of the plane there is a one-to-one, 
onto holomorphic map $f:  \disk = \{ z: |z| < 1\} \to \Omega$.
If $\partial \Omega$ is locally connected then this map extends
continuously to the boundary.  Moreover, we can always take $f(0)$ 
to be any given point of $\Omega$ and $f'(0)$ to have any argument 
we want.  If $\partial \Omega$ is  Jordan curve we can also 
normalize by making any three points on the unit circle map to 
any three points on $\partial \Omega$ (as long as they have the 
same orientation).

The Schwarz-Christoffel formula gives a formula for the
Riemann map of the disk onto a polygonal region $\Omega$:
if the interior angles of $P$  are $ \boldalpha \pi=
\{ \alpha_1 \pi, \dots, \alpha_n \pi\}$, then
$$ f(z) = A + C \int^z \prod_{k=1}^n
        ( 1 - \frac w{ z_k})^{\alpha_k-1}  d w.$$
 See e.g.,  \cite{DT-book}, \cite{Nehari}, \cite{DT99}.
On the half-plane the formula is 
$$ f(z) = A + C \int \prod_{k=1}^{n-1} (w-z_k)^{\alpha_k-1} dw .$$
The formula was discovered independently by  
Christoffel in 1867 \cite{Chr67}  and Schwarz in 1869
\cite{Sch90}, \cite{Sch69a}. For other references 
and a brief history see Section 1.2 of \cite{DT-book}.
It is also possible to formulate it with other base 
domains, such as an infinite strip (see \cite{DT-book}).
See \cite{Hu-98} for a version involving doubly 
connected polygonal regions.
There are also versions for domains other than polygons, 
e.g., circular arc polygons as in \cite{Howell}, \cite{Nehari}.
In this case,  we get a simple formula for  the Schwarzian derivative 
of the conformal map, but it  involves unknown parameters with no
obvious geometric interpretation.

One particular case of the Schwarz-Christoffel formula 
we need (see Section \ref{epsilon-reps}) is for the map 
of the upper half-plane to a triangle with one vertex at 
$\infty$ (i.e., a region bounded by two half-infinite rays
and a finite segment).
 Since there
are only two finite vertices, there is no parameter problem 
to solve; they can be chosen to be any 
two points   we  want, say $\pm 1$, so the formula becomes
\begin{eqnarray} \label{special-SC}
 f(z) = A + C \int (w-1)^{\alpha_1-1} (w+1)^{\alpha_2 -1}dw .
\end{eqnarray}
Note that using the general form of the binomial theorem,
$$ (1+z)^p = \sum_{k=0}^\infty \frac {p (p-1) \cdots (p-k+1)}{k!} z^k,$$
 we
can easily compute power series for these functions in disks
away from the singularities.

The problem in applying the Schwarz-Christoffel formula 
with $n > 3$ vertices  is
that the points $\z=\{ z_1, \dots z_n\}$
 are unknown to us until we know
$f$, so the formula seems circular.
 However, there are
various iterative methods for finding the points $\z$
starting from an initial guess (often taken to be $n$
uniformly distributed points on $\circle$), e.g., see
\cite{DT-book}, \cite{Kythe}.
For example, the method of Davis \cite{Davis} takes an $n$-tuple  of
points $\{ z_1, \dots, z_n\}$ on the unit circle, computes 
an image polygon using the Schwarz-Christoffel formula 
with these parameters (and the known angles) and compares 
the side lengths of this polygon with the desired polygon.
If a side is too short, the corresponding parameter values are moved apart
in the next iteration and conversely.  More precisely, 
if $\{ z_1^k, \dots , z_n^k\}$ is the current guess, and the 
image polygon has vertices $\{ v_1^k, \dots , v_n^k\}$
we define the next set of parameter guesses as 
$$ |z_k^{j+1} - z_{j-1}^{k+1}| = k |z_j^{k} - z_{j-1}^{k}|
 \frac {|v_j-v_{j-1}|}{|v^k_j - v^k_{j-1}|},$$
for $j =0, \dots, n$ where $k$ is a normalizing constant and 
$\v=\{ v_0, \dots , v_n\}$ are the vertices of the target polygon.
The method works in practice in many cases 
but is not known to converge.

 Davis' method is used  in \cite{BT}  by Banjai and Trefethen
to give a $O(n)$ method for finding the  prevertices that
is practical for tens of thousands of vertices (the bound,
however is an average case analysis, not a uniform estimate
for all polygons). 
 Many other methods exist for computing
conformal mappings including integral equation methods that 
are very effective. 
For example, in \cite{O'D-R89} Rokhlin and O'Donnell compute 
conformal maps using the fast 
multipole method to solve an integral equation arising 
from the Kerzman-Stein formula.
Marshall has a fast method called ``zipper''  based on iterating simple maps
(see \cite{MR2008}).
 For  surveys of different numerical
conformal mapping techniques see, e.g., 
\cite{DeLillo},
\cite{Gaier},
\cite{Henrici},
\cite{Ivanov},
\cite{Conf-93},
\cite{Conf-86},
\cite{vonKoppenfels},
\cite{Wegmann05}. 

A circle packing of a domain is a collection of disjoint
(except for tangencies) disks in the domain.  
The Andreev-Thurston theorem say that given such a packing 
one can find a packing of the disk with the same tangency 
relations and that if the circles are small enough the
mapping between the packings is an approximation to the 
Riemann map 
\cite{Rodin-Sullivan},
 \cite{He-Schramm},
\cite{He-Schramm98},
 \cite{Stepehson99},
\cite{Stephenson02},
 \cite{Stephenson03}.
A polynomial time algorithm for computing conformal
mappings is described in \cite{Smith} using a 
polynomial time algorithm for finding circle
packings, but no details are provided about how to 
choose a packing of a domain in time independent of 
the geometry.  A polynomial time algorithm for circle 
packings is also described in \cite{Bojan93}, 
\cite{Bojan97}. Software for computing conformal maps 
via circle packings is available from Ken Stephenson
\cite{CirclePack}.

An alternate approach to the computational complexity of 
conformal mapping is considered by Binder, Braverman and
Yampolsky in \cite{BBY-2005}. They 
consider domains with complicated boundaries (such a fractals)
and assume an oracle is given that will decide whether 
a given point is within $\epsilon$ of the boundary. The complexity 
of a domain is determined by how quickly such an oracle 
works (as a function of $\epsilon$).
Given such an oracle they show the Riemann mapping at a
point can be computed to accuracy $\epsilon^a$ using 
$b \log^2 \epsilon $ space and $\epsilon^{-c}$ time 
for some positive constants $a,b,c$ (their method solves 
a Dirichlet problem using a random walk on an $\epsilon$-grid 
stopped by the oracle).  This is related to 
other notions of the computability  of conformal maps,
such as constructibility in the sense of  Brouwder and 
Errett Bishop, e.g., see 
\cite{Bishop-Bridges}, 
\cite{Cheng-73},
\cite{Herting-99},
\cite{Zhou-96}.


\subsection{Hyperbolic geometry}
On the unit ball, $\ball$, the hyperbolic metric is given by 
$$ |d \rho | = \frac {2 |dz|}{1-|z|^2}. $$ 
More explicitly, it can be written as 
$$ \rho(z_1,z_2)= \frac 12  \log \frac {1+|\sigma(z_2,z_1)|}
{1-|\sigma(z_2,z_1)|)},$$
where $\sigma(z_1,z_2) = (z_2-z_1)/(1-\bar z_1 z_2)$.
In the upper half space model, $ \reals_+^3$, it is 
given by 
$$ |d \rho| = \frac {|dz|}{\dist(z, \partial \reals^2)}.$$
For the ball and upper half-space models, hyperbolic 
geodesics are circular arcs that are orthogonal to the 
boundary (also vertical lines in the case of the half-space).

Hyperbolic area on $\disk$ is defined as $dxdy/(1-|z|^2)^2$ 
and on $\uhp$ as $dxdy /y^2$. The disk and half-plane have 
infinite area; indeed, each Whitney square (Section \ref{cover}) has area
$\simeq 1$ and the area of a hyperbolic $r$-ball is 
grows exponentially with $r$.  A striking feature of 
hyperbolic geometry is that the area of a triangle 
is determined by its three angles, namely 
$\rm{area} (T) = \pi -(\alpha+\beta+ \gamma)$. Thus an 
ideal triangle (one with all three vertices on the boundary)
has area $\pi$ and all other triangles have smaller area.

Simply connected, proper subdomains of the plane inherit a
hyperbolic metric from the unit disk via the Riemann map. 
If $\varphi: \disk \to \Omega$ is conformal and $ w=\varphi(z)$
then $\rho_\Omega(w_1,w_2) = \rho_\disk(z_1,z_2)$ defines the 
hyperbolic metric on $\Omega$ and is independent of the 
particular choice of $\varphi$.  It is often convenient to 
estimate $\rho_\Omega$ in terms of the more geometric 
``quasi-hyperbolic'' metric on $\Omega$ that is 
defined as 
$$ \tilde \rho(w_1,w_2) = \inf \int_{w_1}^{w_2} \frac 
{|dw|} 
{\dist(w,\partial \Omega)},$$
where the infimum is over all arcs in $\Omega$ joining 
$w_1$ to $w_2$.

\begin{thm} [Koebe's distortion theorem]
Suppose $\varphi: \disk \to \Omega$ is a conformal map of the 
disk to a simply connected domain. Then for all $z \in \disk$,
$$ \frac 14 |\varphi'(z)|(1-|z|^2) \leq \dist(\varphi(z), \partial \Omega)
 \leq |\varphi'(z)| (1-|z|^2).$$
\end{thm}

  Because of Koebe's distortion theorem we have 
\begin{eqnarray} \label{hyp-est}
  d \rho_\Omega \leq d \tilde \rho_\Omega \leq 4 d \rho_\Omega.
\end{eqnarray}

\begin{lemma} \label{dist-to-bdy}
Suppose $\Omega$ is simply connected and $z \in \Omega$ satisfies
$\dist(z, \partial \Omega) =1$. 
If either $|w-z| > R$ or $\dist(w,\partial 
\Omega) < 1/R$, then $\rho_\Omega(z,w)  \geq  \frac 14 \log R $.
\end{lemma}

\begin{proof}
Recall that $\rho_\Omega \geq \frac 14 \tilde \rho$ 
where $\tilde \rho$ denotes  the  quasi-hyperbolic 
metric on $\Omega$, defined by 
$d \tilde \rho = |dz|/\dist(z,\partial \Omega)$.
 If $\dist(w,\partial \Omega) \leq 1/R$ then 
$\tilde \rho(z,w) \geq  \int_{1/R}^1 \frac 1t dt = \log R$.
If $|w-z| \geq R$, then $\tilde \rho(z,w) 
\geq \int_1^R \frac 1t dt \geq \log R$.
\end{proof}

%

M{\"o}bius transformations are the only 1-1, onto holomorphic  maps of the 
Riemann sphere to itself.  In the complex plane we write
these maps as 
$z \to (az+b)/(cz+d)$.  Every such map extends uniquely to be a hyperbolic 
isometry of the upper half-space, $\reals^3_+$, and every orientation 
preserving  isometry on $\reals^3_+$  is 
of this form.

\subsection{Conformal modulus:}
Suppose $\Gamma$ is a family of locally rectifiable paths
in a  planar domain $\Omega$ and $\rho$ is a non-negative Borel
function  on $\Omega$. We say $\rho$ is admissible for $\Gamma$ if
$$ \ell(\Gamma) = \inf_{\gamma \in \Gamma} \int_\gamma \rho  d s \geq 1,$$
and define the modulus of $\Gamma$ as
$$ \Mod(\Gamma) = \inf_\rho  \int_\Omega \rho^2 dxdy,$$
where the infimum is over all admissible $\rho$ for $\Gamma$.
The reciprocal of the modulus is called the extremal length 
of the path family.
These are important conformal invariants whose basic properties
are discussed in many  sources such 
 \cite{Ahlfors-QCbook}. A simple result we used in this paper is: 

\begin{lemma} \label{EL-diam}
Suppose $\gamma \subset \disk$ is a Jordan arc with one 
endpoint on $\partial \disk$ and   has diameter $\leq \frac 14$.
 Then the modulus 
of the path family separating $\gamma$ from $B(0, \frac 12)$ 
in $\disk$ is $\pi /\log(\diam(\gamma)) + O(1)$.
\end{lemma}

 A generalized quadrilateral
 $Q$ is a Jordan domain in the plane with four specified
boundary points $x_1,x_2,x_3,x_4$ (in counterclockwise order).
We define the modulus of $Q$,
$M_Q(x_1,x_2,x_3,x_4)$ (or just $M_Q$ or $M(Q)$ if the
points are clear from context),
as the modulus of the path family in $Q$ that
connects the arc $(x_1, x_2)$ to the arc $(x_3, x_4)$.
This is also the unique positive real number $M$
such that $Q$ can be conformally mapped to a $1 \times M$
rectangle with the arcs $(x_1, x_2)$, $(x_3, x_4)$ mapping to the opposite
sides of length $1$.
Given a  generalized quadrilateral
 $Q$  with four boundary points $x_1,x_2,x_3,x_4$,
the quadrilateral $Q'$ with vertices $x_2, x_3,x_4,x_1$ is called
the reciprocal of $Q$ and it is easy to see
that $\Mod(Q') = 1/\Mod(Q)$. We also call $\Mod(Q')$ the 
``extremal distance'' from $(x_1, x_2)$ to the arc $(x_3, x_4)$
in $Q$.

%

\subsection{Cross ratio:}
Given four distinct points $a,b,c,d$ in the plane we
define their  cross ratio as
$$ \cross(a,b,c,d) = \frac { (d-a)(b-c)}{(c-d)(a-b)}.$$
Note that $\cross(a,b,c,z)$ is the unique M{\"o}bius transformation
that sends $a$ to $0$, $b$ to $1$ and $c$ to $\infty$.
This  makes it clear that
cross ratios are invariant under M{\"o}bius transformations; that
$\cross(a,b,c,d)$ is real valued iff the four points lie on a circle;
and is negative iff in addition the points are labeled in
counterclockwise order on the circle.
If the four points lie on $\circle$, then
since $\cross$ and $M_\disk$ are both invariant under M{\"o}bius
transformations of the disk to itself, each must be a function
of the other in this case. The function is explicitly given as
in infinite product in 
Ahlfors' book \cite{Ahlfors-QCbook}.

\subsection{Quasiconformal mappings}
Quasiconformal mappings are a generalization of conformal mappings
that play an important role in modern analysis and a central
role in the current paper.  There are (at least) three  equivalent
definitions of a $K$-quasiconformal mapping between  planar domains.
Suppose $f:\Omega \to \Omega'$ is a homeomorphism.
We say $f$ is $K$-quasiconformal if any of the following equivalent 
conditions holds: 
\begin{description}
\item [Geometric definition] for any generalized quadrilateral
$Q \subset \Omega$, $\Mod(Q)/K \leq \Mod({f(Q)}) \leq K \Mod (Q)$.
\item [Analytic definition] $f$ is absolutely continuous on
  almost every vertical and horizontal line and the partial
  derivatives of $f$ satisfy $ |f_{\bar z}| \leq k |f_z|$
  where $k = (K-1)/(K+1)$.
\item [Metric definition] For every $x \in \Omega$
  $$\limsup_{r \to 0}  \frac {\max_{y:|x-y|=r} |f(x)-f(y)|}
        {\min_{y:|x-y|=r} |f(x)-f(y)|} \leq K.$$
\end{description}
For a proof of the equivalence of the first two, see \cite{Ahlfors-QCbook}
and for a discussion of the third and a generalization to metric
spaces see \cite{HK98} and its references.
In Euclidean space the equivalence of the three definitions is due to Gehring 
\cite{Gehring-60}, \cite{Gehring-61}, \cite{Gehring-61}.
A composition of a $K_1$-quasiconformal map with a $K_2$-quasiconformal
map is $(K_1K_2)$-quasiconformal.
Thus the distance used in
Theorem \ref{main} satisfies the triangle inequality.

A closed curve in the plane is called a $K$-quasicircle 
if it is the image of a circle under a $K$-quasiconformal 
homeomorphism of the plane and is called a quasicircle if 
it is a $K$-quasicircle for some $K < \infty$.
Quasicircles have a geometric characterization in terms of 
Ahlfors' three point condition: a curve $\gamma$ is a 
quasicircle iff for any two points $x,y \in \gamma$ one of 
the two arcs with endpoints $x,y$ has diameter $\leq M|x-y|$.
Quasicircles need not be differentiable; a famous example of 
a non-differentiable quasicircle is the von Koch 
snowflake.

Recall that $\partial f=f_z =\frac 12( f_x- i f_y)$ and 
$\Dbar f=f_{\bar z} =\frac 12( f_x + i f_y)$.
For a $K$-quasiconformal map the ratio $\mu_f = f_{\bar z}/ f_z$
is a well defined complex function
almost everywhere and satisfies $\| \mu_f\|_\infty
\leq k = (K-1)/(K+1)$.
The function $\mu = \mu_f$ is called the  Beltrami coefficient of 
$f$ and  satisfies the following composition 
laws:
$$ \mu_{f^{-1}} \circ f = - (f_z/\overline{f_z})^2 \mu_f, $$
$$ \mu_{g \circ f}(z)  = (f_z(z)/\bar f_z(z)) \frac {\mu_{g} (f(z)) - \mu_f(z)}
             {1-\mu_g(f(z))\overline{\mu_f(z)} }.$$

If a conformal map of the plane to itself fixes two points, it 
must be the identity.  A $(1+\epsilon)$-quasiconformal map of 
fixing two points must be close to the identity in the 
following sense.

\begin{lemma} \label{R-est}
There is a  $0< k < 1$ and  a $C < \infty$ so that the
following holds. Suppose that $f$ is a quasiconformal
mapping of the plane to itself  that preserves $\uhp$,
fixing $0,1$ and $\infty$ and
the Beltrami coefficient of $f$ is $\mu$ with $\| \mu\|_\infty  \leq k $.
Then
$$  |f(w) - [  w - \frac 1 \pi \int_{\plane} \mu(z)  R(z,w) dxdy ] |
      \leq C \|\mu\|_\infty^2  ,$$
for all $|w| \leq 1$, where
$$ R(z,w) = \frac 1{z-w} - \frac {w}{z-1} + \frac {w-1}{z}
            = \frac {w(w-1)}{z(z-1)(z-w)}.$$
\end{lemma}

This precise statement  Lemma 2.6 \cite{Bishop-QCapp} but is based on 
a similar result in  \cite{Ahlfors-QCbook}, Section V.C. If $|\mu|$ is bounded
by $\epsilon$ then since $R$ is integrable, this says $|f(w) -w| =O(\epsilon)$
as long as $|w|$ is bounded. From this one can deduce

\begin{lemma} \label{QC-near-Id}
Suppose $f: \disk \to \disk$ is $(1+\epsilon)$-quasiconformal 
and that it fixes the boundary points $1,-1,i$. Then 
$\sup_{x \in \circle} |f(x) -x| =O(\epsilon)$.
\end{lemma}

This justifies our claim in the introduction that if we
can approximate the prevertices in the QC-sense, then we 
have also approximated them in the uniform sense. 
We will also  use the estimate:

\begin{lemma} \label{f near 1}
Suppose $f$ is a conformal mapping on $\disk$ such that 
$f(0) =0, $ $f'(0)=1$ and $f$ has a $(1+\epsilon)$-quasiconformal 
extension to the plane fixing $\infty$. Then  for $0 \leq r < 1$, 
$\sup_{|z|\leq r} |f'(z)-1| = O(\frac {\epsilon}{1-r}) .$
\end{lemma}

To prove this, we use equation (10) of Section V.B, \cite{Ahlfors-QCbook},
which implies that  for a $(1+\epsilon)$-QC map $f$ fixing $0$ and 
$\infty$  we have 
$$  |f(z) | \leq O(\epsilon)  + |z|,$$
if $|z|\leq 2$.
The same estimate applied to the inverse of $f$  gives 
$\geq |z| - O(\epsilon)$, so we deduce $|f(z)-z|\leq 
O(\epsilon)$  on the closed unit disk. The Cauchy estimates 
then imply $|f'(z) -1| = O(\epsilon/(1-|z|))$, which implies the lemma.

The following lemma from \cite{Bishop-ExpSullivan} is used in this 
paper to show that certain boundaries can be ``flattened'' with 
small quasiconformal distortion. 

\begin{lemma} \label{comp-K}
Suppose $G: (x,y) \to (x,g(x,y))$ is differentiable. Let 
$$ L_1(G,z) = \liminf_{w\to z}  \frac {|G(z)-G(w)|}{|z-w|},$$
$$ L_2(G,z) = \limsup_{w\to z}  \frac {|G(z)-G(w)|}{|z-w|}.$$
Then
\begin{eqnarray} \label{bilip-comp}
 \frac 12  (\sqrt{Y} - \sqrt{X}) =
  L_1(G,z) \leq L_2(G,z) =
     \frac 12 (\sqrt{Y} + \sqrt{X}),
\end{eqnarray}
where
$$ X =
    1+(g_x)^2+(g_y)^2 - 2|g_y|,
   \qquad
   Y = 1+(g_x)^2 + (g_y)^2 + 2|g_y|.$$
Thus the quasiconformal dilatation of $G$ is
$$ K(z)= \limsup_{r\to 0} \frac{\max_{|z-w|=r} |G(z)-G(w)|}
             {\min_{|z-w|=r} |G(z)-G(w)|}
           =  \frac {(\sqrt{Y} + \sqrt{X})}{(\sqrt{Y} - \sqrt{X})},  $$
and $ \mu = \frac {K-1}{K+1} = \sqrt{X}/\sqrt{Y}$. For a map of the 
form $(x,y) \to ( x + y +g(x,y))$ this becomes $ \mu =  
\sqrt{g_x^2+g_y^2} / \sqrt{g_x^2+(2+g_y)^2}$. 
\end{lemma}

The following, which we also use, follows immediately from the above. 

\begin{lemma} \label{radial map} 
Suppose $W \subset \disk$ is a domain defined by $\{ z= r e^{i \theta}:
 r < g(\theta) \}$,  where $ 1 - \delta \leq g \leq 1$ 
and $|g'| \leq \delta$ for all $\theta$.
Then the map $ z \to z/g(\arg(z))$ is a $1+O(\delta)$ quasiconformal 
map of $W$ to the disk.
\end{lemma}

As noted earlier, any quasiconformal map has a dilatation $\mu$ 
so that $\|\mu\|_\infty < 1$.  Conversely, the ``measurable Riemann
Mapping theorem'' says that given any such $\mu$, there is a
$K$-quasiconformal map $f$ with $\mu =  \Dbar f  / \partial f$.
The Beltrami equation $\Dbar f = \mu \partial f$ can 
be solved using a power series in $\mu$ by setting 
$$ f = P[\mu(h+1)] + z,$$
and  
$$ h = T\mu + T \mu T \mu + T \mu  T \mu  T \mu + \dots ,$$
where $T$ is the Beurling transform 
$$ Th (w) = \lim_{r \to 0} \frac 1 \pi \iint_{|z-w| > r} \frac {h(z)}
{ (z-w)^2} dxdy,$$
and $P$ is the  Cauchy integral 
$$ Ph(w) = - \frac 1 \pi \iint h(z) (\frac 1{z-w} - \frac 1{z}) dxdy.$$ 
Formally, $\Dbar P$ is the identity and $\partial P = T$.
So if we choose $f$ as above, then 
$$\Dbar f = \mu(h+1) ,$$
$$
 \partial f = T \mu (h+1) + 1= h+1.
      $$
Hence, $\Dbar f / \partial f = \mu (1+h) /(1+h) = \mu,$
as desired. To make the argument rigorous requires $L^p$
estimates on these operators as described in \cite{Ahlfors-QCbook}, Chapter V.

Even though they don't have to be differentiable everywhere,
 Mori's theorem states that every $K$-quasiconformal map is H{\"older
continuous of order $1/K$, i.e.,
$$ |f(x) - f(y) | \leq C |x-y|^{1/K}.$$
  Moreover, any quasiconformal
map of $\disk$ to itself extends continuously to the boundary.
We shall discuss these boundary values in more detail below.

%

Numerical computation of quasiconformal maps with given 
dilatation is considered in 
 \cite{Daripa92}, \cite{Daripa93}, \cite{Daripa-Mashat}.


\subsection{Quasi-isometries:}
Quasiconformal maps are a generalization of biLipschitz maps, i.e.,
maps that satisfy
$$ \frac 1K \leq \frac {|f(x)-f(y)|}{|x-y|} \leq K.$$
From the metric definition it is clear that any $K$-biLipschitz 
map is $K^2$-quasiconformal.

For $K$-quasiconformal  self-maps of the disk, there is
almost a converse. Although a   quasiconformal map
$f: \disk \to \disk$ need not be biLipschitz, it is
a quasi-isometry of the disk with its hyperbolic
metric $\rho$, i.e., there
are constants $A,B$ such that
$$ \frac 1A \rho(x,y) -B
      \leq \rho(f(x), f(y))
     \leq A \rho(x,y) + B.$$
This says $f$ is biLipschitz for the hyperbolic
metric  at large scales.  A quasi-isometry is
also called a rough isometry in some sources, e.g., 
 \cite{Holopainen}, \cite{Woess}.
We will say $f$ is a quasi-isometry with constant $\epsilon$ 
if we can take $A = 1+ \epsilon$ and $B = \epsilon$.

 In \cite{EMM-QH+CHB} Epstein, Marden and Markovic show that
any  $K$-quasiconformal selfmap of the disk is a quasi-isometry
respect to the hyperbolic metric with $A=K$ and
$B = K\log 2$ if $1 \leq K \leq 2$ and $B= 2.37(K-1)$
if $K > 2$.
 Note that small circles are 
asymptotically the same for the two metrics, so there is no 
difference between ``hyperbolic-quasiconformal'' and 
``Euclidean-quasiconformal'' maps. There is a difference, 
however, between ``hyperbolic biLipschitz'' and 
``Euclidean biLipschitz''.

\begin{thm} \label{QC-bdy}
 For a map $f: \disk \to \disk$   
we have 
(1) $\Rightarrow$
(2) $\Rightarrow$
(3) $\Rightarrow$
(4)  where 
\begin{enumerate}
\item $f$ is biLipschitz with respect to the hyperbolic metric. 
\item $f$ is  quasiconformal.
\item $f$ is a quasi-isometry with respect to the hyperbolic metric.
\item There is a hyperbolic  biLipschitz map $g: \disk 
\to \disk$ so that $g|_\circle = f|_\circle$.
\end{enumerate}
In other words, the three classes of maps 
(hyperbolic biLipschitz, quasiconformal, hyperbolic quasi-isometry) 
all have the same set of boundary values.
\end{thm}

The boundary extension is a quasisymmetric  homeomorphism, i.e.,
there is an $k < \infty$ (depending only on $K$) so that
$1/k\leq |f(I)|/|f(J)| \leq k$, whenever $I, J \subset \circle$ are
adjacent intervals of equal length. Conversely, any
quasisymmetric homeomorphism of $\circle$ can be extended
to a $K$-quasiconformal selfmap of the disk, where $K$ depends
only on $k$.


\subsection{Trees-of-intervals}

The dyadic intervals of generation $n$   in $\real$ 
are of the form $[2^{-n} j, 2^{-n} (j+1))$  and form 
the vertices of a infinite binary tree (we say $I,J$ are adjacent 
if their lengths differ by a factor of $2$ and one contains 
the other). This property is very useful and in this paper 
we make use of it through the following results.

\begin{lemma} \label{make tree}
Suppose ${\cal D} $ is a disjoint, finite collection of dyadic intervals, 
ordered from left to right, and
all with length $\leq L$.  Let ${\cal C}$ be the collection of 
all  dyadic intervals of length $\leq L$ which contain some element 
of ${\cal D}$. If ${\cal C}$ has $m$ elements, it can be enumerated in 
time $O(m)$. 
\end{lemma}

\begin{proof}
Start with the leftmost element $I_1$  of ${\cal D}$ and form the list   $\gamma_1$, of
nested increasing intervals until reaching size $L$. Then move the 
second element $I_2$  of ${\cal D}$, and form the nested, increasing list
until we hit an element of $\gamma_1$. Form $\gamma_2$ by cutting
$\gamma_1$ at this point and replacing the bottom potion by the 
list  starting at $I_2$. Continue moving to the right. When we have 
finished, all $m$ elements of ${\cal C}$ has been found and only $O(m)$
work has been done. 
\end{proof} 

More generally, we define a tree-of-intervals as a collection of 
half-open intervals ${\cal I}$ which contains a maximal element (an interval 
containing all the others) and say $I$ is a child of $J$ if $I\subset J$
and $I$ is maximal in ${\cal I}$ with this property. The tree has degree 
$d$ if every interval has at most $d$ children.
We will say such a tree is complete if each interval is either a 
leaf of the tree or is the union of its children.

\begin{lemma} \label{make cover}
Suppose ${\cal I} = \{ I_j\}_1^n$ is a  tree-of-intervals 
 of degree $d$  and 
$ {\cal J} = \{ J_j\}_1^m$   is a complete binary  tree-of-intervals.
Assume the root of ${\cal J}$ contains the root of ${\cal I}$.
Then for every  element of ${\cal I}$ we can find the minimum element of 
${\cal J}$ containing it  in total time $O(d\cdot n+m))$.
\end{lemma} 

\begin{proof}
This is clearly true if either $n=1$ or $m=1$.  Consider the case 
$n>1$ and $m >1$ and suppose the lemma holds for all pairs where
both coordinates are strictly smaller. 
Let $I_0$, $J_0$ be the largest elements of 
${\cal I}$ and ${\cal J}$ respectively. Clearly $I_0 
\subset J_0$. Check the children of $J_0$ to see if either 
contains $I_0$. Continue this way, creating a path in the ${\cal J}$
tree until we either reach (1) a leaf $J$  of ${\cal J}$ containing 
$I_0$ or (2) an interval $J$ neither of whose children contain 
$I_0$. This takes work $C_1 g$ if $J$ is $g$ levels below $J_0$, since
we only have to do a bounded amount of work per level.

 In  case (1), we assign $J $ to every element of 
${\cal I}$ and we are done. In case (2), we assign 
$J$ to $I_0$  and to every descendent of $I_0$ which contains 
the dividing point of $J$ (the common endpoint it's children
$J_1, J_2$). These descendent form a path  $\gamma$ in the ${\cal I} $
tree and if there are $p$ such descendants we can find them all in 
time $C_2 d p$ since we only have to check  $d$ children at each stage.

Now remove $\gamma$ from the tree ${\cal I}$. Each connected component 
of what remains has its intervals either all in $J_1$ or all in $J_2$, 
and thus satisfies the hypothesis of the lemma with respect to the
one of the subtrees ${\cal J}_1, {\cal J}_2$ 
 of ${\cal J}$ rooted at these points. By induction 
the assignments can be done for each  component ${\cal I}_k $ in time 
$ C_3 d \cdot  i_k \cdot  j_k$ 
where $i_k$ is the number of vertices in the ${\cal I}_k$ 
and $j_k$ is the number of vertices in the choice of ${\cal J}_1$ or ${\cal J} _2$ 
covering this component. Summing over all the components gives 
$ C_3 (d (n-p)+ (m-g))$. Added to the $C_1 g + C_2  dp $ already done, this 
proves the lemma.
\end{proof} 

This seems closely related to merging heaps in computer science
(a heap is a tree whose vertices are labeled by numbers so that the children's 
labels are less than the parents), although here the labels are intervals and
the ordering is set inclusion.

In Section \ref{really thick}, following Lemma \ref{compose rep FB}
 we  claimed that in linear time we 
could find gaps or crescents that were a bounded hyperbolic 
distance from  the image of each Whitney box. 
Here we give a few more details. 
Associate to each Whitney square,  $Q_j$,  its base
$I_j \subset \reals $. Then map $I_j$ to $\partial \Omega_k$ by 
the map $f$ and back to $\reals$ by $\varphi_k$ (the iota map for
$\Omega_k$). Let $W_j$ be the image of $Q_j$ under these two maps. 
Since $\varphi_k$ is a quasi-isometry, $W_j$ is a bounded hyperbolic 
distance from the Whitney box corresponding to the image $K_j$ of $I_j$.
 The $\{ K_j\}$  still form a tree-of-intervals which we denote 
${\cal K}$. Let ${\cal J}$ be the complete binary tree-of-intervals
consisting of the bases of all the bending lamination geodesics.
If $J \in {\cal J}$ is the minimal element containing $K \in {\cal K}$, 
then the ideal triangle (or crescent if $J$ is a leaf) bounded above 
by the geodesic corresponding to $J$ must hit the Whitney box  corresponding 
to $K$ and thus is a bounded hyperbolic distance from $W_j$. This 
is what we wanted. Thus finding the claimed gaps/crescents, reduces 
to an application of Lemma \ref{make cover}.


\section{Fast  power series manipulations (or, $O(n \log n)$ suffices)}
\label{fast-power-series}

In order to prove  
Theorems \ref{thmQC} or  \ref{main} with $C(\epsilon) = O(\log \frac 1 \epsilon 
\log \log \frac 1 \epsilon)$ we can only use operations on 
power series of length $p \sim \log \frac 1 \epsilon$ that take
time at most $ O(p \log p)$. In this section we review 
some basic  results about  power series manipulations 
and check that all the operations we need  can be carried out this 
quickly. If one uses naive manipulations of power series, then one 
simply gets Theorem \ref{main} with a larger constant, e.g., 
$C(\epsilon) = O(\log^c \frac 1\epsilon)$ for some $c$.
To conform with the references, we replace $p$ by $n$; 
from this point 
on $n$ will refer to the number of terms in a power series
(rather than the number of vertices in a polygon, as it did 
in earlier sections).
This  summary is  taken mostly from \cite{Tang} and  \cite{vanderHoeven02}.

The Fourier matrix  is given by 
\begin{eqnarray*}
F_n =\begin{pmatrix}
    1  &  1   &   1  &   \dots  &   1 \\
    1  &  \omega   & \omega^2    &   \dots  &  \omega^{n-1}  \\
    1  &  \omega^2   &  \omega^4   &   \dots  &  \omega^{2(n-1)}  \\
  \vdots    &  \vdots    &  \vdots   &   \ddots  &  \vdots    \\
    1  &  \omega^{n-1}   & \omega^{2(n-1)}     &   \dots  &  \omega^{(n-1)(n-2)}  \\
\end{pmatrix}
\end{eqnarray*} 
where $\omega$ is an $n$th root of unity.  The fast Fourier 
transform (FFT)  applies  $F_n$ to an $n$-vector 
in time $O(n \log n)$ \cite{Cooley-Tukey}.
$F_n$ is unitary (after rescaling) and its conjugate transpose, $F^*$, 
can also be applied in $O(n \log n)$ time.
The discrete Fourier transform (DFT) takes a $n$-long 
sequence of complex numbers $\{ a_k\}_0^{n-1}$ and a
$n$-root of unity $\omega$ and returns the values of 
the  polynomial 
$p(z) = a_0 + a_1 x + \dots a_{n-1} z^{n-1}$
at the points $z =\{ 1, \omega, \omega^2, \dots, 
\omega^{n-1} \}$. Composing DFT with its adjoint returns 
the original sequence times $n$.

Suppose $ f(z) = \sum_{k=0}^n a_k z^k $ and $g(z) = \sum_{k=0}^n 
b_k z^k$.  How fast can we multiply, divide or compose these 
series? Let $M(n)$ denote the number of field operations 
it takes to multiply two power series of length $n$.  The 
usual process of convolving the coefficients shows 
$M(n) =O(n^2)$. A divide and conquer method of
Karatsuba and Ofman \cite{Karatsuba-Ofman} improves this to 
$O(n^\alpha)$ with $\alpha = \log 3 / \log 2$, but the 
fastest known method uses the Fast Fourier Transform 
\cite{Cooley-Tukey}, which 
shows $M(n) = O(n \log n )$
(two power series 
of length $n$ can be multiplied by taking the DFT of each, 
multiplying the results term-by-term,  taking the DFT 
of the result and finally dividing by $n$).

Other operations on power series are generally estimated 
in terms of $M(n)$. For example,  inversion (finding the  
 reciprocal  power series, $1/f$, given the series for $f$) 
is $O(M(n))$. Like several other operations on power series, 
this is most easily proven using Newton's method (applied 
to series rather than numbers). For example, $1/f$ is the 
solution of the equation $\frac 1g - f =0$. If $g_k$ is 
an approximate solution with $n>0$ terms correct, 
then 
$$ g_{k+1} = g_k - \frac { \frac 1g_k -f}{-1/g_k^2} = g_k - \frac {fg_k-1}{z^n} g_k z^n,$$
has $2 n$ correct terms. The right side requires two 
multiplications and so the work to compute inversions is 
$O(M(n)) + O(M(n/2)) + \dots +O(1) = O(M(n))$.

Given inversion, one can divide power series (multiply $f$ by $1/g$) 
compute $\log f$ (integrate $f'/f$ term-by-term) or $\exp(f) $
(solve $\log g =f$ by Newton's method) all in time $O(M(n))$.

Composition  of power series is a little harder.  Brent and Kung showed 
that given power series $f,g$ of order $n$ and $g_0 =0$, the 
composition $f \circ g$ can be computed in time 
${\rm{Comp}}(n)  = O(\sqrt{n \log n} M(n))$.  Using FFT multiplication, this 
gives $O(n^{3/2} \log^{3/2} n \log \log n)$. They also showed
that reversion (i.e., given $f$ find $g$ so $f \circ g(z) = z$) 
can be solved using Newton's method with the 
iteration 
$$ g \to g-\frac{f\circ g}{f' \circ g},$$
which doubles the number of correct terms in $g$ with every step.
Thus ${\rm{Rev}}(n) =O({\rm{Comp}}(n))=O(\sqrt{n \log n} M(n))$.

Fortunately, there are some special cases when composition is faster.
For example, if we want to post-compose with a linear fractional 
transformation $\sigma(z)=(az+b)/(cz+d)$, this is the same as adding and 
dividing series, so is only $O(M(n))$. 

Pre-composing 
by $\sigma$ is more complicated.  A function $f$ is
called algebraic if it satisfies
$$ P_d(z) f(z)^d + \dots + P_0(z) =0,$$
for some polynomials $P_0, \dotsm P_d$. Clearly every rational 
function is algebraic with $d=1$. 
The power series of algebraic functions satisfy linear recursions and
$n$ terms of the series can be computed in $O(n)$.  Moreover, 
pre-composition by algebraic functions is fast; if $f$ has $p$ terms, 
$g$ has $q$ terms and is algebraic 
of degree $d$ then the first $n$ terms of $f \circ g$ can be computed
in time $O(q d^2 \frac{p(q-v)}{n} M(n+pv) \log n)$
where $v$ is the valuation of $P_d$ (the largest power of $z$ that 
divides $P_d(z)$) and $q$ is the maximum of the 
degrees of $P_i$, plus $1$. For linear fractional transformations 
$v=0$ and $q=2$ so the time to pre-compose by such a map 
is $O(M(n) \log n)=O(n\log^2 n)$. (There is an extensive
generalization of the algebraic case to fast manipulations of 
holonomic functions, as developed by van der Hoeven
\cite{vanderHoeven99}, although 
we do not need to use it here.)

This is too slow for our purposes.  Fortunately, the only
times we will have to pre-compose with a M{\"o}bius transformation 
correspond to various manipulations of power and Laurent series
in the fast multipole method 
and all of these can be accomplished in $O(n \log n)$  by 
fast application of Toeplitz, Hankel and Pascal matrices as 
shown by Tang in  \cite{Tang} (the following discussion is based 
on \cite{Tang}).

 A matrix is called circulant if each 
column is a down-shift of the previous one, is called Toeplitz if it 
is constant on diagonals (slope $-1$)  and called Hankel if it is constant on 
antidiagonals (slope $1$). The general forms of these three types are: 
\begin{eqnarray*}
C(x) =\begin{pmatrix}
   x_1   &  x_n   & x_{n-1}    &   \dots  & x_2   \\
   x_2   &  x_1   & x_{n}    &   \dots  & x_3  \\
   x_3   &  x_2   &  x_{1}   &   \dots  &  x_4  \\
  \vdots    &  \vdots    &  \vdots   &   \ddots  &  \vdots    \\
    x_n  &  x_{n-1}   & x_{n-2}    &   \dots  &  x_1  \\
\end{pmatrix}
,\\
T_(x) = \begin{pmatrix}
   x_0   &  x_1   &  x_2   &   \dots  & x_{n-1}   \\
   x_{-1}   &  x_0   &  x_1   &   \dots  & x_{n-2}   \\
    x_{-2}  &  x_{-1}   &  x_0   &   \dots  & x_{n-3}   \\
  \vdots    &  \vdots    &  \vdots   &   \ddots  &  \vdots    \\
   x_{-n+1}   &  x_{-n+2}   & x_{-n+3}    &   \dots  & x_0   \\
\end{pmatrix}
, \\
H(x) = \begin{pmatrix}
  x_{-n+1}    &  x_{-n+2}   &  x_{-n+3}     &   \dots  & x_0   \\
  x_{-n+2}    &   x_{-n+3}    &  x_{-n+4}   &   \dots  & x_1   \\
   x_{-n+3}   &  x_{-n+4}   &  x_{-n+5}   &   \dots  & x_2   \\
  \vdots    &  \vdots    &  \vdots   &   \ddots  &  \vdots    \\
    x_0  &  x_1   & x_2    &   \dots  &  x_{n-1}  \\
\end{pmatrix}
\end{eqnarray*} 
A circulant matrix can be applied to a vector 
using three applications of FFT, i.e., 
because $ C_n(x)$ applied to a vector $y$ is the 
same as  $\IFFT(\FFT(x) \cdot \FFT(y))$.
A Toeplitz matrix can be embedded in a circulant matrix 
of the form 
\begin{eqnarray*}
C_{2n} =\begin{pmatrix}
    T_n & S_n \\
    S_n & T_n \\
\end{pmatrix}
\end{eqnarray*} 
where 
\begin{eqnarray*}
S_n= \begin{pmatrix}
    0  &  x_{-n+1}   &  x_{-n+2}  &   \dots  &  x_{-1}  \\
    x_{n-1}  &  0   & x_{-n+1}    &   \dots  & x_{-2}   \\
   x_{n-2}   &     & 0    &   \dots  &  x_{-3}  \\
  \vdots    &  \vdots    &  \vdots   &   \ddots  &  \vdots    \\
  x_1   &  x_2   &  x_3   &   \dots  & 0   \\
\end{pmatrix}
\end{eqnarray*} 
To apply $T$ to an $n$-vector $y$, append $n$ zeros to $y$ 
to get a $2n$-vector, apply $C_n$ and take the first $n$ 
coordinates of the result.
This takes  $O(n \log n)$ time.
If $H$ is a Hankel  matrix then 
$  R\cdot H$ is a Toeplitz matrix where $R$ is the permutation 
matrix that is $1$'s on the main anti-diagonal and $0$ elsewhere, 
i.e., it reverses the order of the coordinates of a vector. 
Thus $ H = R \cdot( R \cdot H)$, is a Toeplitz matrix followed 
by a permutation and can clearly be applied in time $O(n \log n)$
as well. 

The Pascal matrix is lower triangular with its $(j,k)$th entry 
being the binomial coefficient $ C^j_i=\binom{i}{j}$.
\begin{eqnarray*}
\begin{pmatrix}
   1   &  0   &   0  &   \dots  &  0  \\
   1   &  1  &    0 &   \dots  &  0  \\
    1  &   2  &   1  &   \dots  &  0  \\
  \vdots    &  \vdots    &  \vdots   &   \ddots  &  \vdots    \\
    C^{0}_{n-1}  &   C^{1}_{n-1}   &   C^{2}_{n-1}   &   \dots  & C^{n-1}_{n-1}   \\
\end{pmatrix}
\end{eqnarray*} 
This matrix can be written as $P= {\rm{diag}}(v_1) \cdot T \cdot {\rm{diag}}(v_2)$
where 
$$ v_1= ( 1,1,2!, 3!, \dots, (n-1)!), $$
$v_2= \frac 1{v_1}$ (term-wise) and $T$ is the Toeplitz matrix
\begin{eqnarray*}
T=\begin{pmatrix}
   1   &  0   &   0  &   \dots  &  0  \\
   1   &  1  &    0 &   \dots  &  0  \\
   \frac  1{2!}  &   1  &   1 &   \dots  &  0  \\
  \vdots    &  \vdots    &  \vdots   &   \ddots  &  \vdots    \\
   \frac1{(n-1)! } &   \frac 1{(n-2)!}   &   \frac 1{(n-3)!}   &   \dots  &  1   \\
\end{pmatrix}
\end{eqnarray*} 
The diagonal matrices can be applied in $O(n)$ and the Toeplitz in 
$O(n \log n)$ and hence so can $P$. Similarly for the
transpose of $P$.

Now for the applications to fast multipole translation 
operators.
There are three types of conversions to consider.
First, local to local translation
$$\sum_{k=0}^{n-1} a_k (z-a)^k \quad \to \quad \sum_{k=0}^{n-1} b_k(z-b)^k, $$
then multipole to local 
$$\sum_{k=0}^n a_k (z-a)^{-k} \quad \to \quad \sum_{k=0}^n b_k(z-b)^{k}, $$
and finally, multipole to multipole, 
$$ \sum_{k=0}^n a_k (z-a)^{-k} \quad \to \quad 
\sum_{k=0}^n b_k(z-b)^{-k}.$$

Let $c=b-a$ and consider the local-to-local translation. We have 
$$ \sum_{k=0}^{n-1} a_k (w-c)^k
     =\sum_{k=0}^{n-1} a_k \sum_{j=0}^k w^{j} (-c)^{k-j} \binom{k}{j},$$
so the matrix corresponding to local translation has $k$th column
$$ ( (-c)^{k} , (-c)^{k-1} \binom{k}{1}, \dots, (-c)^{0} \binom{k}{k}, 0, \dots, 0  )^t$$
or 
\begin{eqnarray*}
LL=\begin{pmatrix}
 1 &  -c   &  c^2  &   \dots  &  (-c)^{n-1} \\
 0   &  1   &  -2c   &   \dots  & (-c)^{n-2} C_{n-1}^{1}   \\
 0   &  0  &  1  &   \dots  &  (-c)^{n-3} C_{n-1}^{2}  \\
  \vdots    &  \vdots    &  \vdots   &   \ddots  &  \vdots    \\
  0  &  0  &  0  &   \dots &  1  \\
\end{pmatrix}
\end{eqnarray*} 
This  matrix is equal to 
$$\diag(1,-c,\dots, (-c)^{n-1}) \cdot P' \cdot 
\diag(1, -c^{-1}, \dots (-c)^{-n+1}),$$
where $P'$ is the transpose of $P$.
The diagonal matrices can be applied in $O(n)$ time and $P'$ can 
be applied in $O(n \log n)$. Thus local-to-local translations 
can be done this fast.

Similarly, the multipole-to-multipole and multipole-to-local 
transformations correspond to applying the matrices 
\begin{eqnarray*}
 MM &=&\begin{pmatrix}
 1 &   0  &  0  &   \dots  &  0 \\
 \binom{1}{1} c   &   1  &  0   &   \dots  &   0 \\
   \binom{2}{2} c^2  &   \binom{2}{1} c  &  1  &   \dots  &    \\
  \vdots    &  \vdots    &  \vdots   &   \ddots  &  \vdots    \\
     \binom{n-1}{n-1} c^{n-1}   &    \binom{n-1}{n-2} c^{n-2} 
   &    \binom{n-1}{n-3} c^{n-3}  
              &   \dots  &    1  \\
\end{pmatrix}
 \end{eqnarray*}
\begin{eqnarray*}
ML &=& \begin{pmatrix}
 -c^{-1}  &  c^{-2}   &  c^{-3}  &   \dots  & c^{-n+1}  \\
  -c^{-2}  &  2 c^{-3}   &  -3 c^{-4}   &   \dots  &    \\
   -c^{-3}   &  3 c^{-4}  &  -6 c^{-5}  &   \dots  &    \\
  \vdots    &  \vdots    &  \vdots   &   \ddots  &  \vdots    \\
    -c^{-n+1}  &  (n-1) c^{-n}  & -\binom{n}{2} c^{-n-1}   &   \dots & (-1^{n-1}
                          \binom{2p-2, p-1} c^{-2n-1}   \\
\end{pmatrix}.
\end{eqnarray*} 
We can rewrite these matrices as 
$$ MM =\diag(1,c,\dots, c^{n-1}) \cdot P \cdot 
\diag(1, c^{-1}, \dots c^{-n+1}), $$
$$ ML= \diag(1,c^{-1},\dots, c^{1-n})
 \cdot P \cdot P' \cdot 
\diag( -c^{-1},  c^{-2} \dots (-c)^{-n}), $$
where  $P'$ is the transpose of $P$.
As with local translations, these are compositions of 
diagonal matrices (which  can be applied in $O(n)$) and 
matrices that can be applied in $O(n \log n)$
time.

We will also use structured matrices to compute  
expansions around $\infty$ of functions of the 
form $\int \frac {d \mu(z)}{(z-w)^k}$, $k=1,2,3$.  We will only 
consider the  Cauchy transform ($k=1$) since the 
others can be obtained by term-by-term differentiation 
of that one.
Suppose 
$f(z) = \sum_{k=0}^n  a_k z^k$ is a power
series for an analytic function, bounded 
by one and defined on $\disk$ and 
$\varphi (x,y)$ is a polynomial in $x$ and 
$y$ of uniformly bounded degree. Then 
the Cauchy transform 
$$ F(w)  = \int_S \frac {f(z) \varphi(x,y) dxdy}{z-w},$$
is analytic in $w$ outside $S = [-\frac 12, \frac 12]^2$, 
so has an expansion $ F(w) = \sum_{k=1}^\infty b_n w^{-n}$.
Given $\{a_k\}_0^n$, thinking of $\varphi$ as fixed, we 
want to compute $\{b_k\}_1^n$.  For each monomial 
of the form $z^k x^a y^b$ we can precompute the 
expansion using explicit formulas  ($O(n)$ for each 
of $O(n)$ monomials) and then we simply 
apply the resulting matrix to the vector $\{a_k\}$.
Naively, we can do this in  time $O(n^2)$.

Actually we can compute the expansion in
only $O(n \log n)$. Let $d \mu = x^a y^b dxdy $ 
restricted to $Q = [0,1]^2$. We want to 
compute the expansion at $\infty$ of
\begin{eqnarray*}
F(w) = \iint \frac {z^n}{w-z} d \mu(z)
&=&\iint z^n \frac 1w (1+\frac zw +(\frac zw )^2 + \dots ) d \mu (z) \\
&=& \sum_{k=0}^\infty w^{-k-1} \iint z^{n+k} d \mu (z) \\
& =& \sum_{k=1}^\infty a_{k,n} w^{-k} ,
\end{eqnarray*}
where 
$$a_{k,n} = c(n+k+1,a,b)=\iint_Q (x+iy)^{n+k-1} x^a y^b dx dy.$$
 Since $a_{k,n}$ only 
depends on $k+n$, $A$ is a Hankel matrix.  As noted above,  a $n \times n$
Hankel matrix can be applied to a $n$-vector using FFT in time 
$O(n \log n)$.

The individual coefficients  have explicit formulas involving Euler's 
Beta function. Evaluations for a few small values of $a,b$  (as given 
by {\it Mathematica} are)
$$ c(n,0,0) = \frac {i -i^{n+1} + 2(1+i)^n}{2+3n + n^2},$$
$$ c(n,1,0) =  \frac {2i - i^n + in + 2(1+i)^n((2-i)+n)}
                     { (1+n)(2+n)(3+n)},$$
$$ c(n,2,0) =  \frac {i(6+2i^n + 5n + n^2) + 2(1+i)^n ((4-4i)+n((5-2i)+n))}
                     {(1+n)(2+n)(3+n)(4+n)}$$
$$ c(n,1,1) = -\frac {1+i^n-2(1+i)^n(2+n)}
                     {(1+n)(2+n)(4+n)},$$
$$ c(n,2,1) =  \frac{-3(2i)i^n - n + 2(1+i)^n(1+n)((4-i)+n)}
                    {(1+n)(2+n)(3+n)(5+n)},$$
$$ c(n,0,1) =   \frac{-1-i^{n+1}(2+n) + 2(1+i)^n((2+i)+n)}
                     {(1+n)(2+n)(3+n)}.$$

Thus $n$-term Laurent expansions for 
Beurling transforms of the appropriate 
degree $n$ polynomials  can be 
computed in time $O(n \log n)$.

In Section \ref{epsilon-reps}   we claimed that $\epsilon$-representations
 and partial representations could be computed from each 
 other quickly.
We can now see why this is true. Given a 
$O(n)$-term 
Laurent expansion on an annulus we can clearly create $O(n)$-term
power series expansions that approximate it on disks contained 
in the annulus  in time $O(n \log n)$, just as with the 
multipole to local conversions discussed before. If the double of the
disk is contained in the annulus where the function is bounded
by $1$, then the convergence of the power series is geometric and 
$O(n)$ terms give accuracy of order $\lambda^n$ for some 
$\lambda < 1$. Thus we can do the conversion from representations 
to partial representations.

To go the other direction, we need to compute the coefficients 
of a Laurent expansion for $f$  from knowing power series 
approximations for $f$ on disks that cover a 
contour $\gamma$ around the origin.
The coefficients can be computed exactly as integrals of the 
form $\int_\gamma f(z) z^{k} dz $. This integral can be 
broken into pieces $\int_{\gamma_n}$ where $\gamma = \cup_n 
\gamma_n$ is a decomposition of $\gamma$ into pieces, where 
each piece stays inside one of the disks where we have a 
power series approximation for $f$. If this approximation is
of the form $\sum_j a_j (z-a)^j$, then we wish to evaluate integrals
of the form  
$$ \int_{\gamma_n} (\sum_j a_j (z-a)^j) z^k dz .$$
We can convert $z^k$ to a series with center $a$, e.g., 
$z^k = \sum_q b_{q,k} (z-a)^q$ just as before with Pascal matrices
and then we have to evaluate 
\begin{eqnarray*}
\int_{\gamma_n} (\sum_j a_j (z-a)^j) (\sum_q b_{q,k} (z-a)^q) dz 
  &=& 
  \sum_q \sum_j  a_k   b_{q,k}  \int_{\gamma_n} (z-a)^{j+q} dz  \\
&=&   \sum_q I_{j,k} \sum_j     b_{q,k} a_k    \\
  \end{eqnarray*}
  where $b_{q,k}$ is a Pascal matrix and
$I_{j,q} = \int_{\gamma_n} (z-a)^{j+q} dz$ is an explicit
  Hankel matrix. We have seen above that both types of matrix can 
  be applied in $O(n \log n)$. So that the contour integral 
  using the power series on each piece of the contour can be 
  evaluated this fast.  If the power series approximations  
  agree with the Laurent series to within $\epsilon$ then 
  resulting coefficients will be accurate up to an error 
  of $\int_\gamma \epsilon|d z| = O(|\gamma| \epsilon)$.
  Normalizing so that $\gamma$ and its image both have length 
  about $1$, we see that the coefficients are accurate to 
  $O(\epsilon)$ and so the reconstructed Laurent series agrees
  with the original to within $\sum_k \epsilon = O(n \epsilon)
  = O(\epsilon \log \epsilon) =O(\epsilon^{1-\beta})$ with 
  $\beta$ as close to $0$ as we wish (taking a larger constant 
   in front).

The compositions of power series  considered earlier 
were exact computations in the sense that 
given two $n$-term power series, $f$ and  $g$, we 
are computing the exact first $n$ coefficients of  $f \circ g$.
However, they are inexact in the sense that we are
 truncating up to $n^2 -n$ terms of the full composition.
We will be most interested in the case when $f$ and 
$g$ are conformal maps whose power series coefficients 
decay exponentially. 
Then the  $n$-term truncation of $f \circ g$ equals $f \circ g$ up to 
an error of $e^{- c n}$. Since the truncation introduces an 
error anyway, we may as well tolerate an error of the same 
size coming from an 
inexact computation of the first $n$ terms that we keep. 
Can we approximate the first $n$ terms of $f  \circ g$ 
faster than we can compute them exactly (i.e., faster
than $O(n^{3/2} \log^{3/2}  n \log \log n)$)? The answer is 
yes, at least in the special case that we care about.

\begin{lemma} \label{inexact-comp}
Suppose $f$ and $g$ are conformal maps of $D(0, R)$, $R \geq 2$
such that $f$
has a $(1+\epsilon)$-quasiconformal extension to $\reals^2$ that 
fixes $\infty$.
Assume also that  $f(0)=g(0)$ and $f'(0) = g'(0) =1$.
Then given the first $n$ terms of the power series for 
$f$ and $g$ we can compute in time $O(n \log n)$, an $O(n)$ term power series 
$h$  so that $|h-g \circ f^{-1}| \leq  O(\epsilon^{2(1-\beta)})$ on 
$ D(0, 1)$.
\end{lemma}

\begin{proof}
Let $\{ \omega_j\}_0^{n-1}$ denote the $n$th roots of unity.
Using the FFT we can compute the images of all these points under 
either $f$ or $g$ in time $O(n \log n)$. 
By Lemma \ref{f near 1}, 
$|f'-1|=O(\epsilon)$ on $D(0,2)$ and hence if $|z-\omega_j|=O(\epsilon)$, 
$$ f(z ) = f(\omega_j) + (1+O(\epsilon))(z-\omega_j) + O(|z-\omega_j|^2).$$
 In particular, if $w = f(z)$, 
\begin{eqnarray*}
 f^{-1}( w) &=& w +(\omega_j -f(\omega_j))
              + O(\epsilon |z-\omega_j|) +O(|z-\omega_j|^2)\\
              &=&w +(\omega_j -f(\omega_j))+O(\epsilon^2).
\end{eqnarray*}
Thus taking $w = \omega_j$, 
$$ f^{-1}( \omega_j)  =\omega_j +(\omega_j -f(\omega_j))+O(\epsilon^2).$$
Therefore, we  get
\begin{eqnarray*}
 g \circ f^{-1}(\omega_j) &=& g(\omega_j) +g'(\omega_j)( \omega_j - f(\omega_j)) +
                    O(|\omega_j - f(\omega_j)|^2)+
                    O(\epsilon^2) \\
&=& g(\omega_j) + g'(\omega_j)(\omega_j - f(\omega_j)) +
                                       O(\epsilon^2) .
\end{eqnarray*}
So define  a degree $n$ polynomial by setting its values
on the roots of unity using 
$$h(\omega_j) = g(\omega_j) + g'(\omega_j)(\omega_j - f(\omega_j)).$$
Since we can compute and evaluate $g$ and  $g'$ at the roots of unity in 
$O(n \log n)$, we can find the series expansion for $h$ 
around $0$ in time $O(n \log n)$ using an FFT.

Clearly $h$ is a good approximation to $g \circ f^{-1}$ on the 
$n$th roots of unity (they differ by $O(\epsilon^2)$ on these 
points). What about the rest of $\disk$? We shall show that the 
functions differ by at most $O(\epsilon^2)$ on $\disk$ as well.
Let 
$$ \sigma(z) = \prod_{j=0}^{n-1} (z- \omega_j) = z^n -1.$$
Let 
$$ \psi(z) = \sum_{k=0}^{n-1} \frac {h(\omega_j)/(n\omega_j^{n-1})}
{z - \omega_j}.$$
The zeros of $\sigma$ cancel the simple poles of $\psi$, 
so $P=\psi \cdot \sigma$ is entire and since $|\sigma| \sim |z|^n$
and $|\psi| \sim |z|^{-1}$ near $\infty$, $P$ must 
be a polynomial of degree $n-1$. Moreover, by l'Hopital's rule
\begin{eqnarray*}
\frac {h(\omega_j)}{n\omega_j^{n-1}}
         &=& \lim_{z \to \omega_j} (z-\omega_j)\psi(z)   \\
         &=&  \lim_{z \to \omega_j} (z-\omega_j)\frac{P(z)}{\sigma(z)} \\
         &=&  \frac {1\cdot P(\omega_j) + 
                              (\omega_j-\omega_j)P'(\omega_j)}
                                  {\sigma'(\omega_j)}
         =  \frac {P(\omega_j)}{\sigma'(\omega_j)},
\end{eqnarray*}
which implies $P(\omega_j) = h(\omega_j)$ for all $j=0, \dots , n-1$, 
and this means $P=h$ as polynomials.
Thus 
\begin{eqnarray*}
 |h(z)|=|P(z)| &\leq& |\psi(z)| \cdot |\sigma(z)|\\
      &\leq & (\max_j |h(\omega_j)| )\sum_{j=0}^{n-1} \frac{1/n}{|z-\omega_j|}
                      \cdot (|z^n|+1) \\
      &\leq&  4 M,
\end{eqnarray*}
if $|z| \leq 1/2 $ and $M = \max_j |h(\omega_j)|$.

Suppose $ 1 < r < R/2$.
The  map $g \circ f^{-1}$ is conformal on $D(0, R/2)$ (if 
$\epsilon$ is small enough), so its power series at $0$ 
has terms that  decay like $(2/R)^k$. If $H$ is the $m$-term
truncation of this power series and we choose 
$m \sim n$ large enough, then $|H - g\circ f^{-1}| 
\leq O((2/R)^m)= O(\epsilon^2)$ on $\disk$. Since $H$ is a degree $n-1$
polynomial that interpolates its own values on the roots of 
unity,  $H-h$ is a degree $n-1$ polynomial with 
 values $=O((2/R)^n)$ on the $n$th roots of unity and 
hence, by our calculation above applied to $H-h$, is bounded 
by $O(\epsilon^2)$
on $\disk$.
By Lemma \ref{expand-est}, this means that a truncation $h_t$ of $h$ satisfies 
$$ |h_t(z) - g\circ f^{-1}(z) | =O(\epsilon^{2(1-\beta)}),$$
on $D(0,r)$ where $r= R^\beta$.
\end{proof}

Exact evaluation of a degree $n$ polynomial at $m $ points 
takes $O((n+m) \log^2(n+m))$  \cite{ASU75} and recovering a degree
$n-1$ polynomial from its values at $n$ points takes
$O(n \log^2 n)$, so our approximate method is faster
for the cases we consider.   Multipole methods can 
also be used to give faster approximate evaluation 
and interpolation algorithms. See \cite{DGR96}, 
\cite{Reif}.  These are $O(n \log \frac 1 \epsilon)$
where $n$ is the degree and $\epsilon$ is the desired 
accuracy. In our case, however, $n \sim \log  \frac 1 \epsilon$, 
so this is not faster than exact calculation.

\bibliography{time}

\bibliographystyle{plain}
\end{document}